\documentclass[reqno]{memo-l}

\usepackage{graphicx, float, array, url, mdframed, multicol}
\usepackage{amssymb,mathtools,mathdots,latexsym,stmaryrd,textcomp,gensymb}
\usepackage[shortcuts]{extdash}
\usepackage{tikz}
\usetikzlibrary{arrows, intersections, decorations, fit, shapes}

\usepackage[shortlabels]{enumitem}
\setlist[enumerate]{label={\upshape (\roman*)}}

\theoremstyle{definition}
\newtheorem{definitionx}{Definition}[section]
\newtheorem*{definitionx*}{Definition}
\newtheorem{remarkx}[definitionx]{Remark}
\newtheorem*{remarkx*}{Remark}
\newtheorem{remintro}{Remark}
\newtheorem{examplex}[definitionx]{Example}
\newtheorem*{examplex*}{Example}
\newtheorem{notationx}[definitionx]{Notation}
\newtheorem*{notationx*}{Notation}
\theoremstyle{plain}

\newtheorem*{conjecturex*}{Conjecture}
\newtheorem{theoremx}[definitionx]{Theorem}
\newtheorem*{theoremx*}{Theorem}
\newtheorem{thmintro}{Theorem}
\newtheorem{thmchap}{Theorem}[chapter]
\newtheorem{propositionx}[definitionx]{Proposition}
\newtheorem*{propositionx*}{Proposition}
\newtheorem{lemmax}[definitionx]{Lemma}
\newtheorem*{lemmax*}{Lemma}
\newtheorem{corollaryx}[definitionx]{Corollary}
\newtheorem*{corollaryx*}{Corollary}

\numberwithin{section}{chapter}
\numberwithin{equation}{chapter}
\numberwithin{table}{chapter}

\newenvironment{shbox}{\begin{mdframed}}{\end{mdframed}}
\newcommand{\notacont}[1]{\hspace{0.58cm}\textsc{Notation~#1. (continued) } }

\newcommand{\A}{\mathcal{A}}
\newcommand{\B}{\mathcal{B}}
\newcommand{\C}{\mathcal{C}}
\newcommand{\M}{\mathcal{M}}
\newcommand{\N}{\mathcal{N}}
\renewcommand{\S}{\mathcal{S}}
\newcommand{\T}{\mathcal{T}}

\renewcommand{\a}{\alpha}
\renewcommand{\b}{\beta}
\newcommand{\g}{\gamma}
\renewcommand{\d}{\delta}
\newcommand{\e}{\varepsilon}

\renewcommand{\th}{\theta}
\renewcommand{\i}{\iota}
\renewcommand{\l}{\lambda}
\renewcommand{\r}{\rho}
\newcommand{\wrho}{\widetilde{\rho}}
\newcommand{\s}{\sigma}
\newcommand{\ws}{\widetilde{\sigma}}
\renewcommand{\t}{\tau}

\newcommand{\up}{\upsilon}
\newcommand{\p}{\varphi}

\DeclareSymbolFont{CMletters}{OML}{cmm}{m}{it}
\DeclareMathSymbol{\xi}{\mathord}{CMletters}{"18}

\newcommand{\SL}{\operatorname{SL}}
\newcommand{\GL}{\operatorname{GL}}
\newcommand{\GaL}{\Gamma\!\operatorname{L}}
\newcommand{\PSL}{\operatorname{PSL}}
\newcommand{\PGL}{\operatorname{PGL}}
\newcommand{\PGaL}{\operatorname{P}\!\Gamma\!\operatorname{L}}
\newcommand{\Sp}{\operatorname{Sp}}
\newcommand{\GSp}{\operatorname{GSp}}
\newcommand{\GaSp}{\Gamma\!\operatorname{Sp}}
\newcommand{\PSp}{\operatorname{PSp}}
\newcommand{\PGSp}{\operatorname{PGSp}}

\newcommand{\SU}{\operatorname{SU}}
\newcommand{\GU}{\operatorname{GU}}
\newcommand{\CU}{\Delta\!\operatorname{U}}
\newcommand{\GaU}{\Gamma\!\operatorname{U}}
\newcommand{\PSU}{\operatorname{PSU}}
\newcommand{\PGU}{\operatorname{PGU}}
\newcommand{\PCU}{\operatorname{P}\!\Delta\!\operatorname{U}}
\newcommand{\PGaU}{\operatorname{P}\!\Gamma\!\operatorname{U}}

\newcommand{\SO}{\operatorname{SO}}
\renewcommand{\O}{\operatorname{O}}
\newcommand{\DO}{\operatorname{DO}}
\newcommand{\GO}{\operatorname{GO}}
\newcommand{\GaO}{\Gamma\!\operatorname{O}}

\newcommand{\POm}{\operatorname{P}\!\Omega}
\newcommand{\PSO}{\operatorname{PSO}}
\newcommand{\PO}{\operatorname{PO}}
\newcommand{\PDO}{\operatorname{PDO}}
\newcommand{\PGO}{\operatorname{PGO}}
\newcommand{\PGaO}{\operatorname{P}\!\Gamma\!\operatorname{O}}
\newcommand{\Spin}{\operatorname{Spin}}

\newcommand{\F}{\mathbb{F}}
\newcommand{\FF}{\overline{\F}}
\renewcommand{\:}{\colon}

\newcommand{\ppd}{\operatorname{ppd}}
\newcommand{\tr}{\mathsf{T}}
\newcommand{\<}{\langle}
\renewcommand{\>}{\rangle}

\renewcommand{\leq}{\leqslant}
\newcommand{\leqn}{\trianglelefteqslant}
\renewcommand{\geq}{\geqslant}

\newcommand{\SIsom}{\operatorname{SIsom}}
\newcommand{\Isom}{\operatorname{Isom}}
\newcommand{\Sim}{\operatorname{Sim}}
\newcommand{\Semi}{\operatorname{Semi}}
\newcommand{\soc}{\operatorname{soc}}
\newcommand{\Aut}{\operatorname{Aut}}

\newcommand{\Out}{\operatorname{Out}}
\newcommand{\Inndiag}{\operatorname{Inndiag}}
\newcommand{\Outdiag}{\operatorname{Outdiag}}
\newcommand{\fix}{\operatorname{fix}}
\newcommand{\fpr}{\operatorname{fpr}}
\newcommand{\sgn}{\operatorname{sgn}}
\newcommand{\mx}{\mathrm{max}}

\newcommand{\nonsquare}{\boxtimes}
\newcommand{\rsq}{r_{\square}}
\newcommand{\rns}{r_{\nonsquare}}

\renewcommand{\mod}[1]{\mathrm{ \ } (\mathrm{mod\ } #1)}

\makeatletter
\def\localbig#1#2{%
  \sbox\z@{$\m@th#1
    \sbox\tw@{$#1()$}%
    \dimen@=\ht\tw@\advance\dimen@\dp\tw@
    \nulldelimiterspace\z@\left#2\vcenter to1.2\dimen@{}\right.
  $}\box\z@}

\renewcommand{\div}{\mathrel{\mathpalette\dividesaux\relax}}

\newcommand{\dividesaux}[2]{\mbox{$\m@th#1\localbig{#1}|$}}
\newcommand{\ndividesaux}[2]{%
  \mkern.5mu
  \ooalign{%
    \hidewidth$\m@th#1\localbig{#1}|$\hidewidth\cr
    $\m@th#1\nmid$\cr%
  }%
}
\makeatother

\begin{document}

\frontmatter
\title{The Spread of Almost Simple Classical Groups}

\author{Scott Harper}
\address{School of Mathematics, University of Bristol, BS8 1UG, UK, and Heilbronn Institute for Mathematical Research, UK}
\email{scott.harper@bristol.ac.uk}
\maketitle

\begin{abstract}
Every finite simple group can be generated by two elements, and in 2000, Guralnick and Kantor resolved a 1962 question of Steinberg by proving that in a finite simple group every nontrivial element belongs to a generating pair. Groups with this property are said to be $\frac{3}{2}$\=/generated. 

Which finite groups are $\frac{3}{2}$\=/generated? Every proper quotient of a $\frac{3}{2}$\=/generated group is cyclic, and in 2008, Breuer, Guralnick and Kantor made the striking conjecture that this condition alone provides a complete characterisation of the finite groups with this property. This conjecture has recently been reduced to the almost simple groups and results of Piccard (1939) and Woldar (1994) show that the conjecture is true for almost simple groups whose socles are alternating or sporadic groups. Therefore, the central focus is now on the almost simple groups of Lie type.

In this monograph we prove a strong version of this conjecture for almost simple classical groups, motivated by earlier work of Burness and Guest (2013) and the author (2017). More precisely, we show that every relevant almost simple classical group has uniform spread at least two, unless it is isomorphic to the symmetric group of degree six. We also prove that the uniform spread of these groups tends to infinity if the size of the underlying field tends to infinity.

To prove these results, we are guided by a probabilistic approach introduced by Guralnick and Kantor. This requires a detailed analysis of automorphisms, fixed point ratios and subgroup structure of almost simple classical groups, so the first half of this monograph is dedicated to these general topics. In particular, we give a general exposition of the useful technique of Shintani descent, which plays an important role throughout.
\end{abstract}

\thanks{
Much of the work in this monograph was completed during the author's PhD at the University of Bristol, and he gratefully acknowledges the financial support of the Engineering and Physical Sciences Research Council and the Heilbronn Institute for Mathematical Research. The author sincerely thanks Dr Tim Burness for introducing him to this subject and for his generous support and encouragement over the course of this work. He also wishes to thank Professors Chris Parker and Jeremy Rickard for discussions about this work and reading earlier versions of this monograph.
}

\tableofcontents
\mainmatter

\chapter{Introduction} \label{c:intro}

The topic of generating sets for groups has a history dating to the earliest days of group theory, and it has led to a broad and rich literature, especially in the context of finite simple groups. In 1962, Steinberg \cite{ref:Steinberg62} proved that every finite simple group of Lie type is $2$\=/generated, by exhibiting an explicit pair of generators. In light of the Classification of Finite Simple Groups, together with results on alternating and sporadic groups \cite{ref:AschbacherGuralnick84}, we now know every finite simple group is $2$-generated. 

In the opening of his 1962 paper, Steinberg writes

\begin{quotation}
\emph{
It is possible that one of the generators can be chosen of order 2, as is the case for the projective unimodular group, or even that one of the generators can be chosen as an arbitrary element other than the identity, as is the case for the alternating groups. Either of these results, if true, would quite likely require methods much more detailed than those used here.
}
\end{quotation}

This motivates the following definition, which is central to this monograph.

\begin{definitionx*}
A group $G$ is \emph{$\frac{3}{2}$\=/generated} if for every nontrivial element $g \in G$, there exists an element $h \in G$ such that $\<g,h\> = G$.
\end{definitionx*}

In recent years, probabilistic methods have been very successful in solving many formidable deterministic problems in group theory (see, for example, \cite{ref:Burness16,ref:Liebeck13,ref:Shalev05}). Indeed, through a probabilistic approach, Guralnick and Kantor \cite{ref:GuralnickKantor00} proved that every finite simple group is $\frac{3}{2}$\=/generated, resolving the above question of Steinberg.

Classifying the $1$\=/generated groups is trivial and classifying the $2$\=/generated groups is impossible. Can we classify the $\frac{3}{2}$\=/generated groups? It is straightforward to demonstrate that every proper quotient of an arbitrary $\frac{3}{2}$\=/generated group is necessarily cyclic. In 2008, Breuer, Guralnick and Kantor \cite{ref:BreuerGuralnickKantor08} conjectured that this evidently necessary condition is actually sufficient for finite groups.

\begin{conjecturex*}[$\frac{3}{2}$\=/Generation Conjecture]
A finite group is $\frac{3}{2}$\=/generated if and only if every proper quotient is cyclic.
\end{conjecturex*}

Note that this necessary condition for $\frac{3}{2}$\=/generation is not sufficient for infinite groups; for example, the alternating group $A_\infty$ is simple but not finitely generated, let alone $\frac{3}{2}$\=/generated. However, the author does not know any examples of $2$\=/generated groups with no noncyclic proper quotients that are not $\frac{3}{2}$\=/generated. In \cite{ref:DonovenHarper}, Donoven and the author proved that two natural families of infinite groups generalising Thompson's group $V$ are $\frac{3}{2}$\=/generated, thus providing the first known examples of infinite $\frac{3}{2}$\=/generated groups (other than the infinite cyclic group and Tarski monsters).

The $\frac{3}{2}$\=/Generation Conjecture is true for soluble groups \cite[Theorem~2.01]{ref:BrennerWiegold75}, and for insoluble groups the conjecture has recently been reduced to the almost simple groups \cite{ref:BurnessGuralnickHarper}. Therefore, to prove the $\frac{3}{2}$\=/Generation Conjecture it is enough to prove that $\<T,\th\>$ is $\frac{3}{2}$\=/generated for all nonabelian finite simple groups $T$ and all automorphisms $\th \in \Aut(T)$.

The alternating and symmetric groups of degree at least 5 have been \textbf{}known to be $\frac{3}{2}$\=/generated since the work of Piccard in 1939 \cite{ref:Piccard39}, to which Steinberg refers in the quotation above. In addition, the $\frac{3}{2}$-generation of the relevant almost simple sporadic groups (and the two further almost simple cyclic extensions of $A_6$) follows from the computational results of Breuer, Guralnick and Kantor \cite{ref:BreuerGuralnickKantor08} (see also \cite{ref:Woldar94}). Therefore, to prove the $\frac{3}{2}$\=/Generation Conjecture, it suffices to focus on almost simple groups of Lie type. In this monograph, we prove the $\frac{3}{2}$\=/Generation Conjecture for almost simple classical groups. The exceptional groups pose different challenges and this is the topic of a forthcoming paper \cite{ref:BurnessGuralnickHarper}.

\begin{thmintro}\label{thm:3/2-generation}
Let $G$ be an almost simple classical group. Then $G$ is $\frac{3}{2}$\=/generated if every proper quotient of $G$ is cyclic.
\end{thmintro}

We actually prove a much stronger version of this theorem. To state our main results we must introduce some natural generalisations of $\frac{3}{2}$\=/generation.

\begin{definitionx*}
Let $G$ be a finite noncyclic group.
\begin{enumerate}
\item The \emph{spread} of $G$, written $s(G)$, is the greatest $k$ such that for any $k$ nontrivial elements $x_1,\dots,x_k$, there exists $y \in G$ such that 
\[
\< x_1, y \> = \< x_2, y \> = \cdots = \< x_k, y\> = G. 
\]
\item The \emph{uniform spread} of $G$, written $u(G)$, is the greatest $k$ for which there exists a fixed conjugacy class $C$ such that for any $k$ nontrivial elements $x_1,\dots,x_k$, there exists an element $y \in C$ satisfying the above equalities.
\end{enumerate}
\end{definitionx*}

Observe that $s(G) \geq u(G)$ and that $s(G) \geq 1$ if and only if $G$ is $\frac{3}{2}$\=/generated, so these invariants extend the idea of $\frac{3}{2}$-generation. If $G$ is simple, then Breuer, Guralnick and Kantor \cite{ref:BreuerGuralnickKantor08} proved that $u(G) \geq 2$ with equality if and only if $G \in \{ A_5, \, A_6, \, \Omega_8^+(2) \}$ or $G$ is $\Sp_{2m}(2)$ for $m \geq 3$. This generalises the fact that $s(G) \geq 1$ for simple groups $G$. In addition, Guralnick and Kantor \cite{ref:GuralnickKantor00} proved that if $(G_i)$ is a sequence of simple groups of Lie type where $G_i$ is defined over $\F_{q_i}$, then $u(G_i) \to \infty$ if $q_i \to \infty$. Later Guralnick and Shalev \cite{ref:GuralnickShalev03} determined exactly when $|G_i| \to \infty$ but $u(G_i)$ is bounded.

We may now present the stronger versions of Theorem~\ref{thm:3/2-generation} that we prove.

\begin{thmintro}\label{thm:us_main}
Let $G$ be an almost simple classical group such that $G/\!\soc(G)$ is cyclic. Then $u(G) \geq 2$, unless $G \cong S_6$.
\end{thmintro}

\begin{thmintro}\label{thm:us_asymptotic}
Let $(G_i)$ be a sequence of almost simple classical groups where $G_i$ is defined over $\F_{q_i}$ and $G_i/\!\soc(G_i)$ is cyclic. Then $u(G_i) \to \infty$ as $q_i \to \infty$.
\end{thmintro}

In 2013, Burness and Guest \cite{ref:BurnessGuest13} proved Theorems~\ref{thm:us_main} and~\ref{thm:us_asymptotic} for almost simple groups with socle $\PSL_n(q)$. They followed the probabilistic approach of Guralnick and Kantor in \cite{ref:GuralnickKantor00} but brought a powerful new technique to the problem: \emph{Shintani descent} (see p.\pageref{p:shintani}). Inspired by this work, the author proved Theorems~\ref{thm:us_main} and~\ref{thm:us_asymptotic} for symplectic and odd-dimensional orthogonal groups in \cite{ref:Harper17} using similar methods. 

However, as we explain below, the methods used in these previous papers are not enough to handle the remaining classical groups, which present new challenges. This monograph addresses these challenges and completes the proof of Theorems~\ref{thm:us_main} and~\ref{thm:us_asymptotic} by proving the following two results.

\begin{thmintro}\label{thm:main}
Let $G$ be an almost simple group with socle $\POm^\pm_n(q)$ ($n$ even) or $\PSU_n(q)$ such that $G/\!\soc(G)$ is cyclic. Then $u(G) \geq 2$, unless $G \cong S_6$.
\end{thmintro}

\begin{thmintro}\label{thm:asymptotic}
Let $(G_i)$ be a sequence of almost simple groups, where $G_i$ has socle $\POm^\pm_{n_i}(q_i)$ ($n_i$ even) or $\PSU_{n_i}(q_i)$ and $G_i/\!\soc(G_i)$ is cyclic. Then $u(G_i) \to \infty$ as $q_i \to \infty$.
\end{thmintro}

Therefore, in this monograph, we concentrate on even-dimensional orthogonal groups and unitary groups, with the aim of proving Theorems~\ref{thm:main} and~\ref{thm:asymptotic}. One case when $\soc(G)=\PSL_n(q)$ was omitted in \cite{ref:BurnessGuest13}, so we also prove Theorems~\ref{thm:us_main} and~\ref{thm:us_asymptotic} in this special case (see Remark~\ref{rem:linear}).

Let us make some remarks on the statements of the main theorems.

\begin{remintro}\label{rem:main}
As noted in \cite{ref:BurnessGuest13}, it is straightforward to check that $s(S_6)=2$ and $u(S_6)=0$, so this explains why we exclude the almost simple classical groups that are isomorphic to $S_6$ from the statement of Theorems~\ref{thm:us_main} and~\ref{thm:main}.
\end{remintro}

\begin{remintro}\label{rem:spread_one}
The $\frac{3}{2}$-Generation Conjecture avers that $s(G) \geq 1$ if every proper quotient of $G$ is cyclic. A stronger version of this conjecture is that $s(G) \geq 2$ if every proper quotient of $G$ is cyclic, and this would imply that there do not exist any finite groups with $s(G)=1$ (see \cite[Conjecture~3.16]{ref:Burness19}). Theorem~\ref{thm:us_main} proves this stronger conjecture for almost simple classical groups.
\end{remintro}

\begin{remintro}\label{rem:asymptotic}
Let $(G_i)$ be a sequence of almost simple classical groups where $G_i$ has natural module $\F_{q_i}^{n_i}$ and $G_i/\!\soc(G_i)$ is cyclic. By Theorem~\ref{thm:us_asymptotic}, $u(G_i) \to \infty$ if $q_i \to \infty$, but it is difficult to determine when $n_i \to \infty$ implies $u(G_i) \to \infty$ for bounded $q_i$. Even for simple groups, $G_i = \Sp_{2i+2}(2)$ gives an example where $n_i \to \infty$ but $u(G_i) = 2$ for all $i$ (see \cite[Proposition~2.5]{ref:GuralnickShalev03}). Almost simple groups provide an even greater challenge: for instance, if $G_i = \Aut(\PSL_{2i+1}(2)) = \PSL_{2i+1}(2).2$, then $u(\soc(G_i)) \to \infty$ but $u(G_i) \leq 8$ for all $i$ (see \cite[Theorem~4]{ref:BurnessGuest13}). Determining when the uniform spread of almost simple classical groups is bounded will feature in future work.
\end{remintro}

\begin{remintro}\label{rem:exact_spread}
Let us note that determining the exact value of spread and uniform spread is a difficult task in general. Notably, the spread of odd-degree alternating groups is not known in general (see \cite[Remark~3]{ref:BurnessHarper} and the references therein); however, by \cite[(3.01)--(3.05)]{ref:BrennerWiegold75}, it is known that $s(A_{n})=4$ if $n \geq 8$ is even. In addition, $s(\PSL_2(q))$ is not known when $q \equiv 3 \mod{4}$ (see the discussion in \cite[Remark~5]{ref:BurnessHarper}), and for sporadic groups the exact spread is only known in two cases ($s(\mathrm{M}_{11})=3$ and $s(\mathrm{M}_{23})=8064$, see \cite{ref:Fairbairn12}).
\end{remintro}

\begin{remintro} \label{rem:uniform_spread}
We have already observed that $s(S_6)=2$ but $u(S_6) = 0$. It is worth noting that $s(G)$ and $u(G)$ can be different positive integers. For example, if $n \geq 5$ is odd, then $s(S_n)=3$ and $u(S_n)=2$. However, the only known family of nonabelian simple groups for which $s(G) - u(G)$ is unbounded is $G = \PSL_2(p)$ where $p$ is a prime number satisfying $p \equiv 3 \mod{4}$ (see \cite[Proposition~7.4]{ref:BurnessHarper}).
\end{remintro}

We now turn to a brief discussion of the techniques employed in this monograph; the opening of Chapter~\ref{c:o} gives a more technical account of the particular challenges that we have to overcome. For this discussion, $G$ is an almost simple classical group such that $G/\soc(G)$ is cyclic. 

The framework for proving Theorems~\ref{thm:main} and~\ref{thm:asymptotic} is given by the probabilistic method introduced by Guralnick and Kantor \cite{ref:GuralnickKantor00} (see Section~\ref{s:p_prob}). The general idea is to select an element $s \in G$ and show that $s^G$ witnesses $u(G) \geq k$. To do this, we let $P(x,s)$ be the probability that $\<x,z\> \neq G$ for a random conjugate $z$ of $s$. Evidently, $u(G) \geq 1$ if $P(x,s) < 1$ for all nontrivial $x \in G$. Indeed, $u(G) \geq k$ if $P(x,s) < \frac{1}{k}$ for all prime order $x \in G$ (see Lemma~\ref{lem:prob_method}). 

Let $\M(G,s)$ be the set of maximal subgroups of $G$ that contain $s$. In addition, for $H \leq G$ and $x \in G$, let $\fpr(x,G/H)$ be the \emph{fixed point ratio} of $x$ in the action of $G$ on $G/H$. We will see in Lemma~\ref{lem:prob_method} that
\[
P(x,s) \leq \sum_{H \in \M(G,s)} \fpr(x,G/H).
\]
Therefore, our probabilistic method has three steps: select an appropriate element $s \in G $, determine $\M(G,s)$ and use fixed point ratio estimates to bound $P(x,s)$. 

Selecting a viable element $s \in G$ is perhaps the most interesting and challenging aspect of the proofs. Write $G=\<T,\th\>$ where $T = \soc(G)$ and $\th \in \Aut(T)$. If $s^G$ witnesses $u(G) \geq k > 0$, then $s$ is not contained in any proper normal subgroup of $G$, so we may assume that $s \in T\th$. Consequently, we need to understand the conjugacy classes in the coset $T\th$.

We view the finite groups of Lie type as the fixed points under Steinberg endomorphisms of simple algebraic groups, and this perspective allows us to exploit Shintani descent\label{p:shintani} \cite{ref:Kawanaka77,ref:Shintani76}, a technique which has seen great utility in character theory (see \cite{ref:CabanesSpath19,ref:DigneMichel94,ref:Kessar04,ref:Shoji95} for example). At the heart of this method is a bijection with useful group theoretic properties that, given a connected algebraic group $X$, a Steinberg endomorphism $\s$ of $X$ and an integer $e >1$, provides a correspondence between the conjugacy classes of elements in the coset $X_{\s^e}\s$ and in the subgroup $X_\s$. We use this bijection to transform a problem about almost simple groups into one about simple groups.

Shintani descent was used by Burness and Guest in \cite{ref:BurnessGuest13} in the context of linear groups, and this technique was extended in \cite{ref:Harper17} to overcome various difficulties and subtleties that the symplectic groups posed (such as the disconnected orthogonal subgroups in even characteristic and the graph-field automorphism of $\Sp_4(2^f)$). In this monograph we present a general formalism of Shintani descent for applications to all almost simple groups of Lie type, which we anticipate will be useful more generally. Shintani descent is introduced in Chapter~\ref{c:shintani}, where we provide crucial new results that allow us to handle the novel challenges posed by the twisted minus-type orthogonal and unitary groups.

Our framework for understanding $\M(G,s)$ is provided by Aschbacher's subgroup structure theorem for finite classical groups \cite{ref:Aschbacher84}, which asserts that the maximal subgroups of classical groups are either the stabilisers of geometric structures on the natural module or they arise from an absolutely irreducible representation of a quasisimple group. By studying how our chosen element acts on the natural module, we can constrain the maximal subgroups that could contain this element. 

The common strategy of choosing $s$ to have a large and restrictive order cannot typically be employed for this problem, so we require different techniques. This obstacle occurs because the element $s$ is contained in the nontrivial coset $T\th$ and in many cases this forces $s$ to have a comparatively small order (indeed, even determining the possible element orders in this coset is nontrivial). In Remark~\ref{rem:gpps}, we use Shintani descent to explain this issue more precisely.

Once we have a description of $\M(G,s)$, we use fixed point ratio estimates to bound $P(x,s)$. There is an extensive literature on fixed point ratios for primitive actions of almost simple groups, and these quantities have found applications to a vast range of problems, including the resolution of the Cameron--Kantor conjecture on base sizes of permutation groups \cite{ref:LiebeckShalev99} and the Guralnick--Thompson conjecture on monodromy groups \cite{ref:FrohardtMagaard01}. In Chapter~\ref{c:fpr} we review some general results in this area and prove some new fixed point ratio bounds that we require for our proofs; these bounds may be of independent interest.

Let us now highlight a combinatorial connection to this work. The \emph{generating graph} of a group $G$ is the graph $\Gamma(G)$ whose vertices are the nontrivial elements of $G$ and where two vertices $g$ and $h$ are adjacent if $\<g,h\>=G$. The generating graphs of the dihedral group $D_8$ and the alternating group $A_4$ are given in Figure~\ref{fig:graphs}

\begin{figure}[t]
\begin{tikzpicture}[scale=0.45]
\node[white] (0) at (270:8cm) {};

\node[draw,circle,minimum size=1cm] (1) at (45 :5.5cm) {$ab$};
\node[draw,circle,minimum size=1cm] (2) at (135 :5.5cm) {$b$};
\node[draw,circle,minimum size=1cm] (3) at (225 :5.5cm) {$a^3b$};
\node[draw,circle,minimum size=1cm] (4) at (315 :5.5cm) {$a^2b$};

\node[draw,circle,minimum size=1cm] (5) at (0:1.5cm) {$a^3$};
\node[draw,circle,minimum size=1cm] (6) at (180:1.5cm) {$a$};

\node[draw,circle,minimum size=1cm] (8) at (270:6.5cm) {$a^2$};

\path (1) edge[-] (2);
\path (2) edge[-] (3);
\path (3) edge[-] (4);
\path (4) edge[-] (1);
\path (1) edge[-] (5);
\path (2) edge[-] (5);
\path (3) edge[-] (5);
\path (4) edge[-] (5);

\path (1) edge[-] (6);
\path (2) edge[-] (6);
\path (3) edge[-] (6);
\path (4) edge[-] (6);
\end{tikzpicture} \hspace{1cm}
\begin{tikzpicture}[scale=0.6,circle,inner sep=0.01cm,minimum size=0.9cm]

\node[draw] (1) at (90-360/11 :5cm) {\scalebox{0.9}{\tiny{(1\,2)(3\,4)}}};
\node[draw] (2) at (90 :5cm) {\scalebox{0.9}{\tiny{(1\,3)(2\,4)}}};
\node[draw] (3) at (90+360/11 :5cm) {\scalebox{0.9}{\tiny{(1\,4)(2\,3)}}};

\node[draw] (4) at (90+360/11*2: 5cm) {\tiny{(1\,2\,3)}};
\node[draw] (5) at (90+360/11*3: 5cm) {\tiny{(1\,3\,2)}};
\node[draw] (6) at (90+360/11*4: 5cm) {\tiny{(1\,2\,4)}};
\node[draw] (7) at (90+360/11*5: 5cm) {\tiny{(1\,4\,2)}};
\node[draw] (8) at (90+360/11*6: 5cm) {\tiny{(1\,3\,4)}};
\node[draw] (9) at (90+360/11*7: 5cm) {\tiny{(1\,4\,3)}};
\node[draw] (10) at (90+360/11*8: 5cm) {\tiny{(2\,3\,4)}};
\node[draw] (11) at (90+360/11*9: 5cm) {\tiny{(2\,4\,3)}};

\foreach \x in {1,...,3}
\foreach \y in {4,...,11}
{\path (\x) edge[-] (\y);}

\path (5)  edge[-] (6);
\path (7)  edge[-] (8);
\path (9)  edge[-] (10);
\path (11) edge[-] (4);

\foreach \x in {4,...,11}
\foreach \y in {4,...,11}
{
 \pgfmathtruncatemacro{\a}{\x-\y};
 \ifnum\a=0{}
 \else
 {
  \ifnum\a=1{}
  \else
  {
  \ifnum\a=-1{}
  \else
   {
    \path (\x) edge[-] (\y);
   }
   \fi
  }
  \fi
 }
 \fi
} 
\end{tikzpicture}
\caption{The generating graphs of $D_8$ and $A_4$}\label{fig:graphs}
\end{figure}
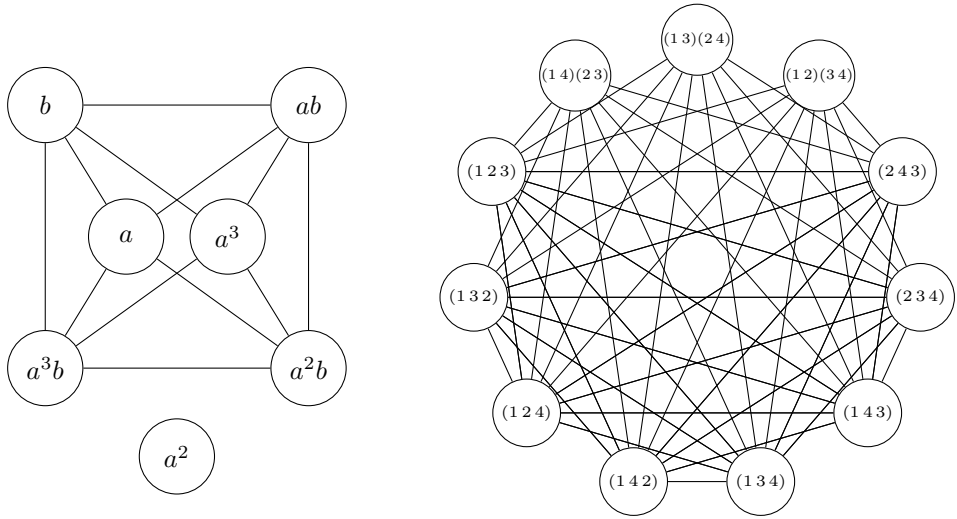

Evidently, $\Gamma(G)$ has no isolated vertices if and only if $G$ is $\frac{3}{2}$\=/generated. This dichotomy is demonstrated by $\Gamma(D_8)$ and $\Gamma(A_4)$, where we note that $D_8$ has a noncyclic quotient whereas $A_4$ does not. Further, if $s(G) \geq 2$, then $\Gamma(G)$ is connected with diameter at most $2$. Therefore, the stronger version of the $\frac{3}{2}$\=/Generation Conjecture in Remark~\ref{rem:spread_one} has the striking interpretation that a generating graph of a finite group either has an isolated vertex or it is connected with diameter at most two. By \cite[Theorem~1.2]{ref:BreuerGuralnickKantor08}, the diameter of the generating graph of any nonabelian finite simple group is two, and Theorem~\ref{thm:us_main} implies that the same conclusion holds for almost simple classical groups $G$ such that $G/\soc(G)$ is cyclic.

Many other natural questions about generating graphs have been investigated in recent years. For instance, if $G$ is a sufficiently large simple group, then $\Gamma(G)$ is \emph{Hamiltonian} (that is, has a cycle containing every vertex exactly once) \cite{ref:BreuerGuralnickLucchiniMarotiNagy10}. Moreover, if $n \geq 120$, then the generating graphs $\Gamma(A_n)$ and $\Gamma(S_n)$ are Hamiltonian \cite{ref:Erdem18}. Indeed, it is conjectured that for all finite groups $G$ of order at least four, the generating graph $\Gamma(G)$ is Hamiltonian if and only if every proper quotient of $G$ is cyclic, which is another strengthening of the $\frac{3}{2}$-Generation Conjecture.

In a different direction, the \emph{total domination number} of a graph $\Gamma$ is the minimal size of a set $S$ of vertices of $\Gamma$ such that every vertex of $\Gamma$ is adjacent to a vertex in $S$. In recent work of Burness and the author \cite{ref:BurnessHarper19,ref:BurnessHarper}, close to best possible bounds on the total domination number of generating graphs of simple groups were obtained, together with related probabilities. For instance, there are infinitely many finite simple groups $G$ for which the total domination number of $\Gamma(G)$ is the minimal possible value of two (for example, $A_p$ when $p \geq 13$ is prime, $\PSL_n(q)$ when $n > 3$ is odd, $E_8(q)$ and the Monster). This is a vast generalisation of the fact that these groups are $\frac{3}{2}$\=/generated.

For further reading on group generation, especially in the context of simple groups and probabilistic methods, see Burness' recent survey article \cite{ref:Burness19}. The recent paper of Burness and the author \cite{ref:BurnessHarper} also features a detailed account of the spread of simple groups and related groups. 

We conclude the introduction with an outline of the structure of this monograph. Chapter~2 introduces the almost simple classical groups, their subgroups, the formed spaces the naturally act on and their connection with simple algebraic groups. As noted above, in Chapter~3 we turn to Shintani descent, where we unify existing results in this area and provide new methods that allow us to handle all almost simple classical groups. Chapter~4 is dedicated to establishing bounds on fixed point ratios. In Chapters~5 and 6, we study automorphisms and special elements of classical groups, before turning to the proofs of our main results on uniform spread.

\chapter{Preliminaries} \label{c:prelims}

In this chapter we introduce the requisite background material from the literature on almost simple classical groups, and we will use this opportunity to establish the notation that we use in this monograph.

\subsection*{Notational conventions}

Let $a,b,n$ be positive integers and let $G,H$ be groups. Throughout we write

\begin{itemize}
\item[] $(a,b)$ for the greatest common divisor of $a$ and $b$
\item[] $a_b$ for the greatest power of $b$ dividing $a$
\item[] $\d_{ab}$ for the Kronecker delta
\item[] $\log{a}$ for the \emph{base two} logarithm of $a$
\item[] $C_n$ (or simply $n$) for the cyclic group of order $n$
\item[] $G.H$ for an unspecified extension of $G$ by $H$ (with quotient $H$)
\item[] $G{:}H$ for an unspecified split extension of $G$ by $H$
\end{itemize}

Groups always act on the right. Accordingly, matrices act on the right of row vectors, $x^g$ denotes $g^{-1}xg$ and $G/H$ is the set of right cosets of $H$ in $G$.

\section{Probabilistic method} \label{s:p_prob}

Probabilistic methods featuring fixed point ratios, introduced below, are a fruitful means of studying a wide range of problems, and the survey article \cite{ref:Burness16} provides an excellent overview of this topic. In this section, we outline the probabilistic method for studying uniform spread introduced by Guralnick and Kantor \cite{ref:GuralnickKantor00}.

Let $G$ be a finite group acting on a finite set $\Omega$. The \emph{fixed point ratio} of $x \in G$ is
\[
\fpr(x,\Omega) = \frac{\fix(x,\Omega)}{|\Omega|} \quad \text{where} \quad \fix(x,\Omega) = |\{ \omega \in \Omega \mid \omega x = \omega \}|.
\] 
If $H \leq G$, then $G$ acts transitively on $G/H$ and one sees that
\[
\fpr(x,G/H) = \frac{|x^G \cap H|}{|x^G|}.
\]
We discuss recent work on fixed point ratios, particularly in the context of primitive actions of almost simple groups, at the opening of Chapter~\ref{c:fpr}.

We now describe the probabilistic method for uniform spread. For $x,s \in G$
\begin{equation} \label{eq:p}
P(x,s) = \frac{|\{z \in s^G \mid \< x, z \> \neq G \}|}{|s^G|}
\end{equation}
is the probability that $x$ does not generate $G$ with a (uniformly) randomly chosen conjugate of $s$. Let $\M(G,s)$ be the set of maximal subgroups of $G$ that contain $s$. The following encapsulates the method (see \cite[Lemmas~2.1 and~2.2]{ref:BurnessGuest13}).

\begin{lemmax} \label{lem:prob_method}
Let $G$ be a finite group and let $s \in G$.
\begin{enumerate}
\item For $x \in G$, 
\[ 
P(x,s) \leq \sum_{H \in \M(G,s)}^{} \fpr(x,G/H).
\]
\item If for all $k$-tuples $(x_1,\dots,x_k)$ of prime order elements of $G$ 
\[
\sum_{i=1}^{k}P(x_i,s) < 1,
\]
then $u(G) \geq k$ with respect to the conjugacy class $s^G$.
\end{enumerate}
\end{lemmax}

We conclude this section with an elementary observation.

\begin{lemmax} \label{lem:fpr_subgroups}
Let $G$ be a finite group, let $H \leq G$ and let $x \in G$. Then the number of $G$-conjugates of $H$ that contain $x$ is $\fpr(x,G/H) \cdot |G:N_G(H)|$.
\end{lemmax}

\section{Classical groups} \label{s:p_groups}

Let $F$ be a finite or algebraically closed field of characteristic $p > 0$ and $V=F^n$. Our notation for classical groups is mainly standard, but there is variation in the literature, so we will briefly outline the notation we use. For further background on classical groups see \cite[Chapter~7]{ref:Aschbacher00}, \cite[Chapter~2]{ref:BurnessGiudici16} and \cite[Chapter~2]{ref:KleidmanLiebeck}.

If $\kappa$ is a bilinear, sesquilinear or quadratic form on $V$, then write $\Isom(V,\kappa)$, $\Sim(V,\kappa)$ and $\Semi(V,\kappa)$ for the groups of isometries, similarities and semisimilarities of $\kappa$, and write $\SIsom(V,\kappa)$ for $\Isom(V,\kappa) \cap \SL(V)$. Using this notation, Table~\ref{tab:groups_notation} gives our notation for the classical groups. For projective groups, we adopt the standard convention that for $G \leq \GaL(V)$ we write $\mathrm{P}G = GZ(V)/Z(V) \cong G/(G \cap Z(V))$, where $Z(V) = Z(\GL(V)) \leq \GaL(V)$ is the group of scalar transformations on $V$.

\begin{table}
\centering
\caption{Notation for classical groups} \label{tab:groups_notation}
{\renewcommand{\arraystretch}{1.2}
\begin{tabular}{cccccc}
\hline
$\kappa$        &                  & $\SIsom(V,\kappa)$ & $\Isom(V,\kappa)$ & $\Sim(V,\kappa)$  & $\Semi(V,\kappa)$  \\
\hline
zero            & $\SL_n(F)$       & $\SL_n(F)$         & $\GL_n(F)$        & $\GL_n(F)$        & $\GaL_n(F)$    \\
unitary         & $\SU_n(F_0)$     & $\SU_n(F_0)$       & $\GU_n(F_0)$      & $\CU_n(F_0)$      & $\GaU_n(F_0)$  \\
symplectic      & $\Sp_n(F)$       & $\Sp_n(F)$         & $\Sp_n(F)$        & $\GSp_n(F)$       & $\GaSp_n(F)$   \\ 
n.d. quadratic  & $\Omega^\e_n(F)$ & $\SO^\e_n(F)$      & $\O^\e_n(F)$      & $\GO^\e_n(F)$     & $\GaO^\e_n(F)$ \\
\hline  
\end{tabular}}
\\[5pt]
{\small Note: see Remark~\ref{rem:groups_notation} for a description of $F_0$ and $\e$}
\end{table}

\begin{remarkx} \label{rem:groups_notation}
Let us remark on the notation introduced in Table~\ref{tab:groups_notation}.
\begin{enumerate}
\item \emph{Unitary groups.} By $\kappa$ being unitary, we require that $F$ be a quadratic extension of $F_0$ with $\mathrm{Gal}(F/F_0) = \<\s\>$, and we mean that $\kappa$ is a nondegenerate $\s$-conjugate symmetric sesquilinear form on $V$. The notation $\CU_n(F_0)$ is uncommon but follows \cite{ref:BurnessGiudici16}; we shall rarely need to refer to this group since $\PCU_n(F_0) = \PGU_n(F_0)$. We write $\GL^+ = \GL$ and $\GL^- = \GU$.
\item \emph{Symplectic groups.} By $\kappa$ being symplectic, we require that $n$ be even, and we mean that $\kappa$ is a nondegenerate alternating bilinear form on $V$. 
\item \emph{Orthogonal groups.} Let $\kappa$ be a nondegenerate (n.d.) quadratic form $Q$ with associated bilinear form $(\cdot,\cdot)$ defined as
\[
(u,v) = Q(u+v) - Q(u) - Q(v).
\]
The \emph{norm} of a vector $v \in V$ is $(v,v)$.
\begin{enumerate}[(a)]
\item If $n=2m$ is even, then there are at most two isometry types of nondegenerate quadratic forms $Q$, distinguished by their \emph{Witt index}, the dimension of a maximal totally singular subspace. If $F = \bar{F}$, then there is a unique type. If $F=\F_q$, there there are exactly two types, plus and minus, and we write $\sgn(Q) = \e$ when $Q$ is $\e$-type. These types are distinguished by the \emph{discriminant} $D(Q) \in \F_{q}/(\F_q)^2 = \{ \square, \nonsquare \}$:
\begin{equation} \label{eq:discriminant_condition}
D(Q) = \square \iff q^m \equiv \sgn(Q) \mod{4},
\end{equation}
interpreting $\sgn(Q)$ as $1$ or $-1$ (see \cite[Proposition~2.5.10]{ref:KleidmanLiebeck}).
\item If $n$ is odd, then $V$ admits a nondegenerate quadratic form if and only if $p$ is odd, in which case there is a unique similarity type of form and we write $\sgn(Q) = \circ$. If $F = \F_q$ where $q$ is odd, then there are two isometry types of nondegenerate quadratic form $Q$, again distinguished by the discriminant $D(Q)$.
\end{enumerate}
\item \emph{The group $\Omega^\e_n(F)$.} Let $Q$ be a nondegenerate quadratic form and assume $(n,F,\sgn(Q)) \neq (4,\F_2,+)$ (see \cite[Proposition~2.5.9]{ref:KleidmanLiebeck} in this case). If $p=2$, then every element of $\SO^\e_n(F) = \O^\e_n(F)$ is a product of reflections and we define $\Omega^\e_n(F)$ as the group of elements that are a product of an even number of reflections (see \cite[22.7--22.9]{ref:Aschbacher00}). If $p$ is odd, then $\SO^\e_n(F)$ is the group of elements that are a product of an even number of reflections and we define $\Omega^\e_n(F)$ as the kernel of the \emph{spinor norm} $\SO^\e_n(F) \to F^\times/(F^\times)^2$ (see \cite[22.10]{ref:Aschbacher00}), so $\Omega_n(F) = \SO_n(F)$ if $F = \bar{F}$. If $F = \bar{F}$, then the algebraic group $\O_n(F)$ is disconnected and $\Omega_n(F)$ is simply $\O_n(F)^\circ$.
\item \emph{Similarities.} If $g \in \Sim(V,\kappa)$, then there exists $\t(g) \in F^\times$ such that for all $u,v \in V$ we have $(ug,vg) = \t(g)(u,v)$ (or $Q(vg) = \t(g)Q(v)$). We refer to $\t\:\Sim(V,\kappa) \to F^\times$ as the \emph{similarity map}.
\item \emph{Warning.} Although we use notation such as $\GO^-_{2m}(q)$, the elements of these groups are linear maps on a fixed vector space $V$ which preserve a fixed quadratic or bilinear form; the elements are not matrices. Indeed, we will use a number of different bases to specify elements in these groups. 
\item \emph{Notation.} Our notation in Table~\ref{tab:groups_notation} is consistent with \cite{ref:Aschbacher00,ref:BreuerGuralnickKantor08,ref:BurnessGiudici16,ref:GorensteinLyonsSolomon98,ref:KleidmanLiebeck}, sources to which we often refer (in general we try to always maintain consistency with \cite{ref:GorensteinLyonsSolomon98}). However, this notation is not universal.
\end{enumerate}
\end{remarkx}

By a \emph{finite simple classical group} we mean one of the groups in Table~\ref{tab:finite}. These groups are simple and each excluded group is either not simple or coincides with a simple group that is included \cite[Theorem~2.1.3 and~Proposition~2.9.1]{ref:KleidmanLiebeck}.

\begin{table}
\centering
\caption{Finite simple classical groups} \label{tab:finite}
{\renewcommand{\arraystretch}{1.2}
\begin{tabular}{ccccc}
\hline
                   & $\PSL_n(q)$      & $\PSU_n(q)$ & $\PSp_n(q)$ & $\POm^\e_n(q)$ \\
\hline
lower bound on $n$ & $2$              & $3$         & $4$         & $7$            \\
excluded $(n,q)$   & $(2,2)$, $(2,3)$ & $(3,2)$     & $(4,2)$     &                \\
\hline 
\end{tabular}}
\end{table}

We conclude this section by coining a useful piece of notation. If $p$ is odd, then $\det(g) = \pm\t(g)^m$ for all $g \in \GO^\e_{2m}(F)$ (see \cite[Lemma~2.8.4]{ref:KleidmanLiebeck}) and we define
\begin{equation}\label{eq:do_odd}
\DO_{2m}^\e(F) = \{ g \in \GO^\e_{2m}(F) \mid \det(g) = \tau(g)^m \}.
\end{equation}
Informally, $\DO^\e_n(F)$ is to $\GO^\e_n(F)$ as $\SO^\e_n(F)$ is to $\O^\e_n(F)$; indeed, we have that $\DO^\e_n(F) \cap \O^\e_n(F) = \SO^\e_n(F)$. If $p=2$, then we simply define
\begin{equation}\label{eq:do_even}
\DO^\e_{2m}(F) = \Omega^\e_{2m}(F).
\end{equation}

\section{Actions of classical groups} \label{s:p_actions}

In this section, $V = F^n$ where $n \geq 1$ and $F$ is a field. We begin by recording some general results on $\GL_n(F)$, which are surely well known but are hard to find direct references for. 

Let $\mathcal{D}$ be a direct sum decomposition $V = V_1 \oplus \cdots \oplus V_k$ or a tensor product decomposition $V = V_1 \otimes \cdots \otimes V_k$, where $\dim{V_i} > 1$ in the latter case. For $G \leq \GL(V)$, the centraliser $G_{(\mathcal{D})}$ and stabiliser $G_{\mathcal{D}}$ of $\mathcal{D}$ are the subgroups that stabilise the $V_1, \dots, V_k$ pointwise and setwise, respectively. If an element $g \in \GL(V)$ centralises the decomposition $\mathcal{D}$ and acts as $g_i$ on $V_i$, then we write $g$ as $g_1 \oplus \cdots \oplus g_k$ or $g_1 \otimes \cdots \otimes g_k$, according to the type of decomposition. 

The following is entirely analogous to Goursat's Lemma from group theory (see \cite[p.75]{ref:Lang02} for example).

\begin{lemmax}[Goursat's Lemma] \label{lem:goursat}
Let $G \leq \GL(V)$ centralise $V=V_1 \oplus V_2$. Let $U$ be an $FG$-submodule of $V$. Then there exist $FG$-submodules $W_1 \leq U_1 \leq V_1$ and $W_2 \leq U_2 \leq V_2$ and an $FG$-isomorphism $\p\:U_1/W_1 \to U_2/W_2$ such that
\[
U = \{ (u_1,u_2) \in U_1 \oplus U_2 \mid \p(W_1+u_1) = W_2+u_2 \}.
\]
\end{lemmax}

\begin{corollaryx} \label{cor:goursat}
Let $G \leq \GL(V)$ centralise $V=V_1 \oplus V_2$. Assume that there are no nonzero $FG$-isomorphisms between $FG$-subquotients of $V_1$ and $V_2$. Let $U$ be an $FG$-submodule of $V$. Then there exist $FG$-submodules $U_1 \leq V_1$ and $U_2 \leq V_2$ such that $U = U_1 \oplus U_2$. 
\end{corollaryx}

The following lemma, which is proved directly in \cite[Lemma~2.10.11]{ref:KleidmanLiebeck}, is an immediate consequence of Corollary~\ref{cor:goursat}.

\begin{lemmax}\label{lem:c1}
Let $G \leq \GL(V)$ centralise $V=V_1 \oplus \dots \oplus V_k$. If $V_1, \dots, V_k$ are pairwise nonisomorphic irreducible $FG$-modules, then they are the only irreducible $FG$-submodules of $V$. 
\end{lemmax}

We use the following straightforward lemma to compute centralisers of elements in classical groups.

\begin{lemmax}\label{lem:centraliser}
Let $g = g_1 \oplus \cdots \oplus g_k \in \GL(V)$ centralise $V=V_1 \oplus \cdots \oplus V_k$. If that there are no nonzero $F\<g\>$-homomorphisms between $V_i$ and $V_j$ when $i \neq j$ (for example, if $V_1,\dots,V_k$ are pairwise nonisomorphic irreducible $F\<g\>$-modules), then
\[
C_{\GL(V)}(g) = C_{\GL(V_1)}(g_1) \times \cdots \times C_{\GL(V_k)}(g_k).
\]  
\end{lemmax}

For the remainder of this section it will be convenient to fix a basis for $V$ and consider the elements of $\GL_n(F)$ as matrices with respect to this basis. For $g \in \GL_n(F)$, if $V$ is an irreducible $F\<g\>$-module, then we say that $g$ is \emph{irreducible}.

\begin{lemmax}\label{lem:irreducible_criterion}
Let $g \in \GL_n(F)$. Then $g$ is irreducible if and only if the characteristic polynomial of $g$ is irreducible over $F$. 
\end{lemmax}

\begin{proof}
Let $\chi$ be the characteristic polynomial of $g$. First assume that $g$ is reducible. Then $g$ is similar to the block lower triangular matrix 
\[
\left( \begin{array}{cc} g_1 & 0 \\ h & g_2 \end{array} \right)
\]
where $g_1$ is a $k \times k$ matrix for some $0 < k < n$. Therefore, the characteristic polynomial $\phi$ of $g_1$ is a proper nonconstant divisor of $\chi$, so $\chi$ is reducible.

For the converse, assume that $g$ is irreducible. From the rational canonical form of $g$, it is evident that the irreducibility of $g$ implies that $\chi$ is the minimal polynomial of $g$. We wish to prove that $\chi$ is irreducible, so write $\chi = \phi\psi$, where $\phi$ and $\psi$ are monic. Since $\chi(g)=0$, without loss of generality, $\phi(g)$ is not invertible. Now let $U$ be the kernel of $\phi(g)$, noting that $U \neq 0$. Let $u \in U$ and note that $(ug)\phi(g) = (u\phi(g))g = 0g = 0$, so $U$ is a submodule of $V$. However, $V$ is irreducible, so $U=V$ and, consequently, $\phi(g) = 0$. Since $\chi$ is the minimal polynomial of $x$, we deduce that $\chi = \phi$. Therefore, $\chi$ is irreducible. This completes the proof.
\end{proof}

\begin{lemmax}\label{lem:irreducible_conjugacy}
Let $g,h \in \GL_n(F)$ be irreducible. Then $g$ and $h$ are similar if and only if they have the same characteristic polynomial.
\end{lemmax}

\begin{proof}
If $g$ and $h$ are similar, then $g$ and $h$ evidently have the same characteristic polynomial. Now assume $\chi$ is the characteristic polynomial of both $g$ and $h$. By Lemma~\ref{lem:irreducible_criterion}, $\chi$ is irreducible. Therefore, the rational canonical form of both $g$ and $h$ is the companion matrix of $\chi$, so $g$ and $h$ are similar.
\end{proof}

We say that an element $g \in \GL_n(F)$ is semisimple if $g$ is similar to a block diagonal matrix $g_1 \oplus \cdots \oplus g_k$ where each $g_i$ is irreducible.

\begin{lemmax}\label{lem:semsimple_conjugacy}
Let $g,h \in \GL_n(F)$ be semisimple. Then $g$ and $h$ are similar if and only if they have the same characteristic polynomial.
\end{lemmax}

\begin{proof}
Assume that $\chi$ is the characteristic polynomial of both $g$ and $h$. Since $g$ and $h$ are semisimple, they are similar to block diagonal matrices $g_1^{a_1} \oplus \cdots \oplus g_k^{a_k}$ and $h_1^{b_1} \oplus \cdots \oplus h_l^{b_l}$, where $g_1,\dots,g_k$ and $h_1,\dots,h_l$ are pairwise non-similar irreducible matrices. For each $i$, let $\phi_i$ and $\psi_i$ be the characteristic polynomials of $g_i$ and $h_i$, respectively. By Lemma~\ref{lem:irreducible_criterion}, the polynomials $\phi_i$ and $\psi_i$ are irreducible since the matrices $g_i$ and $h_i$ are irreducible. Now $\phi_1^{a_1} \cdots \phi_k^{a_k} = \chi = \psi_1^{b_1} \cdots \psi_l^{b_l}$. By the irreducibility of each $\phi_i$ and $\psi_i$, we conclude $k=l$ and we may assume that for each $i$ we have $\phi_i=\psi_i$ and $a_i=b_i$. For each $i$, by Lemma~\ref{lem:irreducible_conjugacy}, $g_i$ and $h_i$ are similar since $g_i$ and $h_i$ are irreducible and have equal characteristic polynomials. Therefore, $g$ and $h$ are similar, as required.
\end{proof}

\section{Standard bases} \label{s:p_bases}

Let $F$ be a finite or algebraically closed field of characteristic $p > 0$ and $V = F^n$. We now fix standard bases for $V$ for each classical form, following \cite[Chapter~2]{ref:KleidmanLiebeck}.

First assume that $\kappa = (\cdot,\cdot)$ is symplectic. Fix $\B = (e_1,f_1,\dots,e_m,f_m)$ such that
\begin{equation} \label{eq:B_symp}
(e_i,e_j) = (f_i,f_j) = 0, \quad (e_i,f_j) = \d_{ij}.
\end{equation} 

Next assume that $\kappa = Q$ is a nondegenerate quadratic form with associated bilinear form $(\cdot,\cdot)$. If $n=2m+1$ is odd, then fix $\B = (e_1,f_1,\dots,e_m,f_m,x)$ such that
\begin{equation} \label{eq:B_odd}
Q(e_i) = Q(f_i) = 0, \quad Q(x) = 1, \quad (e_i,f_j) = \d_{ij}, \quad (e_i,x) = (f_i,x) = 0.
\end{equation} 
Now assume that $n = 2m$ is even. If $\sgn(Q) = +$, then fix $\B^+ = (e_1,f_1,\dots,e_m,f_m)$ such that
\begin{equation} \label{eq:B_plus}
Q^+(e_i) = Q^+(f_i) = 0, \quad (e_i,f_j) = \d_{ij}.
\end{equation} 
If $F = \F_q$ and $\sgn(Q) = -$, then, deviating from \cite{ref:KleidmanLiebeck} and following \cite{ref:GorensteinLyonsSolomon98}, fix $\B^- = (e_1,f_1,\dots,e_{m-1},f_{m-1},u_m,v_m)$ such that
\begin{equation} \label{eq:B_minus}
\begin{array}{c}
Q^-(e_i) = Q^-(f_i) = (e_i,u_m) = (f_i,u_m) = (e_i,v_m) = (f_i,v_m) = 0, \\[5pt]
(e_i,f_j) = \d_{ij}, \quad Q^-(u_m) = Q^-(v_m) = \xi^{q+1}, \quad (u_m,v_m) = \xi^2 + \xi^{-2}
\end{array}
\end{equation}
where $\xi \in \F_{q^2} \setminus \F_q$ satisfies $|\xi| = q+1$ if $q \neq 3$ and $|\xi| = 8$ if $q=3$. (Note that when $q=3$ our definition of the minus-type standard basis corrects that in \cite[Section~2.7]{ref:GorensteinLyonsSolomon98}, where the basis given there is not linearly independent.)

Finally assume that $F = \F_{q^2}$ and $\kappa = (\cdot,\cdot)$ is unitary. We fix two bases. First
\begin{equation}\label{eq:B_u_ortho}
\B_0 = (u_1,\dots,u_n)
\end{equation} 
where $(u_i,u_j) = \d_{ij}$. For the second basis, let $m = \lfloor{\frac{n}{2}}\rfloor$ and fix $e_i = u_{2i-1} + \zeta u_{2i}$ and $f_i = \zeta u_{2i-1} + u_{2i}$ where $\zeta \in \F_{q^2}^\times$ satisfies $\zeta^2-\zeta-1 = 0$. Note that 
\[
(e_i,e_j) = (f_i,f_j) = 0, \quad (e_i,f_j) = \d_{ij}, \quad (e_i,u_n) = (f_i,u_n) = 0 \text{ if $n$ is odd.}
\]
Let $\a,\b \in \F_{q^2}^\times$ satisfy $\a^{q-1} = -1$ and $\b^{q+1} = (-1)^m$ (choose $\a=\b=1$ if $p=2$), and write
\begin{equation} \label{eq:B_u_ef}
\B = 
\left\{
\begin{array}{ll} 
(\a e_1, -\a e_2, \dots, (-1)^{m+1}\a e_m,         f_m, \dots, f_1) & \text{if $n$ is even} \\
(   e_1, -   e_2, \dots, (-1)^{m+1}   e_m, \b u_n, f_m, \dots, f_1) & \text{if $n$ is odd.}  \\
\end{array}
\right.
\end{equation}

\section{Classical algebraic groups} \label{s:p_algebraic}

The finite simple groups of Lie type arise from fixed points of algebraic groups under Steinberg endomorphisms (see \cite[Chapters~1 and~2]{ref:GorensteinLyonsSolomon98}), and this perspective allows us to exploit Shintani descent, which is described in Chapter~\ref{c:shintani}.

Fix a prime $p$. By an \emph{algebraic group} we always mean a linear algebraic group over $\FF_p$. For an indecomposable root system $\Phi$, there exist simple algebraic groups $\Phi^{\mathsf{sc}}$ and $\Phi^{\mathsf{ad}}$ of simply connected and adjoint types, respectively, such that $Z(\Phi^{\mathsf{sc}})$ is finite, $Z(\Phi^{\mathsf{ad}}) = 1$ and $\Phi^{\mathsf{ad}} = \Phi^{\mathsf{sc}}/Z(\Phi^{\mathsf{sc}})$. Moreover, if $X$ is a simple algebraic group with root system $\Phi$, then there exist isogenies $\Phi^{\mathsf{sc}} \to X \to \Phi^{\mathsf{ad}}$ (see \cite[Theorem~1.10.4]{ref:GorensteinLyonsSolomon98}). For ease of notation, we refer to $\Phi^{\mathsf{ad}}$ as $\Phi$.

The classical algebraic groups are given in Table~\ref{tab:algebraic} (see \cite[Theorem~1.10.7]{ref:GorensteinLyonsSolomon98}), where we adopt the notation introduced in Section~\ref{s:p_groups} (but omit reference to the ambient field $\FF_p$). In particular, recall that $\SO_n = \O_n \cap \SL_n$ and $\Omega_n = \SO_n^\circ$.

\begin{table}
\centering
\caption{Simple classical algebraic groups} \label{tab:algebraic}
{\renewcommand{\arraystretch}{1.2}
\begin{tabular}{ccccccc}
\hline
$\Phi$                      & $p$ & $\Phi^{\mathsf{sc}}$ & $\Phi^{\mathsf{ad}}$ & $|Z(\Phi^{\mathsf{sc}})|$ & $\s$     & $(\Phi^{\mathsf{ad}})_\s$ \\        
\hline
$\mathsf{A}_m$ ($m \geq 1$) &     & $\SL_{m+1}$          & $\PSL_{m+1}$         & $(m+1)_{p'}$              & $\p^f$   & $\PGL_{m+1}(q)$           \\[5.5pt]
                            &     &                      &                      &                           & $\g\p^f$ & $\PGU_{m+1}(q)$           \\   
$\mathsf{B}_m$ ($m \geq 2$) & $2$ & $\SO_{2m+1}$         & $\SO_{2m+1}$         & $1$                       &          &                           \\
                            & odd & $\Spin_{2m+1}$       & $\SO_{2m+1}$         & $2$                       & $\p^f$   & $\PSO_{2m+1}(q)$          \\[5.5pt]
$\mathsf{C}_m$ ($m \geq 2$) &     & $\Sp_{2m}$           & $\PSp_{2m}$          & $(p-1,2)$                 & $\p^f$   & $\PGSp_{2m}(q)$           \\[5.5pt]
$\mathsf{D}_m$ ($m \geq 4$) & $2$ & $\Omega_{2m}$        & $\Omega_{2m}$        & $1$                       & $\p^f$   & $\Omega^+_{2m}(q)$        \\
                            &     &                      &                      &                           & $r\p^f$  & $\Omega^-_{2m}(q)$        \\
                            & odd & $\Spin_{2m}$         & $\PSO_{2m}$          & $4$                       & $\p^f$   & $\PDO^+_{2m}(q)$          \\
                            &     &                      &                      &                           & $r\p^f$  & $\PDO^-_{2m}(q)$          \\
\hline
\end{tabular}}
\\[5pt]
{\small Note: $\p$ is $\p_{\B}$ or $\p_{\B^+}$, as appropriate, where $\B$ and $\B^+$ are defined in Section~\ref{s:p_bases}}
\end{table}

By a \emph{Steinberg endomorphism} of an algebraic group $X$, we mean a bijective morphism $\s\:X \to X$ whose fixed point subgroup 
\[
X_\s = \{ x \in X \mid x^\s = x \}
\]
is finite. (In \cite{ref:GorensteinLyonsSolomon98}, Steinberg endomorphisms are assumed to be surjective rather than bijective, but the terminology agrees when $X$ is simple \cite[Proposition~1.15.3]{ref:GorensteinLyonsSolomon98}.)

Let $X$ be a simple algebraic group of adjoint type and let $\s$ be a Steinberg endomorphism of $X$. Then $T = O^{p'}(X_\s)$ is typically a finite simple group (see \cite[Theorem~2.2.7(a)]{ref:GorensteinLyonsSolomon98}) and the groups obtained in this way are the \emph{finite simple groups of Lie type}. In this notation, we say that the \emph{innerdiagonal group} of $T$ is
\begin{equation}
\Inndiag(T) = X_\s. \label{eq:inndiag}
\end{equation}

\begin{definitionx}\label{def:phi_gamma_r}
Let $\mathcal{B}$ be a basis for $\FF_p^n$ and write the elements of $\GL_n(\FF_p)$ as matrices with respect to $\B$.
\begin{enumerate}
\item The \emph{standard Frobenius endomorphism} of $\GL_n(\FF_p)$ with respect to $\B$ is $\p_{\B} \: (x_{ij}) \mapsto (x_{ij}^p)$.
\item The \emph{standard graph automorphism} of $\GL_n(\FF_p)$ with respect to $\B$ is the map $\gamma_{\B}\: x \mapsto (x^{-\tr})^J$, where $J$ is the antidiagonal matrix with entries $1,-1,1,-1,\dots,(-1)^{n+1}$ (from top-right to bottom-left).
\item Let $n=2m$ and $\B=\B^+$ (from \ref{eq:B_plus}). The \emph{standard reflection} $r \in \O^+_n(p)$ is
\[ 
r = I_{n-2} \perp \left( \begin{array}{cc} 0 & 1 \\ 1 & 0 \end{array} \right)
\]
that centralises the decomposition $\<e_1,\dots,f_{m-1}\> \perp \<e_m,f_m\>$. We identify $r$ with the automorphism of $\GO_n(\FF_p)$ that it induces by conjugation.
\end{enumerate}
\end{definitionx}

Observe that each of the graph automorphisms defined in parts (ii) and (iii) of Definition~\ref{def:phi_gamma_r} are involutions and they commute with the standard Frobenius endomorphism defined in part~(i).

\begin{remarkx}\label{rem:phi_gamma_r} Let us allow two notational conveniences.
\begin{enumerate}
\item If the basis $\B$ is understood, then we write $\p = \p_{\B}$.
\item If $\th$ is an endomorphism defined in Definition~\ref{def:phi_gamma_r}, then will identify $\th$ with the map induced on $\th$-stable subgroups of $\GL_n(\FF_p)$ and quotients of such subgroups by $\th$-stable normal subgroups.
\end{enumerate}
\end{remarkx}

Fix $f \geq 1$ and write $q=p^f$. If $X$ is a simple classical algebraic group of adjoint type and $\s$ is a Steinberg endomorphism of $X$ that appears in the sixth column of Table~\ref{tab:algebraic}, then the isomorphism type of $X_\s$ is given in the seventh column of Table~\ref{tab:algebraic}. This is essentially proved in \cite[Section~2.7]{ref:GorensteinLyonsSolomon98}, but we will provide some of the details of the proof, since it will be important later that we understand the group $X_\s$ exactly, not just up to isomorphism.

\begin{lemmax} \label{lem:algebraic_finite}
Let $X$ be a simple classical algebraic group of adjoint type and let $\s$ be a Steinberg endomorphism of $X$ that appears in the sixth column of Table~\ref{tab:algebraic}. Assume that $(X,\s) \neq (\mathsf{D}_m,r\p^f)$. Then $X_\s$ is the group in the seventh column.
\end{lemmax}

\begin{proof}
Assume $(X,\s)$ is $(\mathsf{A}_m,\g\p^f)$ or $(\mathsf{D}_m,\p^f)$; the other cases are similar.

First let $(X,\s) = (\mathsf{A}_m,\g\p^f)$. Let $X=\PSL_n$ and write the elements of $X$ with respect to a fixed basis $\B = (v_1,\dots,v_n)$ for $\FF_p^n$. Write $Z = Z(\GL_n(\FF_p)) \cong \FF_p^\times$. For each $\mu \in \FF_p$, there exists $\l \in \FF_p$ such that $\l^n = \mu$ and hence there exists $\l I_n \in Z$ such that $\det(\l I_n) = \mu$.  Consequently,
\[
X = \PSL_n(\FF_p) = (\SL_n(\FF_p)Z)/Z = \GL_n(\FF_p)/Z = \PGL_n(\FF_p).
\]
If $x \in X_{\g\p^f}$, then $x \in X_{\p^{2f}} = \PGL_n(q^2)$. Moreover, for $x \in \PGL_n(q^2)$
\[
x \in X_{\g\p^f} \iff xJx^{\tr\p^f} = J \iff x\a J x^{\tr\p^f} = \a J,
\] 
where $J$ is the antidiagonal matrix from Definition~\ref{def:phi_gamma_r}(ii) and $\a \in \F_{q^2}^\times$ satisfies $\a^{q-1}=-1$. Observe that the Gram matrix of the nondegenerate unitary form with respect to the the basis $\B$ in \eqref{eq:B_u_ef} is $J$ if $n$ is odd and $\a J$ if $n$ is even, so $X_{\gamma\p^f} = \PGU_n(q)$ in both cases.

Now let $(X,\s) = (\mathsf{D}_m,\p^f)$. First assume that $p=2$. Since $X \leq \O_n(\FF_2)$, we have $X_{\p^f} \leq \O^+_n(q)$. Since $\Omega_n(\FF_2)$ does not contain any reflections, $X_{\p^f} \leq \Omega^+_n(q)$. However,  $|\O_n(\FF_2):\Omega_n(\FF_2)|=2$, so $|\O^+_n(q):X_{\p^f}| \leq 2$. Therefore, $X_{\p^f} = \Omega^+_n(q)$.

Now assume that $p$ is odd. Write $Z = Z(\GO_n(\FF_p)) \cong \FF_p^\times$. Since $\det(\l I_n) = \l^n = \tau(\l I_n)^{n/2}$, we have $Z \leq \DO_n(\FF_p)$. Moreover, for each $\mu \in \FF_p$, there exists $\l \in \FF_p$ such that $\l^2 = \mu$ and hence there exists $\l I_n \in Z$ such that $\tau(\l I_n) = \mu$ and $\det(\l I_n) = \mu^{n/2}$. Consequently, $\SO_n(\FF_p)Z = \DO_n(\FF_p)$ and
\begin{equation} \label{eq:algebraic_finite_plus}
X = \PSO_n(\FF_p) = (\SO_n(\FF_p)Z)/Z = \DO_n(\FF_p)/Z = \PDO_n(\FF_p),
\end{equation}
whence $X_{\p^f} = \PDO^+_n(q)$.
\end{proof}

Next we handle the minus-type orthogonal groups.

\begin{lemmax} \label{lem:algebraic_finite_minus}
Let $X = \mathsf{D}_m$ with $m \geq 4$ and let $\p = \p_{\B^+}$. Then there exists an inner automorphism $\Psi$ of $\GL_{2m}(\FF_p)$ such that $\Psi(X_{r\p^f})$ is $\PDO^-_{2m}(q)$.
\end{lemmax}

\begin{proof}
Let $V = \FF_p^n$ be equipped with the quadratic form $Q$, with bilinear form $(\cdot,\cdot)$, defined in \eqref{eq:B_plus} with respect to the basis $\B^+ = (e_1,f_1,\dots,e_m,f_m)$, where $n=2m$. Let $\Psi$ be the automorphism of $\GL_n(\FF_p)$ induced by conjugation by the element $A = I_{n-2} \perp A'$ that centralises $\<e_1,\dots,f_{m-1}\> \perp \<e_m,f_m\>$, where 
\[
A' = \left(
\begin{array}{cc}
\xi      & \xi^{-1} \\
\xi^{-1} & \xi \\
\end{array}
\right)
\]
and where $\xi \in \F_{q^2} \setminus \F_q$ satisfies $|\xi| = q+1$ if $q \neq 3$ and $|\xi| = 8$ if $q=3$.

Write $u_m = e_mA$ and $v_m = f_mA$. It is straightforward to check that $Q(u_m)=Q(v_m)=\xi^{q+1}$ and $(u_m,v_m) = \xi^2+\xi^{-2}$, so, without loss of generality, we may assume that $\B^+A$ is the basis $\B^-$ defined in \eqref{eq:B_minus}.

Let $\s_\e = (\p_{\B^\e})^f$. A straightforward calculation yields $AA^{-(q)}=r$ where $A = (a_{ij})$ and $A^{(q)} = (a_{ij}^q)$. Consequently, $\Psi(X_{r\s_+}) = A^{-1}X_{r\s_+}A = X_{\s_-}$ for any subgroup $X \leq \GL_n(\FF_p)$. Let $V^\e$ be the $\F_q$-span of $\B^\e$. Then $(V^\e,Q)$ is the $\e$-type formed space from \eqref{eq:B_plus} or \eqref{eq:B_minus}. Therefore, if $X = \SO_n(\FF_p)$, then $X_{\s_+} = \SO^+_n(q)$ and $\Psi(X_{r\s_+}) = X_{\s_-} = \SO^-_n(q)$.

We are ready to prove the main claims of the lemma.

First assume that $p=2$ and $X = \Omega_n(\FF_2)$. We know that $\Psi(Y_{r\p^f}) = \O^-_n(q)$, where $Y = \O_n(\FF_2)$. Since $\Psi^{-1}$ maps the reflections in $\O^-_n(q)$ to reflections in $Y_{r\p^f}$ and $X$ contains no reflections, we conclude that $\Psi(X_{r\p^f}) = \Omega^-_n(q)$.

Now assume that $p$ is odd and $X = \PSO_n(\FF_p)$. We recorded in \eqref{eq:algebraic_finite_plus} that $X = \PSO_n(\FF_p) = \PDO_n(\FF_p)$. The above discussion now implies that $\Psi(X_{r\p^f}) = \PDO^-_n(q)$. This completes the proof.
\end{proof}

We now see the significance of the notation $\DO^\pm_n(q)$ introduced in Section~\ref{s:p_groups}. Namely, in light of Lemmas~\ref{lem:algebraic_finite} and~\ref{lem:algebraic_finite_minus}, with a slight abuse of notation for minus-type groups, for even $n \geq 8$,
\begin{equation}\label{eq:inndiag_o}
\Inndiag(\POm^\pm_n(q)) = \PDO^\pm_n(q).
\end{equation}

\section{Maximal subgroups of classical groups} \label{s:p_subgroups}

An understanding of the subgroup structure of almost simple classical groups will be essential in Chapters~\ref{c:o} and~\ref{c:u}. Let $G$ be an almost simple classical group and let $V = \F_{q^d}^n$ be the natural module for $\soc(G)$, where $q=p^f$ and $d \in \{1,2\}$ (here $d=2$ if and only if $\soc(G) =\PSU_n(q)$). Theorem~\ref{thm:aschbacher} was proved by Aschbacher \cite{ref:Aschbacher84}, except for the special case when $\soc(G)=\POm_8^+(q)$ and $G$ contains a triality automorphism; this latter case was proved by Kleidman \cite{ref:Kleidman87}.

\begin{theoremx}[Aschbacher's Subgroup Theorem] \label{thm:aschbacher}
Let $G$ be an almost simple classical group and let $H$ be a maximal subgroup of $G$ not containing $\soc(G)$. Then $H$ belongs to one of the subgroup collections $\C_1, \dots, \C_8, \S, \N$.
\end{theoremx}

Regarding Theorem~\ref{thm:aschbacher}, notice that the subgroups of $G$ that contain $\soc(G)$ correspond to subgroups of $G/\!\soc(G) \leq \Out(\soc(G))$, which is a well-known soluble group. This explains our focus on maximal subgroups not containing $\soc(G)$.

The collections $\C_1, \dots, \C_8$ contain the \emph{geometric subgroups}, and each such collection corresponds to a different geometric structure on the natural module for $\soc(G)$. We adopt the definition of each $\C_i$ given in \cite[Section~4.$i$]{ref:KleidmanLiebeck}, which differs slightly from Aschbacher's original definition. These eight collections are summarised in Table~\ref{tab:geometric}. Each $\C_i$ collection is a union of \emph{types} of geometric subgroup. The type of a subgroup is a rough indication of both its group theoretic structure and the geometric structure it stabilises; this notion is formally introduced in \cite[p.58]{ref:KleidmanLiebeck}. The main theorem in \cite[Chapter~3]{ref:KleidmanLiebeck} establishes the structure, conjugacy and, when $n \geq 13$, maximality of each geometric subgroup of each almost simple classical group. If $n \leq 12$, then complete information on the maximal subgroups of almost simple classical groups is given in \cite{ref:BrayHoltRoneyDougal}. 

\begin{table}
\centering
\caption{Geometric subgroups} \label{tab:geometric}
\begin{tabular}{ccc}
\hline
       & structure stabilised                                   & rough description in $\GL_n(q)$                            \\
\hline
$\C_1$ & n.d. or totally singular subspace                   & maximal parabolic                                          \\[7pt]
$\C_2$ & $V = \bigoplus_{i=1}^{k}V_i$ where $\dim{V_i}=a$       & $\GL_a(q) \wr S_k$ with $n=ak$                             \\[7pt]
$\C_3$ & prime degree field extension of $\F_q$                 & $\GL_{a}(q^k).k$ with $n=ak$ for prime $k$                 \\[7pt]
$\C_4$ & tensor product $V=V_1 \otimes V_2$                     & $\GL_{a}(q) \circ \GL_{b}(q)$ with $n=ab$                  \\[7pt]
$\C_5$ & prime degree subfield of $\F_q$                        & $\GL_{n}(q_0)$ with $q=q_0^k$ for prime $k$                \\[7pt]
$\C_6$ & symplectic-type $r$-group                              & $(C_{q-1} \circ r^{1+2a}).\Sp_{2a}(r)$ with $n=r^a$        \\[7pt]
$\C_7$ &  $V = \bigotimes_{i=1}^{k}V_i$ where $\dim{V_i}=a$     & $(\GL_a(q) \circ \cdots \circ \GL_a(q)).S_k$ with $n=a^k$  \\[7pt]
$\C_8$ & nondegenerate classical form                           & $\GSp_n(q)$, \ $\GO^{\e}_n(q)$, \ $\GU_n(q^{\frac{1}{2}})$ \\
\hline
\end{tabular}
\end{table}  

If $H \leq G$ is contained in the collection $\S$, then $H$ is almost simple with socle $H_0$ and the embedding $H_0 \leq G$ is afforded by an absolutely irreducible representation $\widehat{H}_0 \to \GL_n(V)$ for some quasisimple extension $\widehat{H}_0$ of $H_0$. If $\soc(G)$ is $\Sp_4(2^f)$ or $\POm_8^+(q)$ additional subgroups arise in a collection $\N$, described in \cite[Table~5.9.1]{ref:BurnessGiudici16}; a feature of the subgroups $H \in \N$ is that they are \emph{novelty}, that is, $H \cap \soc(G)$ is not maximal in $\soc(G)$. 

A key aspect of the proofs in Chapters~\ref{c:o} and~\ref{c:u} is to determine which maximal subgroups of a given almost simple classical group $G$ contain a carefully chosen element $s \in G$. While we cannot typically use the order of $s$ to do this (see Remark~\ref{rem:gpps}), when we can, we use the main theorem of \cite{ref:GuralnickPenttilaPraegerSaxl97}, which we now discuss.

For positive integers $a,b$ such that $a \geq 2$, we say that a positive integer $r$ is a \emph{primitive divisor} of $a^b-1$ if $r$ divides $a^b-1$ but $r$ does not divide $a^k-1$ for any $k < b$. Write $\ppd(a,b)$ for the set of \emph{primitive prime divisors} of $a^b-1$. The following is due to Zsigmondy \cite{ref:Zsigmondy82} (see also \cite[Theorem~3.1.5]{ref:BurnessGiudici16}).
 
\begin{theoremx}\label{thm:zsigmondy}
Let $(a,b)$ be a pair of positive integers satisfying 
\begin{equation}\label{eq:zsigmondy}
\text{$a \geq 2$ and $(a,b) \neq (2,6)$ and $a+1$ is not a power of 2 if $b=2$.} 
\end{equation}
Then there exists a primitive prime divisor of $a^b-1$.
\end{theoremx}
 
The main theorem of \cite{ref:GuralnickPenttilaPraegerSaxl97} classifies the maximal subgroups of $\GL_n(q)$ that contain an element whose order is divisible by a primitive prime divisor of $q^k-1$ for $k > \frac{n}{2}$, and we will use the version given in \cite[Theorem~2.2]{ref:GuralnickMalle12JAMS}.

\section{Computational methods} \label{s:p_computation}

Based on the work of Breuer in \cite{ref:Breuer07}, we implemented an algorithm in \textsc{Magma} \cite{ref:Magma} that takes as input a finite group $G$, an element $s \in G$ and nonnegative integers $k$ and $N$, with the aim of determining whether $s^G$ witnesses $u(G) \geq k$. An overview of this algorithm is given in \cite[Section~2.3]{ref:Harper17} and the relevant code is in Appendix~\ref{c:code}. The computations were carried out in \textsc{Magma}~2.24-4 on a 2.7\,GHz machine with 128\,GB RAM. The largest computation took 472\,s and 417\,MB of memory, and this was to prove that $u(\< \Omega_8^+(4), \th\>) \geq 2$ where $\th$ is an involutory field automorphism.

\chapter{Shintani Descent} \label{c:shintani}

In this chapter, we describe Shintani descent, which is the main technique for understanding the conjugacy classes in almost simple groups. Shintani descent is crucial to this project and also useful more generally. In Section~\ref{s:shintani_intro}, we follow the account given in \cite[Section~2.6]{ref:BurnessGuest13}, and Section~\ref{s:shintani_applications} records some of the key applications of Shintani descent. We hope that this will serve as a general reference for future use, so we prefer to give our own treatment of these existing results and we take the opportunity to set these results in a general context. 

However, for our application, the existing techniques of Shintani descent are not sufficient and we need to develop further results that allow us to handle, for example, twisted groups of Lie type. In particular,  Section~\ref{s:shintani_properties} features three new technical lemmas that explain how we can manipulate Shintani maps, and Section~\ref{s:shintani_substitute} introduces a new result that allows us to use Shintani descent in contexts that previously were not amenable to this approach.

\section{Introduction} \label{s:shintani_intro}

For this entire chapter, let $X$ be a connected algebraic group over $\FF_p$ and let $\s$ be a Steinberg endomorphism of $X$. The following is \cite[Theorem~10.13]{ref:Steinberg68}.

\begin{theoremx}[Lang--Steinberg Theorem] \label{thm:lang_steinberg}
The map $L\:X \to X$ defined as $L(x) = xx^{-\s}$ is surjective.
\end{theoremx}

\begin{corollaryx} \label{cor:lang_steinberg}
The map $L'\:X \to X$ where $L'(x) = xx^{-\s^{-1}}$ is surjective.
\end{corollaryx}

\begin{proof}
Let $g \in X$. Theorem~\ref{thm:lang_steinberg} implies that there exists $x \in X$ such that $g^{-\s} = xx^{-\s}$. Consequently, $g = xx^{-\s^{-1}}$ and $L'$ is surjective.
\end{proof}
 
Fix $e>1$. The subgroup $X_{\s^e}$ is $\s$-stable, so $\s$ restricts to an automorphism $\ws = \s|_{X_{\s^e}}$ of $X_{\s^e}$, and we can consider the finite semidirect product $X_{\s^e}{:}\<\ws\>$, where $g^{\ws} = g^\s$ for all $g \in X_{\s^e}$, noting that $|\ws|=e$.

\begin{definitionx} \label{def:shintani}
A \emph{Shintani map} of $(X,\s,e)$ is a map of conjugacy classes
\[
F\: \{(g\ws)^{X_{\s^e}} \mid g \in X_{\s^e} \} \to \{x^{X_{\s}} \mid x \in X_{\s} \} \quad (g\ws)^{X_{\s^e}} \mapsto (a^{-1}(g\ws)^ea)^{X_{\s}}
\] 
where $a \in X$ satisfies $g=aa^{-\s^{-1}}$ (which exists by Corollary~\ref{cor:lang_steinberg}).
\end{definitionx}

We will often abuse notation by using $F(g\ws)$ to refer to a representative of the $X_\s$-class $F((g\ws)^{X_{\s^e}})$. 

The following theorem establishes the main properties of the Shintani map. It was first proved by Kawanaka in \cite{ref:Kawanaka77}, building on earlier work of Shintani who introduced the key ideas in \cite{ref:Shintani76}. We follow the proof of \cite[Lemma~2.13]{ref:BurnessGuest13}.

\begin{theoremx}[Shintani Descent] \label{thm:shintani_descent}
Let $F$ be a Shintani map of $(X,\s,e)$.
\begin{enumerate}
\item The map $F$ is a well-defined bijection, independent of the choice of $a \in X$.
\item If $g \in X_{\s^e}$ then $C_{X_{\s}}(F(g\ws)) = a^{-1}C_{X_{\s^e}}(g\ws)a$.
\end{enumerate}
\end{theoremx}

\begin{proof}
Let $g \in X_{\s^e}$ and write $g=aa^{-\s^{-1}}$. First note that 
\begin{align*}
a^{-1}(g\s)^ea &= a^{-1}gg^{\s^{-1}}\!\!\!\! \cdots g^{\ws^{-(e-1)}}a \\
               &= a^{-1}(aa^{-{\s^{-1}}})(a^{\s^{-1}}a^{-{\s^{-2}}}) \cdots (a^{\s^{-(e-1)}}a^{-{\s^{-e}}})a = a^{-\s^{-e}}a.
\end{align*}
Since $g=aa^{-\s^{-1}} \in X_{\s^e}$ we know that $aa^{-\s^{-1}} = (aa^{-\s^{-1}})^{\s^{-e}} = a^{\s^{-e}}a^{-\s^{-(e+1)}}$, whence $a^{-\s^{-e}}a = a^{-\s^{-(e+1)}}a^{\s^{-1}} = (a^{-\s^{-e}}a)^{\s^{-1}}$, so $a^{-\s^{-e}}a \in X_{\s}$. 

Let $h\ws$ be $X_{\s^e}$-conjugate to $g\ws$. Fix $k \in X_{\s^e}$ such that $h\ws = k^{-1}(g\ws)k$ and consequently $h = k^{-1}gk^{-\s^{-1}}$. Writing $g = aa^{-\s^{-1}}$, we obtain $h = (k^{-1}a)(k^{-1}a)^{-\s^{-1}}$, whence
\[ 
(k^{-1}a)^{-1}(h\ws)^e(k^{-1}a) = a^{-1}k(h\ws)^ek^{-1}a = a^{-1}(g\ws)^ea.
\] 
Therefore, $F$ does not depend on the choice of representative of the $X_{\s^e}$-class.

Write $g = aa^{-\s^{-1}} = bb^{-\s^{-1}}$. Then $a^{-1}b = a^{-\s^{-1}}b^{\s^{-1}} = (a^{-1}b)^{\s^{-1}}$, so $a^{-1}b \in X_{\s}$ and
\[
b^{-1}(g\ws)^eb = (a^{-1}b)^{-1}(a^{-1}(g\ws)^ea)(a^{-1}b),
\]
so $F$ is independent of the choice of $a$. Therefore, $F$ is a well-defined function.

To see that $F$ is surjective, let $x \in X_{\s}$ and write $x^{-1} = bb^{-\s^{-e}}$. Therefore, writing $a = b^{-1}$, we have $x = a^{-\s^{-e}}a$. As we argued in the first paragraph, $a^{-1}(aa^{-\s^{-1}}\ws)^ea = x$ and $aa^{-\s^{-1}} \in X_{\s^e}$ since $a^{-\s^{-e}}a \in X_{\s}$. We will complete the proof that $F$ is bijective after proving (ii).

Turning to (ii), let $z \in C_{X_{\s^e}}(g\ws)$. Then $a^{-1}za$ centralises $a^{-1}(g\ws)^ea$. The fact that $z \in C_{X_{\s^e}}(g\ws)$ implies that $zg\ws = g\ws z$, so $z^{\s^{-1}} = g^{-1}zg$. Therefore, 
\[ 
(a^{-1}za)^{\s^{-1}} = a^{-\s^{-1}}g^{-1}zga^{\s^{-1}} = a^{-1}gg^{-1}zgg^{-1}a = a^{-1}za.
\]
Therefore, $a^{-1}za \in X_{\s}$, so $a^{-1}za \in C_{X_{\s}}(a^{-1}(g\ws)^ea) = C_{X_{\s}}(F(g\ws))$. This proves that $a^{-1}C_{X_{\s^e}}(g\ws)a \subseteq C_{X_{\s}}(F(g\ws))$. For the reverse inclusion, let $w \in C_{X_{\s}}(F(g\ws))$. Then 
\[
awa^{-1} = (g\ws)^{-e}(awa^{-1})(g\ws)^e = (aa^{-\s^{-1}}\s)^{-e}(awa^{-1})(aa^{-\s^{-1}}\s)^e = (awa^{-1})^{\s^{-e}},
\]
which implies that $awa^{-1} \in X_{\s^e}$. Moreover,
\[
(g\ws)^{-1}(awa^{-1})(g\ws) = (\s^{-1}a^{\s^{-1}}a^{-1})awa^{-1}(aa^{-\s^{-1}}\s) = aw^{\s^{-1}}a^{-1} = awa^{-1},
\]
so $awa^{-1} \in C_{X_{\s^e}}(g\ws)$. This implies that $a^{-1}C_{X_{\s^e}}(g\ws)a = C_{X_{\s}}(F(g\ws))$, as claimed.

We may now prove that $F$ is bijective. Let $\{c_1, \dots, c_t\}$ be representatives of the $X_{\s}$-classes in $X_{\s}$. Then there exist $X_{\s^e}$-classes $C_1,\dots,C_t$ in $X_{\s^e}\ws$ such that $F(C_i)=c_i$ for each $i$, by the surjectivity of $F$. By (ii), $|C_i|=|c_i||X_{\s^e}:X_{\s}|$. This implies that
\[ 
\sum_{i=1}^{t} |C_i| = |X_{\s^e}:X_{\s}| \sum_{i=1}^{t}|c_i| = |X_{\s^e}| = |X_{\s^e}\ws|, 
\] 
so $\{C_1,\dots,C_t\}$ is the set of $G$-classes in $X_{\s^e}\ws$, which proves that $F$ is bijective.
\end{proof}

Theorem~\ref{thm:shintani_descent}(i) justifies our reference to $F$ as \emph{the} Shintani map of $(X,\s,e)$.

The following concrete example highlights how we apply Shintani descent.

\begin{examplex}\label{ex:shintani_descent}
Let $e \geq 2$, let $m \geq 4$ and let $q=2^e$. Write $X = \Omega_{2m}(\FF_2)$. Let $\p = \p_{\B^+}$ be the standard Frobenius endomorphism $(a_{ij}) \mapsto (a_{ij}^2)$ of $X$ (see Definition~\ref{def:phi_gamma_r}(i)).

Let $F$ be the Shintani map of $(X,\p,e)$. Note that $X_\p = \Omega^+_{2m}(2)$ and that $X_{\p^e} = \Omega^+_{2m}(q)$. Now 
\[
F \: \{ (g\p)^{\Omega^+_{2m}(q)} \mid g \in \Omega^+_{2m}(q) \} \to \{ x^{\Omega^+_{2m}(2)} \mid x \in \Omega^+_{2m}(2) \}.
\]
Therefore, we can specify a conjugacy class in the coset $\Omega^+_{2m}(q)\p$ of the almost simple group $\<\Omega^+_{2m}(q), \p\>$ as the preimage under $F$ of a conjugacy class in $\Omega^+_{2m}(2)$.

Recall the element $r$ from Definition~\ref{def:phi_gamma_r}(iii). Let $E$ be the Shintani map of $(X,r\p,e)$. Then $X_{r\p} \cong \Omega^-_{2m}(2)$ and $X_{(r\p)^e} \cong \Omega^\e_{2m}(q)$ where $\e=(-)^e$. Therefore, the map
\[
E \: \{ (gr\p)^{\Omega^\e_{2m}(q)} \mid g \in \Omega^\e_{2m}(q) \} \to \{ x^{\Omega^-_{2m}(2)} \mid x \in \Omega^-_{2m}(2) \}
\]
allows us, for example, to specify elements in the coset $\Omega^+_{2m}(q)r\p$ of $\<\Omega^+_{2m}(q), r\p\>$ when $e$ is even. However, this setup does not shed light on this coset when $e$ is odd. This is problematic as we will need to select an element in this coset in order to study the uniform spread of $\< \Omega^+_{2m}(q), r\p\>$, and this shows the limitations of the current Shintani descent techniques. In Example~\ref{ex:shintani_substitute}, we will see how to handle this case using our new methods.
\end{examplex}

\section{Properties} \label{s:shintani_properties}

In this section, we will establish three new properties of the Shintani map, which justify techniques that we repeatedly employ. Each of these properties relies on the fact that the Shintani map does not depend on the choice of element afforded by the Lang--Steinberg Theorem (see Theorem~\ref{thm:shintani_descent}(i)). 

Throughout, we assume that $X$ is a connected algebraic group, $\s$ is a Steinberg endomorphism of $X$ and $e > 1$. Let $F$ be the Shintani map of $(X,\s,e)$ and let $\ws = \s|_{X_{\s^e}}$.

We begin with a preliminary observation. If $Y$ is a closed $\s$-stable subgroup of $X$, then the restriction $\s_Y$ of $\s$ to $Y$ is a Steinberg endomorphism. Similarly, if $\pi\: X \to  Y$ is an isogeny with a $\s$-stable kernel, then $\s$ induces a Steinberg endomorphism $\s_Y$ on $Y$ such that $\s_Y \circ \pi = \pi \circ \s$. For ease of notation, in both cases we write $\s$ for $\s_Y$. 

The first property concerns subgroups (an application is Proposition~\ref{prop:o_Ia_max_reducible}).
\begin{lemmax} \label{lem:shintani_subgroups}
Let $Y$ be a closed connected $\s$-stable subgroup of $X$ and let $E$ be the Shintani map of $(Y,\s,e)$.
\begin{enumerate}
\item For all $g \in Y_{\s^e}$, any representative of $E((g\ws)^{Y_{\s^e}})$ represents $F((g\ws)^{X_{\s^e}})$.
\item For all $x \in Y_\s$, any representative of $E^{-1}(x^{Y_\s})$ represents $F^{-1}(x^{X_\s})$.
\end{enumerate}
\end{lemmax}

\begin{proof}
We prove only (i) since (ii) is very similar. Let $g \in Y_{\s^e}$ and let $x$ be a representative of $E((g\ws)^{Y_{\s^e}})$. Then $x = a^{-1}(g\ws)^ea$ for an element $a \in Y$ such that $aa^{-\s^{-1}} = g$. Since $Y \leq X$, the element $a^{-1}(g\ws)^ea=x$ represents $F((g\ws)^{X_{\s^e}})$. 
\end{proof}

The second property concerns quotients.
\begin{lemmax} \label{lem:shintani_quotients}
Let $\pi\: X \to Y$ be an isogeny with a $\s$-stable kernel and let $E$ be the Shintani map of $(Y,\s,e)$. 
\begin{enumerate}
\item For all $h \in \pi(X_{\s^e}) \leq Y_{\s^e}$, there exists $y \in \pi(X_\s) \leq Y_\s$ that represents the class $E(h\ws)$.
\item For all $y \in \pi(X_\s) \leq Y_\s$, there exists $h \in \pi(X_{\s^e}) \leq Y_{\s^e}$ such that $h\ws$ represents the class $E^{-1}(y)$.
\end{enumerate}
Moreover, if $\< \pi(X_{\s^e}), \ws \> \leqn \< Y_{\s^e}, \ws \>$ and $\pi(X_\s) \leqn Y_\s$, then $E$ restricts to a bijection
\[
E_1\: \{ (h\ws)^{Y_{\s^e}} \mid h \in \pi(X_{\s^e}) \} \to \{ y^{Y_\s} \mid y \in \pi(X_\s) \}.
\]
\end{lemmax}

\begin{proof}
For (i), let $g \in X_{\s^e}$ and let $x$ be a representative of $F(g\ws)$. Then $x = a^{-1}(g\ws)^ea$ for an element $a \in X$ such that $aa^{-\s^{-1}} = g$. Therefore, we have $\pi(x) = \pi(a)^{-1}(\pi(g)\ws)^e\pi(a)$. Note that $\pi(x) \in \pi(X_\s) \leq Y_\s$. Moreover, $\pi(a) \in Y$ and $\pi(a)\pi(a)^{-\s^{-1}} = \pi(g)$, so $\pi(a)^{-1}(\pi(g)\ws)^e\pi(a)=\pi(x)$ is a valid representative of $E(\pi(g)\ws)$, as required. As with Lemma~\ref{lem:shintani_subgroups}, (ii) is similar to (i).

If $\< \pi(X_{\s^e}), \ws \> \leqn \< Y_{\s^e}, \ws \>$ and $\pi(X_\s) \leqn Y_\s$, then for all $h \in \pi(X_{\s^e})$ and for all $y \in \pi(X_\s)$ we have $(h\ws)^{Y_{\s^e}} \subseteq \pi(X_{\s^e})\ws$ and $y^{Y_\s} \subseteq \pi(X_\s)$, which implies, given (i) and (ii), that $E$ restricts to the bijection $E_1$.
\end{proof}

\begin{corollaryx} \label{cor:shintani_quotients}
Let $Y$ be a simple algebraic group of adjoint type, let $\s$ be a Steinberg endomorphism of $Y$ and let $e > 1$. Write $T = O^{p'}(Y_{\s^e})$ and assume that $\< T, \ws \> \leqn \< Y_{\s^e}, \ws \>$. Then the Shintani map of $(Y,\s,e)$ restricts to a bijection
\[
\{ (t\ws)^{Y_{\s^e}} \mid t \in T \} \to \{ y^{Y_\s} \mid y \in O^{p'}(Y_\s) \}.
\]  
\end{corollaryx}

\begin{proof}
Let $X$ be the simply connected version of $Y$, so $Y = X/Z(X)$, and let $\pi\: X \to Y$ be the isogeny arising from taking the quotient by $Z(X)$. By \cite[Theorem~2.1.2(e)]{ref:GorensteinLyonsSolomon98}, since $X$ is simply connected, there is a unique Steinberg endomorphism $\s_X$ of $X$ such that $\pi\circ\s_X = \s\circ\pi$, so in particular, $\ker(\pi)$ is $\s_X$-stable. As usual, for ease of notation, we write $\s = \s_X$. We aim to apply Lemma~\ref{lem:shintani_quotients}, with the isogeny $\pi\:X \to Y$ and the Shintani maps $F$ and $E$ of $(X,\s,e)$ and $(Y,\s,e)$, respectively. By \cite[Proposition~2.5.9 and Theorem~2.2.6(c)]{ref:GorensteinLyonsSolomon98}, 
\begin{gather}
\pi(X_\s)     = X_\s/Z(X_\s) = O^{p'}(Y_\s) \leqn Y_\s \\
\pi(X_{\s^e}) = X_{\s^e}/Z(X_{\s^e}) = O^{p'}(Y_{\s^e}) = T,
\end{gather}
and, by hypothesis, $\< T, \ws \> \leqn \< Y_{\s^e}, \ws \>$. Therefore, Lemma~\ref{lem:shintani_quotients} implies that $E$ restricts to the bijection
\[
\{ (t\ws)^{Y_{\s^e}} \mid t \in T \} \to \{ y^{Y_\s} \mid y \in O^{p'}(Y_\s) \}. \qedhere
\] 
\end{proof}

The following example elucidates the utility of Corollary~\ref{cor:shintani_quotients} and it provides an alternative proof of \cite[Proposition~2.4]{ref:Harper17} (see also Lemmas~\ref{lem:o_Ia_tau} and~\ref{lem:o_Ia_eta}).

\begin{examplex} \label{ex:shintani_quotients}
Let $m \geq 2$, let $p$ be an odd prime and let $q=q_0^e=p^f$, where $e \geq 2$ divides $f$. Write $Y = \PSp_{2m}(\FF_q)$ and let $\s=\p^{f/e}$ where $\p = \p_{\B^+}$ is the standard Frobenius endomorphism. The Shintani map $E$ of $(Y,\s,e)$ is
\[
E\: \{ (g\ws)^{\PGSp_{2m}(q)} \mid g \in \PGSp_{2m}(q) \} \to \{ x^{\PGSp_{2m}(q_0)} \mid x \in \PGSp_{2m}(q_0) \}.
\]
The map $E$ allows us to identify a $\PGSp_{2m}(q)$-class $(g\ws)^{\PGSp_{2m}(q)}$ in the coset $\PGSp_{2m}(q)\ws$ by specifying a conjugacy class $x^{\PGSp_{2m}(q_0)}$ of $\PGSp_{2m}(q_0)$, but we do not know which coset of $\PSp_{2m}(q)$ this class is contained in. However, Corollary~\ref{cor:shintani_quotients} implies that $E$ restricts to the bijection
\[
\{ (t\ws)^{\PGSp_{2m}(q)} \mid t \in \PSp_{2m}(q) \} \to \{ y^{\PGSp_{2m}(q_0)} \mid y \in \PSp_{2m}(q_0) \},
\]
which informs us that $g\ws \in \PSp_{2m}(q)\ws$ if and only if $x \in \PSp_{2m}(q_0)$. 
\end{examplex}

We conclude with a property that relates Shintani maps to taking powers.

\begin{lemmax}\label{lem:shintani_powers}
Let $x \in X_{\s}$ and assume that $F(g\ws) = x^{X_{\s}}$. Let $d > 1$.
\begin{enumerate}
\item Let $E_1$ be the Shintani map of $(X,\s^d,e)$. Then $E_1((g\ws)^d) = (x^d)^{X_{\s^d}}$.
\item Assume $d$ is a proper divisor of $e$ and let $E_2$ be the Shintani map of $(X,\s^d,e/d)$. Then $E_2((g\ws)^d) = x^{X_{\s^d}}$.
\end{enumerate}
\end{lemmax}

\begin{proof}
Assume that $g \in X_{\s^e}$ satisfies $F(g\ws) = x$. Fix an element $a \in X$ satisfying $a^{-1}(g\ws)^ea = x$ and $aa^{-\s^{-1}} = g$. Write 
\[
h = gg^{\s^{-1}}\cdots g^{\s^{-(d-1)}}.
\]
Then $(g\ws)^d = h\ws^d$ and $h= aa^{-\s^{-d}}$. Therefore, 
\[
E_1((g\ws)^d) = E_1(h\ws^d) = a^{-1}(h\ws^d)^ea = a^{-1}(g\ws)^{de}a = x^d,
\]
and if $d$ is a proper divisor of $e$, then also
\[
E_2((g\ws)^d) = E_2(h\ws^d) = a^{-1}(h\ws^d)^{e/d}a = a^{-1}(g\ws)^ea = x,
\]
which completes the proof.
\end{proof}

\begin{remarkx}\label{rem:shintani_descent}
Let $g,h \in X_{\s^e}$. If $g\ws$ and $h\ws$ are $\<X_{\s^e},\ws\>$-conjugate, then there exist $k \in X_{\s^e}$ and an integer $i$ such that 
\[ 
h\ws = (k\ws^i)^{-1}g\ws(k\ws^i) = (h\ws)^i(k\ws^i)^{-1}g\ws(k\ws^i)(h\ws)^{-i},
\]
but $(k\ws^i)(h\ws)^{-i} \in X_{\s^e}$, so $g\ws$ and $h\ws$ are $X_{\s^e}$\=/conjugate. In particular,
\begin{equation}
|C_{ \< X_{\s^e}, \ws \>}(g\ws)| = e|C_{X_{\s^e}}(g\ws)|. \label{eq:centraliser}
\end{equation}
\end{remarkx}

\section{Applications} \label{s:shintani_applications}

Theorem~\ref{thm:shintani_descent}(ii) demonstrates that Shintani maps preserve important group theoretic data. We now exploit this by providing three applications of Shintani descent to determining maximal overgroups of elements. We continue to assume that $X$ is a connected algebraic group, $\s$ is a Steinberg endomorphism of $X$,  $e > 1$, $F$ is the Shintani map of $(X,\s,e)$ and $\ws = \s|_{X_{\s^e}}$. 

We begin with an important theorem of Shintani descent \cite[Theorem~2.14]{ref:BurnessGuest13}.

\begin{theoremx} \label{thm:shintani_descent_fix}
Let $Y$ be a closed connected $\s$-stable subgroup of $X$. For all elements $g \in X_{\s^e}$,
\[
\fix(g\ws, X_{\s^e}/Y_{\s^e}) = \fix(F(g\ws), X_\s/Y_\s).
\]
\end{theoremx}

The first application extends \cite[Corollary~2.15]{ref:BurnessGuest13} to the natural general setting of Shintani descent.

\begin{lemmax}\label{lem:shintani_descent_fix}
Let $Y$ be a closed connected $\s$-stable subgroup of $X$ such that $N_{X_{\s}}(Y_{\s}) = Y_{\s}$ and $N_{X_{\s^e}}(Y_{\s^e}) = Y_{\s^e}$. For all $g \in X_{\s^e}$, the number of $X_{\s^e}$-conjugates of $Y_{\s^e}$ normalised by $g\ws$ equals the number of $X_{\s}$-conjugates of $Y_{\s}$ containing $F(g\ws)$.
\end{lemmax}

\begin{proof}
Since $Y_{\s^e}$ is $\s$-stable and $N_{X_{\s^e}}(Y_{\s^e}) = Y_{\s^e}$, the conjugation action of $\<X_{\s^e},\ws\>$ on the set of $X_{\s^e}$-conjugates of $Y_{\s^e}$ is equivalent to the action of $\<X_{\s^e},\ws\>$ on cosets of $Y_{\s^e}$ in $X_{\s^e}$. Therefore, the number of $X_{\s^e}$-conjugates of $Y_{\s^e}$ normalised by $g\ws$ is $\fix(g\ws,X_{\s^e}/Y_{\s^e})$. Similarly, the number of $X_{\s}$-conjugates of $Y_{\s}$ containing $F(g\ws)$ is $\fix(F(g\ws),X_{\s}/Y_{\s})$. The result now follows from Theorem~\ref{thm:shintani_descent_fix}. 
\end{proof}

The following example demonstrates a typical application of Lemma~\ref{lem:shintani_descent_fix}.

\begin{examplex} \label{ex:shintani_descent_fix}
Let $n \geq 2$ and let $q=q_0^e=p^f$ where $e \geq 2$ divides $f$. Let $X = \SL_n(\FF_p)$ and let $\s = \p^{f/e}$, where $\p$ is the standard Frobenius endomorphism $(a_{ij}) \mapsto (a_{ij}^p)$ of $X$, with respect to some fixed basis $\B = (u_1,\dots,u_n)$ for $\FF_p^n$. Let $F$ be the Shintani map of $(X,\s,e)$. Note that $X_{\s} = \SL_n(q_0)$ and $X_{\s^e} = \SL_n(q)$.

Let $1 \leq k < n$. We may fix a $\s$-stable maximal $P_k$ parabolic subgroup $Y \leq X$; for example, let $Y$ be the stabiliser in $X$ of the subspace $\<u_1,\dots,u_k\>$. In particular, $Y$ is a closed connected subgroup of $X$. Moreover, $N_{X_\s}(Y_\s) = Y_\s$ and $N_{X_{\s^e}}(Y_{\s^e}) = Y_{\s^e}$, so we are in a position to apply Lemma~\ref{lem:shintani_descent_fix}. 

Let $g \in X_{\s^e}$. By Lemma~\ref{lem:shintani_descent_fix}, the number of $\SL_n(q)$-conjugates of $Y_{\s^e}$ normalised by $g\ws$ equals the number of $\SL_n(q_0)$-conjugates of $Y_\s$ containing $F(g\ws)$.  

There is a unique $\SL_n(q)$-class of maximal subgroups of $G = \< \SL_n(q), \ws \>$ of type $P_k$ and this class is represented by $H = \< Y_{\s^e}, \ws \>$ (see, for example, \cite[Proposition~4.1.17]{ref:KleidmanLiebeck}). In addition, for each $x \in \SL_n(q)$, the element $g\ws$ is contained in $H^x$ if and only if $g\ws$ normalises $Y_{\s^e}^x$. Therefore, the number of $G$-conjugates of $H$ containing $g\ws$ equals the number of $\SL_n(q_0)$-conjugates of $Y_\s$ containing $F(g\ws)$. 
\end{examplex}

Example~\ref{ex:shintani_descent_fix} highlights the key idea of Shintani descent: we can deduce information about $g\ws$ from information about $F(g\ws)$.

Our second application is a minor generalisation of \cite[Proposition~2.16(i)]{ref:BurnessGuest13}. Here we write $\widetilde{G}=X_{\s^e}{:}\<\ws\>$.

\begin{lemmax} \label{lem:centraliser_bound}
Let $g \in \widetilde{G}$ and let $H \leq \widetilde{G}$. Then $g\ws$ is contained in at most $|C_{X_{\s}}(F(g\ws))|$ distinct $\widetilde{G}$-conjugates of $H$. 
\end{lemmax}

\begin{proof}
By Lemma~\ref{lem:fpr_subgroups}, the number of $\widetilde{G}$-conjugates of $H$ that contain $g\ws$ is 
\[
N = \frac{|(g\ws)^{\widetilde{G}} \cap H|}{|(g\ws)^{\widetilde{G}}|} \cdot \frac{|\widetilde{G}|}{|N_{\widetilde{G}}(H)|} = \frac{|(g\ws)^{\widetilde{G}} \cap H||C_{\widetilde{G}}(g\ws)|}{|N_{\widetilde{G}}(H)|}.
\]
First note that $(g\ws)^{\widetilde{G}} \subseteq X_{\s^e}g\ws$, and for $0 \leq i < e$, the cosets $(X_{\s^e} \cap H)(g\ws)^i$ in $H$ are distinct. Therefore, $|(g\ws)^{\widetilde{G}} \cap H| \leq |H|/e$. Next, by \eqref{eq:centraliser} and Theorem~\ref{thm:shintani_descent}(ii), 
\[
|C_{\widetilde{G}}(g\ws)| = |C_{X_{\s^e}}(g\ws)|e = |C_{X_{\s}}(F(g\ws))|e.
\]
Together these observations give 
\[\
N \leq \frac{|H||C_{X_\s}(F(g\ws))|e}{e|N_{\widetilde{G}}(H)|} \leq |C_{X_\s}(F(g\ws))|. \qedhere
\]
\end{proof}

The third application is based on \cite[Proposition~2.16(ii)]{ref:BurnessGuest13} and is more specialised than the previous two. To state this result, we need to fix some notation. 

Let $X$ be a simple classical algebraic group of adjoint type. Let $\s = \rho\p^i$, where $\p$ is a standard Frobenius endomorphism of $X$ and one of the following holds
\begin{enumerate}
\item $\rho$ is trivial,
\item $X = \mathsf{A}_m$ and $\rho$ is the standard graph automorphism $\gamma$
\item $X = \mathsf{D}_m$ and $\rho$ is the reflection $r$
\end{enumerate}
(see Definition~\ref{def:phi_gamma_r}). Let $x \in X_\s$ act on the natural module for $X_\s$ as $A_1 \oplus \cdots \oplus A_k$, where for each $1 \leq i \leq k$, one of the following holds
\begin{enumerate}
\item $A_i$ is irreducible on a $d_i$-space
\item $(X,\rho) \neq (\mathsf{A}_m,1)$ and $A_i = B_i \oplus B_i^*$, where $A_i$ stabilises a dual pair of totally singular $d_i$-spaces and $B_i$ and $B_i^*$ are irreducible and not similar.
\end{enumerate}

\begin{lemmax}\label{lem:shintani_subfield}
Let $g \in X_{\s^e}$ such that $F(g\ws)$ is the element $x$ defined above. If $e$ is prime and $(d_i,d_j) = 1$ when $i \neq j$, then the number of $X_{\s^e}$-conjugates of $X_{\s}$ normalised by $g\ws$ is at most $e^k$.
\end{lemmax}

\begin{proof}
Write $H = X_\s$ and $\widetilde{H} = N_{\widetilde{G}}(H)$, noting that $\widetilde{H} = H \times \< \ws \>$ since $H$ is adjoint. For a subset $S \subseteq \FF_q^\times$, let $S^*$ be $S^{-\s}$ if $H$ is a unitary group and $S^{-1}$ otherwise. The restrictions on $F(g\ws)$ in the statement imply that the eigenvalue multiset (over $\FF_p$) of $F(g\ws)$ is $S_1 \cup \cdots \cup S_k$ where $S_i$ is either $\Lambda_i$ or $\Lambda_i \cup \Lambda_i^*$ where $\Lambda_i = \{\l_i, \dots, \l_i^{q_0^{d_i-1}} \}$, and $\Lambda_i \neq \Lambda_i^*$ in the latter case. 

Let $h\ws \in \widetilde{H}$ be $\widetilde{G}$-conjugate to $g\ws$. Then $F(h\ws)$ is $X_{\s}$-conjugate to $F(g\ws)$. Let the eigenvalue multiset of $h \in H$  be $\{ \a_1, \dots, \a_n \}$. Therefore, the eigenvalue multiset of $F(h\ws)$ is the eigenvalue multiset of $(h\ws)^e = h^e$, which is $\{ \a_1^e, \dots, \a_n^e \}$. Therefore, without loss of generality, $\a_i^e = \l_i$ for each $1 \leq i \leq k$. Now note that $\a_1, \dots, \a_k$ determine all of the eigenvalues of $h$. Thus, there are $e^k$ choices for the eigenvalues of $h$ and consequently $e^k$ choices for $h$ and, hence, $h\ws$ up to $H$-conjugacy. Therefore, $(g\ws)^{\widetilde{G}} \cap \widetilde{H}$ splits into $e^k$ $H$-classes. Since $(d_i,d_j)=1$ for $i \neq j$, we know that $h$ stabilises the same type of decomposition of $\F_{q_0}^n$ as $g$, acting irreducibly on the corresponding summands. Therefore, $|C_{X_\s}(h\ws)| = |C_{X_\s}(h)| = |C_{X_\s}(F(g\ws))|$. Consequently, the $H$-classes into which $g\ws^{\widetilde{G}} \cap \widetilde{H}$ splits have size $|F(g\ws)^H|$. 

By Lemma~\ref{lem:fpr_subgroups}, the number of $\widetilde{G}$-conjugates of $\widetilde{H}$ which contain $g\ws$ is 
\[
\frac{|(g\ws)^{\widetilde{G}} \cap \widetilde{H}|}{|(g\ws)^{\widetilde{G}}|}\frac{|\widetilde{G}|}{|\widetilde{H}|} \leq \frac{e^k|F(g\ws)^{X_{\s}}||C_{\widetilde{G}}(g\ws)|}{|\widetilde{H}|} = \frac{e^k|X_{\s}||C_{\widetilde{G}}(g\ws)|}{|\widetilde{H}||C_{X_{\s}}(F(g\ws))|} \leq e^k. \qedhere
\]
\end{proof}

\section{Generalisation} \label{s:shintani_substitute}

We saw in Example~\ref{ex:shintani_descent} that there are situations that we will encounter in the proof of our main theorems where Theorem~\ref{thm:shintani_descent} alone does not allow us to identify and work with elements in the relevant coset of the almost simple group. These are the situations that we will encounter in Sections~\ref{ss:o_Ib}, \ref{ss:o_IIb} and~\ref{ss:u_Ib}. We now provide a useful new result that allows us to handle these cases. 

Let $X$ be a connected algebraic group, $\s$ a Steinberg endomorphism of $X$ and $e > 1$. Further, let $\rho$ be an automorphism (of algebraic groups) of $X$. 

Suppose that we wish to identify an element in the coset $X_{\rho\s^e}\ws$ (compare with Example~\ref{ex:shintani_descent}). The rough idea of Lemma~\ref{lem:shintani_substitute} is that while we cannot find elements $x \in X_{\r\s}$ and $g \in X_{\r\s^e}$ such that $(g\ws)^e$ is $X$-conjugate to $x$, if we insist that $x$ is contained in $(C_X(\rho)^\circ)_{\r\s} \leq X_{\r\s}$, then there does exist $g \in X_{\r\s^e}$ such that $(g\ws)^e$ is $X$-conjugate to $x\wrho^{-1}$.

\begin{lemmax}\label{lem:shintani_substitute}
Let $Z$ be a closed connected $\s$-stable subgroup of $X$ contained in $C_{X}(\rho)$. Let $G = X_{\rho\s^e}{:}\<\wrho,\ws\>$ where $\ws = \s|_{X_{\rho\s^e}}$ and $\wrho = \rho|_{X_{\rho\s^e}}$. Let $x \in Z_{\s} \leq X_{\rho\s^e}$. 
\begin{enumerate}
\item There exists $g \in Z_{\s^e} \leq X_{\rho\s^e}$ such that $(g\ws)^e$ and $x\wrho^{-1}$ are $X$-conjugate elements of $G$. 
\item Moreover, if there exists a positive integer $d$ such that $(\rho\s^e)^d = \s^{ed}$, then for any $g$ as in (i) the following hold.
\begin{itemize}
\item[{\rm (a)}] Let $H \leq \< X_{\rho\s^e}, \ws\>$. Then the number of $X_{\rho\s^e}$-conjugates of $H$ that contain $g\ws$ is at most $|C_{X_{\s}}(x^d)|$.
\item[{\rm (b)}] Let $Y$ be a closed connected $\s$-stable subgroup $X$ such that $N_{X_{\s}}(Y_{\s}) = Y_{\s}$ and $N_{X_{\s^{de}}}(Y_{\s^{de}}) = Y_{\s^{de}}$. Then the number of $X_{\s^{de}}$-conjugates of $Y_{\s^{de}}$ normalised by $g\ws$ equals the number of $X_{\s}$-conjugates of $Y_{\s}$ containing $x^d$.
\end{itemize}
\end{enumerate}
\end{lemmax}

\begin{proof}
Let $F$ be the Shintani map of $(Z,\s,e)$ and fix $x \in Z_\s$. Let $\widehat{\s} = \s|_{Z_{\s^e}}$, noting that $\widehat{\s}^e=1$. By Theorem~\ref{thm:shintani_descent} applied to $F$, there exists $g \in Z_{\s^e}$ such that 
\[
a^{-1} (g\widehat{\s})^e a = a^{-1}(gg^{\s^{e-1}}g^{\s^{e-2}} \cdots g^{\s})a = x
\] 
as elements of $Z_{\s^e}{:}\<\widehat{\s}\>$, where $a \in Z \leq X$ satisfies $aa^{-\s^{-1}}=g$. Now $g \in Z_{\s^e} \leq X_{\rho\s^e}$ and $\ws^e=\wrho^{-1}$ as an element of $G = X_{\rho\s^e}{:}\<\ws,\wrho\>$. Therefore, as elements of $G$, 
\[ 
a^{-1}(g\ws)^ea = a^{-1}(gg^{\s^{e-1}}g^{\s^{e-2}} \cdots g^{\s})\ws^ea = a^{-1}(gg^{\s^{e-1}}g^{\s^{e-2}} \cdots g^{\s})\wrho^{-1}a = x\wrho^{-1}.
\] 
This proves part (i).

Now turn to part (ii) and assume that $(\rho\s^e)^d = \s^{ed}$. Let $E$ be the Shintani map of $(X,\s,de)$, recording that $Z_{\s} \leq X_{\s}$ and $X_{\rho\s^e} \leq X_{(\rho\s^e)^d} = X_{\s^{de}}$. Write $\overline{\s} = \s|_{X_{\s^{de}}}$. Since $\overline{\s}|_{X_{\rho\s^e}} = \ws$ and $|\overline{\s}|=de=|\ws|$, we can consider $\< X_{\rho\s^e}, \ws\>$ as a subgroup of $\< X_{\s^{de}}, \overline{\s}\>$, where we identify $\ws$ with $\overline{\s}$. Consequently, $E(g\ws) = a^{-1}(g\ws)^{de}a = x^d$. By Lemma~\ref{lem:centraliser_bound}, if $H \leq \< X_{\rho\s^e}, \ws\>$, then the number of $X_{\s^{de}}$-conjugates of $H$ containing $g\ws$ is at most $|C_{X_{\s}}(x^d)|$, which implies (a). If $Y$ is a connected $\s$-stable subgroup $X$ such that $N_{X_{\s}}(Y_{\s}) = Y_{\s}$ and $N_{X_{\s^{de}}}(Y_{\s^{de}}) = Y_{\s^{de}}$, then Lemma~\ref{lem:shintani_subgroups} implies that the number of $X_{\s^{de}}$-conjugates of $Y_{\s^{de}}$ normalised by $g\ws$ equals the number of $X_{\s}$-conjugates of $Y_{\s}$ containing $x^d$, as claimed in (b).
\end{proof}

\begin{examplex} \label{ex:shintani_substitute}
This continues Example~\ref{ex:shintani_descent}. Let $e \geq 3$ be odd, $m \geq 4$ and $q=2^e$. Write $X = \Omega_{2m}(\FF_2)$ and recall the standard Frobenius endomorphism $\p$ and the involutory automorphism $r$. The existing Shintani descent methods did not provide information about the coset $\Omega^+_{2m}(q)r\p$ of $\<\Omega^+_{2m}(q), r\p\>$. We now use Lemma~\ref{lem:shintani_substitute} to overcome this obstacle.

Let $Z \cong \Omega_{2m-2}(\FF_2)$ be the subgroup of $X$ that centralises $\<e_1,\dots,f_{m-1}\> \perp \<e_m,f_m\>$ and acts trivially on the second summand. Evidently $Z \leq C_X(r)$. Therefore, Lemma~\ref{lem:shintani_substitute}(i) implies that for all $x \in Z_{r\p} \cong \Omega^+_{2m-2}(2)$, there exists $g \in X_{r(r\p)^e} = \Omega^+_{2m}(q)$ such that $(gr\p)^e$ is $X$-conjugate to $xr$. Crucially, parts~(a) and~(b) of Lemma~\ref{lem:shintani_substitute}(ii) translate information about $x$ into information about $gr\p$, so, in this way, we can select and work with elements in the coset $\Omega^+_{2m}(q)r\p$.
\end{examplex}

\chapter{Fixed Point Ratios} \label{c:fpr}

This chapter presents upper bounds on fixed point ratios that we will use as part of the probabilistic method we described in Section~\ref{s:p_prob}. Much is known about fixed point ratios for primitive actions of almost simple groups. One reason for this is the important applications these bounds have to a diverse range of problems, such as monodromy groups and base sizes of permutation groups, via probabilistic methods (see \cite{ref:Burness16}).

For groups of Lie type, the most general such bound is \cite[Theorem~1]{ref:LiebeckSaxl91} of Liebeck and Saxl, which establishes that 
\begin{equation}\label{eq:fpr}
\fpr(x,G/H) \leq \frac{4}{3q}
\end{equation}
for any almost simple group of Lie type $G$ over $\F_q$, core-free maximal subgroup $H$ of $G$ and nontrivial element $x \in G$, unless $\soc(G) \in \{ \PSL_2(q) \} \cup \{ \PSL_4(2), \PSp_4(3) \}$. This bound is essentially best possible, since $\fpr(x,G/H)$ is roughly $q^{-1}$ when $q$ is odd, $G = \PGL_n(q)$, $H$ is a maximal $P_1$ parabolic subgroup (the stabiliser of a $1$-space of $\F_q^n$) and $x$ is a reflection (see \cite[Example~1.5]{ref:Burness16}).

Let $G$ be an almost simple classical group. A maximal subgroup $H \leq G$ is a \emph{subspace subgroup} if $H \cap \soc(G)$ acts reducibly on the natural module for $\soc(G)$ or if $\soc(G)$ is $\Sp_n(2^f)$ and $H \cap \soc(G) = \O^\pm_n(2^f)$, and $H$ is a \emph{nonsubspace subgroup} otherwise. In Section~\ref{s:fpr_s} we record and prove bounds on fixed point ratios for subspace subgroups. 

Notice that the bound in \eqref{eq:fpr} does not depend on the element $x$. The sequence of papers \cite{ref:Burness071,ref:Burness072,ref:Burness073,ref:Burness074} gives an upper bound on $\fpr(x,G/H)$ depending on $x$ when $H \leq G$ is nonsubspace and $x \in G$ has prime order. We present and apply this result in Section~\ref{s:fpr_ns}, before giving some tighter bounds on the fixed point ratios for nonsubspace actions of almost simple unitary groups in low dimensions.

\section{Subspace actions} \label{s:fpr_s}

We begin with a general theorem that combines several results of Guralnick and Kantor \cite[Propositions~3.1, 3.15 and~3.16]{ref:GuralnickKantor00}. Here $d=2$ if $\soc(G) = \PSU_n(q)$ and $d=1$ otherwise, so $\F_{q^d}^n$ is the natural module for all of the classical groups $\GL_n(q)$, $\Sp_n(q)$, $\O^\e_n(q)$ and $\GU_n(q)$.

\begin{theoremx}\label{thm:fpr_s}
Let $G \leq \PGaL(V)$ be an almost simple classical group with natural module $V=F^n$ where $F=\F_{q^d}$. Assume that $n \geq 6$. Let $H \leq G$ be a reducible maximal subgroup, stabilising a subspace $0 < U < V$ of dimension $k$ and Witt index $l$. Let $1 \neq x \in G$. Let $m,a,b,c$ be the parameters defined in Table~\ref{tab:fpr_s}.
\begin{enumerate}
\item If $\soc(G) = \PSL_n(q)$, then
\[
\fpr(x,G/H) \leq 2|F|^{-\min\{k,\,n-k\}}.
\]
\item If $\soc(G) \neq \PSL_n(q)$ and $U$ is nondegenerate, then 
\[
\fpr(x,G/H) \leq 2|F|^{-m+a} + |F|^{-m+b} + |F|^{-l} + |F|^{-n+k}.
\]
\item If $\soc(G) \neq \PSL_n(q)$ and $U$ is totally singular, then
\[
\fpr(x,G/H) \leq 2|F|^{-m+c} + |F|^{-\frac{m}{d}+\frac{b}{d}} + |F|^{-k}.
\]
\end{enumerate}
\end{theoremx}

\begin{table}
\centering
\caption{Fixed point ratios: Values of $a$, $b$, $c$} \label{tab:fpr_s}
{\renewcommand{\arraystretch}{1.2}
\begin{tabular}{llccc}
\hline
$\soc(G)$          &          & $a$ & $b$            & $c$ \\
\hline
$\PSp_{2m}(q)$     & $q$ even & $2$ & $0$            & $1$ \\
                   & $q$ odd  & $1$ & $0$            & $1$ \\
$\Omega_{2m+1}(q)$ &          & $1$ & $0$            & $1$ \\
$\POm^\e_{2m}(q)$  & $\e = +$ & $2$ & $1$            & $2$ \\
                   & $\e = -$ & $2$ & $0$            & $1$ \\
$\PSU_n(q)$        & $n=2m$   & $2$ & $\frac{1}{2}$  & $1$ \\
                   & $n=2m+1$ & $1$ & $-\frac{1}{2}$ & $0$ \\
\hline
\end{tabular}
}
\end{table}

Frohardt and Magaard established upper and lower bounds on the fixed point ratio of an element $x$ of an almost simple classical group $G$ on an appropriate set of $k$-spaces of $V$ that depend on $G$ and $k$ but also take into account the element $x$ \cite[Theorems~1--6]{ref:FrohardtMagaard00}. In particular, these bounds depend on the following invariant.

\begin{notationx}\label{not:nu}
For $x \in \PGL_n(q)$, let $\hat{x}$ be a preimage of $x$ in $\GL_n(q)$ and define $\nu(x)$ as the codimension of the largest eigenspace of $\hat{x}$ on $V \otimes_{\F_q} \FF_p$.
\end{notationx}

For example, if $G = \POm^\e_n(q)$ and $H \leq G$ is the stabiliser of a nondegenerate $k$-space, where $k < \frac{n}{2}$, then for all $x \in G$ satisfying $s = \nu(x) < \frac{n}{2k}$, \cite[Theorem~2]{ref:FrohardtMagaard00} gives
\[
q^{-sk} - 3q^{-(n-1)/2} < \fpr(x,G/H) < q^{-sk} + 200q^{-(n-1)/2}.
\]
As in \cite{ref:Harper17}, the constants in these upper bounds are too large for our application, so we derive our own bounds for the particular cases we will require. Our bounds have no restriction on $s$ in terms of $k$.

In part~(i) of the statement of Proposition~\ref{prop:fpr_s_o12}, if $q$ is even, then the nonsingular $1$-space in question is degenerate and has a stabiliser of type $\Sp_{2m-2}(q)$ (see \cite[Proposition~4.1.7]{ref:KleidmanLiebeck}).

\begin{propositionx}\label{prop:fpr_s_o12}
Let $G = \PO^\e_{2m}(q)$ where $m \geq 4$. Let $x \in G$ have prime order and $\nu(x)=s$. 
\begin{enumerate}
\item If $H \leq G$ is the stabiliser of a nonsingular $1$-space, then 
\[
\fpr(x,G/H) \leq \frac{1}{q^s} + \frac{1}{q^{2m-s}} + \frac{2}{q^m-\e}.
\]
\item If $H \leq G$ is the stabiliser of a nondegenerate $2$-space, then
\[
\fpr(x,G/H) \leq \frac{1}{q^{2s}} + \frac{1}{q^{m-1}-1} + \frac{4}{q^{2m-3}} + \frac{1}{q^{2m-2s}}.
\]
\end{enumerate}
\end{propositionx}

\begin{proof}
Let $r$ be the order of $x$. If $x$ is not contained in a $G$-conjugate of $H$, then $\fpr(x,G/H) = 0$. Therefore, assume that $x \in H$. Let $V = \F_q^{2m}$. \vspace{5pt}

\emph{Proof of part~(i).} Write $H = G_{\<u\>}$ and $U = \<u\>^\perp$. We consider four cases depending on the parity of $r$ and $p$.

\emph{Case 1: $r \not\in \{2,p\}$.} In this case, $x$ is a semisimple element of odd prime order. By \cite[Proposition~3.5.4]{ref:BurnessGiudici16}, $x$ is $G$-conjugate to an element that lifts to a block diagonal matrix $[M_1,\dots,M_d,I_{2l}]$ centralising $V = V_1 \perp \cdots \perp V_d \perp W$ where, for some even $k$, each $V_j$ is a nondegenerate $k$-space and $W$ is the (nondegenerate) $1$-eigenspace of $x$. Moreover, either each matrix $M_j$ acts irreducibly on $V_j$ or each matrix $M_j$ centralises the decomposition $V_j = U_j \oplus U_j^*$, where $U_j$ and $U_j^*$ are totally singular subspaces on which $M_j$ acts irreducibly. The submodules $V_j$ are pairwise nonisomorphic.

Since $x \in H$, we deduce that $x$ fixes $u$. Therefore, $2l > 0$ and on $U$ the element $x$ acts as $[M_1,\dots,M_d,I_{2l-1}]$. Therefore, \cite[Lemma~3.5.3]{ref:BurnessGiudici16} implies that $x^G \cap H =  x^H$. Moreover, from the centraliser orders in \cite[Appendix~B]{ref:BurnessGiudici16} we obtain
\[
\frac{|x^G \cap H|}{|x^G|} = \frac{|H|}{|G|}\frac{|C_G(x)|}{|C_H(x)|} \leq \frac{(2,q-1)}{q^{m-1}(q^m-\e)}\frac{q^{l-1}(q^l+1)}{(2,q-1)} \leq \frac{1}{q^{2m-2l}} + \frac{1}{q^m-\e}.
\]
Since $2l$ is the dimension of the $1$-eigenspace of $x$, we know that $2m-2l \geq s$. The result now follows in this case.

\emph{Case 2: $r=p=2$.} Here $x$ is a unipotent involution and we adopt the notation of Aschbacher and Seitz \cite{ref:AschbacherSeitz76}. Since $p=2$, the subgroup $H$ is the centraliser in $G$ of a $b_1$ involution, and we may write $H \cong \< b_1 \> \times \Sp_{2m-2}(q)$. Now an element $x = (y,z) \in H$, where $y \in \<b_1\>$ and $z \in \Sp_{2m-2}(q)$, embeds as the block diagonal element $[y,z]$ in $G$. Therefore, in light of \cite[Lemma~3.4.14]{ref:BurnessGiudici16}, it is straightforward to determine how $x^G \cap H$ splits into $H$-classes. For example, if $x=b_s$ for odd $s \geq 3$, then $x^G \cap H$ is the union of $x_1^H \cup x_2^H \cup x_3^H$ where $x_1$, $x_2$ and $x_3$ are the elements $(I_2,b_s)$, $(b_1,a_{s-1})$ and $(b_1,c_{s-1})$ of $H$. Therefore, using the centraliser orders that are given in \cite[Appendix~B]{ref:BurnessGiudici16}, we deduce that
\begin{align*}
\fpr(x,G/H) &= \frac{|H|}{|G|} \sum_{i=1}^{3} \frac{|C_{G}(x_i)|}{|C_{H}(x_i)|} = \frac{q^{s-1}(q^{2m-2s}-1)+1+(q^{s-1}-1)}{q^{m-1}(q^m-\e)} \\ &= \frac{q^{m-s}}{q^m-\e} \leq \frac{1}{q^s}+\frac{1}{q^m-\e}.
\end{align*}
In the remaining cases, $x^G \cap H$ splits into $H$-classes in the following ways
\begin{itemize}
\item[] $b_1^G \cap H = (b_1,I_{n-2})^H \cup (I_2,b_1)^H$
\item[] $b_m^G \cap H = (b_1,a_{m-1})^H \cup (b_1,c_{m-1})^H$, where $m$ is necessarily odd
\item[] $a_s^G \cap H = (I_2,a_s)^H$, if $s < m$
\item[] $c_s^G \cap H = (I_2,c_2)^H \cup (b_1,b_{s-1})^H$, if $s < m$
\item[] $c_m^G \cap H = (b_1,b_{s-1})^H$, where $m$ is necessarily even,
\end{itemize}
noting that $x$ does not have type $a_m$ (when $m$ is even), since $a_m^G \cap H$ is empty. In each case, we can verify the claimed bound as above.

\emph{Case 3: $r=2$ and $p >2$.} The $G$-classes of semisimple involutions are described in detail in \cite[Section~3.5.2]{ref:BurnessGiudici16}. Since $x \in H$ we may deduce that $x$ has type $t_i$, $t_i'$ or $\gamma_i$ for some $i$, in the notation of \cite{ref:GorensteinLyonsSolomon98}. (In particular, \cite[Table~B.9]{ref:BurnessGiudici16} makes clear that involutions arising from matrices of order four do not stabilise nondegenerate 1-spaces.) Said otherwise, $x$ lifts to an involution $-I_a \perp I_b$ centralising a decomposition $U_1 \perp U_2$ where $U_1$ and $U_2$ are nondegenerate $a$- and $b$-spaces. Therefore, either $x$ fixes $u$ and acts as $-I_a \perp I_{b-1}$ on $U$, or $x$ negates $u$ and acts as $-I_{a-1} \perp I_b$ on $U$. Therefore, $x^G \cap H = x_1^H \cup x_2^H$ where $x_1$ and $x_2$ correspond to the two possible actions of $x$ on $u$. Consequently,
\[
\frac{|x^G \cap H|}{|x^G|} = \frac{|H|}{|G|} \left( \frac{|C_G(x_1)|}{|C_H(x_1)|} + \frac{|C_G(x_2)|}{|C_H(x_2)|} \right). 
\]
Assume that $a=2k$ and $b=2l$; the case where $a$ and $b$ are odd is very similar. From the centraliser orders in \cite[Appendix~B]{ref:BurnessGiudici16} we can compute that
\[
\frac{|C_G(x_1)|}{|C_H(x_1)|} \leq \frac{1}{2} q^{l-1}(q^l+1) \quad \text{and} \quad \frac{|C_G(x_2)|}{|C_H(x_2)|} \leq \frac{1}{2} q^{k-1}(q^k+1).
\] 
Therefore,
\[
\frac{|H|}{|G|} \left( \frac{|C_G(x_1)|}{|C_H(x_1)|} + \frac{|C_G(x_2)|}{|C_H(x_2)|} \right) \leq \frac{q^{l-1}(q^l+1)+q^{k-1}(q^k+1)}{q^{m-1}(q^m-\e)} \leq \frac{1}{q^{2k}} + \frac{1}{q^{2l}} + \frac{1}{q^m-\e}.
\]
Since $\{ 2k, 2l \} = \{ s, 2m-s \}$, we have verified the result in this case.

\emph{Case 4: $r = p > 2$.} In this case, $x$ is a unipotent element of odd prime order, so, by \cite[Proposition~3.5.12]{ref:BurnessGiudici16}, $x$ is $G$-conjugate to an element that lifts to a matrix with Jordan form $[J_p^{a_p},\dots,J_2^{a_2},J_1^{a_1}]$ where $\sum_{i=1}^{p} i a_i = 2m$. Indeed, the conjugacy class $x^G$ is characterised by this Jordan form together with a sequence $(\d_1,\d_3,\dots,\d_p)$ in $\{\square,\nonsquare\}$ that satisfies $\d_1\d_3\cdots\d_p = D(Q)$, where $Q$ is the form defining $G$. 

Note that $V = \<u\> \perp U$ since $p$ is odd. Since $x \in H$ and the only eigenvalue of $x$ is $1$, the vector $u$ is fixed by $x$. Since the $1$-eigenspace of $J_i$ is totally singular when $i > 1$, we deduce that $a_1 > 0$ and $x$ acts on $U$ as an element whose Jordan form is $[J_p^{a_p},\dots,J_2^{a_2},J_1^{a_1-1}]$. Moreover, the corresponding sequence of discriminants for the element $x|_U$ is $(\d_1\d,\d_3,\dots,\d_p)$, where $\d = D(Q|_{\<u\>})$. By \cite[Proposition~3.5.12]{ref:BurnessGiudici16}, this completely determines the $H$-class of $x$. Therefore, $x^G \cap H = x^H$. Using the centraliser orders in \cite[Appendix~B]{ref:BurnessGiudici16}, noting that $s = 2m-\sum_{j \geq 1}a_j$, we have
\begin{align*}
\fpr(x,G/H) \leq \frac{|H|}{|G|} \frac{|C_G(x)|}{|C_H(x)|} = \frac{q^{2m-s-a_1}}{q^{m-1}(q^m-\e)}\frac{|\O^{\e_1}_{a_1}(q)|}{|\O^{\e_1'}_{a_1-1}(q)|} \leq \frac{1}{q^s} + \frac{1}{q^m-\e}.
\end{align*}
This completes the proof of part~(i).

\vspace{5pt}

\emph{Proof of part~(ii).} We proceed in the same way as for part~(i).

\emph{Case~1: $r \not\in \{2,p\}$.} In this case, $x$ is a semisimple element of odd prime order. By \cite[Proposition~3.5.4]{ref:BurnessGiudici16}, $x$ is $G$-conjugate to an element that lifts to a block diagonal matrix $[A_1^{a_1},\dots,A_t^{a_t},I_e]$ centralising a decomposition $V = V_1^{a_1} \perp \cdots \perp V_t^{a_t} \perp W$ where, for some even $k$, each $V_j$ is a nondegenerate $k$-space and $W$ is the (nondegenerate) $1$-eigenspace of $x$. Moreover, either each matrix $A_j$ acts irreducibly on $V_j$ or each matrix $A_j$ centralises the decomposition $V_j = U_j \oplus U_j^*$, where $U_j$ and $U_j^*$ are totally singular subspaces on which $A_j$ acts irreducibly. The submodules $V_j$ are pairwise nonisomorphic.

Let us now determine how $x^G \cap H$ splits into $H$-classes. Let $h \in H$ be $G$-conjugate to $x$. Then $h$ lifts to $(A,B) \in \O^{\e_1}_2(q) \times \O^{\e_2}_{2m-2}(q)$. If $A = I_2$, then $e \geq 2$ and $h$ is $H$-conjugate to $x_0$, an element lifting to $(I_2, [A_1^{a_1},\dots,A_t^{a_t},I_{e-2}])$. If $A \neq I_2$, then let $\l \in \FF_q$ be a nontrivial eigenvalue of $A$. Then $\l$ is an eigenvalue of $A_j$ for some $j$. Since the set of eigenvalues of $A$ is closed under the map $\mu \mapsto \mu^q$, we deduce that $k=2$ and $A=A_j$. Therefore, $h$ is $H$-conjugate to $x_j$, an element lifting to $(A_j, [A_1^{a_1},\dots,A_j^{a_j-1},\dots,A_t^{a_t},I_e])$. 

This information is enough to determine how $x^G \cap H$ splits into $H$-classes. If $k > 2$, then $e > 0$ and $x^G \cap H = x_0^H$. If $k = 2$, then, writing $e=2a_0$, we have
\[
x^G \cap H = \bigcup_{\substack{0 \leq j \leq t \\[1pt] a_j >0}} x_j^H
\]
We now use this information about $x^G \cap H$ to find an upper bound on $\fpr(x,G/H)$. First note that 
\[
\frac{|H|}{|G|} = \frac{|\O^{\e_1}_2(q)||\O^{\e_2}_{2m-2}(q)|}{|\O^\e_{2m}(q)|} = \frac{2(q-\e_1)}{q^{2m-2}(q^{m-1}+\e_2)(q^m-\e)}.
\]
Similarly, if $e=2a_0 > 0$, then
\[
\frac{|C_G(x)|}{|C_H(x_0)|} \leq \frac{q^{e-2}(q^{a_0-1}+1)(q^{a_0}+1)}{2(q-\e_1)}.
\]
Now assume that $k=2$. Let $\eta=+$ if $r$ divides $q-1$ and let $\eta=-$ otherwise (when $r$ necessarily divides $q+1$). Then for all $1 \leq j \leq t$ such that $a_j > 0$ we have 
\[
\frac{|C_G(x)|}{|C_H(x_j)|} = \frac{|\GL^\eta_{a_j}(q)|}{|\GL^{\e_1}_1(q)||\GL^\eta_{a_j-1}(q)|} \leq \frac{q^{a_j-1}(q^{a_j}+1)}{q-\e_1} 
\]
Now,
\[
\fpr(x,G/H) = \frac{|H|}{|G|} \sum_{\substack{0 \leq j \leq t \\[1pt] a_j > 0}} \frac{|C_{G}(x)|}{|C_{H}(x_j)|}.
\]
Therefore, with the above bounds, we maximise our upper bound on $\fpr(x,G/H)$ when $a_j=0$ for all $j \geq 2$. In this case, $a_0+a_1=m$ and $s=2a_1=2m-e$. Therefore,
\begin{align*}
\fpr(x,G/H) &\leq \frac{2(q-\e_1)}{q^{2m-2}(q^{m-1}+\e_2)(q^m-\e)} \frac{q^{e-2}(q^{a_0-1}+1)(q^{a_0}+1)}{2(q-\e_1)} \\
            & \qquad + \frac{2(q-\e_1)}{q^{2m-2}(q^{m-1}+\e_2)(q^m-\e)} \frac{q^{a_1-1}(q^{a_1}+1)}{q-\e_1} \\ 
            &\leq \frac{1}{q^{2s}} + \frac{4}{q^{2m-3}} + \frac{1}{q^{m-1}-1}. 
\end{align*}

\emph{Case 2: $r=p=2$ and Case 3: $r=2$ and $p >2$. } In these case, $x$ is an involution and we proceed exactly as described in part~(i); we omit the details.

\emph{Case 4: $r = p > 2$.} In this case, the $G$-conjugacy class of $x$ is determined by the Jordan form $[J_p^{a_p},\dots,J_2^{a_2},J_1^{a_1}]$ where $\sum_{i=1}^{p} i a_i = 2m$ and a sequence $(\d_1,\d_3,\dots,\d_p)$ in $\{\square,\nonsquare\}$ where $\d_1\d_3\cdots\d_p = D(Q)$. Let $x = (y,z) \in \O^{\e_1}_2(q) \times \O^{\e_2}_{2m-2}(q)$. Since Jordan blocks of even size occur with even multiplicity in orthogonal groups, we deduce that $y$ is trivial and $z$ has Jordan form $[J_p^{a_p},\dots,J_2^{a_2},J_1^{a_1-2}]$ and sequence of discriminants $(\d_1\d,\d_3,\dots,\d_p)$ where $\d = D(Q|_{U})$. Therefore, as in part~(i), $x^G \cap H = x^H$ and the result again follows from the centraliser orders in \cite[Appendix~B]{ref:BurnessGiudici16}. We have completed the proof.
\end{proof}

\section{Nonsubspace actions} \label{s:fpr_ns}

We now turn to fixed point ratios for nonsubspace actions of classical groups, which, in general, are smaller than fixed point ratios for subspace actions. Building on work of Liebeck and Shalev \cite{ref:LiebeckShalev99}, the following general theorem was established by Burness in \cite[Theorem~1]{ref:Burness071} (see \cite[Definition~2]{ref:Burness071} for a precise definition of the \emph{dimension of the natural module}).

\begin{theoremx}\label{thm:fpr_ns}
Let $G$ be an almost simple classical group such that the natural module of $G$ is $n$-dimensional. If $H \leq G$ is a maximal nonsubspace subgroup and $x \in G$ has prime order, then 
\[
\fpr(x,G/H) < |x^G|^{-\frac{1}{2}+\frac{1}{n}+\iota}
\]
where $\iota$ is given in \cite[Table~1]{ref:Burness071}.
\end{theoremx}

In the statement of Theorem~\ref{thm:fpr_ns}, for most subgroups $H \leq G$ the parameter $\iota$ is simply $0$, and whenever $n \geq 10$ we have $\iota \leq \frac{1}{n-2}$. Theorem~\ref{thm:fpr_ns} is essentially best possible. For example, if $G = \PGL_n(q_0^2)$ and $H = \PGL_n(q_0)$, then $|x^G \cap H|$ is roughly $|x^G|^{\frac{1}{2}}$ (see also \cite[Example~2.17]{ref:Burness16}).

\begin{propositionx}\label{prop:fpr_ns_o}
Let $G$ be an almost simple group with socle $\POm^\e_{2m}(q)$ where $m \geq 4$ and $q = p^f$. Let $H \leq G$ be a maximal nonsubspace subgroup and let $x \in G$ be nontrivial. Then
\[
\fpr(x,G/H) < 2q^{-(m-2+2/(m+1))}
\]
Moreover,
\begin{enumerate}
\item if $f \geq 2$ and either $\nu(x) \geq 2$ or $x \not\in \PGO^\e_{2m}(q)$, then 
\[
\fpr(x,G/H) < 3q^{-(2m-5+3/m-\ell)}
\] 
where $\ell = 0$, unless $H$ has type $\GL^\pm_m(q)$, in which case $\ell = 2$
\item if $\soc(G)=\POm_8^+(q)$ and $H$ is almost simple with socle $G_2(q)$ or $\PSL^\pm_3(q)$, then 
\[
\fpr(x,G/H) < 2q^{-9/2}.
\]
\end{enumerate}
\end{propositionx}

\begin{proof}
Part~(i) is an immediate consequence of \cite[Corollary~2]{ref:Burness074}, which in turn Burness deduces from Theorem~\ref{thm:fpr_ns}.

Now let us consider part~(ii). Write $T = \POm^\e_{2m}(q)$. From the bounds presented in \cite[Section~3]{ref:Burness072}, if $x \in \PGO^\e_{2m}(q)$ and $\nu(x) \geq 2$, then 
\[
|x^G| \geq |x^T| \geq \frac{2^{\d_{2,p}}}{8} \left(\frac{q}{q+1}\right) q^{4m-6}
\]
and if $x \in \Aut(T) \setminus \PGO^\e_{2m}(q)$, then by \cite[Corollary~3.49]{ref:Burness072},
\[
|x^G| \geq \frac{1}{8} q^{m(m-1/2)} \geq \frac{1}{4} \left(\frac{q}{q+1}\right) q^{4m-6}.
\]
Theorem~\ref{thm:fpr_ns} now implies that if $\nu(x) \geq 2$ or $x \not\in \PGO^\e_{2m}(q)$, then
\begin{align*}
\fpr(x,G/H) &< |x^G|^{-\frac{1}{2}+\frac{1}{2m}+\i} < \frac{\left(8/2^{\d_{2,p}} \cdot \frac{q+1}{q}\right)^{1/2}}{q^{(4m-6)(\frac{1}{2}-\frac{1}{2m}-\i)}} \leq \frac{3}{q^{2m-5-(4m-6)\i}},
\end{align*}
where $\i=0$ unless $H$ has type $\GL^\pm_m(q)$ and $\i=(2m-2)^{-1}$, as claimed in (ii).

We now turn to part~(iii). By \cite[Theorem~7.1]{ref:GuralnickSaxl03}, if $x \in H \cap \PGO^+_8(q)$, then $\nu(x) \geq 3$, so from the bounds in \cite[Section~3]{ref:Burness072},
\[
|x^G| > \frac{1}{4}\left(\frac{q}{q+1}\right)q^{12}.
\]
In addition, by \cite[Corollary~3.49]{ref:Burness072}, if $x \in \Aut(T) \setminus \PGO^+_8(q)$, then
\[
|x^G| > \frac{1}{8}q^{14} \geq \frac{1}{4}\left(\frac{q}{q+1}\right)q^{12}.
\]
Therefore, by Theorem~\ref{thm:fpr_ns}, we conclude that
\[
\fpr(x,G/H) < |x^G|^{-3/8} \leq \frac{\left(4 \cdot \frac{q+1}{q}\right)^{3/8}}{q^{9/2}} \leq \frac{2}{q^{9/2}}. \qedhere
\]
\end{proof}

\begin{propositionx}\label{prop:fpr_ns_u}
Let $G$ be an almost simple group with socle $\PSU_n(q)$ where $n \geq 7$. Let $H \leq G$ be a maximal nonsubspace subgroup and let $x \in G$ be nontrivial. Then
\[
\fpr(x,G/H) < \frac{2}{q^{n-3+2/n}}.
\]
\end{propositionx}

\begin{proof}
We may assume that $x \in H$. By \cite[Lemma~2.1]{ref:Burness071}, $|x^G| > \frac{1}{4}q^{2n-2}$. Therefore, if $H$ does not have type $\Sp_n(q)$, then Theorem~\ref{thm:fpr_ns}, implies that 
\[
\fpr(x,G/H) < \frac{2}{q^{(2n-2)(1/2-1/n)}} = \frac{2}{q^{n-3+2/n}}.
\]
For the remainder of the proof we can assume that $H$ has type $\Sp_n(q)$. By \cite[Corollary~3.38]{ref:Burness072}, if $x \in \PGU_n(q)$ and $\nu(x) > 1$, then
\[
|x^G| > \frac{1}{2}\left(\frac{q}{q+1}\right)q^{4(n-2)},
\]
and if $x \in \PGaU_n(q) \setminus \PGU_n(q)$ has odd order, then, by \cite[Lemma~3.48]{ref:Burness072}
\[
|x^G| > \frac{1}{2}\left(\frac{q}{q+1}\right)q^{2n^2/3-5/3}.
\]
In both cases, we obtain the desired bound.  Now assume that $x \in \PGU_n(q)$ and $\nu(x)=1$. Since $x \in H$ we know that $x = [J_2,I_{n-2}]$ and we can compute
\[
|x^G| > \frac{q^{2n-1}}{2(q+1)} \ \ \text{and} \ \ |x^G \cap H| \leq (2-\d_{2,p})|x^H| < q^n,
\]
which gives
\[
\fpr(x,G/H) < \frac{2(q+1)}{q^{n-1}} < \frac{2}{q^{n-3+2/n}}.
\]
Finally assume that $x \in \PGaU_n(q)$ is an involutory graph automorphism. Here \cite[Lemma~3.48]{ref:Burness071} implies that
\[
|x^G| > \frac{1}{2}\left(\frac{q}{q+1}\right)q^{(n^2-n-4)/2}
 \]
and the bounds in the statement hold.
\end{proof}

The rest of this chapter is dedicated to deriving upper bounds on fixed point ratios of nonsubspace actions of low-dimensional almost simple unitary groups. 

\begin{theoremx}\label{thm:fpr_ns_u_low}
Let $G$ be an almost simple group with socle $\PSU_n(q)$ where $3 \leq n \leq 6$. Assume that $q \geq 11$ if $n \in \{3,4\}$. Let $H \leq G$ be a maximal nonsubspace subgroup. Let $x \in G$ be nontrivial.
\begin{enumerate}
\item If $n \in \{5,6\}$, then
\[
\fpr(x,G/H) \leq (q^4-q^3+q^2-q+1)^{-1}. 
\]
\item If $n \in \{3,4\}$ and $H$ does not have type $\Sp_4(q)$, then 
\[
\fpr(x,G/H) \leq (q^2-q+1)^{-1}. 
\]
\item If $n=4$ and $H$ has type $\Sp_4(q)$, then
\[
\fpr(x,G/H) \leq (2,q+1) \cdot \frac{q^4+1}{q^5+q^2}.
\]
\end{enumerate}
\end{theoremx}

The following will be used in the proof of Theorem~\ref{thm:fpr_ns_u_low} and will also be used in its own right in Chapter~\ref{c:u}.

\begin{propositionx}\label{prop:u_max}
Let $n \geq 6$ and let $G$ be an almost simple group with socle $T = \PSL^\e_n(q)$. Let $H$ be a maximal subgroup of $G$ such that $T \not\leq H$. Let $x \in G \cap \PGL^\e_n(q)$ with $\nu(x) \leq 2$. If $x \in H$, then one of the following holds
\begin{enumerate}
\item $H \in \C_1 \cup \C_2 \cup \C_5 \cup \C_8$ 
\item $H$ appears in Table~\ref{tab:u_max_c}
\item $n \in \{6,7\}$, $q=p$, $H \in \S$ and $\soc(H)$ appears in Table~\ref{tab:u_max_s}.
\end{enumerate}
\end{propositionx}

\begin{table}
\centering
\caption{The subgroups in Proposition~\ref{prop:u_max}(ii)}\label{tab:u_max_c}
{\renewcommand{\arraystretch}{1.2}
\begin{tabular}{cccccc}
\hline
       & type of $H$                       & $n$  & $\e$        & $q$          & $x$                                          \\
\hline
$\C_3$ & $\GL_m(q^2)$                      & $2m$ & $+$         & any          & $[J_2^2,J_1^{n-4}]$                          \\
       &                                   &      &             &              & $[\l,\l^q, I_{n-2}]$ with $|\l| \div q^2-1$  \\
$\C_4$ & $\GL^\e_2(q) \otimes \GL^\e_m(q)$ & $2m$ & $\pm$       & any          & $[J_2^2,J_1^{n-4}]$                          \\
       &                                   &      &             &              & $[\l I_2, I_{n-2}]$ with $|\l| \div q+1$     \\
$\C_6$ & $2^6.\Sp_6(2)$                    & $8$  & $\pm$       & $p$          & $\nu(x) = 2$                      \\
$\C_7$ & $\GL_2(q) \wr S_3$                & $8$  & $+$         & any          & $\nu(x) = 2$                                 \\
$\S$   & $\PSL^\e_3(q)$                    & $6$  & $\pm$       & odd          & $[-I_2,I_4]$                                 \\ 
\hline
\end{tabular}}
\end{table}

\begin{table}
\centering
\caption{The subgroups in Proposition~\ref{prop:u_max}(iii)}\label{tab:u_max_s}
{\renewcommand{\arraystretch}{1.2}
\begin{tabular}{ccc}
\hline
$n$  & $\soc(H)$         & conditions                        \\
\hline
$7$  & $\PSU_3(3)$       & $p \equiv \e \mod{3}$, $p \geq 5$ \\[5.5pt]
$6$  & $A_6$             & $p \equiv \e \mod{3}$, $p \geq 5$ \\
     & $A_7$             & $p \equiv \e \mod{3}$, $p \geq 5$ \\
     & $\PSL_3(4)$       & $p \equiv \e \mod{3}$, $p \geq 5$ \\
     & $\PSU_4(3)$       & $p \equiv \e \mod{3}$             \\ 
     & $\mathrm{M}_{12}$ & $\e=+$ and $p=3$                  \\  
     & $\mathrm{M}_{22}$ & $\e=-$ and $p=2$                  \\                
\hline
\end{tabular}}
\end{table}

\begin{proof}
Assume neither (i) nor (iii) hold. Then, by \cite[Theorem~7.1]{ref:GuralnickSaxl03}, we have one of the following
\begin{enumerate}[(a)]
\item $H \in \C_3$
\item $H \in \C_4$
\item $n=8$, $q=p \equiv \e \mod{4}$ and $H \in \C_6$ has type $2^{1+6}.\Sp_6(2)$
\item $n=8$, $\e=+$ and $H \in \C_7$ has type $\GL_2(q) \wr S_3$
\item $n=6$, $p > 2$, $H \in \S$ and $\soc(H) = \PSL^\e_3(q)$ via the the symmetric square of the natural representation.
\end{enumerate}
We need to prove that the only cases that (a)--(e) give rise to are those in Table~\ref{tab:u_max_c}.

For (a), the conclusion is given by \cite[Lemma~5.3.2]{ref:BurnessGiudici16} noting that $G$ does not have any degree two field extension subgroups if $T = \PSU_n(q)$.

We now turn to (b). Assume that $g = g_1 \otimes g_2$ centralises a tensor product decomposition $V = V_1 \otimes V_2$ where $1 < \dim{V_1} \leq \dim{V_2}$. Then \cite[Lemma~3.7]{ref:LiebeckShalev99} implies that $n=2m$ and
\[
\text{$n=2m$, \ $\nu(g) = 2$, \ $\dim{V_1} = 2$, \ $\dim{V_2} = m$, \ $\nu(g_1) = 0$, \ $\nu(g_2) = 1$.}
\]
Without loss of generality $g_1 = I_2$. If $g$ is unipotent, then $g_2 = [J_2,J_1^{m-2}]$ and $g = [J_2^2,J_1^{n-2}]$. If $g$ is semisimple, then $g_2 = [\l,I_{m-1}]$ and $g = [\l I_2, I_{n-2}]$ where $\l \in \F_{q^2}$ and $|\l|$ divides $q+1$.

For (c), \cite[Lemma~6.3]{ref:Burness072} implies that $\nu(x) > 1$. 

Now assume that (d) holds. Here $g$ stabilises a tensor product decomposition $V = V_1 \otimes V_2 \otimes V_3$ where $\dim{V_i} = 2$. From the discussion in (b), $g$ permutes the factors nontrivially. Now \cite[Lemma~5.7.2]{ref:BurnessGiudici16} implies that $\nu(x) > 1$.

Finally consider (e). First assume that $g$ is unipotent. A direct computation verifies that the possible Jordan forms on $\F_p^6$ of order $p$ elements of $\GL_3(p)$ acting on the symmetric square are $[J_3,J_2,J_1]$ and $[J_5,J_1]$ if $p>3$ or $[J_3,J_3]$ if $p=3$, so $g \not\in H$. Now assume that $g$ is semisimple. Then the eigenvalues of $g$ are of the form $\a^2,\b^2,\g^2,\a\b,\a\g,\b\g$. Since $\nu(g)\leq2$, at least $4$ of these eigenvalues are equal. Therefore, without loss of generality, $\a\b=\a\g$, so $\b=\g$ and the eigenvalues of $g$ are in fact $\a^2,\a\b,\a\b,\b^2,\b^2,\b^2$. Since the eigenvalues of $g$ are not all equal, we know that $\a \neq \b$ and therefore $\a\b \neq \b^2$. This implies that $\a^2=\b^2$, so $\b=-\a$ and we conclude that $g = [-\a I_2,\a I_4] = [-I_2,I_4]$ modulo scalars. 
\end{proof}

Before proving Theorem~\ref{thm:fpr_ns_u_low} we handle several cases in a series of lemmas.

\begin{lemmax}\label{lem:fpr_ns_u_low_comp}
Theorem~\ref{thm:fpr_ns_u_low} is true if $(n,q) \in \{ (3,11), (5,2), (6,2) \}$.
\end{lemmax}

\begin{proof}
This is a straightforward computation in \textsc{Magma}.
\end{proof}

For the rest of this chapter, $3 \leq n \leq 6$ and $q \geq 11$ if $n \in \{3,4\}$. In addition, $G$ is an almost simple group with socle $\PSU_n(q)$, $H \leq G$ is maximal and $x \in G$ has prime order. 

Let us specify some particular elements that will demand extra attention. 
\begin{equation}\label{eq:fpr_ns_u_low_elt}
{\renewcommand{\arraystretch}{1.2}
\begin{array}{c}
   [J_2,J_1^3], \, [J_2^2,J_1], \, [\l,I_4], \, [\l,\l,I_3] \in \PGU_5(q) \\ 
{} [J_2,J_1^4], \, [\l,I_5] \in \PGU_6(q) \\
\end{array}}
\end{equation}
where $\l \in \F_{q^2}^\times$ and $|\l|$ is a prime divisor of $q+1$.

\begin{notationx}
Let $X$ be a finite subset of a group $G$ and let $r$ be prime. Then we write
\[
I_r(X) = \{x \in X \mid |x|=r \} \qquad i_r(X) = |I_r(X)| \qquad  i_\mx(X) = \max_{\text{$r$ prime}} i_r(X).
\]
\end{notationx}

As in the proof of the Proposition~\ref{prop:fpr_s_o12}, in the proofs that follow, we will extensively refer the information presented in \cite[Chapter~3]{ref:BurnessGiudici16} on the conjugacy classes of elements of prime order in almost simple classical groups, but for clarity of exposition we will not constantly cite this source. In particular, conjugacy class sizes that are asserted in these proofs can be deduced from the centraliser orders summarised in \cite[Appendix~B]{ref:BurnessGiudici16}.

\begin{lemmax} \label{lem:fpr_ns_u_low_sp}
Let $H$ have type $\Sp_n(q)$. Then
\[
\fpr(x,G/H) \leq 
\left\{
\begin{array}{ll}
(2,q+1)(q^4+1)(q^5+q^2)^{-1} & \text{if $n=4$}  \\
(q^4-q^3+q^2-q+1)^{-1}       & \text{if $n=6$.} \\
\end{array}
\right.
\]
\end{lemmax}

\begin{proof}
Write $H_0 = H \cap \PGU_n(q)$ and $|x|=r$. First assume that $x \in \PGU_n(q)$ and $r \neq p$. If $n=6$, then $x$ is a semisimple element with $\nu(x)=1$, so $x^G \cap H$ is empty. Therefore, we will assume that $n=4$ and $x \in H$. For now assume that $r > 2$. By \cite[Proposition~3.3.1 and Lemma~3.4.3]{ref:BurnessGiudici16}, conjugacy of semisimple elements in $\PSp_n(q)$ and $\PGU_n(q)$ is determined by eigenvalues. Therefore, $x^G = x^{\PGU_4(q)}$ and $x^G \cap H = x^H = x^{\PGSp_4(q)}$, so by a straightforward calculation, 
\[
\frac{|x^G \cap H|}{|x^G|} = \frac{|x^{\PGSp_4(q)}|}{|x^{\PGU_4(q)}} \leq \frac{q^2+1}{q(q^3+1)}
\]
with equality if $x = [\l,\l,\l^{-1},\l^{-1}]$, where $r = |\l|$ divides $q-1$, and this is sufficient to establish the desired bound. 

Now assume that $r = 2 \neq p$ (still with $n=4$). If $x^G \cap H$, then $x$ does not have a $1$-dimensional $1$-eigenspace, so $x$ (as an element of $\PGU_4(q)$) has type $t_2$ or $t_2'$, so $|x^{\PSU_4(q)}| \geq \frac{1}{2}q^4(q^2+1)(q^2-q+1)$. There are four classes of semisimple involutions in $\PGSp_4(q)$, with centralisers of order
\begin{gather*}
|C_{\PGSp_4(q)}(t_1)| = 2|\Sp_2(q)|^2, \ \ |C_{\PGSp_4(q)}(t_1')| = 2|\Sp_2(q^2)| \\
|C_{\PGSp_4(q)}(t_2)| = 2|\GL_2(q)|,   \ \ |C_{\PGSp_4(q)}(t_2')| = 2|\GU_2(q)|.  
\end{gather*}
Therefore,
\[
\frac{|x^G \cap H|}{|x^G|} \leq \frac{i_2(\PGSp_4(q))}{|x^{\PSU_4(q)}|} \leq \frac{2q^4(q^2+2)}{q^4(q^2+1)(q^2-q+1)} < \frac{2(q^4+1)}{q^5+q^2}.
\]

Next assume that $x \in \PGU_n(q)$ and $r=p$. For now assume further that $n=6$, so $x=[J_2,J_1^4]$. There is a unique class of elements with this Jordan form in $\PSU_6(q)$ and in $\PGSp_6(q)$, so
\[
\frac{|x^G \cap H|}{|x^G|} \leq \frac{|x^{\PGSp_6(q)}|}{|x^{\PSU_6(q)}|} = \frac{(q^6-1)(q+1)}{(q^6-1)(q^5+1)} = \frac{1}{q^4-q^3+q^2-q+1}.
\]
Now assume that $n=4$. Here the possible Jordan forms in $H$ are $[J_2,J_1^2]$, $[J_2^2]$ and (if $p \geq 5$) $[J_4]$; let $k$ be $1$, $2$, $(4,q+1)$ in these three cases, respectively. There is a unique $\PGSp_4(q)$-class of elements of one of these Jordan forms, and there is a unique $\PGU_4(q)$-class, which splits into $k$ distinct $\PSU_4(q)$-classes. This gives
\[
\frac{|x^G \cap H|}{|x^G|} \leq k \cdot \frac{|x^{\PGSp_4(q)}|}{|x^{\PGU_4(q)}|}.
\]
It is easy to compute that
\[
\frac{|x^{\PGSp_4(q)}|}{|x^{\PGU_4(q)}|} \leq \frac{1}{q^2-q+1}
\]
with equality if $x$ has Jordan form $[J_2,J_1^2]$ or $[J_4]$, so the claimed bound holds.

Finally assume that $n=4$ and $x \in \PGaU_4(q) \setminus \PGU_4(q)$. If $x$ is a field automorphism of (odd) order $r$, then,
\[
\frac{|x^G \cap H|}{|x^G|} \leq \frac{|x^{\PGSp_4(q)}|}{|x^{\PSU_4(q)}|} \leq \frac{4|\PGSp_4(q)||\PGU_4(q^{1/r})|}{|\PGSp_4(q^{1/r})||\PGU_4(q)|} \leq \frac{4}{q^{2/3}(q^2-q+1)},
\]
which gives the bound. 

It remains to assume that $x$ is a graph automorphism. Let $\g$ be a symplectic-type graph automorphism, and write $\widetilde{G} = \PGU_4(q)$ and $\widetilde{H} = C_{\widetilde{G}}(\g) = \PGSp_4(q)$. Then $\<T,x\> \leq \widetilde{G}{:}\<\g\>$ and $\<H_0,x\> \leq \widetilde{H} \times \<\g\>$; moreover, 
\[
x^G \cap H \subseteq \{ h \in \widetilde{H} \mid \text{$h^2=1$ and $(h\g)^{\widetilde{G}} = x^{\widetilde{G}}$} \}.
\]
The conjugacy classes of involutions in $\widetilde{H}$ are labelled $t_1$, $t_1'$, $t_2$, $t_2'$ if $p \neq 2$ and $a_2$, $b_1$, $c_2$ if $p=2$. From the proof of \cite[Proposition~8.1]{ref:Burness072}, if $|h|=2$, then $h\g$ is $\widetilde{G}$-conjugate to $\g$ if and only if $h$ has type $t_2$ or $t_2'$ if $p \neq 2$ or type $a_2$ if $p=2$. Therefore, if $x$ is not symplectic-type, then for even $q$,
\begin{align*}
\frac{|x^G \cap H|}{|x^G|} 
&\leq \frac{|b_1^{\widetilde{H}}| + |c_2^{\widetilde{H}}|}{|x^{\PSU_4(q)}|} = \frac{|C_{\PSp_4(q)}(b_1)|}{|\PGU_4(q)|}\left(\frac{|\PSp_4(q)|}{q^4(q^2-1)} + \frac{|\PSp_4(q)|}{q^4}\right)  \\
&= \frac{1}{q^2(q^3+1)(q^4-1)} \left((q^4-1) + (q^2-1)(q^4-1)\right) = \frac{2}{q^3+1},
\end{align*}
and for odd $q \geq 5$, by \cite[Proposition~4.5.5]{ref:KleidmanLiebeck}, $|x^G| \geq \frac{1}{2}|\PGU_4(q):\PGO^\e_4(q)|$, so
\begin{align*}
\frac{|x^G \cap H|}{|x^G|} 
&\leq \frac{|t_2^{\widetilde{H}}| + |(t_2')^{\widetilde{H}}|}{|x^{\PSU_4(q)}|} \leq \frac{|\PGO^\e_4(q)|}{|\PGU_4(q)|}\left(\frac{|\PGSp_4(q)|}{2|\GL_2(q)|} + \frac{|\PGSp_4(q)|}{2|\GU_2(q)|}\right) \\
&= \frac{4 \cdot (q^2+1)}{(q^3+1)(q^2+\e)} \leq \frac{1}{q^2-q+1}.
\end{align*}
If $x$ is symplectic-type, then for even $q$
\[
\frac{|x^G \cap H|}{|x^G|} = \frac{1+|a_2^{\Sp_4(q)}|}{|x^{\PSU_4(q)}|} = \frac{|\Sp_4(q)|}{|\PSU_4(q)|}\cdot\left(1+\frac{|\Sp_4(q)|}{q^4(q^2-1)}\right) = \frac{q^2}{q^3+1},
\]
and for odd $q$, by \cite[Proposition~4.5.6]{ref:KleidmanLiebeck}, $|x^G| = \frac{1}{2}|\PGU_4(q):\PGSp_4(q)|$, so
\begin{align*}
\frac{|x^G \cap H|}{|x^G|} 
&\leq \frac{1+|t_1^{\widetilde{H}}| + |(t_1')^{\widetilde{H}}|}{|x^{\PSU_4(q)}|} = \frac{2|\PGSp_4(q)|}{|\PGU_4(q)|}\left(1 + \frac{|\PGSp_4(q)|}{2|\Sp_2(q)|^2} + \frac{|\PGSp_4(q)|}{2|\Sp_2(q^2)|}\right) \\
&= \frac{2}{q^2(q^3+1)} \left(1 + \tfrac{1}{2} q^2(q^2+1) + \tfrac{1}{2} q^2(q^2-1)\right) = \frac{2(q^4+1)}{q^2(q^3+1)}.
\end{align*}
This completes the proof.
\end{proof}

\begin{lemmax} \label{lem:fpr_ns_u_low_c5}
Let $H$ have type $\GU_n(q^{1/k})$ or $\SO^\e_n(q)$. Then
\[
\fpr(x,G/H) \leq 
\left\{
\begin{array}{ll}
(q^2-q+1)^{-1}         & \text{if $n \in \{3,4\}$} \\
(q^4-q^3+q^2-q+1)^{-1} & \text{if $n \in \{5,6\}$ and $x$ in \eqref{eq:fpr_ns_u_low_elt}.} \\
\end{array}
\right.
\]
\end{lemmax}

\begin{proof}
Write $|x|=r$ and $H_0 = H \cap \PGU_n(q)$. We begin by considering $x \in \PGU_n(q)$. Now $x^G \cap H$ is a subset of all the elements of $\Inndiag(H_0)$ with the same eigenvalues as $x$ if $x$ is semisimple or the same Jordan form as $x$ if $x$ is unipotent. Using this estimate, together with the information on the conjugacy classes of unitary and orthogonal groups in \cite[Chapter~3]{ref:BurnessGiudici16}, it is easy to verify the result. We just give the details when $n=6$ as the arguments are very similar in the remaining cases. 

First assume that $r=p$, so $x = [J_2,J_1^4]$ and $|x^G| \geq |x^{\PSU_6(q)}| \geq \frac{(q^5+1)(q^6-1)}{6(q+1)}$. If $H$ has type $\GU_6(q^{1/k})$, then $|x^G \cap H| \leq |x^{\PGU_6(q^{1/k})}| \leq \frac{(q^{5/3}+1)(q^2-1)}{q^{1/3}+1}$, so
\[
\frac{|x^G \cap H|}{|x^G|} \leq \frac{6(q+1)(q^{5/3}+1)(q^2-1)}{(q^5+1)(q^6-1)(q^{1/3}+1)} < \frac{q+1}{q^5+1} = \frac{1}{q^4-q^3+q^2-q+1}.
\]
If $H$ has type $\SO^\e_6(q)$ then $p$ is odd, so $|x^G \cap H|= 0$ as $\SO^\e_n(q)$ does not contain elements with Jordan form $[J_2,J_1^{n-2}]$ in odd characteristic.

Now assume that $r \neq p$, so $x = [\l,I_5]$ where $|\l| \in \ppd(q,2) \cup \{2\}$. In this case, $|x^G| = \frac{q^5(q^6-1)}{q+1}$, and if $H$ has type $\GU_6(q^{1/k})$, then $|x^G \cap H| = |x^H| = \frac{q^{5/3}(q^2-1)}{q^{1/3}+1}$, which gives the bound. Now assume that $H$ has type $\SO^\e_6(q)$. In this case, if $x^G \cap H$ is not empty, then $\l=-1$. Here $x^G \cap H = x_\square^H \cup x_\nonsquare^H$ where $x_\d \in \PGO^\e_6(q)$ acts as $-I_1 \perp I_5$ with respect to an orthogonal decomposition $U \perp U^\perp$ where $U$ is a nondegenerate $1$-space with discriminant $\d$. Therefore, 
\[
|x^G \cap H| \leq |x_\square^{\PGO^\e_6(q)}|+|x_\nonsquare^{\PGO^\e_6(q)}| = \frac{1}{2}q^2(q^3+1) + \frac{1}{2}q^2(q^3-1) = q^5
\]
and we conclude that
\[
\frac{|x^G \cap H|}{|x^G|} \leq \frac{q^5(q+1)}{q^5(q^6-1)} < \frac{1}{q^4-q^3+q^2-q+1}.
\]

We now turn to the case where $x \in \PGaU_n(q) \setminus \PGU_n(q)$ (so $n \in \{3,4\}$ and $q \geq 11$). For now assume that $x$ is a field automorphism. If $H$ has type $\GU_n(q^{1/k})$, then
\[
\frac{|x^G \cap H|}{|x^G|} \leq \frac{|H||C_G(x)|}{|G|} \leq \frac{(n,q+1)|\PGU_n(q^{1/k})||\PGU_n(q^{1/r})|}{|\PGU_n(q)|}
\]
which gives the claimed bound since $k,r \geq 3$ and $q \geq 8$. 

Now assume that $H$ has type $\SO^\e_n(q)$. By \cite[Proposition~3.5.20]{ref:BurnessGiudici16}, 
\[
|x^G \cap H| \leq i_r(H_0x) = |x^{\Inndiag(H_0)}|
\]
and this gives the desired bound. For instance, if $n=3$, then
\[
\frac{|x^G \cap H|}{|x^G|} \leq \frac{|x^{\Inndiag(H_0)}|}{|x^{\PSU_3(q)}|} = \frac{(3,q+1)|\PSO_3(q)||\PGU_3(q^{1/r})|}{|\PSO_3(q^{1/r})||\PGU_3(q)|} 
\]
which allows us to conclude that
\[
\frac{|x^G \cap H|}{|x^G|} \leq \frac{3 \cdot q(q^2-1)\cdot q(q^{2/3}-1)(q+1)}{q^3(q^2-1)(q^3+1) \cdot q^{1/3}(q^{2/3}-1)} < \frac{1}{q^2-q+1}.
\]

Finally assume that $x$ is an involutory graph automorphism. We follow the proof of \cite[Proposition~5.1]{ref:Burness072} (where the relevant case of Theorem~\ref{thm:fpr_ns} is proved). First assume that $H$ has type $\GU_n(q^{1/k})$. Then $x$ induces an involutory graph automorphism on $H$ of the same type as it induces on $G$ (that is, symplectic on both or non-symplectic on both). We obtain the desired bound. For example, if $n=3$ and $q$ is odd, then
\[
\frac{|x^G \cap H|}{|x^G|} \leq \frac{|x^{\Inndiag(H_0)}|}{|x^{\PSU_3(q)}|} = \frac{(3,q+1)|\PGU_3(q^{1/k})||\PSO_3(q)|}{|\PSO_3(q^{1/l})||\PGU_3(q)|} < \frac{1}{q^2-q+1}
\]
as we computed above. 

Now assume that $H$ has type $\SO^\e_n(q)$, where we follow the proof of \cite[Proposition~8.2]{ref:Burness072}. If $x$ is nonsymplectic,
\[
\frac{|x^G \cap H|}{|x^G|} \leq \frac{1+i_2(H_0)}{|x^G|} \leq{|\PGO^\eta_n(q)|}{|\PSU_n(q)|} 2(q+1)q^{1+2\d_{n,4}},
\]
which gives the result; in particular, if $n=3$, then
\[
\frac{|x^G \cap H|}{|x^G|} \leq \frac{2(3,q+1)}{q(q^2-q+1)} \leq \frac{1}{q^2-q+1}.
\]
If $x$ is symplectic, then $n=4$ and from the splitting of $x^G \cap H$ into $H$-classes described in the proof of  \cite[Proposition~8.2]{ref:Burness072}, we obtain $|x^G \cap H| \leq 2q^2$, so
\[
\frac{|x^G \cap H|}{|x^G|} \leq \frac{2(4,q+1)q^2}{q^2(q^3+1)} = \frac{2(4,q+1)}{q^2-q+1} \leq \frac{1}{q^2-q+1}.
\]
\end{proof}

\begin{lemmax} \label{lem:fpr_ns_u_low_c3}
Let $n = 3$ and let $H$ have type $\GU_1(q^3)$. Then
\[
\fpr(x,G/H) \leq (q^2-q+1)^{-1}.
\]
\end{lemmax}

\begin{proof}
Write $H \cap \PGU_3(q) = H_0 = B{:}\<\phi\>$ where $B \leq C_{q^2-q+1}$ and $|\phi|=3$. Let $x \in H$ with $|x|=r$. 

First assume that $x \in \PGU_3(q)$. The order of any element in $B$ is a primitive divisor of $q^6-1$ and any element in $H_0 \setminus B$ has order $3$ (for it is conjugate to $\phi$ or $\phi^2$). Therefore, $r \in \ppd(q,6) \cup \{3\}$. If $r \in \ppd(q,6)$, then $|x^G| = q^3(q+1)(q^2-1)$ and $|x^G \cap H_0| = 3$. Now assume that $r=3$, so $|x^G \cap H_0| = i_3(H) = 2(q^2-q+1)$. If $p=3$, then $x = [J_3]$, so $|x^G| = q(q^2-1)(q^3+1)$. If $p \neq 3$, then $x = [\xi,\xi^{-1},1]$ where $|\xi|=3$, so $|x^G| \geq q^3(q-1)(q^2-q+1)$. The required bound holds in every case.

Now assume that $x \in \PGaU_3(q) \setminus \PGU_3(q)$. If $r \geq 5$, then
\[
\frac{|x^G \cap H|}{|x^G|} \leq \frac{|Bx|}{|x^G|} = (q^2-q+1) \cdot \frac{|\PGU_3(q^{1/r})|}{|\PSU_3(q)|} < (q^2-q+1)^{-1},
\]
and if $r=3$, then $|x^G \cap H|=0$, since all elements of order $3$ in $H$ are contained in $H_0 \leq \PGU_3(q)$. Finally if $r=2$, then $x$ is a graph automorphism and 
\[
\frac{|x^G \cap H|}{|x^G|} \leq \frac{|Bx|}{|x^G|} = (q^2-q+1) \cdot \frac{|\PSO_3(q)|}{|\PSU_3(q)|} < (q^2-q+1)^{-1}. \qedhere
\]
\end{proof}

\begin{lemmax} \label{lem:fpr_ns_u_low_c2}
Let $H$ have type $\GU_{n/k}(q) \wr S_k$ or $\GL_{n/2}(q^2)$. Then
\[
\fpr(x,G/H) \leq 
\left\{
\begin{array}{ll}
(q^2-q+1)^{-1}         & \text{if $n \in \{3,4\}$} \\
(q^4-q^3+q^2-q+1)^{-1} & \text{if $n \in \{5,6\}$ and $x$ in \eqref{eq:fpr_ns_u_low_elt}.} \\
\end{array}
\right.
\]
\end{lemmax}

\begin{proof}
We prove this lemma only when $n=4$ and $H$ has type $\GU_1(q) \wr S_4$ since the other cases are similar (bearing in mind, only the elements in \eqref{eq:fpr_ns_u_low_elt} need to be considered when $n \in \{5,6\}$). Write $H \cap \PGU_4(q) = H_0 = B{:}S_4$, where $B \leq C_{q+1}^4/\Delta$ with $\Delta = \{ (\l,\l,\l,\l) \mid \l \in C_{q+1}\}$. 

First assume that $x \in \PGU_4(q)$. Let us begin by considering the case where $x \in B$. Then $x$ is diagonal and it is easy to compute $|x^G|$. In addition, note that $((\l_1,\l_2,\l_3,\l_4)\Delta)^G \cap B$ is the set of elements $(\l_{1\s},\l_{2\s},\l_{3\s},\l_{4\s})$ for some permutation $\s \in S_4$. Therefore, one of the following hold, for distinct $\l,\mu,\nu \in \F_{q^2}^\times$ of order $r$ dividing $q+1$,
\[
{\renewcommand{\arraystretch}{1.2}
\begin{array}{cc}
\hline
x                & |x^G \cap B| \\
\hline
{}[\l,I_3]       &  4           \\ 
{}[\l,\l,I_2]    &  6           \\ 
{}[\l,\mu,I_2]   & 12           \\
{}[\l,\mu,\nu,1] & 24           \\ 
\hline
\end{array}}
\]

Now consider the case where $x \in H_0 \in B$. Let $\s \in S_4$ be the permutation that $x$ induces on the factors of $B$. Then by \cite[Lemma~5.2.6]{ref:BurnessGiudici16}, one of the following holds, where $|\xi|=3$.
\[
{\renewcommand{\arraystretch}{1.2}
\begin{array}{cccc}
\hline
\s^{S_4}                 & x \text{ (if $r=p$)} & x \text{ (if $r \neq p$)} & |x^G \cap (H \setminus B)| \\           
\hline   
(1 \, 2 \, 3)^{S_4}      & [J_3,J_1]            & [\xi,\xi^{-1},I_2]        & 8(q+1)^2                   \\
(1 \, 2)(3 \, 4)^{S_4}   & [J_2^2]              & [-I_2,I_2]                & 3(q+1)                     \\
(1 \, 2)^{S_4}           & [J_2,J_1^2]          & [-1,I_3]                  & 6(q+1)                     \\
                         &                      & [-I_2,I_2]                & 6(q+1)                     \\
\hline
\end{array}}
\]
(Regarding the final two rows of the table above, there are two $G$-classes of elements that transpose two factors: one negates one of the fixed factors and the other acts trivially on both.) Let us now justify the final column in this table. We concentrate on the case where $\s \in (1 \, 2 \, 3)^{S_4}$ as the remaining cases are similar. First note that $|\s^{S_4}|=8$. Now assume that $x$ induces $\s = (1 \, 2 \, 3)$ on the factors of $B$. It is easy to check that an element $x=(\l_1,\l_2,\l_3,1)\Delta\s \in H_0$ has order $3$ if and only if $\l_1\l_2\l_3=1$, so there are $(q+1)^2$ choices for $x$. 

With this information, it is easy to check that the required bound holds.

Now assume that $x \in \PGaU_4(q) \setminus \PGU_4(q)$. If $x$ is a field automorphism, then
\[
\frac{|x^G \cap H|}{|x^G|} \leq \frac{(4,q+1)|H||\PGU_4(q^{1/r})|}{|\PGU_4(q)|} < \frac{1}{q^2-q+1}.
\]
Finally assume that $x$ is an involutory graph automorphism. The argument for this case is given in detail in the proof of \cite[Proposition~2.7]{ref:Burness073} and the bound can easily be verified. In particular, if $x$ is symplectic, then $x$ acts as a double transposition on the four factors of $B$, so
\[
\frac{|x^G \cap H|}{|x^G|} \leq \frac{3(q+1)\cdot(4,q+1)|\PGSp_4(q)|}{|\PGU_4(q)|} \leq \frac{3(4,q+1)}{q^2(q^2-q+1)} < \frac{1}{q^2-q+1}. \qedhere
\]
\end{proof}

We are now ready to prove Theorem~\ref{thm:fpr_ns_u_low}.

\begin{proof}[Proof of Theorem~\ref{thm:fpr_ns_u_low}]
By Lemma~\ref{lem:fpr_ns_u_low_comp}, we will assume that $(n,q) \not\in \{(3,11),(5,2),(6,2)\}$. In addition, we will assume that $(G,H,x)$ does not appear in Lemmas~\ref{lem:fpr_ns_u_low_c5}--\ref{lem:fpr_ns_u_low_c2} as the required bound was shown to hold in these cases.

\emph{Case 1. $n=3$.} Consulting \cite[Tables~8.5 and~8.6]{ref:BrayHoltRoneyDougal}, we see that $H \in \C_6 \cup \S$. The conjugacy classes of elements of $\PGaU_3(q)$, together with the centraliser order, are given in \cite[Appendix~B]{ref:BurnessGiudici16} and with this information it is easy to check that $|x^G| \geq (q^2-1)(q^2-q+1)$, with equality if $x = J_2 \perp I_1$. Therefore,
\[
\frac{|x^G \cap H|}{|x^G|} \leq \frac{i_\mx(H)}{|x^G|} \leq \frac{i_\mx(H)}{(q^2-1)(q^2-q+1)}.
\]
Therefore, it suffices to prove that $i_\mx(H) \leq q^2-1$. First assume that $H \in \C_6$. Here $q \equiv 2 \mod{3}$ and $H$ has type $3^{1+2}{:}Q_8$. It is easy to check that 
\[
i_\mx(H) \leq |H| \leq 216 \leq q^2-1,
\] 
since $q \geq 17$ in this case. Now assume that $H \in \S$ and write $S = \soc(H)$. If $S = \PSL_2(7)$ and $q \equiv 3,5,6 \mod{7}$, then, $q \geq 13$ and a computation in \textsc{Magma} shows
\[
i_\mx(H) \leq i_\mx(\Aut(S)) = 84 \leq q^2-1.
\] 
Similarly, if $S = A_6$ and $q \equiv 11, 14 \mod{15}$, then $q \geq 29$ and
\[
i_\mx(H) \leq 360 \leq q^2-1.
\]

\emph{Case 2. $n=4$.} By \cite[Tables~8.10 and~8.11]{ref:BrayHoltRoneyDougal}, $H \in \C_6$ has type $2^4{:}\Sp_4(2)$ or $H \in \S$ is an almost simple group with socle $S \in \{ \PSL_2(7), A_7, \PSU_4(2) \}$. By Lemma~\ref{lem:fpr_ns_u_low_comp}, we assume that $q \geq 11$ and, proceeding as in the previous case, we obtain
\[
\frac{|x^G \cap H|}{|x^G|} \leq \frac{i_\mx(H)}{|x^G|} < \frac{1}{q^2-q+1}.
\]

\emph{Case 3. $n=5$.} By Theorem~\ref{thm:fpr_ns}, we obtain the desired bound provided that
\begin{equation} \label{eq:fpr_ns_u_low_bound_5}
|x^G| \geq (q^4-q^3+q^2-q+1)^{10/3}.
\end{equation}

If $x \in \Aut(\PSU_5(q)) \setminus \PSU_5(q)$, then \cite[Lemma~3.48]{ref:Burness072} gives \eqref{eq:fpr_ns_u_low_bound_5}. From the information in \cite[Appendix~B]{ref:BurnessGiudici16}, we see that $|x^G| > q^{40/3}$ unless $x$ is conjugate to one of
\begin{equation}
[J_2,J_1^3], \ [J_2^2,J_1], \ [\l,I_4], \ [\l,\l,I_3]. \label{eq:fpr_ns_u_low_5elt}
\end{equation}
Therefore, for the remainder of this case we will assume that $x$ is in one of these specific classes. 

The maximal subgroups of $G$ are given in \cite[Tables~8.20 and~8.21]{ref:BrayHoltRoneyDougal}, and we see that either $H$ has type $\GU_1(q^5)$ or $H \in \C_6 \cup \S$. In the first case, $x^G \cap H$ is empty by \cite[Lemma~5.3.2]{ref:BurnessGiudici16}. Now assume that $H \in \C_6 \cup \S$. In this case, it suffices to prove that $|x^G| \geq i_\mx(H)(q^4-q^3+q^2-q+1)$, whenever $x^G \cap H$ is nonempty.

First assume that $H \in \S$. In this case, $q \geq 5$ and $\soc(H)$ is either $\PSL_2(11)$ or $\PSU_4(2)$, so
\[
i_\mx(H) \leq i_\mx(\Aut(S)) = 170 \leq q^4-1.
\]
Observe that $|x^G| \geq (q^4-1)(q^4-q^3+q^2-q+1)$, with equality if $x = [J_2,I_3]$, so we obtain the desired bound.

Now assume that $H \in \C_6$. Here either $q = p \equiv 4 \mod{5}$ or $q=p^2$ with $p \equiv 2,3 \mod{5}$ and $H$ has type $5^{1+2}{:}\Sp_2(5)$. In particular, 
\[
i_\mx(H) \leq 3124 \leq q(q^3+1)(q^4-1).
\] 
If $|x^G \cap H| \neq 0$, then $\nu(x) \geq 2$ by \cite[Lemma~5.6.3]{ref:BurnessGiudici16}, so 
\[
|x^G| \geq q(q^3+1)(q^4-1)(q^4-q^3+q^2-q+1).
\]

\emph{Case 4. $n=6$.} By Theorem~\ref{thm:fpr_ns}, it suffices to show that
\begin{equation} \label{eq:fpr_ns_u_low_bound_6}
|x^G| \geq (q^4-q^3+q^2-q+1)^3,
\end{equation}
noting that the parameter $\i$ is $0$ since $H$ does not have type $\Sp_6(q)$. If $x \not\in \PGU_6(q)$, then \eqref{eq:fpr_ns_u_low_bound_6} is given by \cite[Lemma~3.48]{ref:Burness072}. Now assume that $x \in \PGU_6(q)$. If $\nu(x) \geq 2$, then, \cite[Corollary~3.38]{ref:Burness072} gives \eqref{eq:fpr_ns_u_low_bound_6}, so it remains to assume that with $\nu(x)=1$. Proposition~\ref{prop:u_max} implies that $|x^G \cap H| = 0$ unless $H$ has type $\GL_3(q^2)$ or $H \in \S$ and
\begin{equation}
\soc(H) \in \{ A_6, \ A_7, \ \PSL_3(4), \ \PSU_4(3)\}. \label{eq:fpr_ns_us_s}
\end{equation} 

First assume that $H$ has type $\GL_3(q^2)$. We claim that $|x^G \cap H|=0$. Write $V = \F_{q^2}^6$. Then $H$ stabilises a decomposition $V = U \oplus U^*$ where $U$ is a maximal totally singular subspace of $V$ and $H = B.2$ where $B$ centralises this decomposition. By \cite[Lemma~5.2.6]{ref:BurnessGiudici16}, $|x^G \cap (H \setminus B)|=0$, and all of the elements of $B$ are of the form $g \oplus g^{-(q)\tr}$, so $|x^G \cap B|=0$ also.

Now assume that $H \in \S$. Here we see that $|x^G| \geq q^5(q^2-q+1)(q^3-1)$ and, via computation in \textsc{Magma}, $i_\mx(H) \leq (q^2-q+1)(q^3-1)$, which gives the result in the familiar way. This completes the proof.
\end{proof}

\chapter{Symplectic and Orthogonal Groups} \label{c:o}

\section{Introduction}\label{s:o_intro}

We now turn our focus to proving our main results on uniform spread: Theorems~\ref{thm:main} and~\ref{thm:asymptotic}. The aim of this chapter is to prove Theorems~\ref{thm:main} and~\ref{thm:asymptotic} for even-dimensional orthogonal groups. We will consider unitary (and some linear) groups in Chapter~\ref{c:u}. For this entire chapter, write $q=p^f$ and
\begin{gather}
\T = \{ \POm^\e_{2m}(q) \mid \text{$m \geq 4$ and $\e \in \{+,-\}$} \} \\
\A = \{ \< T, \th \> \mid \text{$T \in \T$ and $\th \in \Aut(T)$} \}.
\end{gather}

The main results of this chapter are the following.
\begin{thmchap}\label{thm:o_main}
If $G \in \A$, then $u(G) \geq 2$.
\end{thmchap}

\begin{thmchap}\label{thm:o_asymptotic}
Let $(G_i)$ be a sequence of groups in $\A$ with $\soc(G_i) = \POm^{\e_i}_{2m_i}(q_i)$. Then $u(G_i) \to \infty$ if $q_i \to \infty$.
\end{thmchap}

Let us now discuss the proofs. Let $G = \<T,\th\> \in \A$ with $T \in \T$. As we explained in the introduction, to prove that $u(G) \geq k$ for some $k \geq 1$, we adopt the probabilistic approach introduced by Guralnick and Kantor in \cite{ref:GuralnickKantor00} (see Section~\ref{s:p_prob}). Recall that this approach has three stages. First we must fix an element $s \in G$. In order for $s^G$ to witness $u(G) \geq k$, the element $s$ cannot be contained in a proper normal subgroup of $G$, so we may assume that $s \in T\th$. Consequently we need to understand the conjugacy classes in the coset $T\th$. We then study the set $\M(G,s)$ of maximal subgroups of $G$ that contain $s$, before showing that every prime order element $x \in G$ satisfies
\[
P(x,s) \leq \sum_{H \in \M(G,s)} \fpr(x,G/H) < \frac{1}{k}.
\]

We must first determine the automorphisms $\th$ it suffices to consider, and this will require a detailed analysis of the automorphism group of $T$. 

Generically, $\th$ will be a field or graph-field automorphism (possibly multiplied by a nontrivial element of $\Inndiag(T)$). In this case, we view $G=\<T,\th\>$ from the perspective of algebraic groups, which allows us to employ Shintani descent. The main idea, therefore, is to write $\Inndiag(T) = X_{\s^e}$ and $\th \in \Inndiag(T)\s$ for a suitable connected algebraic group $X$, Steinberg endomorphism $\s$ and integer $e > 1$ (see Example~\ref{ex:shintani_descent}). We may then select an element $s \in T\th$ as the preimage under $F$ of a judiciously chosen element $x \in X_\s$ (see Proposition~\ref{prop:o_Ia_elt}).

However, unlike in the previous study of symplectic and odd-dimensional orthogonal groups in \cite{ref:Harper17}, it will not always be possible to write $\Inndiag(T) = X_{\s^e}$ and $\th \in \Inndiag(T)\s$ for the same Steinberg endomorphism $\s$, and we need to apply Shintani descent differently and use Lemma~\ref{lem:shintani_substitute} (see Examples~\ref{ex:shintani_descent} and~\ref{ex:shintani_substitute}).

Of course, there are other types of automorphisms $\th$ that must be considered. If $\th$ is diagonal, then we can employ methods similar to those used by Breuer, Guralnick and Kantor in \cite{ref:BreuerGuralnickKantor08}. When $\th$ is an involutory graph automorphism (for example, a reflection), then we must necessarily select an element $s \in T\th$ that fixes a $1$-space of $\F_q^{2m}$, which makes bounding $P(x,s)$ more difficult (recall from Chapter~\ref{c:fpr} that the fixed point ratio of an element of prime order on $1$-spaces can be as large as roughly $q^{-1}$). Consequently, we give a constructive proof that some specific pairs of elements generate $G$ in addition to a probabilistic argument which deals with the general case (see Proposition~\ref{prop:o_IIb_reflection}). This constructive argument is of a different flavour to much of the rest of the proofs. Finally, when $T = \POm_8^+(q)$, we must also take into account triality graph and graph-field automorphisms. Here we cannot rely on the action of $G$ on a natural module.

In light of the above discussion, it is natural to partition our analysis into the following cases
\begin{enumerate}[I]
\item $\th \in \PGaO^\e_{2m}(q) \setminus \PGO^\e_{2m}(q)$
\item $\th \in \PGO^\e_{2m}(q)$
\item $\th \in \Aut(\POm_8^+(q)) \setminus \PGaO^\e_{2m}(q)$.
\end{enumerate}
In Cases~I and~II, we define the following two subcases
\begin{enumerate}[(a)]
\item $G \cap \PGO^\e_{2m}(q) \leq \PDO^\e_{2m}(q)$
\item $G \cap \PGO^\e_{2m}(q) \not\leq \PDO^\e_{2m}(q)$.
\end{enumerate}
Recall that $\PDO^\e_{2m}(q)$ is our nonstandard notation for an index two subgroup of $\PGO^\e_{2m}(q)$ (see \eqref{eq:do_odd} and~\eqref{eq:do_even} in Section~\ref{s:p_groups}). In \eqref{eq:inndiag} in Section~\ref{s:p_algebraic}, we observed that $\PDO^\e_{2m}(q) = \Inndiag(\POm^\e_{2m}(q))$.

In short, Cases~I(b) and~II(b) are more difficult than Cases~I(a) and~II(a). Case~I(b) is exactly the situation in which Shintani descent does not apply directly, and in Case~II(b) we encounter the obstacle of graph automorphisms we discussed above. We will partition Case~III further but we reserve the details of this until the introduction to Section~\ref{s:o_III}.

This chapter is organised as follows. We begin with two sections that determine general properties about almost simple symplectic and orthogonal groups. Our reason for including symplectic and odd-dimensional orthogonal groups, in addition to being comprehensive, is that they will feature in our analysis of centralisers of linear and unitary groups in Chapter~\ref{c:u}. In particular, in Section~\ref{s:o_cases}, we will determine the conjugacy classes of the outer automorphism group and Section~\ref{s:o_elements} will introduce the elements that will play a central role in our proofs. We will then prove Theorems~\ref{thm:o_main} and~\ref{thm:o_asymptotic}, considering Cases~I--III in Sections~\ref{s:o_I}--\ref{s:o_III}, respectively.

\clearpage
\section{Automorphisms}\label{s:o_cases}

Let $T \in \T$. The main result of this section is Proposition~\ref{prop:o_cases}, which details the automorphisms $\th \in \Aut(T)$ it suffices to consider to prove Theorems~\ref{thm:o_main} and~\ref{thm:o_asymptotic}. 

\subsection{Preliminaries} \label{ss:o_cases_prelims}

Let us fix some notation. For $g \in \Aut(T)$, write $\ddot{g}$ for the set $Tg$. Therefore, $\Out(T) = \{ \ddot{g} \mid g \in \Aut(T) \}$. We begin with a preliminary elementary observation, which we will also use in Section~\ref{s:u_cases}. 

\begin{lemmax}\label{lem:division}
Let $S =\<a\>{:}\<b\> $ be a semidirect product of finite cyclic groups. For all $i >0$ there exist $j,k \in \mathbb{N}$ such that $\<ab^i\> = \<a^jb^k\>$ and $k$ divides $|b|$.
\end{lemmax}

\begin{proof}
Let $i>0$. We repeatedly use the fact that, since $\<a\> \leqn S$, for all $l \in \mathbb{N}$
\begin{equation}
(ab^i)^l \in \<a\>b^{il}. \label{eq:semidirect_product}
\end{equation}  

Write $|b|=n$, and let $k$ divide $n$ and satisfy $\< b^i \> = \< b^k \>$. Now let $r$ be the least positive integer such that $b^{ir} = b^k$. By \eqref{eq:semidirect_product}, $|ab^i| = s|b^i|$. Let $d$ be the product of the distinct prime divisors of $s$ which do not divide $r$. Then, by \eqref{eq:semidirect_product}, $(ab^i)^{r+d|b^i|} = a^jb^k$ for some $j \in \mathbb{N}$. Therefore, $\<a^jb^k\> \leq \<ab^i\>$.

Recall that $|ab^i| = s|b^i|$. Note that $(r+d|b^i|,|b^i|) = (r, |b^i|)=1$ as $\< b^{ir} \> = \< b^i \>$. Let $t$ be a prime divisor of $s$. If $t$ does not divide $r$, then $t$ does not divide $r+d|b^i|$ since $t$ divides $d$. Now assume that $t$ divides $r$. If $t$ divides $r+d|b^i|$, then $t$ divides $d|b^i|$, so $t$ divides $|b^i|$ since $t$ does not divide $d$. However, this implies that $t$ divides $(r,|b^i|) =1$, which is a contradiction. Therefore, $t$ does not divide $r+d|b^i|$. Consequently, $(r+d|b^i|,s)=1$. We now conclude that $(r+d|b^i|,s|b^i|)=1$, so $\< a^jb^k \> = \< ab^i \>$, which proves the claim.
\end{proof}

For the remainder of this section, write $n=2m$, $q=p^f$ and $V = \F_q^n$. Further, let $\B^\e$ be the basis from \eqref{eq:B_plus} or \eqref{eq:B_minus}. Write $\F_q^\times = \<\a\>$. In addition, if $q$ is odd, then let $\b \in \F_q^\times$ with $|\b|=(q-1)_2$ and note that $\a, \b \not\in (\F_q^\times)^2$. 

\subsection{Plus-type}\label{s:o_cases_plus}

Let $T =\POm^+_{2m}(q)$ with $m \geq 4$. Fix the standard Frobenius endomorphism $\p = \p_{\B^+}\:(a_{ij})\mapsto(a_{ij}^p)$ and the standard reflection $r \in \PO^+_{2m}(q)$ from Definition~\ref{def:phi_gamma_r}. It will be useful to fix $\rsq$ and $\rns$ as the images in $\PO^+_{2m}(q)$ of reflections in vectors of square and nonsquare norm respectively (evidently, if $q$ is even, then we do not use the notation $\rns$). In \cite[Section~2]{ref:KleidmanLiebeck}, the symbols $\rsq$ and $\rns$ (and also $\d$, introduced below) refer to elements of $\GO^+_{2m}(q)$, but we prefer to use these symbols for elements of $\PGO^+_{2m}(q)$.

\begin{definitionx} \label{def:delta}
Let $q$ be odd. With respect to the basis $\B^+$ for $\F_q^{2m}$, define $\hat{\d}^+ \in \GL_{2m}(q)$ as $\b I_m \oplus I_m$, which centralises the decomposition $\<e_1,\dots,e_m\> \oplus \<f_1,\dots,f_m\>$ and let $\d^+ \in \PGL_{2m}(q)$ be the image of $\hat{\d}^+$.
\end{definitionx}

\begin{remarkx} \label{rem:delta}
We comment on Definition~\ref{def:delta}.
\begin{enumerate}
\item Note that $\hat{\d}^+$ is a similarity with $\tau(\hat{\d}^+) = \b$ and $\det(\hat{\d}^+) = \b^m$.
\item We will refer to $\d^+$ simply as $\d$ if the sign is understood. (A different element $\d^- \in \PDO^-_{2m}(q)$ will be introduced in Section~\ref{s:o_cases_minus}.) 
\item Our definition of $\d$ differs from that in \cite{ref:KleidmanLiebeck}: both versions centralise the decomposition $\<e_1,\dots,e_m\> \oplus \<f_1,\dots,f_m\>$, but we work with $\b I_m \oplus I_m$ rather than $\a I_m \oplus I_m$. However, both versions give the same element $\ddot{\d}$. To see this, write $k = ((q-1)_{2'}-1)/2$ and note that
\[
(\a I_m \oplus I_m) \cdot (\a^k I_m \oplus \a^{-k} I_m) \cdot \a^kI_{2m} = \b I_m \oplus I_m
\]
where $(\a^k I_m \oplus \a^{-k} I_m) \in \Sp_{2m}(q)$ and $\a^kI_{2m}$ is a scalar.
\end{enumerate} 
\end{remarkx}

By \cite[Proposition~2.7.3]{ref:KleidmanLiebeck}, if $T=\POm_{2m}^+(q)$ with $m \geq 5$, then
\begin{equation}\label{eq:o_out_plus}
\Out(T) = \left\{
\begin{array}{ll}
\< \ddot{r}_{\square} \> \times \< \ddot{\p} \>                                 \cong C_2 \times C_f             & \text{if $q$ is even} \\
\< \ddot{\d} \> \times  \< \ddot{r}_{\square} \> \times \< \ddot{\p} \>         \cong C_2 \times C_2 \times C_f  & \text{if $q$ is odd \& $D(Q)=\nonsquare$} \\
\< \ddot{\d}, \, \ddot{r}_{\square}, \, \ddot{r}_{\nonsquare}, \, \ddot{\p} \>  \cong D_8 \times C_f             & \text{if $q$ is odd \& $D(Q)=\square$.}
\end{array}
\right.
\end{equation}

Now assume that $m=4$. The group $\POm_8^+(q)$ has a \emph{triality} automorphism $\t$ such that $C_G(\t) \cong G_2(q)$ (see \cite[pp.200--202]{ref:Carter72}). From \cite[Section~1.4]{ref:Kleidman87}, if $T=\POm_8^+(q)$, then
\begin{equation}\label{eq:o_out_plus_8}
\Out(T) = \left\{
\begin{array}{ll}
\< \ddot{r}_{\square}, \ddot{\t} \> \times \< \ddot{\p} \>           \cong S_3 \times C_f & \text{if $q$ is even} \\
\< \ddot{\d}, \ddot{r}_{\square}, \ddot{\t} \> \times \<\ddot{\p} \>  \cong S_4 \times C_f & \text{if $q$ is odd.}
\end{array}
\right.
\end{equation}

\begin{remarkx}\label{rem:o_out_plus}
Let $T = \POm^+_{2m}(q)$. Assume that $q$ is odd and $D(Q) = \square$. By \cite[Proposition~2.7.3(iii)]{ref:KleidmanLiebeck}, $\< \ddot{r}_{\square}, \ddot{r}_{\nonsquare}, \ddot{\d}\> \cong D_8$. Moreover, if $m$ is even, then 
\[
|\ddot{r}_{\square}\ddot{\d}|=4, \quad |\ddot{\d}|=2, \quad (\ddot{r}_{\square}\ddot{\d})^{\ddot{\d}} = (\ddot{r}_{\square}\ddot{\d})^{-1}, \quad (\ddot{r}_{\square}\ddot{\d})^2 = \ddot{r}_{\square}\ddot{r}_{\nonsquare},
\]
and if $m$ is odd, then
\[
|\ddot{\d}|=4, \quad |\ddot{r}_{\square}\ddot{\d}|=2, \quad \ddot{\d}^{\ddot{r}_{\square}\ddot{\d}} = \ddot{\d}^{-1}, \quad \ddot{\d}^2 = \ddot{r}_{\square}\ddot{r}_{\nonsquare}.
\]
In both cases, $Z(\< \ddot{r}_{\square}, \ddot{r}_{\nonsquare}, \ddot{\d}\>) = \< \ddot{r}_{\square}\ddot{r}_{\nonsquare} \>$.

It will be convenient to write $\Out_0(T)$ for $\GaO^+_{2m}(q)/T$, so $\Out_0(T)=\Out(T)$ if $m \geq 5$ and $|\Out(T):\Out_0(T)|=3$ if $m=4$. Since $\p$ arises from an automorphism of $\GL_{2m}(q)$, the group $\Out_0(T)$ splits as the semidirect product $\< \ddot{r}_{\square}, \ddot{r}_{\nonsquare}, \ddot{\d}\>{:}\<\ddot{\p} \>$. If $\ddot{\p} \in Z(\Out_0(T))$, then evidently we have $\Out_0(T) \cong D_8 \times C_f$. However, $\ddot{\p}$ need not be central in $\Out_0(T)$. In particular, by \cite[Proposition~2.7.3(iii)]{ref:KleidmanLiebeck}, 
\[
[\ddot{r}_{\square}, \ddot{\p}] = [\ddot{r}_{\nonsquare}, \ddot{\p}] = 1
\]
but 
\[
\ddot{\p} \not\in Z(\Out_0(T)) \iff [\ddot{\d},\ddot{\p}] \neq 1 \iff \text{$m$ is odd and $p \equiv 3 \mod{4}$}.
\] 
If $\ddot{\p} \not\in Z(\Out_0(T))$, then $\ddot{\d}$ has order 4 and $\ddot{\d}^{\ddot{\p}} = \ddot{\d}^{-1}$, which implies that $\Out_0(T) = \< \ddot{r}_{\square}, \ddot{r}_{\nonsquare}, \ddot{\d}\> \times \<\ddot{r}_{\square}\ddot{\p} \>$. In this case, $p \equiv 3 \mod{4}$ and $q \equiv 1 \mod{4}$, so $f$ is even and $\ddot{r}_{\square}\ddot{\p}$ has order $f$; this shows that $\Out_0(T) \cong D_8 \times C_f$ in this case also.
\end{remarkx}

\begin{remarkx}\label{rem:o_out_plus_8}
Let $m=4$. In this case $\ddot{\p} \in Z(\Out(T))$, and $\{1, \ddot{r}_{\square}, \ddot{\tau}\}$ is a set of conjugacy class representatives of $\< \ddot{r}_{\square}, \ddot{\t} \> \cong S_3$ if $q$ is even and $\{1, \ddot{r}_{\square}, \ddot{\d}, \ddot{\d}\ddot{r}_{\square}, \ddot{\tau} \}$ is a set of conjugacy class representatives of $\< \ddot{r}_{\square}, \ddot{r}_{\nonsquare}, \ddot{\d}, \ddot{\tau}\>$ if $q$ is odd.
\end{remarkx}

The following lemma provides further information when $q$ is odd and $D(Q) = \square$. It is useful to record the following set of conditions
\begin{equation}
\text{$m$ is odd \ \emph{and} \  $p \equiv 3 \mod{4}$ \ \emph{and} \  $i$ is odd \ \emph{and} \ $f$ is even.} \label{eq:o_cases_condition}
\end{equation}

\begin{lemmax} \label{lem:o_out_plus_facts}
Let $T = \POm^+_{2m}(q)$. Assume that $q$ is odd and $D(Q) = \square$. For $0 \leq i < f$, the following hold
\begin{enumerate}
\item $\ddot{\d}\ddot{\p}^i$ and $\ddot{r}_{\square}\ddot{r}_{\nonsquare}\ddot{\d}\ddot{\p}^i$ are $\Out(T)$-conjugate
\item $\ddot{\d}\ddot{r}_\square\ddot{\p}^i$ and $\ddot{\d}\ddot{r}_\nonsquare\ddot{\p}^i$ are $\Out(T)$-conjugate
\item $\ddot{\p}^i$ and $\ddot{r}_\square\ddot{r}_\nonsquare\ddot{\p}^i$ are $\Out(T)$-conjugate if \eqref{eq:o_cases_condition} holds
\item $\ddot{r}_\square\ddot{\p}^i$ and $\ddot{r}_\nonsquare\ddot{\p}^i$ are $\Out(T)$-conjugate if \eqref{eq:o_cases_condition} does not hold.
\end{enumerate}
\end{lemmax}

\begin{proof}
Write $A = \< \ddot{r}_{\square}, \ddot{r}_{\nonsquare}, \ddot{\d} \>$. The description of $\Out_0(T)$ in Remark~\ref{rem:o_out_plus} allows us to deduce that the conjugacy classes of $A$ are
\[
\{\ddot{1}\}, \ \ \{\ddot{r}_{\square}\ddot{r}_{\nonsquare}\}, \ \ \{ \ddot{r}_\square, \, \ddot{r}_\nonsquare \}, \ \ \{ \ddot{\d}, \, \ddot{r}_\square\ddot{r}_\nonsquare\ddot{\d} \}, \ \ \{ \ddot{\d}\ddot{r}_\square, \, \ddot{\d}\ddot{r}_\nonsquare \}.
\]
If the condition~\eqref{eq:o_cases_condition} is not satisfied, then $\ddot{\p}^i \in Z(\Out_0(T))$ and (i), (ii) and~(iv) follow. Now assume that condition~\eqref{eq:o_cases_condition} is satisfied. In this case $\ddot{r}_\square\ddot{\p}^i \in Z(\Out_0(T))$. Writing 
\begin{gather*}
\ddot{\d}\ddot{\p}^i = \ddot{r}_\nonsquare \ddot{\d} (\ddot{r}_\square \ddot{\p}^i) \quad \text{and} \quad \ddot{r}_\square \ddot{r}_\nonsquare \ddot{\d} \ddot{\p}^i = \ddot{r}_\square \ddot{\d} (\ddot{r}_\square \ddot{\p}^i) \\
\ddot{\d} \ddot{r}_\square \ddot{\p}^i = \ddot{\d} (\ddot{r}_\square \ddot{\p}^i) \quad \text{and} \quad \ddot{\d} \ddot{r}_\nonsquare \ddot{\p}^i = \ddot{r}_\square \ddot{r}_\nonsquare \ddot{\d} (\ddot{r}_\square \ddot{\p}^i) \\
\ddot{\p}^i = \ddot{r}_\square (\ddot{r}_\square \ddot{\p}^i) \quad \text{and} \quad \ddot{r}_\square \ddot{r}_\nonsquare \ddot{\p}^i = \ddot{r}_\nonsquare (\ddot{r}_\square \ddot{\p}^i)
\end{gather*}
reveals that (i), (ii) and~(iii) hold.
\end{proof}

Recall the definition of $\PDO^+_{2m}(q)$ from Section~\ref{s:p_groups} (see \eqref{eq:do_odd} and \eqref{eq:do_even}). The following is  \cite[Proposition~2.7.4]{ref:KleidmanLiebeck}, but it can be quickly deduced from \eqref{eq:inndiag_o}.

\begin{lemmax} \label{lem:o_inndiag_plus}
Let $T = \POm^+_{2m}(q)$ with $m \geq 4$. Then 
\[
\Inndiag(T) = \PDO^+_{2m}(q) = \left\{
\begin{array}{ll}
T                    & \text{if $q$ is even} \\
\<T, \d \>           & \text{if $q$ is odd and $D(Q) = \nonsquare$} \\
\<T, \rsq\rns, \d \> & \text{if $q$ is odd and $D(Q) = \square$.} \\
\end{array}
\right.
\]
\end{lemmax}

\subsection{Minus-type}\label{s:o_cases_minus}

Now let $T = \POm^-_{2m}(q)$ with $m \geq 4$. To describe $\Out(T)$ in this case we deviate from \cite{ref:KleidmanLiebeck} and work more in the spirit of \cite{ref:GorensteinLyonsSolomon98}. This is because we want to work with a copy of $\POm^-_{2m}(q)$ that arises naturally from the perspective of algebraic groups. However, we do want to be able to concretely work with the action of $\POm^-_{2m}(q)$ on the natural module $\F_q^{2m}$, so we will recover some of the key results from \cite[Section~2.8]{ref:KleidmanLiebeck} in our context. In this section, the isomorphism $\Psi$ from Lemma~\ref{lem:algebraic_finite_minus} will be the key tool for relating our two viewpoints.

Recall the standard Frobenius endomorphism $\p = \p_{\B^+}\:(a_{ij})\mapsto(a_{ij}^p)$ and the reflection $r \in \PO^+_{2m}(q)$ from Definition~\ref{def:phi_gamma_r}. Recall from Lemma~\ref{lem:algebraic_finite_minus} that $\PDO^-_n(q) = \Psi(X_{r\p^f})$, where $X = \PSO_n(\FF_p)$. Define $\psi\: \Psi(X) \to \Psi(X)$ as 
\begin{equation}\label{eq:psi}
\psi = \Psi \circ \p \circ \Psi^{-1}. 
\end{equation}
Then
\[
\Aut(T) = \PDO^-_n(q){:}\<\psi\> = \PGaO^-_n(q)
\]
and $\psi^f=\Psi(r)=r$. We use $\rsq$ and $\rns$ as in plus-type, but we often, instead, work with the reflection $r$, which we may assume is contained in $\{ \rsq,\rns \}$. 

If $q$ is odd, then we define a further element.

\begin{definitionx} \label{def:delta_minus}
Let $q$ be odd. With respect to $\B^+$, define $\Delta \in \GO^+_{2m}(q^2)$ as $\b I_{m-1} \oplus I_{m-1} \perp [\b_2, \b_2^q]$, centralising $\<e_1,\dots,e_{m-1}\> \oplus \<f_1,\dots,f_{m-1}\> \perp \<e_m,f_m\>$, where $\b_2 \in \F_{q^2}^\times$ has order $(q^2-1)_2$. Let $\hat{\d}^-$ be $\Psi(\Delta)$ and $\d^- \in \PGO^-_{2m}(q)$ its image.
\end{definitionx}

\begin{remarkx} \label{rem:delta_minus}
We comment on Definition~\ref{def:delta_minus}.
\begin{enumerate}
\item If the sign $-$ is understood, then we omit reference to it.
\item Since $\Delta \in \GO^+_{2m}(q^2)$ is fixed by $r\p^f$, we have $\hat{\d} \in \GO^-_{2m}(q)$. 
\item Evidently, $\det(\Delta) = \b^m$, so $\det(\hat{\d}) = \det(\Phi(\Delta)) = \b^m$.
\item It is straightforward to verify that $\tau(\Delta) = \b_2^{q+1} = \b$, with respect to the plus-type form defined in terms of $\B^+$. This implies that $\tau(\hat{\d}) = \b$ with respect to the minus-type form defined in terms of $\B^-$. 
\end{enumerate} 
\end{remarkx}

\begin{lemmax}\label{lem:o_inndiag_minus}
Let $T = \POm^-_{2m}(q)$. Then
\[
\Inndiag(T) = \PDO^-_{2m}(q) = \left\{
\begin{array}{ll}
T            & \text{if $q$ is even} \\
\<T, \d \>   & \text{if $q$ is odd.}
\end{array}
\right.
\]
\end{lemmax}

\begin{proof}
By \eqref{eq:inndiag_o}, $\Inndiag(T) = \PDO^-_{2m}(q)$. If $q$ is even, then $\PDO^-_{2m}(q) = T$ (see \eqref{eq:do_even}). Now assume that $q$ is odd. Note that $\tau(\hat{\d}) = \b$, so $\d \not\in \PO^-_{2m}(q)$. Since $|\PGO^-_{2m}(q):\PO^-_{2m}(q)|=2$, we deduce that $\PGO^-_{2m}(q) = \< \PO^-_{2m}(q), \d \>$. Now $\PDO^-_{2m}(q) \cap \PO^-_{2m}(q) = \PSO^-_{2m}(q)$ and $\d^- \in \PDO^-_{2m}(q)$ since $\det(\hat{\d}) = \b^m = \tau(\hat{\d})^m$, so $\PDO^-_{2m}(q) = \< \PSO^-_{2m}(q), \d \>$. Since $\Inndiag(T)/T$ is cyclic (see \cite[Theorem~2.5.12]{ref:GorensteinLyonsSolomon98}) in fact, $\PDO^-_{2m}(q) = \< T, \d\>$, which completes the proof.
\end{proof}

\begin{remarkx}\label{rem:o_minus_kl}
In light of Lemma~\ref{lem:o_inndiag_minus}, let us compare our notation for $\PGO^-_{2m}(q)$ with the notation in \cite[Section~2.8]{ref:KleidmanLiebeck}. Their symbol $\ddot{r}_\square$ is also our $\ddot{r}_\square$, but their $\ddot{\d}$ is our $\ddot{r}^m\ddot{\d}$. Therefore, we may conclude from \cite[Section~2.8]{ref:KleidmanLiebeck} that, in our notation, if $D(Q)=\nonsquare$ then $|\ddot{\d}|=2$, and if $D(Q)=\square$ then $|\ddot{\d}|=4$ with $\dddot{\d}^2 = \ddot{r}_\square\ddot{r}_\nonsquare$.
\end{remarkx}

We now describe $\Out(T)$.

\begin{lemmax} \label{lem:o_out_minus}
Let $T = \POm^-_{2m}(q)$. Then
\[
\Out(T) = \left\{
\begin{array}{ll}
\< \ddot{\psi} \>                        \cong C_{2f}            & \text{if $q$ is even} \\
\< \ddot{\d} \> \times \< \ddot{\psi} \> \cong C_2 \times C_{2f} & \text{if $q$ is odd and $D(Q)=\nonsquare$} \\
\< \ddot{\d} \> {:} \< \ddot{\psi} \>    \cong C_4{:}C_{2f}      & \text{if $q$ is odd and $D(Q)=\square$.} \\
\end{array}
\right.
\]
\end{lemmax}

\begin{proof}
Recall that $\Aut(T) = \Inndiag(T){:}\<\psi\>$. By Lemma~\ref{lem:o_inndiag_minus} we see that $\Out(T) = \<\ddot{\psi}\>$ when $q$ is even and $\Out(T) = \<\ddot{\d}\>{:}\<\ddot{\psi}\>$ when $q$ is odd. Since $|\ddot{\psi}|=|\psi|=2f$, we have proved the claim when $q$ is even.

Now assume that $q$ is odd. If $D(Q)=\nonsquare$, then, by Remark~\ref{rem:o_minus_kl}, $|\ddot{\d}|=2$, so $\ddot{\psi}$ centralises $\ddot{\d}$. It remains to assume that $D(Q)=\square$. In this case, $f$ is necessarily odd (see \eqref{eq:discriminant_condition}), so $\<\ddot{\psi}\> = \<\ddot{r}\ddot{\psi}^2\>$, since $\psi^f=r$. By Remark~\ref{rem:o_minus_kl}, $|\ddot{\d}|=4$, so $\ddot{\psi}^2$, having odd order, centralises $\ddot{\d}$. Since $r_v^\d = r_{v\d}$, for any $v \in V$, we know that $\ddot{r}_\square^{\ddot{\d}} = \ddot{r}_\nonsquare$. Therefore,
\[
\ddot{\d}^{\ddot{\psi}} = \ddot{\d}^{\ddot{r}} = \ddot{\d}\ddot{r}_\square\ddot{r}_\nonsquare = \ddot{\d}^{-1}.
\]
This completes the proof.
\end{proof}

\begin{remarkx} \label{rem:o_out_minus}
Let $T = \POm^-_{2m}(q)$. Assume that $q$ is odd and $D(Q) = \square$. From the proof of Lemma~\ref{lem:o_out_minus}, $|\ddot{\d}|=4$, $|\ddot{r}_\square|=2$ and $\ddot{\d}^{\ddot{r}_\square} = \ddot{\d}^{-1}$, so  $\<\ddot{\d},\ddot{r}\> \cong D_8$. Moreover, $[\ddot{\d},\psi^2] = 1$, so
\[
\Out(T) \cong \<\ddot{\d},\ddot{r}\> \times \<\ddot{\psi}^2\> \cong D_8 \times C_f.
\]
\end{remarkx}

\begin{lemmax} \label{lem:o_out_minus_facts}
Let $T = \POm^-_{2m}(q)$. Assume that $q$ is odd and $D(Q)=\square$. For $0 \leq i < 2f$, the following hold
\begin{enumerate}
\item $\ddot{\d}\ddot{\psi}^i$ and $\ddot{\d}^{-1}\ddot{\psi}^i$ are $\Out(T)$-conjugate
\item if $i$ is odd, then $\ddot{\psi}^i$ and $\ddot{r}_\square\ddot{r}_\nonsquare\ddot{\psi}^i$ are $\Out(T)$-conjugate.
\end{enumerate}
\end{lemmax}

\begin{proof}
From Remark~\ref{rem:o_out_minus}, $\ddot{\d}^{\ddot{r}} = \ddot{\d}^{-1}$ and $[\ddot{r},\ddot{\psi}] = 1$, so $(\ddot{\d}\ddot{\psi}^i)^{\ddot{r}} = \ddot{\d}^{-1}\ddot{\psi}^i$. Moreover, if $i$ is odd, then $(\ddot{\psi}^i)^{\ddot{\d}} = \ddot{\d}^{-1}\ddot{\d}^{\ddot{\psi}^i} \ddot{\psi}^i = \ddot{\d}^{-1}\ddot{\d}^{-1}\ddot{\psi}^i = \ddot{r}_\square\ddot{r}_\nonsquare\ddot{\psi}^i$.
\end{proof}

\subsection{Conjugacy of outer automorphisms}\label{ss:o_cases}

For this section, define
\begin{equation}
d = \left\{
\begin{array}{ll}
1 & \text{if $\e=+$} \\
2 & \text{if $\e=-$}
\end{array}
\right.
\end{equation}

\begin{propositionx}\label{prop:o_cases}
Let $G \in \A$ with $\soc(G) = T = \POm^\e_{2m}(q)$. Then $G$ is $\Aut(T)$-conjugate to $\<T,\th\>$ for exactly one of the following
\begin{enumerate}
\item $\th$ in Row~(1) of Table~\ref{tab:o_cases}
\item $\th$ in Row~(2) of Table~\ref{tab:o_cases}, if $q$ is odd
\item $\th$ in Row~(3) of Table~\ref{tab:o_cases}, if $q$ is odd and $D(Q) = \square$
\item $\th = \tau\p^i$ where $i$ is $0$ or a proper divisor of $f$, if $m=4$ and $\e=+$.
\end{enumerate} 
\end{propositionx}

\begin{table}
\centering
\caption{The relevant automorphisms $\th$ in when $T=\POm^\e_{2m}(q)$} \label{tab:o_cases}
{\renewcommand{\arraystretch}{1.2}
\begin{tabular}{ccccccccc}
\cline{1-8}  
           & I(i)           & I(ii)            & I(iii)     & I(iv)            & I(v)       & II(a)      & II(b)  &     \\
\cline{1-8}  
$\e$       & $+$            & $+$              & $+$        & $-$              & $-$        &            &        &     \\
\cline{1-8}   
$   $      & $\p^i$         & $r\p^i$          & $r\p^i$    & $\psi^i$         & $\psi^i$   & $1$        & $r$    & (1) \\
$\th$      & $\d\p^i$       & $\d r\p^i$       & $\d r\p^i$ & $\d\psi^i$       & $\d\psi^i$ & $\d$       & $\d r$ & (2) \\
$   $      & $\rsq\rns\p^i$ & $\rsq\rns r\p^i$ &            & $\rsq\rns\psi^i$ &            & $\rsq\rns$ &        & (3) \\
\cline{1-8} 
$df/i$     & any            & even             & odd        & odd              & even       &            &        &     \\
notes      & $\star$        & $\dagger$        &            &                  &            &            &        &     \\
\cline{1-8} 
\end{tabular}
} 
\\[5pt]
{ \small Note: $i$ is a proper divisor of $df$ and the notes are given in Remark~\ref{rem:o_cases}}
\end{table}

Before proving Proposition~\ref{prop:o_cases} we must comment on Table~\ref{tab:o_cases}. 

\begin{remarkx}\label{rem:o_cases}
Let us explain how to read Table~\ref{tab:o_cases}.
\begin{enumerate}
\item In Case~I, the possibilities for $\th$ depend on whether $\e$ is $+$ or $-$. Moreover, we have used the conditions on $\e$ and $i$ to define five subcases. Observe that Case~I(a) is the union of Cases~I(i), I(ii) and I(iv), whereas Case~I(b) is the union of Cases~I(iii) and~I(v). We will often refer to these subcases.
\item In Case~II, the description is uniform for both signs $\e$, but we have noted which of Cases~II(a) and~II(b) the automorphism $\th$ arises in.
\item Observe that part~(iv) of Proposition~\ref{prop:o_cases} corresponds to Case~III. We define subcases of Case~III in the introduction to Section~\ref{s:o_III}.
\item We now comment on the notes. 
\begin{itemize}
\item[$\star$]   $\ddot{\p}^i$ and $\ddot{r}_\square\ddot{r}_\nonsquare\ddot{\p}^i$ are $\Out(T)$-conjugate if the condition \eqref{eq:o_cases_condition} holds.
\item[$\dagger$] $\ddot{r}\ddot{\p}^i$ and $\ddot{r}_\square\ddot{r}_\nonsquare\ddot{r}\ddot{\p}^i$ are $\Out(T)$-conjugate \emph{unless} \eqref{eq:o_cases_condition} holds.
\end{itemize} \label{item:o_cases_notes}
\end{enumerate}
\end{remarkx}

\begin{proof}[Proof of Proposition~\ref{prop:o_cases}]
Write $G = \<T, g\>$ where $g \in \Aut(T)$. We will study the $\Out(T)$-conjugacy classes, since two groups $\<T,g_1\>$ and $\<T,g_2\>$ are $\Aut(T)$-conjugate if and only if $\ddot{g}_1$ and $\ddot{g}_2$ are $\Out(T)$-conjugate.

Begin by assuming that $\e=+$. By inspecting the structure of $\Out(T)$ given above, it is manifest that we may write $g = h\p^i$ where $h$ is a product of diagonal and graph automorphisms. Assume for now that $i > 0$. Since $\<\ddot{h}, \ddot{\p}\> = \<\ddot{h}\>{:}\<\ddot{\p}\>$, by Lemma~\ref{lem:division}, there exist $j,k \in \mathbb{N}$ with $k$ dividing $f$ such that $\<\ddot{h}\ddot{\p}^i\> = \<\ddot{h}^j\ddot{\p}^k\>$. Therefore, we assume that $i$ divides $f$. That is, we may assume that $\ddot{g} = \ddot{h}\ddot{\p}^i$ where $h$ is a product of diagonal and graph automorphisms and where either $i=0$ or $i$ divides $f$. If either $m \geq 5$ and $q$ is even or $q$ is odd and $D(Q)=\nonsquare$, then $\ddot{g}$ is clearly equal to an automorphism in Table~\ref{tab:o_cases}. Moreover, if $m=4$ or if $q$ is odd and $D(Q)=\square$, then Remark~\ref{rem:o_out_plus_8} and Lemma~\ref{lem:o_out_plus_facts} establish that $\ddot{g}$ is $\Out(T)$-conjugate to an automorphism featuring in the statement of the proposition. This proves the result when $\e=+$.

Now assume that $\e=-$. As in plus-type, we can assume that $\ddot{g} = \ddot{h}\ddot{\psi}^i$ where $h$ is a diagonal automorphism and where either $i=0$ or $i$ divides $2f$. Noting that $\psi^f = r$, it follows that $\ddot{g}$ is $\Out(T)$-conjugate to an automorphism $\ddot{\th}$ in the statement, where we apply Lemma~\ref{lem:o_out_minus_facts} when $q$ is odd and $D(Q) = \square$. This completes the proof.
\end{proof}

\begin{remarkx}\label{rem:s_cases}
Proposition~\ref{prop:o_cases} determines the $\Out(T)$-classes when $T = \POm_{2m}^\e(q)$ and Proposition~\ref{prop:u_cases} does when $T = \PSL_n^\e(q)$. For completeness let us record these classes when $T$ is $\PSp_{2m}(q)$ or $\Omega_{2m+1}(q)$ (see \cite[Propositions~2.4.4 and~2.6.3]{ref:KleidmanLiebeck}). In this case, if $G = \<T,g\>$ for $g \in \Aut(T)$, then $G$ is $\Aut(T)$-conjugate to $\<T,\th\>$ for exactly one of the automorphisms $\th$ below, where $i$ is a divisor of $f$:
\[
{\renewcommand{\arraystretch}{1.2}
\begin{array}{cccc}
\hline
T                & \text{conditions}         & \Out(T)                                                                 &                                \\
\hline
\PSp_{2m}(q)     & \text{$p = 2$ \& $m > 2$} & \<\ddot{\p}\>                                      \cong C_f            & \p^i                           \\
                 & \text{$p = 2$ \& $m = 2$} & \<\ddot{\r}\>                                      \cong C_{2f}         & \r^j\ \text{(for $j \div 2f$)} \\
                 & \text{$p > 2$}            & \<\ddot{\d}, \ddot{\p}\>                           \cong C_2 \times C_f & \p^i,\, \d\p^i                 \\
\Omega_{2m+1}(q) & \text{$p > 2$}            & \<\ddot{r}_\square\ddot{r}_\nonsquare, \ddot{\p}\> \cong C_2 \times C_f & \p^i,\, \rsq\rns\p^i           \\
\hline
\end{array}}
\]
\vspace{3pt}
\end{remarkx}

\begin{remarkx}\label{rem:bray_holt_roneydougal}
We note in passing that our approach of considering each simple group $T$ and each automorphism $\th \in \Aut(T)$ (with the reductions justified by Proposition~\ref{prop:o_cases}) allows us to avoid mentioning the classical groups that Bray, Holt and Roney-Dougal \cite{ref:BrayHoltRoneyDougal09} highlight are not well-defined (such as the one often referred to as $\mathrm{P}\Sigma\mathrm{O}_{2m}^+(q)$).
\end{remarkx}

Now that we have established the cases to consider, let us conclude this section by immediately handling some small orthogonal groups. This result can be established by way of computation in \textsc{Magma} (see Section~\ref{s:p_computation}). 

\begin{propositionx}\label{prop:o_computation}
Let $G \in \A$. Then $u(G) \geq 2$ if the socle of $G$ is one of the following groups
\begin{equation}\label{eq:o_computation}
\Omega_8^\pm(2), \ \POm_8^\pm(3), \ \Omega_8^\pm(4), \ \Omega^\pm_{10}(2), \ \Omega^\pm_{12}(2).
\end{equation}
\end{propositionx}

\clearpage
\section{Elements} \label{s:o_elements}

For this section, write $V = \F_q^n$ where $n \geq 1$ and $q=p^f$. Write $\F_q^\times = \< \a \>$. We will define several \emph{types} of semisimple elements in symplectic and orthogonal groups that will play an important part in the proofs both later in this chapter and in Chapter~\ref{c:u}. (Indeed our reason for considering symplectic groups is that, in addition to orthogonal groups, they arise as centralisers of graph automorphsisms in unitary groups and these feature significantly in Section~\ref{ss:u_IIb}.) 

\subsection{Preliminaries} \label{ss:o_elements_prelims}

The following technical result will be useful.

\begin{lemmax}\label{lem:technical}
Let $r$ be a primitive divisor of $q^n-1$. Let $g \in \GL_n(q)$ and assume that $g$ has an eigenvalue over $\FF_p$ of order $r$. Then $g$ is irreducible on $\F_q^n$ and the eigenvalues of $g$ over $\FF_p$ are $\l,\l^q,\dots,\l^{q^{n-1}}$, which are all distinct.
\end{lemmax}

\begin{proof}
Let $\l \in \FF_p$ be an eigenvalue of $g$ of order $r$ and let $\phi$ be the minimal polynomial of $\l$ over $\F_q$. Since $r$ is a primitive divisor of $q^n-1$, the element $\l$ is contained in $\F_{q^n}$ and is not contained in any proper subfield of $\F_{q^n}$. Therefore the degree of $\phi$ is $n$, so $\phi$ is the characteristic polynomial of $g$. This implies that $g$ has an irreducible characteristic polynomial, so, by Lemma~\ref{lem:irreducible_criterion}, $g$ is irreducible on $\F_q^n$. Moreover, the eigenvalues of $g$ are the roots of $\phi$, which are the $n$ distinct Galois conjugates $\l, \l^q, \dots, \l^{q^{n-1}}$. This completes the proof.
\end{proof}

Applying Lemma~\ref{lem:technical} gives the following familiar result.

\begin{lemmax}\label{lem:irreducible}
Let $r$ be a primitive divisor of $q^n-1$ and let $\l \in \F_{q^n}^\times$ of order $r$. Then $\GL_n(q)$ has an irreducible element of order $r$ and eigenvalues $\l, \l^q, \dots, \l^{q^{n-1}}$.
\end{lemmax}

\begin{proof}
Consider the a field extension embedding $\pi\:\GL_1(q^n)\to\GL_n(q)$. Now $g=\pi((\l)) \in G$ has order $r$ and $\l$ is an eigenvalue of $g$. Therefore, by Lemma~\ref{lem:technical}, $g$ is irreducible and has eigenvalues $\l, \l^q, \dots, \l^{q^{n-1}}$.
\end{proof}

For the remainder of this section write $n=2m$. Extending the argument in the previous proof to symplectic and orthogonal groups yields the following two results. We only prove the latter since the former is similar but easier.

\begin{lemmax}\label{lem:irreducible_minus}
Let $G$ be $\Sp_{2m}(q)$ or $\SO^-_{2m}(q)$. Let $r$ be a primitive divisor of $q^{2m}-1$ that divides $q^m+1$ and let $\l \in \F_{q^{2m}}^\times$ have order $r$.  Then $G$ contains an irreducible element of order $r$ and eigenvalues $\l, \l^q, \dots, \l^{q^{2m-1}}$.
\end{lemmax}

\begin{lemmax} \label{lem:irreducible_minus_similarity}
Let $q$ be odd and let $G$ be either $\GSp_{2m}(q)$ or $\DO^-_{2m}(q)$. Let $r$ be a divisor of $q^m+1$ that is divisible by $(q^m+1)_2$. Assume that $r/2$ is a primitive divisor of $q^{2m}-1$. Then $G$ contains an element $g$ of order $(q-1)r$ such that $\t(g) = \a$ and $g^{q-1}$ is irreducible.
\end{lemmax}

\begin{proof}
First assume that $G = \GSp_{2m}(q)$. Let $\l \in \F_{q^{2m}}^\times$ have order $(q-1)r$. The order of $\l^{q^m+1}$ is $(q-1)r/(q^m+1,(q-1)r)$. Since $r$ divides $q^m+1$,
\[
(q^m+1, \ (q-1)r) = r \, \left( \tfrac{1}{r}(q^m+1), \, q-1 \right) = r,
\]
since $(q^m+1,q-1) = 2$ and $(q^m+1)_2$ divides $r$. Therefore, $\l^{q^m+1}$ has order $q-1$. Consequently, we may choose $\l$ such that $\l^{q^m+1} = \a$. 

There is a field extension embedding $\pi_1\:H \to \GSp_{2m}(q)$, where
\[
H = \{ h \in \GSp_2(q^m) \mid \tau(h) \in \F_q \} = \{ h \in \GL_2(q^m) \mid \det(h) \in \F_q \},
\] 
where the second equality holds since $\GSp_2(q^m)=\GL_2(q^m)$ and $\tau(h) = \det(h)$ for all $h \in \GSp_2(q^m)$ (see \cite[Lemma~2.4.5]{ref:KleidmanLiebeck}, for example). In addition, there is a field extension embedding $\pi_2\: K \to H$, where
\[
K = \{ (\mu) \in \GL_1(q^{2m}) \mid \mu^{q^m+1} \in \F_q \}.
\]
Now $g = \pi_1(\pi_2((\l))) \in G$ has order $(q-1)r$. Moreover, 
\[
\tau(g) = \tau(\pi_2((\l)) = \det(\pi_2((\l))) = \l^{q^m+1} = \a.
\] 
Now $\l$ is an eigenvalue of $g$, so $\l^{q-1}$ is an eigenvalue of $g^{q-1}$. Since $\l^{q-1}$ has order $r$, by Lemma~\ref{lem:technical}, $g^{q-1}$ is irreducible.

Now assume that $G = \DO^-_{2m}(q)$. In this case, let $\l \in \F_{q^{2m}}^\times$ have order $r$. There is a field extension embedding $\pi\:H \to \DO^-_{2m}(q)$, where
\[
H = \{ h \in \DO^-_2(q^m) \mid \tau(h) \in \F_q \} \cong C_{(q^m+1)(q-1)}.
\]
Now fix $h \in \DO^-_2(q^m)$ of order $(q-1)r$ and $\tau(h) = \a$. Without loss of generality, the eigenvalues of $h$ are $\l$ and $\a\l^{-1}$. Let $g=\pi(h)$. Then $g$ has order $(q-1)r$ and $\tau(g) = \tau(h) = \a$. Moreover, $\l^{q-1}$ is an eigenvalue of $g^{q-1}$ of order $r/(r,q-1) = r/2$, so Lemma~\ref{lem:technical} implies that $g^{q-1}$ is irreducible. This completes the proof.
\end{proof}

Let $(G,C)$ be $(\Sp_{2m}(q),\, \GSp_{2m}(q))$ or $(\O^+_{2m}(q),\, \GO^+_{2m}(q))$ and let $V=\F_q^{2m}$ be the natural module for $G$. Then $V$ admits a decomposition $\mathcal{D}(V)$
\begin{equation} \label{eq:totally_singular}
\text{$V = V_1 \oplus V_2$ \quad where \quad $V_1 = \<e_1,\dots,e_m\>$ and $V_2 = \<f_1,\dots,f_m\>$,}
\end{equation}
noting that $V_1$ and $V_2$ are totally singular $m$-spaces (with respect to the bases in \eqref{eq:B_symp} and \eqref{eq:B_plus}). The following describes the centraliser of the decomposition $\mathcal{D}(V)$.

\begin{lemmax}\label{lem:totally_singular}
Let $(G,C)$ be $(\Sp_{2m}(q),\, \GSp_{2m}(q))$ or $(\O^+_{2m}(q),\, \GO^+_{2m}(q))$. Then
\begin{enumerate}
\item $G_{(\mathcal{D}(V))} = \{ g \oplus g^{-\tr}    \mid g \in \GL_m(q) \}$
\item $C_{(\mathcal{D}(V))} = \{ \l g \oplus g^{-\tr} \mid \text{$g \in \GL_m(q)$ and $\l \in \F_q^\times$} \}$
\item If $g \in \GL_m(q)$ and $\l \in \F_q^\times$, then $\tau(\l g \oplus g^{-\tr}) = \l$.
\end{enumerate}
\end{lemmax}

\begin{proof}
The matrix of the underlying bilinear form with respect to the basis $(e_1,\dots,e_m,f_1,\dots,f_m)$ is
\[
M = \left( \begin{array}{cc} 0 & I_m \\ I_m & 0 \end{array} \right).
\]
Let $x = g \oplus h \in \GL(V)$ centralise $\mathcal{D}(V)$. If $x$ is a similarity of the form, then, for some $\l \in \F_q^\times$, we have $xMx^{-\tr} = \l M$ and consequently $g = \l h^{-\tr}$. It is straightforward to see that all such elements are indeed similarities. This proves (ii). Now let $\l \in \F_q^\times$ and $g \in \GL(V)$. Write $x = \l g \oplus g^{-\tr}$. Then $xMx^{-\tr} = \l M$, so $\tau(x) = \l$. This proves (iii) and consequently (i).
\end{proof}

\begin{lemmax}\label{lem:irreducible_plus}
Let $G$ be $\Sp_{2m}(q)$ or $\SO^+_{2m}(q)$. Let $r$ be a primitive divisor of $q^m-1$. Then $G$ contains an element of order $r$ that centralises $\mathcal{D}(V)$ and acts irreducibly on both $V_1$ and $V_2$.
\end{lemmax}

\begin{proof}
By Lemma~\ref{lem:irreducible}, there exists an irreducible element $g \in \GL_m(q)$ of order $r$. The corresponding element $g \oplus g^{-\tr} \in G_{(\mathcal{D}(V))}$ satisfies the statement.
\end{proof}

\subsection{Types of semisimple elements} \label{ss:o_elements_types}

Write $V = \F_q^{2m}$ and $\F_q^\times = \< \a \>$. By applying the results of Section~\ref{ss:o_elements_prelims}, in this section we will define some important types of semisimple elements in symplectic and orthogonal groups. The general idea that motivates these definitions is that we are interested in elements that stabilise few subspaces, which are contained in particular cosets of $\Sp_{2m}(q)$ in $\GSp_{2m}(q)$ or $\Omega^\pm_{2m}(q)$ in $\GO^\pm_{2m}(q)$ and whose orders have few prime divisors.

\begin{definitionx}\label{def:elt_a_plus}
Let $m$ be odd and let $G$ be $\Sp_{2m}(q)$ or $\SO^+_{2m}(q)$. An element $g \in G$ has \emph{type $(2m)^+_q$} if $|g| \in \ppd(q,m)$ and $g$ centralises  $V = V_1 \oplus V_2$ where $V_1$ and $V_2$ are totally singular nonisomorphic irreducible $\F_q\<g\>$-modules.
\end{definitionx}

\begin{lemmax}\label{lem:elt_a_plus}
Let $G$ be $\Sp_{2m}(q)$ or $\SO^+_{2m}(q)$ and assume that $m$ is odd. Then $G$ contains an element of type $(2m)^+_q$.
\end{lemmax}

\begin{proof}
Theorem~\ref{thm:zsigmondy} implies that $q^m-1$ has a primitive prime divisor $r$ and Lemma~\ref{lem:irreducible_plus} establishes that $G$ contains an element $g \oplus g^{-\tr}$ of order $r$ that centralises $\mathcal{D}(V)$ and acts irreducibly on both $V_1$ and $V_2$. By \cite[Lemma~3.1.13]{ref:BurnessGiudici16}, since $m$ is odd, the eigenvalue sets of $g$ and $g^{-\tr}$ are distinct, so $g$ and $g^{-\tr}$ are nonisomorphic. Therefore, $g \oplus g^{-\tr}$ has type $(2m)^+_q$.
\end{proof}

\begin{definitionx}\label{def:elt_a_minus}
Let $G$ be $\Sp_{2m}(q)$ or $\SO^-_{2m}(q)$. An element $g \in G$ has \emph{type $(2m)^-_q$} if $g$ is irreducible on $V$ and $|g| \in \ppd(q,2m)$, or $q$ is Mersenne, $m=1$ and $|g|=q+1$, or $q=2$, $m=6$ and $|g|=9$.
\end{definitionx}

\begin{lemmax}\label{lem:elt_a_minus}
Let $G$ be $\Sp_{2m}(q)$ or $\SO^-_{2m}(q)$. Then $G$ contains an element of type $(2m)^-_q$.
\end{lemmax}

\begin{proof}
If $q$ is Mersenne and $m=1$, or $q=2$ and $m=3$, then let $r=q^m+1$. Otherwise, Theorem~\ref{thm:zsigmondy} implies that $q^{2m}-1$ has a primitive prime divisor $r$. Now Lemma~\ref{lem:irreducible_minus} implies that $G$ contains an irreducible element of order $r$.
\end{proof}

\begin{lemmax}\label{lem:omega_a}
Let $g \in \SO^\e_{2m}(q)$ have type $(2m)^\e_q$. Then $g \not\in \Omega^\e_{2m}(q)$ if and only if $\e=-$, $m=1$ and $q$ is Mersenne.
\end{lemmax}

\begin{proof}
First assume that $\e=-$, $m=1$ and $q$ is Mersenne. Then $|g| = q+1$ and $|\Omega^-_2(q)| = \frac{1}{2}(q+1)$, so $g \not\in \Omega^-_2(q)$. Now assume otherwise. Therefore, $g$ has odd prime order, so $g \in \Omega^\e_{2m}(q)$.
\end{proof}

\begin{lemmax}\label{lem:elt_eigenvalues}
Let $g$ be an element of $\Sp_{2m}(q)$ or $\SO^\e_{2m}(q)$ of type $(2m)^\e_q$. Then the eigenvalues of $g$ (over $\FF_p$) are distinct.
\end{lemmax}

\begin{proof}
If $\e=-$, then $g$ is irreducible, so the characteristic polynomial of $g$ over $\F_q$ is irreducible and the eigenvalues of $g$ are distinct. Now assume that $\e=+$. Then $g=x \oplus x^{-\tr}$, centralising the decomposition $\mathcal{D}(V)$ (see \eqref{eq:totally_singular}) where $x$ and $x^{-\tr}$ act irreducibly on $V_1$ and $V_2$. Therefore, the characteristic polynomial of $x$ is irreducible. Moreover, $V_1$ and $V_2$ are nonisomorphic $\F_q\<x\>$-modules, so the characteristic polynomials of $x$ and $x^{-\tr}$ are distinct irreducible polynomials. Consequently, $g$ has distinct eigenvalues in this case too. This completes the proof.
\end{proof}

Now assume that $q$ is odd. Fix $\b \in \F_q^\times$ with $|\b|=(q-1)_2$. We will define some variants on the types of elements defined above, which have a very similar action on the natural module. Consequently, in the first instance the reader is encouraged to think of elements of type $(2m)^\pm_q$ upon encountering ${\,}^\Delta(2m)^\pm_q$ and ${\,}^\Sigma(2m)^\pm_q$.

\begin{definitionx}\label{def:elt_c} 
Let $q$ be odd, let $\e \in \{+,-\}$ and let $G$ be $\GSp_{2m}(q)$ or $\DO^\e_{2m}(q)$. An element $g \in G$ has \emph{type ${\,}^\Delta(2m)^\e_q$} if $\tau(g) = \b$ and $g^k$ has type $(2m)^\e_q$ where
\[
k = \left\{
\begin{array}{ll}
(q^m+1)_2(q-1)_2 & \text{if $\e=-$ and either $m  > 1$ or $q$ is not Mersenne} \\
(q-1)_2          & \text{otherwise.}
\end{array}
\right.
\]
\end{definitionx}

\begin{lemmax}\label{lem:elt_c}
Let $q$ be odd, let $\e \in \{+,-\}$ and let $G$ be $\GSp_{2m}(q)$ or $\DO^\e_{2m}(q)$. 
\begin{enumerate}
\item If $\e=+$ and $m > 1$ is odd, then $G$ contains an element of type ${\,}^\Delta(2m)^+_q$.
\item If $\e=-$, then $G$ contains an element of type ${\,}^\Delta(2m)^-_q$
\end{enumerate}
\end{lemmax}

\begin{proof}
First assume that $\e=+$. By Lemma~\ref{lem:elt_a_plus}, $G$ contains an element $g \oplus g^{-\tr}$ of type $(2m)^+_q$. Let $h = \b g \oplus g^{-\tr}$, noting that $h \in G$ (see Lemma~\ref{lem:totally_singular}(ii)). We claim that $h$ has type ${\,}^\Delta(2m)^+_q$. By Lemma~\ref{lem:totally_singular}(iii), $\tau(h) = \b$. Now $|g|$ is odd, since $|g| \in \ppd(q,m)$, and $|\b|=(q-1)_2$, so $h^{(q-1)_2} = g^{(q-1)_2} \oplus (g^{(q-1)_2})^{-\tr}$ has order $|g|$. Therefore, $h^{(q-1)_2}$ has type $(2m)^+_q$ and, consequently, $h$ has type ${\,}^\Delta(2m)^+_q$.

Now assume that $\e=-$. For now assume further that $m > 1$ or $q$ is not Mersenne. Theorem~\ref{thm:zsigmondy} implies that we may fix $r \in \ppd(2m,q)$. By Lemma~\ref{lem:irreducible_minus_similarity}, there exists an element $g \in G$ of order $r(q^m+1)_2(q-1)$ such that $\tau(g)=\a$ and $g^{(q-1)}$ is irreducible. Let $h=g^{(q-1)_{2'}}$. Then $h^{(q^m+1)_2(q-1)_2}$ has type $(2m)^-_q$ and $\tau(h)$ has order $(q-1)_2$, so without loss of generality is $\tau(h)=\b$. Therefore, $h$ has type ${\,}^\Delta(2m)^-_q$.

It remains to assume that $\e=-$, $m=1$ and $q$ is Mersenne. Then Lemma~\ref{lem:irreducible_minus_similarity} implies that there exists $g \in G$ of order $(q+1)(q-1)$ such that $\tau(g)=\a$ and $g^{q-1}$ is irreducible. As before, $g^{(q-1)_{2'}}$ has type ${\,}^\Delta(2)^-_q$. We have completed the proof.
\end{proof}

\begin{definitionx}\label{def:elt_e} 
Let $q$ be odd. An element $g \in \SO^\e_{2m}(q) \setminus \Omega^\e_{2m}(q)$ has \emph{type ${\,}^\Sigma(2m)^\e_q$} if $g^k$ has type $(2m)^\e_q$ where $k=(q^m-\e)_2$.
\end{definitionx}

\begin{lemmax}\label{lem:elt_e}
Let $q$ be odd.
\begin{enumerate}
\item If $m > 1$ is odd, then $\SO^+_{2m}(q)$ contains an element of type ${\,}^\Sigma(2m)^+_q$.
\item If $m > 1$, then $\SO^-_{2m}(q)$ contains an element of type ${\,}^\Sigma(2m)^-_q$.
\end{enumerate}
\end{lemmax}

\begin{proof}
First assume that $\e=+$ and $m > 1$ is odd. By Theorem~\ref{thm:zsigmondy}, we may fix $r \in \ppd(m,q)$. Let $\l \in \F_{q^{2m}}^\times$ have order $r(q^m-1)_2$. By Lemma~\ref{lem:irreducible}, $\GL_m(q)$ contains an element of $r(q^m-1)$ and determinant $\l^{q^{m-1}+\cdots+q+1}$. Let $h = g \oplus g^{-\tr}$. By Lemma~\ref{lem:totally_singular}(i), $h \in \SO^+_{2m}(q)$. We know that $\l \not \in (\F_{q^m}^\times)^2$ since $(q^m-1)_2$ divides the order of $\l$. Therefore, $\det(g) = \l^{q^{m-1}+\cdots+q+1} \not\in (\F_q^\times)^2$. Consequently, $h \not\in \Omega^+_{2m}(q)$ by \cite[Lemma~4.1.9]{ref:KleidmanLiebeck}. Now $h^{(q^m-1)_2}$ has type $(2m)^+_q$, so $h$ has type ${\,}^\Sigma(2m)^+_q$. 

Now assume that $\e=-$ and $m > 1$. By Theorem~\ref{thm:zsigmondy}, we may fix $r \in \ppd(2m,q)$. By Lemma~\ref{lem:irreducible_minus}, $\SO^-_{2m}(q)$ contains an irreducible element $h$ of order $r(q^m+1)_2$. By \cite[Theorem~4]{ref:ButurlakinGrechkoseeva07}, $(q^m+1)_2$ does not divide the order of a maximal torus of $\Omega^-_{2m}(q)$, so $g \not\in \Omega^-_{2d}(q)$. Since $h^{(q^m+1)_2}$ has type $(2m)^-_q$, $h$ has type ${\,}^\Sigma(2d)^-_q$, which completes the proof.
\end{proof}

For all of the elements introduced in this section, if the field size $q$ is clear from the context, then we omit the subscript of $q$ from the notation. However, in general, the field size is pertinent, as Lemma~\ref{lem:elt_splitting} demonstrates.

\begin{lemmax}\label{lem:elt_splitting}
Let $m > 1$ and $q=q_0^e$. Let $G$ be $\Sp_{2m}(q)$ or $\SO^\eta_{2m}(q)$. Let $g \in G$ have odd order and type $(2m)^\eta_{q_0}$. Assume that $m$ is odd if $\eta=+$ and that $(q_0,m) \neq (2,6)$ if $\eta=-$. Then $g$ is similar to $g_1 \oplus \cdots \oplus g_t$ where each of $g_1,\dots,g_t$ has type $\left(\frac{2m}{t}\right)^\e_q$ where $t = (m,e)$ and $\e = \eta^{e/t}$. 
\end{lemmax}

\begin{proof}
First assume that $\e=+$. Then $|g| \in \ppd(q_0,m)$ and the eigenvalue set of $g$ is $\Lambda \cup \Lambda^{-1}$ where $\Lambda = \{ \l, \l^{q_0}, \dots, \l^{q_0^{m-1}} \}$. There are $t=(m,e)$ distinct $\mu \mapsto \mu^q$ orbits on $\Lambda$, say $\Lambda_1, \dots, \Lambda_t$, each of size $m/t$. Fix $1 \leq j \leq t$ and $\l_j \in \Lambda_j$. By Lemma~\ref{lem:technical}, there exists an irreducible element $x_j \in \GL_{m/t}(q)$ with eigenvalue set $\Lambda_j$. Then $g_j = x_j \oplus x_j^{-\tr}$ has type $\left(\frac{2m}{t}\right)^+_q$ and eigenvalue set $\Lambda_j \cup \Lambda_j^{-1}$. Therefore, $g$ has the same eigenvalues as $g_1 \oplus \cdots \oplus g_t$. Noting that $g$ is a semisimple element of odd order, Lemma~\ref{lem:semsimple_conjugacy} implies that $g$ is similar to $g_1 \oplus \cdots \oplus g_t$. This proves the claim in this case.

Now assume that $\e=-$. Then $|g| \in \ppd(q_0,2m)$ and $\Lambda = \{ \l, \l^{q_0}, \dots, \l^{q_0^{2m-1}} \}$ is the eigenvalue set of $g$. There are $k=(2m,e)$ distinct $\mu \mapsto \mu^q$ orbits of $\Lambda$, say $\Lambda_1, \dots, \Lambda_k$, each of size $2m/k$. Assume for now that $2m/k$ is odd. Then $k=(2m,e)=2(m,e)=2t$ and we may assume that $\Lambda_{t+j} = \Lambda_j^{-1}$ for each $1 \leq j \leq t$. As we argued in the previous case, there exists an element $g_j$ of type  $\left(\frac{2m}{t}\right)^+_q$ whose eigenvalue set is $\Lambda_i \cup \Lambda_i^{-1}$ and $g$ is similar to $g_1 \oplus \cdots \oplus g_t$. 

It remains to assume that $2m/k$ is even. In this case, $k=(2m,e)=(m,e)=t$. Fix $1 \leq j \leq t$ and let $\l_j \in \Lambda_j$. Lemma~\ref{lem:irreducible_minus} implies that there exists an irreducible element $g_j \in \SO^-_{2m/t}(q)$ with eigenvalue set $\Lambda_j$. Therefore, $g_j$ has type $\left(\frac{2m}{t}\right)^-_q$. Lemma~\ref{lem:semsimple_conjugacy} now implies that $g$ is similar to $g_1 \oplus \cdots \oplus g_t$, completing the proof.
\end{proof}

We conclude with a comment on centralisers.

\begin{lemmax} \label{lem:elt_centraliser}
Let $G$ be $\PGSp_{2m}(q)$ or $\PDO^\e_{2m}(q)$. Let $g \in G$ lift to an element of type ${\,}^\ast(2m)^\e_q$, where $\ast$ is the empty symbol, $\Delta$ ($q$ odd) or $\Sigma$ ($q$ odd and $G = \PDO^\e_{2m}(q)$). Then $|C_G(g)| \leq q^m-\e$.
\end{lemmax}

\begin{proof}
A suitable power $h$ of $g$ has type $(2m)_q^\e$. For $x \in \GL_{2m}(q)$, write $\overline{x}$ for the image in $\PGL_{2m}(q)$. First assume that $\e=+$. Then $h = h_1 \oplus h_1^{-\tr}$ and $|h| \in \ppd(q,m)$. By \cite[Appendix~B]{ref:BurnessGiudici16}, $|C_G(\overline{g})| \leq |C_G(\overline{h}) = q^m-1$.

Next assume that $\e=-$. If $m > 1$ or $q$ is not Mersenne, then $|h| \in \ppd(q,2m)$ and from \cite[Appendix~B]{ref:BurnessGiudici16}, $|C_G(\overline{g})| \leq |C_G(\overline{h})| = q^m+1$. It is straightforward to verify the special case where $|h| = q+1$ and $G$ is $\PGSp_2(q)$ or $\PDO^-_2(q)$.
\end{proof}

\subsection{Reflections} \label{ss:o_elt_reflections}

We conclude this section by discussing reflections. We continue to write $V=\F_q^{2m}$ and $\F_q^\times = \< \a \>$. The standard bases $\B^+$ and $\B^-$ were introduced in \eqref{eq:B_plus} and \eqref{eq:B_minus}. Recall that if $q$ is odd, then $\b \in \F_q^\times$ has order $(q-1)_2$, so $\b \not\in (\F_q^\times)^2$. If $\e=-$, then we will make use of the isomorphism $\Psi\:\< X_{r\p^f}, r \> \to \PGO^-_{2m}(q)$ (see Lemma~\ref{lem:algebraic_finite_minus}). 

\begin{definitionx}\label{def:elt_r}
With respect to the basis $\B^\e$ for $\F_q^2$, define
\[ 
r^\e = \left( 
\begin{array}{cc} 
0 & 1 \\ 
1 & 0 \\ 
\end{array} 
\right) \in \O^\e_2(q)
\]
and if $q$ is odd, then also
\[
{\,}^\Delta r^+_q = \left(
\begin{array}{cc}
0 & \b \\
1 & 0  \\
\end{array}
\right) \in \GO^+_2(q)
\]
and, for $\b_2 \in \F_{q^2}^\times$ of order $(q^2-1)_2$,
\[ 
{\,}^\Delta r^-_q = \Psi(R) \in \GO^-_2(q) \quad \text{where} \quad
R = \left(
\begin{array}{cc} 
0      & \b_2 \\ 
\b_2^q & 0    \\
\end{array}
\right) \in \GO^+_2(q^2).
\]
\end{definitionx}

\begin{lemmax}\label{lem:elt_r_even}
Let $q$ be even and let $F$ be a finite extension of $\F_q$. Then $r^\e$ is a reflection that stabilises a unique (nonsingular) $1$-space of $F^2$.
\end{lemmax}

\begin{proof}
Evidently $r^\e$ stabilises the nonsingular $1$-space $\<e_1+f_1\>$ if $\e=+$ and $\<u_1+v_1\>$ if $\e=-$, and this is the unique subspace stabilised by $r^\e$. 
\end{proof}

\begin{lemmax}\label{lem:elt_r_odd}
Let $q$ be odd and let $F$ be a finite extension of $\F_q$. Then
\begin{enumerate}
\item $r^+$ is a reflection in a vector of norm $-2$
\item $r^-$ is a reflection in a vector of norm $-2\l^2$ for some $\l \in \F_q^\times$
\item $r^\e$ stabilises exactly two (orthogonal nondegenerate) $1$-spaces of $F^2$.
\item ${\,}^\Delta r^\e$  acts irreducibly on $F^2$ if $|F:\F_q|$ is odd
\item ${\,}^\Delta r^\e$ stabilises exactly two (orthogonal nondegenerate) $1$-spaces of $F^2$ if $|F:\F_q|$ is even
\item $\tau({\,}^\Delta r^\e) = \b$ and $\det({\,}^\Delta r^\e) = -\b$.
\end{enumerate}
\end{lemmax}

\begin{proof}
Observe that $r^+=r_{e_1-f_1}$ and $(e_1-f_1,e_1-f_1) = -2$. Similarly, $r^-=r_{u_1-v_1}$ and 
\[
(u_1-v_1,u_1-v_1) = 2- 2(\xi^2+\xi^{-2}) + 2 = -2(\xi-\xi^{-1})^2
\]
(see the definition of $\B^-$ in \eqref{eq:B_minus}). This proves (i) and (ii). 

For (iii), the characteristic polynomial of $r^\e$ is $X^2-1$, so $r^\e$ has a $1$-dimensional $1$- and $-1$-eigenspace and these two $1$-spaces are exactly the proper nonzero subspaces stabilised by $r^\e$. Smilarly, (iv) and (v) hold since the characteristic polynomial of ${\,}^\Delta r^\e$ is $X^2-\b$. 

Finally consider (vi). If $\e=+$, then this is a straightforward calculation. If $\e=-$, then we easily see that  $\det([\b_2,\b_2^q]) = -\b_2^{q+1} = -\b$ and $\Psi$ is induced by conjugation, so $\det({\,}^\Delta r^\e) = -\b$. Similarly, $\tau([\b_2,\b_2^q]) = \b_2^{q+1} = \b$, with respect to the standard plus-type form on $\F_{q^2}^{2m}$ and the definition of $\Psi$ implies that $\tau(\hat{\d}) = \b$ with respect to the standard minus-type form on $\F_q^{2m}$.
\end{proof}

\begin{remarkx}\label{rem:elt_r}
Let us comment on reflections.
\begin{enumerate}
\item The element $r \in \GO_{2m}(\FF_p)$ from Definition~\ref{def:phi_gamma_r} is simply $I_{2m-2} \perp r^+$, centralising $\<e_1,\dots,f_{m-1}\> \perp \<e_m,f_m\>$. Additionally, $\Psi(r) = I_{2m-2} \perp r^-$, centralising $\<e_1,\dots,f_{m-1}\> \perp \<u_m,v_m\>$. Thus, we often identify $r$ and $r^\e$ as elements of $\O^\e_{2m}(q)$.
\item Assume $q$ is odd. By Lemma~\ref{lem:elt_r_odd}, the norm of $r^\e$ is square if and only if $-2 \in (\F_q^\times)^2$. This latter condition holds if and only if 
\begin{equation} \label{eq:square}
\text{$f$ is even or $p \equiv 1 \ \text{or} \ 3 \mod{8}$.}
\end{equation}
Therefore, $\ddot{r}^\e$ is $\ddot{r}_\square$ if \eqref{eq:square} holds and $\ddot{r}^\e$ is $\ddot{r}_\nonsquare$ otherwise.
\item If $q$ is odd, then ${\,}^\Delta r^+ = \d^+r$ and ${\,}^\Delta r^- = \d^-r$.
\end{enumerate}
\end{remarkx}

\subsection{Field extension subgroups} \label{ss:o_prelims_c3}

In this final preliminary section, we briefly discuss maximal field extension overgroups of certain elements. We begin by stating \cite[Lemma~5.3.2]{ref:BurnessGiudici16} for future reference. 

\begin{lemmax}\label{lem:c3}
Let $k$ be a prime divisor of $n$, let $\pi\:\GL_{n/k}(q^k).k \to \GL_n(q)$ be a field extension embedding and let $x \in \GL_{n/k}(q^k).k$ have prime order $r \neq p$.
\begin{enumerate}
\item If $x \in \GL_{n/k}(q^k)$ and has eigenvalues $\l_1,\dots,\l_{n/k}$ over $\FF_p$, then $\pi(x)$ has eigenvalues $\Lambda_1 \cup \cdots \cup \Lambda_{n/k}$ where $\Lambda_i = \{ \l_i^{q^j} \mid 0 \leq j < k \}$.
\item If $x \not\in \GL_{n/k}(q^k)$, then $r=k$ and each $r$th root of unity occurs as an eigenvalue of $\pi(x)$ with multiplicity $n/k$.
\end{enumerate}
\end{lemmax}

\begin{corollaryx} \label{cor:c3}
Let $G$ be $\PSp_{2m}(q)$ or $\PSO^\pm_{2m}(q)$ and let $g$ lift to $g_1 \oplus \cdots \oplus g_t \oplus I_\ell$ where $g_1,\dots,g_t$ have type $(2d)_q^\e$ for $d > 1$ and have distinct eigenvalues.
\begin{enumerate}
\item If $d$ is odd, then $g$ is not contained in the base of a subgroup of type $\Sp_m(q^2)$ (where $m$ is even) or $\O^{\up}_m(q^2)$ (where ${\up} \in \{+,-\}$ if $m$ is even and ${\up}=\circ$ if $m$ is odd).
\item If $\e \neq (-)^d$, then $g$ is not contained in the base of a $\GU_m(q)$ subgroup.
\end{enumerate}
\end{corollaryx}

\begin{proof}
Let $\pi\:H=B.2 \to G$ be the field extension embedding in question, where $B$ is the base of $H$. Write $|g|=r$. For a contradiction, suppose that $g \in B$.

First assume that $\e=+$, so we may assume that $d$ is odd. Let $\Lambda$ be the set of nontrivial eigenvalues of $g$. If $g = \pi(x)$ for $x \in B$, then, by Lemma~\ref{lem:c3}(i), $\Lambda = \Lambda_0 \cup \Lambda_0^q$, where $\Lambda_0$ is the set of eigenvalues of $x$. Since $x$ is an element defined over $\F_{q^2}$ we know that $\Lambda_0^{q^2} = \Lambda_0$. However, the elements of $\Lambda_0$ have order $r$, where $r \in \ppd(q,d)$. Since $d$ is odd, $\Lambda_0^{q^2} = \Lambda_0^q$. Thus, every eigenvalue of $g$ occurs with multiplicity at least two, which contradicts the distinctness of the eigenvalues of $g$.

Next assume that $\e=-$. Let $\Lambda_i$ be the set of $2d$ distinct eigenvalues of $g_i$. For now consider part~(i), so we may assume that $d$ is odd. Then $r \in \ppd(q,2d)$ and there are two $\mu \mapsto \mu^{q^2}$ orbits on $\Lambda_i$, say $\Lambda_{i1}$ and $\Lambda_{i2} = \Lambda_{i1}^q = \Lambda_{i1}^{-1}$. By Lemma~\ref{lem:c3}(i), without loss of generality, the eigenvalues of $g$ as an element of $\GL_m(q^2)$ are $\cup_{i=1}^{t} \Lambda_{i1}$, which is not closed under inversion (see \cite[Lemma~3.1.13]{ref:BurnessGiudici16}), which is a contradiction to \cite[Lemma~3.4.1]{ref:BurnessGiudici16}.

Continuing to assume $\e=-$, now consider part~(ii). We may now assume that $d$ is even. Therefore, $r \in \ppd(q,d)$ and again write $\Lambda_{i1}$ and $\Lambda_{i2} = \Lambda_{i1}^{-q}$ for the two $\mu \mapsto \mu^{q^2}$ orbits on $\Lambda_i$. Then, by Lemma~\ref{lem:c3}(i), without loss of generality, the eigenvalues of $g$ as an element of $\GU_m(q)$ are $\cup_{i=1}^{t} \Lambda_{i1}$, which is not closed under the map $\mu \mapsto \mu^{-q}$, which is a contradiction to \cite[Proposition~3.3.1]{ref:BurnessGiudici16}. This completes the proof.
\end{proof}

Combining Corollary~\ref{cor:c3} with Lemma~\ref{lem:elt_splitting} gives the following.

\begin{corollaryx} \label{cor:c3_subfield}
Let $G$ be $\PSp_{2m}(q)$ or $\PSO^\pm_{2m}(q)$. Let $g \in G$ have type $(2d)^\eta_{q_0} \perp I_\ell$ for $q_0^e=q$.
\begin{enumerate}
\item If $d$ is odd, then $g$ is not contained in the base of a subgroup of type $\Sp_m(q^2).2$ or $\O^{\up}_m(q^2).2$.
\item If $d$ is odd and $\eta = +$; or $d$ is even, $\eta = -$ and $e$ is odd; or $d$ is odd, $\eta = -$ and $e$ is even, then $g$ is not contained in the base of a $\GU_m(q)$ subgroup.
\end{enumerate}
\end{corollaryx}

\clearpage 
\section{Case I: semilinear automorphisms}\label{s:o_I}

Having established the cases to consider, we now start proving Theorems~\ref{thm:o_main} and~\ref{thm:o_asymptotic}. In this section, we begin with Case~I. Accordingly, write $G=\<T,\th\>$ where $T = \POm^\e_{2m}(q)$ for $m \geq 4$ and $\th \in \PGaO^\e_{2m}(q) \setminus \PGO^\e_{2m}(q)$. Recall the cases
\begin{enumerate}[(a)]
\item $G \cap \PGO^\e_{2m}(q) \leq \PDO^\e_{2m}(q)$
\item $G \cap \PGO^\e_{2m}(q) \not\leq \PDO^\e_{2m}(q)$.
\end{enumerate}

The main motivation for this case distinction is that Shintani descent applies directly in Case~I(a) but in Case~I(b) we need to use this technique in a more flexible manner. A side effect of this distinction is that in Case~I(a), $\nu(x) > 1$ for all $x \in G \cap \PGL(V)$ and this makes the probabilistic method easier to apply. Recall that Table~\ref{tab:o_cases} further partitions Cases~I(a) and~I(b). In particular, I(a) is the union of I(i), (ii) and~(iv), and I(b) is the union of I(iii) and~(v), where the definitions of Cases~I(i)-(v) are summarised in Table~\ref{tab:o_cases_I}. We consider Cases~I(a) and~I(b) in Sections~\ref{ss:o_Ia} and~\ref{ss:o_Ib}, respectively.
 
\begin{table}[b]
\centering
\caption{Definition of Cases~I(i)--(v)} \label{tab:o_cases_I}
{\renewcommand{\arraystretch}{1.2}
\begin{tabular}{ccccc}
\hline  
case  & $\e$ & $\th$             & condition      \\    
\hline   
(i)   & $+$  & $\th_0 \p^i$      & none           \\[2pt]
(ii)  &      & $\th_0 r \p^i$    & $f/i$ is even  \\[2pt]
(iii) &      &                   & $f/i$ is odd   \\[2pt]
(iv)  & $-$  & $\th_0 \psi^i$    & $2f/i$ is odd  \\[2pt]
(v)   &      &                   & $2f/i$ is even \\
\hline
\end{tabular}}
\\[5pt]
{\small Note: $\th_0 \in \Inndiag(T)$}
\end{table}

\subsection{Case I(a)}\label{ss:o_Ia}

In this section, we first we identify an element $t\th \in G$, then we determine $\M(G,t\th)$ and apply the probabilistic method.

Shintani descent (see Chapter~\ref{c:shintani}) will play an indispensable role in identifying an appropriate element $t\th \in T\th$ for each automorphism $\th$ (see Example~\ref{ex:shintani_descent}). With this in mind let us fix the following notation for Section~\ref{ss:o_Ia}. \vspace{5pt}

\begin{shbox}
\begin{notationx} \label{not:o_Ia}
\begin{enumerate}[label={}, leftmargin=0cm, itemsep=3pt]
\item Write $q=p^f$ where $f \geq 2$. Let $V = \F_q^{2m}$.
\item Fix the simple algebraic group
\[
X = \left\{ 
\begin{array}{ll}
\Omega_{2m}(\FF_2) & \text{if $p=2$}       \\
\PSO_{2m}(\FF_p)   & \text{if $p$ is odd.} \\
\end{array}
\right.
\]
\item Fix the standard Frobenius endomorphism $\p = \p_{\B^+}$ of $X$, defined with respect to the standard basis $\B^+$, as $(a_{ij}) \mapsto (a_{ij}^p)$, modulo scalars.
\item Fix the diagonal element $\d^+$ and reflection $r$  (see Definitions~\ref{def:delta} and~\ref{def:phi_gamma_r}).
\item If $\e=-$, fix the map  $\Psi$ from Lemma~\ref{lem:algebraic_finite_minus}, which restricts to an isomorphism $\Psi\: \< X_{r\p^f}, r \> \to \PGO^-_{2m}(q)$. Moreover, fix $\psi = \Psi \circ \p \circ \Psi^{-1}$ and $\d^- = \Psi(\d^+)$ (see \eqref{eq:psi} and Definition~\ref{def:delta_minus}).
\end{enumerate}
\end{notationx}
\end{shbox} \vspace{5pt}

As a consequence of Proposition~\ref{prop:o_cases}, we can assume that $\th \in \PGO^+_{2m}(q)\p^i$ when $\e=+$ and $\th \in \PGO^-_{2m}(q)\psi^i$ when $\e=-$. In the latter case, the definition of Case~I(a) ensures that $2f/i$ is odd, so $i$ is even and it is straightforward to show, for $j=i/2$, we have $2f/(2f,f+j) = 2f/(2f,i)$. Consequently, when $\e=-$, we may, and will, work with 
\[
\th = \th_0\psi^{f+j} = \th_0r\psi^j
\] 
instead of $\th_0\psi^i$, noting that $j$ divides $f$ and $2f/i = f/j$ is odd. \vspace{5pt}

\begin{shbox}
\notacont{\ref{not:o_Ia}} 
\begin{enumerate}[label={}, leftmargin=0cm, itemsep=3pt]
\item Write $q=q_0^e$, where $(\eta,\sigma,e)$ are as follows
\begin{center}
{\renewcommand{\arraystretch}{1.2}
\begin{tabular}{cccc}
\hline  
case & $\eta$ & $\s$    & $e$    \\
\hline
(i)  & $+$    & $\p^i$  & $f/i$  \\
(ii) & $-$    & $r\p^i$ & $f/i$  \\
(iv) & $-$    & $r\p^j$ & $2f/i$ \\
\hline
\end{tabular}}
\end{center}
\item Let $F$ be the Shintani map of $(X,\s,e)$, so
\[
F\: \{ (g\ws)^{X_{\s^e}} \mid g \in X_{\s^e} \} \to \{ x^{X_\s} \mid x \in X_\s \}.
\]
Observe that $X_{\s^e} \cong \Inndiag(T)$ and $X_{\s} = \Inndiag(T_0)$ for a subgroup $T_0$ of $T$ isomorphic to $\POm^\eta_{2m}(q_0)$. We will harmlessly identify $T_0$ with $\POm^\eta_{2m}(q_0)$ and write $\Inndiag(T_0) = \PDO^\eta_{2m}(q_0) = \<\PSO^\eta_{2m}(q_0),\d_0\>$.
\end{enumerate}
\end{shbox} \vspace{5pt}

\begin{remarkx}\label{rem:not_o_Ia}
Let us make some observations regarding Notation~\ref{not:o_Ia}.
\begin{enumerate}
\item The definition of Case~I(a) implies that $\e=\eta^e$.
\item If $\e=+$, then $\Inndiag(T)\th = X_{\s^e}\ws$.
\item If $\e=-$, then, via the isomorphism $\Psi$, we identify $X_{\s^e}$ with $\Inndiag(T)$ and we identify $\ws = r\p^j$ with $\th = r\psi^j$, so $\Inndiag(T)\th = X_{\s^e}\ws$ in this case also.
\end{enumerate}
\end{remarkx}

In light of Remark~\ref{rem:not_o_Ia}, the main idea is to select the element $t\th \in \Inndiag(T)\s$ as the preimage under $F$ of a carefully chosen element $y \in \Inndiag(T_0)$. If $q$ is even, then $\Inndiag(T) = T$ and this is a transparent process. When $q$ is odd, the following two results facilitate this selection procedure (compare with Example~\ref{ex:shintani_quotients}.)

\begin{lemmax}\label{lem:o_Ia_tau}
Let $q$ be odd. The Shintani map $F$ restricts to bijections
\begin{enumerate}
\item{$F_1\:\! \{(g\ws)^{\PDO^\e_{2m}(q)}   \mid g \in \PSO_{2m}^\e(q) \} \to \{x^{\PDO^\eta_{2m}(q_0)}       \mid x \in \PSO_{2m}^{\eta}(q_0) \}$}
\item{$F_2\:\! \{(g\d\ws)^{\PDO^\e_{2m}(q)}\! \mid g \in \PSO_{2m}^\e(q) \} \to \{(x\d_0)^{\PDO^\eta_{2m}(q_0)}\! \mid x \in \PSO_{2m}^{\eta}(q_0) \}.$}
\end{enumerate} 
\end{lemmax}

\begin{proof}
This is Lemma~\ref{lem:shintani_quotients} with $\pi\:\SO_{2m}(\FF_q) \to \PSO_{2m}(\FF_q)$, noting that $\<\PSO^\e_{2m}(q),\ws\>$ and $\PSO^\eta_{2m}(q_0)$ are index two subgroups of $\<\PDO^\e_{2m}(q),\ws\>$ and $\PDO^\eta_{2m}(q_0)$.
\end{proof}

\begin{lemmax}\label{lem:o_Ia_eta}
Let $q$ be odd and assume that $q_0^m \equiv \eta \mod{4}$. The map $F_1$ restricts to bijections
\begin{enumerate}
\item{$F_{11}\: \{(g\ws)^{\PDO^\e_{2m}(q)} \mid g \in T \} \to \{x^{\PDO^\eta_{2m}(q_0)} \mid x \in T_0 \}$}
\item{$F_{12}\: \{(g\rsq\rns\ws)^{\PDO^\e_{2m}(q)} \mid g \in T \} \to \{(x\rsq\rns)^{\PDO^\eta_{2m}(q_0)} \mid x \in T_0 \}.$}
\end{enumerate} 
\end{lemmax}

\begin{proof}
The condition $q_0^m \equiv \eta \mod{4}$ ensures that $|\PSO^\eta_{2m}(q_0):T_0|=2$ (see \eqref{eq:discriminant_condition}). We claim $|\PSO^\e_{2m}(q):T|=2$. If $\e=\eta=+$, then $q^m \equiv 1 \mod{4}$ and $|\PSO_{2m}^+(q):T|=2$. Next, if $\e=+$ and $\eta=-$, then $e$ is even, so again $q^m \equiv 1 \mod{4}$ and $|\PSO_{2m}^+(q):T|=2$. Finally, if $\e=\eta=-$, then $e$ is odd and $q^m \equiv 3 \mod{4}$, so $|\PSO_{2m}^-(q):T|=2$.

Write $W=\Spin_{2m}(\FF_q)$ and let $\pi\:W \to X$ be the natural isogeny. Now $\pi(W_{\s^e}) = T$ where $W_{\s^e} = \Spin^\e_{2m}(q)$, and $\pi(W_{\s}) = T_0$ where $W_\s = \Spin^\eta_{2m}(q_0)$ (see \cite[Theorem~2.2.6(f)]{ref:GorensteinLyonsSolomon98}). Evidently, $T_0 \leqn \Inndiag(T_0)$. Moreover, if $\e=+$, then the condition $q_0^m \equiv \eta \mod{4}$ implies that condition~\eqref{eq:o_cases_condition} is satisfied, so, in light of Remark~\ref{rem:o_out_plus}, $\<\ddot{\s}\> \leqn \< \Inndiag(T)/T, \ddot{\s}\>$ and hence $\<T,\ws\> \leqn \<\Inndiag(T),\ws\>$. Similarly, if $\e=-$, then $i$ is even, so $[\ddot{\psi}^i,\ddot{\d}]=1$ (see Lemma~\ref{lem:o_out_minus}), which implies that $\<\ddot{\s}\> \leqn \< \Inndiag(T)/T, \ddot{\s}\>$ and hence, again, $\<T,\ws\> \leqn \<\Inndiag(T),\ws\>$. Therefore, by Lemma~\ref{lem:shintani_quotients}, the Shintani map $F$ of $(X,\s,e)$ restricts to the map $F_{11}$. By Lemma~\ref{lem:o_Ia_tau}, $F$ restricts to $F_1$, so, in fact, $F_1$ restricts to the bijections $F_{11}$ and $F_{12}$, as required.
\end{proof}

We will now define the elements we will use to prove Theorems~\ref{thm:o_main} and~\ref{thm:o_asymptotic} in Case~I(a). In light of the probabilistic method outlined in Section~\ref{s:p_prob}, we need to select $t\th \in G$ in a way which allows us to control both the maximal subgroups of $G$ which contain it and the fixed point ratios associated with these subgroups. 

Recall that in Definitions~\ref{def:elt_a_plus} and~\ref{def:elt_a_minus}, we defined standard \emph{types} of elements denoted $(2d)^\pm_q$ for some $d \geq 1$. Moreover, in Definitions~\ref{def:elt_c} and~\ref{def:elt_e}, for odd $q$ we also defined variants indicated by superscripts $\Delta$ and $\Sigma$. These variants have a very similar action on the natural module but crucially are contained in a different coset of the simple group. By working with the latter, we will be able to select an element that lies in the precise coset $T\th$.

\begin{definitionx}\label{def:ab}
Let $\th \in \Aut(T)$. 
\begin{enumerate}
\item Define
\[
a = a(\th) = \left\{
\begin{array}{ll}
\Delta & \text{if $\th \not\in \< \PO^\e_{2m}(q),\p\>$} \\
       & \text{if $\th \in \< \PO^\e_{2m}(q),\p\>$}     \\
\end{array}
\right. 
\]
where we mean the empty symbol in the second case.
\item Define
\[
b = b(\th) = \left\{
\begin{array}{ll}
\Delta & \text{if $\th \not\in \< \PO^\e_{2m}(q),\p\>$} \\
\Sigma & \text{if $\th \in \< \POm^\e_{2m}(q),\p\>\rsq\rns$} \\
       & \text{otherwise.}
\end{array}
\right.
\]
\item Define
\[
c = c(\th,q_0) = \left\{
\begin{array}{ll}
\Delta & \text{if $\th \not\in \< \PO^\e_{2m}(q),\p\>$} \\
\Sigma & \text{if $\th \in \< \POm^\e_{2m}(q),\p\>$ and $q_0$ is Mersenne} \\
\Sigma & \text{if $\th \in \< \POm^\e_{2m}(q),\p\>\rsq\rns$ and $q_0$ is not Mersenne} \\
       & \text{otherwise.}
\end{array}
\right.
\]
\end{enumerate}
\end{definitionx}

\begin{remarkx}\label{rem:ab}
The dependence on whether $q_0$ is Mersenne in Lemma~\ref{lem:omega_a} has to be taken into account in our arguments and defining $c$ as a variant on $b$ that depends on $q_0$ allows us to do this. Notice that $a=b=c$ is empty when $q$ is even.  
\end{remarkx}

\begin{propositionx}\label{prop:o_Ia_elt}
Let $T=\POm^\e_{2m}(q)$ and let $\th$ be an automorphism in Table~\ref{tab:o_cases} (in Case~I(i), I(ii) or~I(iv)). Let $y \in \PDO_{2m}^\eta(q_0)$ be the element in Table~\ref{tab:o_Ia_elt}. Then there exists $t \in T$ such that $(t\th)^e$ is $X$-conjugate to $y$.
\end{propositionx} 

\begin{table}
\centering
\caption{Case~I(a): The element $y$ for the automorphism $\th$} \label{tab:o_Ia_elt}
{\setlength{\tabcolsep}{0pt}
\renewcommand{\arraystretch}{1.2}
\begin{tabular}{ccc}
\hline
\multicolumn{3}{c}{Generic case} \\
\hline   
$\begin{array}{c}m \\\mod{4}\end{array}$ & \multicolumn{2}{c}{\begin{tabular}{cc} \multicolumn{2}{c}{$y$} \\ $\eta=+$ \hspace{1.9cm} & \hspace{1.9cm} $\eta=-$ \end{tabular}} \\   
\hline
\begin{tabular}{c} $0$ \\ $2$ \end{tabular} & \begin{tabular}{c} ${\,}^c(m)^- \perp {\,}^a(m-2)^+ \perp {\,}^a(2)^-$ \\ ${\,}^c(m)^+ \perp {\,}^a(m-2)^- \perp {\,}^a(2)^-$ \end{tabular} & ${\,}^c(2m-2)^+ \perp {\,}^a(2)^-$ \\
\begin{tabular}{c} $3$ \\ $1$ \end{tabular} & ${\,}^c(2m-2)^- \perp {\,}^a(2)^-$ & \begin{tabular}{c} ${\,}^c(m+1)^- \perp {\,}^a(m-3)^- \perp {\,}^a(2)^-$ \\ \begin{tabular}{c} ${\,}^c(m+3)^- \perp {\,}^a(m-5)^- \perp {\,}^a(2)^-$ \end{tabular} \end{tabular}  \\
\hline \\[-2pt]
\hline
\multicolumn{3}{c}{Specific cases} \\
\hline
$m$ & \multicolumn{2}{c}{\begin{tabular}{cc} \multicolumn{2}{c}{$y$} \\ $\eta=+$ \hspace{1.9cm} & \hspace{1.9cm} $\eta=-$ \end{tabular}} \\
\hline
$4$ & ${\,}^c(6)^- \perp {\,}^a(2)^-$ & ${\,}^b(8)^-$                   \\
$5$ &                                 & ${\,}^b(6)^+ \perp {\,}^a(4)^-$ \\
\hline
\end{tabular}}
\\[5pt]
{\small Note: we describe $y$ by specifying its type over $\F_{q_0}$ }
\end{table}

\begin{proof}[Proof of Proposition~\ref{prop:o_Ia_elt}]
As $y \in \PDO^\eta_{2m}(q_0)$, by Theorem~\ref{thm:shintani_descent}, there exists $g \in \Inndiag(T)$ such that $(g\ws)^e$ is $X$-conjugate to $y$. We will now prove that $g\ws$ is contained in the coset $T\th$. It is routine to deduce information about which coset of $T_0$ contains $y$. For example, assume that $m$ is even and $\eta=-$. If $q$ is even, then $y$ has type $(2m-2)^+_{q_0} \perp (2)^-_{q_0}$ and $y$ is clearly an element of $T_0$. Now assume that $q$ is odd and fix $\widehat{y} = \widehat{y}_1 \perp \widehat{y}_2 \in \DO^-_{2m}(q_0)$ where $\widehat{y}_1$ has type ${\,}^b(2m-2)^+_{q_0}$ and $\widehat{y}_2$ has type ${\,}^a(2)^-_{q_0}$ such that $y = \widehat{y}Z(\DO^-_{2m}(q_0))$. If $\th \in \{\d,\rsq\rns\d\}$, then $\widehat{y}$ has type ${\,}^\Delta(2m-2)^+_{q_0} \perp {\,}^\Delta(2)^-_{q_0}$, so $\tau(\widehat{y}_1) = \tau(\widehat{y}_2) = \beta$ and we deduce that $y \in \PSO^+_{2m}(q_0)\d_0$. For now assume that $q_0$ is not Mersenne. If $\th = 1$, then $\widehat{y}$ has type $(2m-2)^+_{q_0} \perp (2)^-_{q_0}$, so, by Lemma~\ref{lem:omega_a}, $y \in T_0$, and if $\th = \rsq\rns$, then $\widehat{y}$ has type ${\,}^\Sigma(2m-2)^+_{q_0} \perp (2)^-_{q_0}$, so Lemma~\ref{lem:omega_a} implies that $(2)^-_{q_0} \in \Omega_2^-(q_0)$ and Lemma~\ref{lem:elt_c} implies that ${\,}^\Sigma(2m-2)^+_{q_0} \in \SO^+_{2m-2}(q_0)$, so $y \in T\rsq\rns$. Now assume that $q_0$ is Mersenne. If $\th = 1$, then $\widehat{y}$ has type ${\,}^\Sigma(2m-2)^+_{q_0} \perp (2)^-_{q_0}$, so $\widehat{y}_1 \in \SO^+_{2m-2}(q_0) \setminus \Omega^+_{2m-2}(q_0)$ and $\widehat{y}_2 \in \SO^-_{2m}(q_0) \setminus \Omega^-_{2m}(q_0)$ and therefore $y \in \Omega^-_{2m}(q_0)$. Similarly, if $\th = \rsq\rns$, then $\widehat{y}$ has type $(2m-2)^+_{q_0} \perp (2)^-_{q_0}$ and we deduce that $y \in \SO^-_{2m}(q_0) \setminus \Omega^-_{2m}(q_0)$.

We will now use Shintani descent (in particular Lemmas~\ref{lem:o_Ia_tau} and~\ref{lem:o_Ia_eta}) to deduce information about which coset of $T$ contains $g\ws$.

If $q$ is even, then $\ws = \th$ (one of $\p^i$, $r\p^i$ and $\psi^i$) and $X_{\s^e} = T$, so $g\ws \in T\th$. 

Therefore, from now on we may assume that $q$ is odd. Assume that $\th$ appears in Row~(2) of Table~\ref{tab:o_cases}. Then $\t(y) = \b_0$, so $y \in \PSO_{2m}^\eta(q_0)\d_0$. By Lemma~\ref{lem:o_Ia_tau}, this implies that $g\ws \in \PSO^\e_{2m}(q)\d\ws$. Therefore, $g\ws = t\th$ where $t \in T$ and $\th \in \{ \d\ws, \rsq\rns\d\ws\}$. In Case~I(i), $\th = \d\p^i$, in Case~I(ii) $\th = \d r\p^i$ and in Case~I(iv) $\th = \d\psi^i$, which suffices to prove the claim, since in all three cases, $\ddot{\th}$ and $\ddot{r}_\square\ddot{r}_\nonsquare\ddot{\th}$ are $\Out(T)$-conjugate (see Lemmas~\ref{lem:o_out_plus_facts} and~\ref{lem:o_out_minus_facts}).

Now assume $\th$ appears in Row~(1) or~(3). Then $\t(y) = 1$, so $y \in \PSO_{2m}^\eta(q_0)$ and $g\ws \in \PSO^\e_{2m}(q)\ws$, by Lemma~\ref{lem:o_Ia_tau}. If $D(Q) = \nonsquare$, then $\ws = \th$ (one of $\p^i$, $r\p^i$ and $\psi^i$) and $T=\PSO^\e_{2m}(q)$, so $g\ws \in T\th$. 

Therefore, it remains to assume that $D(Q) = \square$. In this case, $q^m \equiv \e \mod{4}$. For now assume that $q_0^m \equiv \eta \mod{4}$, so that we may apply  Lemma~\ref{lem:o_Ia_eta} (this always holds when $\e=-$). By the choice of $a$ and $c$, if $\th$ is in Row~(1), then the spinor norm of $y$ is square, so $y \in \POm_{2m}^\eta(q_0)$ and $g\ws \in T\ws$, by Lemma~\ref{lem:o_Ia_eta}, and, since $\th = \ws$ (one of $\p^i$, $r\p^i$ and $\psi^i$), we conclude that $g\ws \in T\th$. Similarly, if $\th$ is in Row~(3), then $y \in \PSO_{2m}^\eta(q_0) \setminus \POm_{2m}^\eta(q_0)$ and $g\ws \in T\rsq\rns\ws$, so  $g\ws \in T\th$ since $\th = \ws$ (one of $\rsq\rns\p^i$, $\rsq\rns r\p^i$ or $\rsq\rns\psi^i$).

We now need to assume that $q^m \equiv \e \mod{4}$ but $q_0^m \not\equiv \eta \mod{4}$. In this case $\e=+$. First assume that $\eta=+$. Therefore, $q_0 \equiv 3 \mod{4}$ and $m$ is odd. This forces $q \equiv 1 \mod{4}$. Together this implies that $m$ is odd, $p \equiv 3 \mod{4}$, $i$ is odd, $f$ is even. Under these conditions, we need only consider one of $\p^i$ and $\rsq\rns\p^i$ (see Remark~\ref{rem:o_cases}\ref{item:o_cases_notes}), so we can choose $\th$ such that $g\ws \in T\th$. Now assume that $\eta=-$. Therefore, $q_0 \equiv 1 \mod{4}$, so $m$ is even or $i$ is even or $p \equiv 1 \mod{4}$. This allows us to only consider one of $r\p^i$ and $\rsq\rns r\p^i$ (see Remark~\ref{rem:o_cases}\ref{item:o_cases_notes}), so, as above, we can choose $\th$ such that $g\ws \in T\th$. This completes the proof.
\end{proof}

Continue to let $T$ be the simple group $\POm^\e_{2m}(q)$ and let $\th$ be an automorphism from Table~\ref{tab:o_cases}. Fix $y \in \PDO_{2m}^\eta(q_0)$ from Table~\ref{tab:o_Ia_elt} and $t\th \in G = \<T,\th\>$ from Proposition~\ref{prop:o_Ia_elt}. We will now study the set $\M(G,t\th)$ of maximal overgroups of $t\th$ in $G$. For now we will assume that $T \neq \POm^\e_8(q)$ and we will return to this special case at the end of the section.

\begin{theoremx}\label{thm:o_Ia_max}
Assume that $T \neq \POm_8^\e(q)$. The maximal subgroups of $G$ which contain $t\th$ are listed in Table~\ref{tab:o_Ia_max}, where $m(H)$ is an upper bound on the multiplicity of the subgroups of type $H$ in $\M(G,t\th)$.
\end{theoremx}

\begin{table}
\centering
\caption{Case~I(a): Description of $\M(G,t\th)$}\label{tab:o_Ia_max}
{\renewcommand{\arraystretch}{1.2}
\begin{tabular}{lllp{6.7cm}}
\hline    
       & type of $H$                                   & $m(H)$ & conditions                                              \\
\hline
$\C_1$ & $\O^{\up}_2(q) \times \O^{\e\up}_{2m-2}(q)$   & $1$    & $(\eta,m) \neq (-,5)$                                   \\[5.5pt]
       & $P_{m/2}$                                     & $2$    & $\eta = +$ and $m \equiv 0 \mod{4}$                     \\
       & $P_{m/2-1}$                                   & $2$    & $\eta = +$ and $m \equiv 2 \mod{4}$                     \\      
       & $\O^{\up}_{m-2}(q) \times \O^{\e\up}_{m+2}(q)$& $1$    & $\eta = +$ and $m$ even                                 \\[5.5pt]
       & $P_{m-1}$                                     & $2$    & $\eta = -$ and $m$ even                                 \\
       & $\O^{\up}_{m-3}(q) \times \O^{\e\up}_{m+3}(q)$& $1$    & $\eta = -$ and $m$ odd                                  \\
       & $\O^{\up}_{m-5}(q) \times \O^{\e\up}_{m+5}(q)$& $1$    & $\eta = -$ and $m \equiv 1 \mod{4}$ with $m \neq 5$     \\
       & $\O^{\up}_{m-1}(q) \times \O^{\e\up}_{m+1}(q)$& $1$    & $\eta = -$ and $m \equiv 3 \mod{4}$                     \\[5.5pt]    
       & $\O^{\up}_4(q) \times \O^{\e\up}_6(q)$        & $1$    & $\eta = -$ and $m = 5$                                  \\
       & $P_3$                                         & $2$    & $\eta = -$ and $m = 5$                                  \\[5.5pt]   
$\C_2$ & $\O^{\up}_{2m/k}(q) \wr S_k$                  & $N$    & $k \div m$, \, $k > 1$, \, ${\up}^k = \e$               \\
       & $\O_{2m/k}(q) \wr S_k$                        & $N$    & $k \div 2m$, \, $2m/k > 1$ odd                          \\[5.5pt] 
       & $\GL_m(q)$                                    & $2N$   & $\eta = +$, \, $m$ even                                 \\
       &                                               & $N$    & $\e = +$, \, $\eta = -$, \, $m$ odd                     \\[5.5pt]
$\C_3$ & $\O_m(q^2)$                                   & $2N$   & $m > 5$ odd                                             \\ 
       & $\GU_m(q)$                                    & $2N$   & $\e = \eta = +$, \, $m$ even                            \\
       &                                               & $N$    & $\e = \eta = -$, \, $m$ odd                             \\[5.5pt]
$\C_5$ & $\O^{\up}_{2m}(q^{1/k})$                      & $N$    & $k \div f$, \, $k$ is prime, \, $\up^k=\e$              \\[5.5pt]
$\S$   & $\PSp_4(q)$                                   & $2N$   & $\eta = -$, \, $m=5$, \, $q \equiv \e \mod{4}$          \\ 
\hline
\end{tabular}}
\\[5pt]
{\small Note: $N=|C_{\PDO^\eta_{2m}(q_0)}(y)|$ and in $\C_1$ there is a unique choice of $\up$}
\end{table}

Let us outline the proof of Theorem~\ref{thm:o_Ia_max}. If $T \leq H$, then we deduce that $\th \in H$, since $t\th \in H$, but then we would have $H=G$, which is not the case. Therefore, $T \not\leq H$, so Theorem~\ref{thm:aschbacher} implies that $H$ is contained in one of the geometric families $\C_1, \dots, \C_8$ or is an almost simple irreducible group in $\S$.

Our general idea is to consider each of these families in turn and determine which possible types of subgroup could contain the element $t\th$, by exploiting the restrictive properties that we have chosen $t\th$ to have. For types of subgroups which could contain the element $t\th$ we will find an upper bound on the number of subgroups of this type that contain $t\th$. We will not concern ourselves with determining \emph{exactly} which subgroups contain $t\th$; sometimes it will be sufficient, for example, to use an overestimate on the number of subgroups of a given type which contain $t\th$.

\begin{remarkx} \label{rem:gpps}
If $s \in \GL_n(q)$ has order divisible by a primitive prime divisor of $q^k-1$ for $k > \frac{n}{2}$, then the subgroups $H \leq \GL_n(q)$ that contain $s$ are classified by the main theorem of \cite{ref:GuralnickPenttilaPraegerSaxl97}. However, this will not be useful in proving Theorem~\ref{thm:o_Ia_max}. To see why, consider the example where $G = \<\POm^+_{2m}(p^f),\p\>$ and $\p$ is the standard order $f$ field automorphism. Then, via Shintani descent, we choose an element $t\p \in G$ such that $(t\p)^f$ is conjugate to an element of $\POm^+_{2m}(p)$. Therefore, $t\p$, and even more so $(t\p)^f$, which is the element we typically have better information about, has a small order compared with the order of $G$. Consequently, we will need to use other properties of the element $t\p$ in order to constrain its maximal overgroups.
\end{remarkx}

We will present a result on multiplicities of subgroups in $\M(G,t\th)$, before proving Theorem~\ref{thm:o_Ia_max} in three parts, by considering the cases where $H \in \M(G,t\th)$ is reducible, irreducible imprimitive and primitive. We write
\[
\widetilde{G} = \< X_{\s^e}, \ws \>
\]
noting that $\Inndiag(T) \leq \widetilde{G} \leq \Aut(T)$ and $G \leq \widetilde{G}$.

The following result will apply to Case~I(b) also.

\begin{propositionx}\label{prop:o_I_conjugacy}
Assume that $T \neq \POm_8^+(q)$. Let $T \leq A \leq \Aut(T)$ and let $H$ be a maximal $\C_1$, $\C_2$, $\C_3$ or $\C_5$ subgroup of $A$. Then there is a unique $\widetilde{G}$-conjugacy class of subgroups of type $H$, unless $H$ has one of the following types, in which case there are two $\widetilde{G}$-classes:
\begin{center}
\begin{tabular}{cccccc}
\hline
type & $P_m$ & $\GL_m(q)$ & $\GU_m(q)$ & $\O^+_m(q^2)$ & $\O_m(q^2)$ \\
$\e$ & $+$   & $+$        & $+$        & $+$           & $-$         \\
$m$  & any   & odd        & even       & even          & odd         \\
\hline
\end{tabular}
\end{center}
\end{propositionx}

\begin{proof}
If $m \leq 6$, then the result follows from the tables in \cite[Chapter~8]{ref:BrayHoltRoneyDougal}. Now assume that $m \geq 7$. We will apply the main theorem of \cite{ref:KleidmanLiebeck}. 

Let $H$ be a maximal geometric subgroup of $G$. Let $\mathcal{H}=\{H_1, \dots, H_c\}$ be a set of representatives of the $c$ distinct $T$-classes of subgroups of $T$ of the same type as $H$. In the terminology of \cite[Chapter~3]{ref:KleidmanLiebeck}, for each $1 \leq i \leq c$, let $H_{G,i}$ be the $G$-associate of $H_i$. In particular, $H_{G,i}$ is a geometric subgroup of $G$ of the same type as $H_i$ and $H_i \leq H_{G,i}$ (see \cite[Section~3.1]{ref:KleidmanLiebeck} for a precise definition). There is a natural action of $\Out(T)$ on the set $\mathcal{H}$, and the permutation representation $\pi\:\Out(T) \to S_c$ associated to this action is described in \cite[Tables~3.5.A--3.5.G]{ref:KleidmanLiebeck}. As a consequence of the proof of \cite[Lemma~3.2.2(iii)]{ref:KleidmanLiebeck}, for $G \leq A \leq \Aut(T)$, the groups $H_{G,i}$ and $H_{G,j}$ are $A$-conjugate if and only if $H_i$ and $H_j$ are in the same $\pi(\ddot{A})$-orbit. By \cite[Tables~3.5E and~3.5G]{ref:KleidmanLiebeck}, $\pi(\widetilde{G}/T)$ is transitive, except for the exceptional cases in the statement, when $c=2$ and $\pi(\widetilde{G}/T)$ is intransitive. This proves the statement, but we provide some examples, with $\e=+$.

For example, consider the case where $m$ is odd, $H$ has type $\O^-_{2m}(q^{1/2})$ and $p \equiv 1 \mod{4}$. In this situation, $c=4$, $\ker(\pi) = \< \ddot{\p} \>$ and the stabiliser of $H_1$ is $\< \ddot{\p}, \ddot{r}_{\square}\>$. Therefore, $\pi(\widetilde{G}/T) = \< \ddot{\d} \> \cong C_4$ is transitive, so there is exactly one $\widetilde{G}$-class of subgroups of $G$ of the same type as $H$.

For another example, let $m$ be even and let $H$ have type $\GL_m(q)$. In this situation, $c=2$, $\ker(\pi) = \< \Inndiag(T)/T, \ddot{\p} \>$ and the stabiliser of $H_1$ is $\< \ddot{\p}, \ddot{r}_{\square}\>$. Therefore, $\pi(\widetilde{G}/T) = 1$, so there are exactly two $\widetilde{G}$-classes of subgroups of $G$ of the same type as $H$.
\end{proof}

\begin{propositionx}\label{prop:o_Ia_max_reducible}
Theorem~\ref{thm:o_Ia_max} is true for reducible subgroups.
\end{propositionx}

\begin{proof}
We will apply Lemma~\ref{lem:shintani_descent_fix} (see Example~\ref{ex:shintani_descent_fix}). 

\emph{Case 1: stabilisers of totally singular subspaces.} Let $H$ be a maximal parabolic subgroup of $G$. Then $H \leq \widetilde{H} = \< Y_{\s^e}, \ws \>$ for a $\ws$-stable parabolic subgroup $Y \leq X$. In particular, $Y$ is a closed connected subgroup of $X$. Moreover, $\widetilde{H}$ and $Y_\s$ are maximal (and hence self-normalising) subgroups of $\widetilde{G}$ and $X_\s$, respectively. Therefore, Lemma~\ref{lem:shintani_descent_fix} implies that the number of $X_{\s^e}$-conjugates of $H$ which contain $t\th$ equals the number of $X_{\s}$-conjugates of $H \cap X_{\s}$ which contain $F(t\th) = y$. 

Assume that $\eta=+$ and $m \geq 5$ with $m \mod{4} \in \{0,1\}$; the other cases are similar. First assume that $m \equiv 1 \mod{4}$. By Lemma~\ref{lem:c1}, $y$ does not stabilise any totally singular subspaces of $\F_{q_0}^n$ and thereore is not contained in any parabolic subgroups of $X_{\s}$. Therefore, $t\th$ is not contained in any parabolic subgroups of $G$.

Now assume that $m \equiv 0 \mod{4}$. Here $y$ stabilises exactly two totally singular subspaces of $\F_{q_0}^n$, each of dimension $m/2-1$, so $y$ is contained in exactly two parabolic subgroups of $X_{\s}$, of type $P_{m/2-1}$, and consequently $t\th$ is contained in exactly two parabolic subgroups of $G$, of type $P_{m/2-1}$, as claimed in Theorem~\ref{thm:o_Ia_max}.

\emph{Case 2: stabilisers of nondegenerate subspaces.} Let $H$ be the stabiliser in $G$ of a nondegenerate $k$-space. Let $L= \SL_n(\FF_p)/\<-I_n\>$ and extend the domain of $\s$ to $L$. Let $E$ be the Shintani map of $(L,\s,e)$. Observe that $t\th \in G \leq \< L_{\s^e}, \th \>$ and $F(t\th) \in X_{\s} \leq L_{\s}$. Accordingly, Lemma~\ref{lem:shintani_subgroups} implies that $F(t\th) = E(t\th)$. Let $M \leq L$ be a $P_k$ parabolic subgroup. Applying Lemma~\ref{lem:shintani_descent_fix} to the Shintani map $E$ for $L$ and the subgroup $M \leq L$, we see that the number of $k$-spaces of $V=\F_q^n$ fixed by $t\th$ equals the number of $k$-spaces of $V_0=\F_{q_0}^n$ fixed by $E(t\th) = F(t\th) = y$. 

Again let us assume that $m \mod{4} \in \{0,1\}$ and $\eta = +$, beginning with the case where $m \equiv 1 \mod{4}$. Lemma~\ref{lem:c1} implies that $y$ stabilises exactly two proper nonzero subspaces of $V_0$, of dimensions $2$ and $(2m-2)$, so $t\th$ stabilises exactly two proper nonzero subspaces of $V$, of the same dimensions. In Case~1, we demonstrated that $t\th$ is not contained in a parabolic subgroup of $G$. Therefore, both subspaces of $V$ must be nondegenerate, for otherwise $t\th$ would stabilise its (totally singular) radical and therefore be contained in a parabolic subgroup. Consequently, the only reducible maximal subgroup of $G$ containing $t\th$ has type $\O^{\up}_2(q) \times \O^{-\up}_{2m-2}(q)$ for some sign $\up \in \{+,-\}$ (it is exactly for the reason that we pass to the linear group $L$ that we cannot determine the sign $\up$).

Now assume that $m \equiv 0 \mod{4}$. Then $y$ stabilises exactly 14 proper nonzero subspaces of $V_0$, of dimensions
\begin{gather*}
2, \quad m/2-1 \text{ \ (2)}, \quad m/2+1 \text{ \ (2)}, \quad m-2 \text{ \ (2)}, \quad m \text{ \ (2)}, \\ m+2, \quad 3m/2-1 \text{ \ (2)}, \quad 3m/2+1 \text{ \ (2)}, \quad 2m-2,
\end{gather*}
wher $\text{(2)}$ denotes the fact that there are two subspaces of each of these dimensions. 

From Case~1, we know that $t\th$ stabilises exactly two totally singular subspaces, each of dimension $m/2-1$. Since $t\th$ stabilises a (necessarily not totally singular) $2$-, $m$- and $(m+2)$-space, we deduce that the stabilised $(m/2+1)$-, $(3m/2-1)$- and $(3m/2+1)$-spaces must be the direct sum of 2 the $2$-, $m$- and $(m+2)$-spaces with the two $(m/2-1)$-spaces. These subspaces are neither totally singular, since there are only two such subspaces stabilised by $t\th$. We now claim that these six subspace are degenerate. Indeed, the $(m/2+1)$-space has a $(m/2-1)$-dimensional totally singular subspace, which implies that it is degenerate. In addition, if one of the $(3m/2 \pm 1)$-spaces were nondegenerate, then $t\th$ would stabilise its $(m/2 \mp 1)$-dimensional nondegenerate orthogonal complement, but we have already shown that all subspaces of this dimension stabilised by $t\th$ are degenerate. Therefore, the only possible nondegenerate subspaces of $V$ stabilised by $t\th$ are those of dimension $2$, $m-2$, $m$ (of which there are two), $m+2$, $2m-2$, which give rise to maximal irreducible subgroups of type  $\O^{\e_1}_2(q) \times \O^{-\e_1}_{2m-2}(q)$ and $\O^{\e_2}_{m-2}(q) \times \O^{-\e_2}_{m+2}(q)$ for particular signs $\e_1$ and $\e_2$ (the stabiliser of a nondegenerate $m$-space is not maximal).
\end{proof}

\begin{propositionx}\label{prop:o_Ia_max_imprimitive}
Theorem~\ref{thm:o_Ia_max} is true for imprimitive subgroups.
\end{propositionx}

\begin{proof}
By \cite[Table~3.5.E]{ref:KleidmanLiebeck}, all possible types of irreducible imprimitive subgroup feature in Table~\ref{tab:o_Ia_max}. If $\e=+$, then we claim that maximal subgroups of type $\GL_m(q)$ only arise if $\eta=+$ and $m$ is even, or $\eta=-$ and $m$ is odd.

First consider $\eta=+$ and $m$ odd. In this case, $G \leq \< \Inndiag(T), \p^i\>$, so there are no elements in $G$ which interchange the totally singular subspaces $\<e_1, \dots, e_m\>$ and $\<f_1, \dots,f_m\>$ (see \cite[Proposition~2.7.4]{ref:KleidmanLiebeck}). Therefore, a subgroup of $G$ of type $\GL_m(q)$ is contained in two subgroups of type of $P_m$, and no maximal subgroups of type $\GL_m(q)$ occur. 

Now consider $\eta=-$ and $m$ even. In this case, $G \not\leq \< \Inndiag(T), \p^i\>$, so by \cite[Tables~3.5.E and~3.5.G]{ref:KleidmanLiebeck}, any subgroup of $G$ of type $\GL_m(q)$ is contained in a proper normal subgroup of $G$ and is, therefore, not maximal.

The multiplicities follow quickly from Lemma~\ref{lem:centraliser_bound} and Proposition~\ref{prop:o_I_conjugacy}. 
\end{proof}

\begin{lemmax}\label{lem:o_Ia_max_power}
Assume that $m \neq 4$ and $(\eta,m) \neq (-,5)$. A suitable power of $y$ has type $I_{n-2} \perp A$ where
\[
A = 
\left\{
\begin{array}{ll}
(2)^-_{q_0}  & \text{if $q_0$ is not Mersenne} \\
-I_2         & \text{otherwise.}
\end{array}
\right.
\]
\end{lemmax}

\begin{proof}
All of the types of elements that we discuss in this proof are defined over $\F_{q_0}$ but we omit the subscripts $q_0$ for simplicity of notation.

\emph{Case 1: $q_0$ is not Mersenne.} First assume that $\th \in \{ \p^i, \psi^i\}$. In this case, $y = y_1 (\perp y_2) \perp (2)^-$, where $(2)^-$ has order $r \in \ppd(q_0,2)$, and $y_i$ has type $(d_i)^{\e_i}$ and order $r_i \in \ppd(q_0,\ell_i)$, where we write $\ell_i = d_i/(d_i/2-1,2)$ (we put the middle term in brackets to indicate that depending on $m$ and $\eta$, the element might centralise a decomposition into either two or three subspaces). In particular, $\ell_i > 2$, so $r$ and $r_i$ are coprime. Consequently, a power of $y$ has type $I_{n-2} \perp (2)^-$, as required.

Next assume that $\th \in \{ \d\p^i, \d\psi^i\}$, so $y = {\,}^\Delta y_1 (\perp {\,}^\Delta y_2) \perp {\,}^\Delta (2)^-$, where $y_1$ and $y_2$ are as in the previous case.  Noting that $(2)^-$ has odd order, by Definition~\ref{def:elt_c}, we may raise $y$ to a suitable power of $2$ in order to obtain an element of type $y_1 (\perp y_2) \perp (2)^-$, which reduces to the previous case.

Now assume that $\th \in \{ \rsq\rns\p^i, \rsq\rns\psi^i \}$, so $y$ has type ${\,}^\Sigma y_1 (\perp y_2) \perp (2)^-$. By Definition~\ref{def:elt_e}, we may again obtain an element of type $y_1 (\perp y_2) \perp (2)^-$ by raising $y$ to some suitable power of two, thus reducing this case to the first one.

\emph{Case 2: $q_0$ is Mersenne.} In this case, an element of type $(2)^- = (2)^-_{q_0}$ has order $q_0+1$, which is a power of two, so we must be more careful when raising elements to even powers. However, note that elements of type $(2)^-$ and $(d)^\pm$ for $d > 2$ still have coprime order.

If $\th \in \{ \rsq\rns\p^i, \rsq\rns\psi^i\}$, then $y = y_1 (\perp y_2) \perp (2)^-$, where $y_i$ has type $(d_i)^{\e_i}$ and a power of $y$ has type $I_{n-2} \perp (2)^-$.

Now assume that $\th \in \{ \d\p^i, \d\psi^i\}$. For concreteness consider the case where $\eta=+$ and $m \equiv 0 \mod{4}$; the other cases are no harder to analyse. Here $y$ has type ${\,}^\Delta(m)^- \perp {\,}^\Delta(m-2)^+ \perp {\,}^\Delta(2)^-$. Since $q_0$ is Mersenne, $q_0 \equiv 3 \mod{4}$ and consequently $(q_0^{m/2-1}+1)_2 = 2$, noting that $m/2-1$ is odd. Therefore, $y^{2(q_0-1)_2}$ has type $(m)^- \perp (m-2)^+ \perp x^2$, where $x$ has type $(2)^-$. Now $|x| = q_0+1 \geq 4$ is a power of two, so $y^{(q_0-1)_2(q_0+1)}$ has type $(m)^- \perp (m-2)^+ \perp -I_2$, a suitable (odd) power of which has type $I_{2m-2} \perp -I_2$.

Finally assume that $\th = \{\p^i, \psi^i\}$. Then $y$ has type ${\,}^\Sigma y_1 (\perp y_2) \perp (2)^-$.  Definition~\ref{def:elt_e} informs us that $y^2$ has type $y_1 (\perp y_2) \perp w$, where $w$ has order $\frac{1}{2}(q_0+1) \geq 2$, so a power of $y^2$ has type $-I_2 \perp I_{n-2}$. This completes the proof.
\end{proof}

\begin{propositionx}\label{prop:o_Ia_max_primitive}
Theorem~\ref{thm:o_Ia_max} is true for primitive subgroups.
\end{propositionx}

\begin{proof}
For now assume that $(\eta,m) \neq (-,5)$. By construction, a suitable power of $t\th$ is $X$-conjugate to $y$.  By Lemma~\ref{lem:o_Ia_max_power}, fix a power $z = z_1 \perp I_{2m-2}$ of $y$, where 
\[
z_1 = 
\left\{
\begin{array}{ll}
(2)^-_{q_0}       & \text{if $q_0$ is not Mersenne} \\
-I_2              & \text{otherwise,}
\end{array}
\right.
\] 
noting that $z \in T$ has prime order.

Now let $H \in \M(G,t\th)$ be primitive. By Theorem~\ref{thm:aschbacher}, $H$ is contained in one of the geometric families $\C_3, \dots, \C_8$ or is an almost simple irreducible group in the $\S$ family. We consider each family in turn.

Consider $\C_3$ subgroups. First suppose that $H$ has type $\O^{\up}_{2m/k}(q^k)$ for a prime divisor $k$ of $2m$ and a sign ${\up} \in \{\circ,\e\}$. Write $H \cap T = B.k$. From the definition of $z$, Lemma~\ref{lem:c3}(ii) implies that $z \in B$. Moreover, since $\nu(z) = 2$, Lemma~\ref{lem:c3}(i) implies that $k=2$. Therefore, to verify the claim in Table~\ref{tab:o_Ia_max}, we can assume that $m$ is even. In this case, a power of $y$ has type $(2d)^+ \perp I_{2m-2d}$, where $d \in \{\frac{m}{2}, \frac{m-2}{2}, m-1 \}$ is odd, which contradicts Corollary~\ref{cor:c3_subfield}. Therefore, $H$ does not have type $\O^\up_{2m/k}(q^k)$ unless $m$ is odd and $k=2$. 

Now suppose that $H$ has type $\GU_m(q)$. These maximal subgroups only occur when $\e = +$ and $m$ is even, or $\e=-$ and $m$ is odd (see \cite[Tables~3.5.E and~3.5.F]{ref:KleidmanLiebeck}). Suppose that $\e=+$ but $\eta=-$ (and $m$ is even). In this case a power of $y$ has type $I_2 \perp (2m-2)^-$, but this is a contradiction to Corollary~\ref{cor:c3_subfield}(ii)(a). Therefore, $H$ has type $\GU_m(q)$ and $\e=\eta=(-)^m$.

Now let us turn to $\C_4$ subgroups. Suppose that $H$ is the centraliser of a decomposition $V_1 \otimes V_2$ where $\dim{V_1} \geq \dim{V_2} > 1$. Since $z \in H$, we may write $z = z_1 \otimes z_2$. Since $\nu(z)=2$, \cite[Lemma~3.7]{ref:LiebeckShalev99} implies that $\nu(z_1) = 1$, $\nu(z_2) = 0$ and $\dim{V_2} = 2$. Inspecting the conditions on $\dim{V_1}$ and $\dim{V_2}$ in \cite[Tables~3.5.E and~3.5.F]{ref:KleidmanLiebeck}, this is impossible unless $\e=+$ and $H$ has type $\Sp_2(q) \otimes \Sp_m(q)$. Since $\nu(z_2) = 0$, we must have that $z_1$ is a semisimple element of $\Sp_m(q)$ such that $\nu(z_1)=1$, and there are no such elements. Therefore, $H \not\in \C_4$.

If $H \in \C_5$, then $H$ has type $\O^{\up}_{2m}(q_1)$ where $q=q_1^k$ for a prime divisor $k$ of $f$ and a sign $\up \in \{+,-\}$ such that ${\up}^k = \e$.

The $\C_6$ family is empty since $q$ is not prime.

We now treat $\C_7$ subgroups, which only arise when $\e=+$. Suppose that $H$ is the stabiliser of a decomposition $U_1 \otimes U_2 \otimes \cdots \otimes U_k$ with $\dim{U_i}>1$. Let $H_0 = H \cap \PGL(V)$ and write $H_0 = B.S_k$. Since $z$ does not centralise a tensor product decomposition (see the discussion of $\C_4$ subgroups), $z \not\in B$. Therefore, $z$ cyclically permutes the $k$ factors. However, $z$ has prime order and exactly two nontrivial eigenvalues which contradicts the eigenvalue pattern required by  \cite[Lemma~5.7.2]{ref:BurnessGiudici16}. Therefore, $H \not\in \C_7$.

The $\C_8$ family is empty. 

Finally, consider the $\S$ family. Since $\nu(z) = 2$, $2m \geq 10$ and $q$ is not prime, \cite[Theorem~7.1]{ref:GuralnickSaxl03} implies that no such subgroups arise.

It remains to assume that $(\eta,m)=(-,5)$. To prove the result in this case, we simply note that $y$ has type ${\,}^a(4)^- \perp {\,}^c(6)^+$, so a power of $y$ has type $(6)^+ \perp I_4$, which, in light of Corollary~\ref{cor:c3_subfield}, implies that $y$ is not contained in subgroups of type $\O_5(q^2)$ or $\GU_5(q)$. 

To complete the proof, we note that the stated upper bounds on the multiplicities of nonsubspace subgroups follow from Lemma~\ref{lem:centraliser_bound} and Proposition~\ref{prop:o_I_conjugacy}.
\end{proof}

We have now proved Theorem~\ref{thm:o_Ia_max} and are, consequently, in the position to prove Theorems~\ref{thm:o_main} and~\ref{thm:o_asymptotic} in Case~I(a).

\begin{propositionx}\label{prop:o_Ia}
Let $G = \<T, \th\> \in \A$ with $T \neq \POm_8^\e(q)$. In Case~I(a), $u(G) \geq 2$ and as $q \to \infty$ we have $u(G) \to \infty$. 
\end{propositionx}

\begin{proof}
We apply the probabilistic method encapsulated by Lemma~\ref{lem:prob_method}. Let $x \in G$ have prime order. We will obtain an upper bound on
\[
P(x,t\th) \leq \sum_{H \in \M(G,t\th)} \fpr(x,G/H).
\]
By Lemma~\ref{lem:prob_method} we need to show that $P(x,t\th) < \frac{1}{2}$ and $P(x,t\th) \to 0$ as $q \to \infty$.

Theorem~\ref{thm:o_Ia_max} gives a superset of $\M(G,t\th)$. Moreover, referring to Table~\ref{tab:o_Ia_max}, it is straightforward to show that
\[
N = |C_{\PDO^\eta_{2m}(q_0)}(y)| \leq 2q_0^m.
\]
For instance, if $\eta=-$ and $m$ is even, then Lemmas~\ref{lem:centraliser} and~\ref{lem:elt_centraliser} imply that
\[
|C_{X_{\s}}(y)| \leq (q_0+1)(q_0^{m-1}-1) \leq 2q_0^m.
\]

The relevant fixed point ratios are given in Theorem~\ref{thm:fpr_s} and Proposition~\ref{prop:fpr_ns_o}, where we make use of the observation that $\nu(x) \geq 2$ for all $x \in G \cap \PGO^\e_{2m}(q)$. 

Write $d(n)$ for the number of proper divisors of a number $n$.

First assume that $\eta=+$ and $m$ is odd, or $\eta=-$ and $m$ is even. Then
\[
P(x,t\th) \leq \frac{1}{q^2}+\frac{7}{q^{m-2}}+\frac{5}{q^{m-1}} +  (2+\log\log{q} + 2d(2m)) \cdot 2q_0^m \cdot \frac{3}{q^{2m-5}},
\]
which proves $P(x,t\th) \to 0$ as $q \to \infty$ and $P(x,t\th) < \frac{1}{2}$ unless $(\eta,m,q)=(+,5,4)$. (Here we make use of the fact that when $\e=-$, we know that $2f/i$ is odd, so $i > 1$ and consequently $q \geq q_0^3$.)

In the exceptional case, $t\th$ is not contained in a maximal parabolic subgroup, and we can discount subgroups of type $\O^-_{10}(2)$ since they do not contain elements of order $|y| = 51$. These observations, together with a refined bound on the centraliser $|C_{X_\s}(y)|$, give
\[
P(x,t\th) \leq \frac{1}{4^2}+\frac{3}{4^3}+\frac{1}{4^4} + (1+1)\cdot(2+1)(2^4+1) \cdot\frac{3}{4^5} < \frac{1}{2}.
\] 

Next assume that $\eta=+$ and $m$ is even. Then
\[
P(x,t\th) \leq \frac{1}{q^2}+\frac{3}{q^{m/2-1}}+\frac{14}{q^{m-2}} + (1+\log\log{q} + 2d(2m)) \cdot 2q_0^m \cdot \frac{3}{q^{2m-5}} + 8q_0^m \cdot \frac{3}{q^{2m-7}},
\]
so $P(x,t\th) \to 0$ as $q \to \infty$ and $P(x,t\th) < \frac{1}{2}$, unless $(m,q) = (6,4)$. 

In this exceptional case, we will show that $t\th$ is contained in no subgroups of type $\GL_6(4)$ or $\GU_6(4)$; omitting the corresponding term gives $P(x,\th) < \frac{1}{2}$. The type of $y$ is $(2)^-_2 \perp (4)^-_2 \perp (6)^+_2$. First suppose that $y$ is contained in a subgroup $H$ of type $\GU_6(4)$. Write $H \cap \PGL(V) = B.2$. A power $y_1$ of $y$ has type $2^-_2 \perp I_{10}$, whose order is $3$. Therefore, $y_1 \in B$; however, $e=2$, so this contradicts Corollary~\ref{cor:c3_subfield}, so $t\th$ is not contained in a $\GU_6(4)$ subgroup. Next suppose that $y$ is contained in a subgroup $H$ of type $\GL_6(4)$. Again we write $H \cap \PGL(V) = B.2$. A power $y_2$ of $y$ has type $4^-_2 \perp I_8$, whose order is $5$. Therefore $y_2 \in B$. This implies that $y_2 = M \oplus M^{-\tr}$. The four nontrivial eigenvalues of $y_2$ are $\l, \l^{2}, \l^{2^2}, \l^{2^3}$, where $|\l| = 5$. Without loss of generality, $\l$ is an eigenvalue of $M$. On the one hand, $\l^4$ must be an eigenvalue of $M$, but, on the other hand, $\l^{-1} = \l^4$ is an eigenvalue of $M^{-\tr}$, which is a contradiction. Therefore, $t\th$ is not contained in a $\GL_6(4)$ subgroup.

Now assume that $\eta=-$ and $m \geq 7$ is odd. Then 
\begin{align*}
P(x,t\th) \leq \frac{1}{q^2} &+ \frac{2}{q^{(m-1)/2}} + \frac{11}{q^{m-3}} + \frac{1}{q^{m-5}} \\ & + (2+\log\log{q} + 2d(2m)) \cdot 2q_0^m \cdot \frac{3}{q^{2m-5}} + 2q_0^m \cdot \frac{3}{q^{2m-7}} < \frac{1}{2}
\end{align*}
and $P(x,t\th) \to 0$ as $q \to \infty$.

Finally assume that $(\eta,m) = (-,5)$. Then 
\[
P(x,t\th) \leq \frac{1}{q^2} + \frac{8}{q^3} + \frac{4}{q^4} + (6 + \log{\log{q}})\cdot 2q_0^5 \cdot \frac{3}{q^5} + 2q_0^5 \cdot \frac{3}{q^3},
\]
which proves $P(x,t\th) \to 0$ as $q \to \infty$ and $P(x,t\th) < \frac{1}{2}$ unless $\e=+$ and $e=2$. In this case, by arguing as above we can show that $y$ is not contained in a subgroup of type $\GL_5(q)$ and omitting the corresponding term gives  $P(x,t\th) < \frac{1}{2}$ unless $q=4$. Now assume that $q=2$. Here we can discount subgroups of type $\O^+_{10}(2)$ since they do not contain elements of order $|y|=35$ and, by Lemma~\ref{lem:shintani_subfield}, $t\th$ is contained in at most $e^2=4$ subgroups of type $\O^-_{10}(2)$. Therefore,
\[
P(x,t\th) \leq \frac{1}{4^2} + \frac{8}{4^3} + \frac{4}{4^4} + (2\cdot(2^2+1)(2^3-1)+4) \cdot \frac{3}{4^5} < \frac{1}{2}.
\]
and $P(x,t\th) \to 0$ as $q \to \infty$. This completes the proof.
\end{proof}

In Case~I(a), it remains to prove Theorems~\ref{thm:o_main} and~\ref{thm:o_asymptotic} with $T = \POm^+_8(q)$. Recall the element $y$ was defined in Table~\ref{tab:o_Ia_elt} and Proposition~\ref{prop:o_Ia_elt} guarantees the existence of an element $t\th \in T$ such that $F(t\th)=y$.

\begin{propositionx}\label{prop:o_Ia_4_plus}
Let $G = \<T, \th\> \in \A$ where $T=\POm^\e_8(q)$. In Case~I(a), $u(G) \geq 2$ and as $q \to \infty$ we have $u(G) \to \infty$. 
\end{propositionx}

\begin{proof}
We apply Lemma~\ref{lem:prob_method}. First assume that $\eta=-$. Recall that $e$ is even if $\e=+$ and $e$ is odd if $\e=-$. In this case $y \in X_\s = \PDO^-_8(q_0)$ has type ${\,}^b(8)^-$, so $|C_{X_\s}(y)| \leq q_0^4+1$, by Lemma~\ref{lem:elt_centraliser}. Now $y$ is not contained in any reducible subgroups of $\PDO^-_8(q_0)$, so by arguing as in the proof of Proposition~\ref{prop:o_Ia_max_reducible}, using Lemma~\ref{lem:shintani_descent_fix}, we deduce that $t\th$ is not contained in any reducible subgroups of $G$. By \cite[Tables~8.50-53]{ref:BrayHoltRoneyDougal}, there are at most $M+\log\log{q}$ conjugacy classes of irreducible maximal subgroups of $G$ where
\[
M = \left\{
\begin{array}{ll}
 6 & \text{if $\e=+$} \\
 2 & \text{if $\e=-$}
\end{array}
\right.
\] 
and $G$ does not have any maximal subgroups of type $\GL^\pm_4(q)$. Therefore, from the bound in Proposition~\ref{prop:fpr_ns_o}(ii), for all prime order $x \in G$ we have
\[
P(x,t\th) < (M+\log\log{q})(q_0^4+1) \cdot \frac{3}{q^3} < \frac{1}{2}
\]
and $P(x,t\th) \to 0$ as $q \to \infty$, unless $e=2$ and $f \in \{2,4\}$ (so $\e=+$). 

Now assume that $e=2$ and $f \in \{2,4\}$. Then a suitable power of $y$ has type $(8)^-_{q_0} = (4)^-_q \perp (4)^-_q$ (see Lemma~\ref{lem:elt_splitting}). Let us consider the possible imprimitive maximal overgroups of $y$ of type $B{:}S_k$. Since the order of $y$ is coprime to $24$, $y \in B$, which implies that $H$ has type $\O^-_4(q) \wr S_2$ and $y^{12}$ (and hence $t\th$) is contained in a unique conjugate of $H$. Now consider primitive overgroups. The subgroups of type $\PSL_3(q)$ and $\O_8^+(q^{1/2})$ contain no elements of order $r \in \ppd(q^{1/2},16)$, so $y$ is contained in no subgroups of these types, and by Lemma~\ref{lem:shintani_subfield}, $t\th$ is contained in at most $4$ subgroups of type $\O_8^-(q^{1/2})$. Therefore,
\[
P(x,t\th) < (1 + 3) \cdot \frac{3}{q^3} < \frac{1}{2}
\]
and $P(x,t\th) \to 0$ as $q \to \infty$.

Now assume that $\e = \eta = +$. Here $y \in \PDO^+_8(q_0)$ has type ${\,}^c(6)^- \perp {\,}^a(2)^-$, so $|C_{X_\s}(y)| \leq (q_0+1)(q_0^3+1)$. We now study $\M(G,t\th)$, beginning with reducible subgroups. Since $y$ is contained in a unique reducible maximal subgroup of $\PDO^+_8(q_0)$ (of type $\O^-_2(q_0) \times \O^-_6(q_0)$), by Lemma~\ref{lem:shintani_descent_fix}, we deduce that $t\th$ is contained in a unique reducible maximal subgroup of $G$ (of type $\O^\up_2(q) \times \O^\up_6(q)$ for some choice $\up \in \{+,-\}$). Next note that $G$ has six $\widetilde{G}$-classes of maximal imprimitive subgroups, exactly two of which have type $\GL_4(q)$. Finally consider primitive maximal subgroups. For each prime divisor $k$ of $f$, there is one $\widetilde{G}$-class of subfield subgroups of type $\O^+_8(q^{1/k})$, and if $f$ is even, then also one $\widetilde{G}$-class of $\O^-_8(q^{1/2})$ subgroups. There are at most 11 further $\widetilde{G}$-classes of maximal primitive subgroups, exactly two of which have type $\GU_4(q)$.

Let $x \in G$ have prime order. For now assume that $e \geq 3$. Then Theorem~\ref{thm:fpr_s} and Proposition~\ref{prop:fpr_ns_o} imply that
\[
P(x,t\th) < \frac{4}{q^2} + \frac{1}{q^3} + (14+\log\log{q})(q_0+1)(q_0^3+1) \cdot \frac{3}{q^{15/4}} + 4(q_0+1)(q_0^3+1) \cdot \frac{2}{q^{12/5}} \to 0
\]
$q \to \infty$ and $P(x,t\th) < \frac{1}{2}$, unless $q=2^3$.

Now assume that $q=2^3$. Here $(t\th)^3$ is $X$-conjugate to $y = (2)^-_2 \perp (6)^-_2$. We will consider more carefully the maximal overgroups of $t\th$. 

We begin with imprimitive subgroups. Note that $y = (2)^-_2 \perp (6)^-_2 = A_0 \perp A_1 \perp A_2 \perp A_3$ centralising $\F_8^8 = U_0 \perp U_1 \perp U_2 \perp U_3$ where each $U_i$ is a nondegenerate minus-type $2$-space on which $A_i$ acts irreducibly (indeed $|A_0|=3$ and $|A_i|=9$ if $i > 0$). This implies that $t\th$ is not contained in any subgroups of types $\GL_4(8)$, $\O^-_4(8) \wr S_2$ or $\O^+_2(8) \wr S_4$ and is contained in at most $1$ subgroup of type $\O^-_2(8) \wr S_4$ and at most $3$ subgroups of type $\O^+_2(8) \wr S_2$. 

We now turn to primitive subgroups. For subfield subgroups, by Lemma~\ref{lem:shintani_subfield}, $t\th$ is contained in at most $9$ subgroups of type $\O^+_8(2)$. For field extension subgroups, we claim that $y$ is not contained in any subgroups of type $\O^+_4(8^2)$ and is contained in at most $16$ subgroups of type $\GU_4(8)$. The first claim follows from Lemma~\ref{lem:c3} noting that $y^3 = [1,1,\l,\l,\l,\l^2,\l^2,\l^2]$ where $|\l|=3$. For the second claim let $H$ have type $\GU_4(8)$, write $H \cap T = B.2$ and let $\pi$ be the field extension embedding. Now $y = [\l,\l^{-1},\mu,\mu^{-1},\mu^2,\mu^{-2},\mu^{4},\mu^{-4}]$ where $|\l|=3$ and $|\mu|=9$. Let $b \in B$ satisfy $\pi(b) = y$. Then $y = [\l^{\e_1},\mu^{\e_2},\mu^{\e_3},\mu^{\e_4}]$ where $\e_i \in \{+,-\}$. Therefore, there are $16$ possibilities for $y$ up to $B$-conjugacy and consequently $8$ possibilities up to $H_0$-conjugacy. Therefore, $|y^T \cap H_0| = 8|b^{H_0}|$. In addition, $|C_T(z)| = |\GU_1(8)||\GU_3(8)| = |C_{H_0}(b)|$, so by Lemma~\ref{lem:fpr_subgroups}, we deduce that $y$ is contained in $8$ $T$-conjugates of $H_0$ and consequently $8$ $G$-conjugates of $H$. Since there are two $G$-classes of subgroups of type $\GU_4(8)$, we conclude that $t\th$ is contained in at most $16$ subgroups of $G$ of type $\GU_4(8)$, as claimed. Finally, there are $5$ further $G$-classes of maximal irreducible subgroups, so
\[
P(x,t\th) < \frac{4}{q^2} + \frac{1}{q^3} + (1+3+9+5(2+1)(2^3+1))\cdot\frac{3}{8^{15/4}} + 16\cdot\frac{2}{8^{12/5}} < \frac{1}{2}.
\]

It remains to assume that $e=2$. If $q_0 = 2$, then Proposition~\ref{prop:o_computation} implies that $u(G) \geq 2$, so we can assume that $q_0 \geq 3$. A power of $y$ is $[\l, \l^{-1}] \perp [\mu,\mu^q,\mu^{q^2},\mu^{-1},\mu^{-q},\mu^{-q^2}]$ with respect to $V = (U \oplus U^*) \perp (W \oplus W^*)$, where $|\l| > 2$ and $|\mu|$ is a primitive divisor of $q^3-1$. Therefore, by Lemma~\ref{lem:c1}, $U \perp W$, $U \perp W^*$, $U^* \perp W$ and $U^* \perp W^*$ are the only totally singular subspaces stabilised by $y$, so $y$ is contained in exactly two subgroups of $G$ of type $\GL_4(q)$. Moreover, a power of $y$ has type $I_2 \perp (6)^-_{q_0}$, so Corollary~\ref{cor:c3_subfield}, $y$ is not contained in any subgroups of type $\GU_4(q)$. Therefore,
\[
P(x,t\th) < \frac{4}{q^2} + \frac{1}{q^3} + (14+\log\log{q})(q_0+1)(q_0^3+1) \cdot \frac{3}{q^{15/4}} + 2\cdot \frac{2}{q^{12/5}} \to 0
\]
as $q \to \infty$ and $P(x,t\th) < \frac{1}{2}$, unless $q=3^2$. Let $q=3^2$. In this case, $|y|$ is divisible by $7$, the unique primitive prime divisor of $3^6-1$, and the only types of irreducible maximal subgroup of $G$ with order divisible by $7$ are $\GL_4(9)$ (2 classes), $\O^+_8(3)$ (4 classes), $\O^-_8(3)$ and $\Omega_7$. We know that $t\th$ is contained in at most subgroups of type $\GL_4(9)$ and Lemma~\ref{lem:shintani_subfield} implies that $t\th$ is contained in at most $4$ subgroups of type $\O^+_8(3)$. Thus we conclude that
\[
P(x,t\th) < \frac{4}{9^2} + \frac{1}{9^3} + (4+3(3+1)(3^3+1)) \cdot \frac{3}{9^{15/4}} + 2\cdot \frac{2}{9^{12/5}} < \frac{1}{2}.
\] 
This completes the proof.
\end{proof}

\clearpage
\subsection{Case I(b)}\label{ss:o_Ib}

For Case~I(b), we cannot select an element $t\th \in T\th$ by directly considering a Shintani map as we did in Case~I(a). Indeed, this is precisely the reason for the distinction between Cases~I(a) and~I(b). Nevertheless, we can use Shintani descent indirectly to select appropriate elements in $T\th$ via Lemma~\ref{lem:shintani_substitute} (see Example~\ref{ex:shintani_substitute}). \vspace{7pt}

\begin{shbox}
\begin{notationx} \label{not:o_Ib}
\begin{enumerate}[label={}, leftmargin=0cm, itemsep=3pt]
\item Write $q=p^f$ where $f \geq 2$. Let $V = \F_q^{2m}$.
\item Fix the simple algebraic group
\[
X = \left\{ 
\begin{array}{ll}
\Omega_{2m}(\FF_2) & \text{if $p=2$}       \\
\PSO_{2m}(\FF_p)   & \text{if $p$ is odd.} \\
\end{array}
\right.
\]
\item Fix the standard Frobenius endomorphism $\p = \p_{\B^+}$ of $X$, defined with respect to the standard basis $\B^+$, as $(a_{ij}) \mapsto (a_{ij}^p)$, modulo scalars.
\item With respect to the $\B^+$, write $V_E = \<e_1,\dots,e_{m-1}\>$ and $V_F = \<f_1,\dots,f_{m-1}\>$. With respect to the decomposition 
\[
V = (V_E \oplus V_F) \perp \<e_m,f_m\>
\]
recall that $r = I_{2m-2} \perp r^+$ and $\d = \d^+ = (\b I_{m-1} \oplus I_{m-1}) \perp [\b,1]$, where, in the latter case $q$ is odd and $\b \in \F_q^\times$ has order $(q-1)_2$.
\item Fix $Z_1 = X_{(\<e_m,f_m\>)} \cong \SO_{2m-2}(\FF_p)$ and $Z_2 = (Z_1)_{(V_E \oplus V_F)} \cong \GL_{m-1}(\FF_p)$, so $Z_1$ acts trivially on $\<e_m,f_m\>$ and $Z_2 \leq Z_1$ centralises $V_E \oplus V_F$.
\end{enumerate}
\end{notationx}
\end{shbox} \vspace{7pt}

By Proposition~\ref{prop:o_cases}, we may, and will, assume $\th \in \PGO^+_{2m}(q)\p^i$ if $\e=+$ and $\th \in \PGO^-_{2m}(q)\psi^i$ if $\e=-$. \vspace{7pt}

\begin{shbox}
\notacont{\ref{not:o_Ib}} Write $q=q_0^e$ and $e=f/i$.
\begin{enumerate}[label={}, leftmargin=0cm, itemsep=3pt]
\item Fix $(\sigma,\rho,d,Z)$ as follows, where $\Delta = \d \d^{\s^{-1}} \d^{\s^{-2}} \cdots \d^{\s^{-(e-1)}}$ 
\begin{center}
{\renewcommand{\arraystretch}{1.1}
\begin{tabular}{cccccc}
\hline  
$\e$ & $\th$        & $\s$       & $\rho$         & $d$          & $Z$   \\
\hline
$+$ & $r\p^i$       & $r\p^i$    & $r$            & $2$          & $Z_1$ \\
    & $\d^- r\p^i$  & $\d r\p^i$ & $r\Delta^{-1}$ & $2(q_0-1)_2$ & $Z_2$ \\
$-$ & $\psi^i$      & $\p^i$     & $r$            & $2$          & $Z_1$ \\
    & $\d^- \psi^i$ & $\d\p^i$   & $r\Delta^{-1}$ & $2(q_0-1)_2$ & $Z_2$ \\
\hline
\end{tabular}}
\end{center}
\end{enumerate}
\end{shbox} \vspace{5pt}

\begin{remarkx}\label{rem:not_o_Ib}
Let us comment on Notation~\ref{not:o_Ib}.
\begin{enumerate}
\item Note that $Z_1$ and $Z_2$ are connected $\p$-stable subgroups of $X$.
\item We have $Z_1 \leq C_X(r)$ since the map $r$ is supported on $\<e_m,f_m\>$.
\item If $q$ is odd, then $Z_2 \leq C_{Z_1}(\d|_{V_E \oplus V_F})$ since $\d|_{V_E \oplus V_F}$ centralises the decomposition $V_E \oplus V_F$ and acts as a scalar on each summand.
\item The automorphisms $\psi$ and $\d^-$ of $\POm^-_{2m}(q)$, where $q$ is odd in the latter case, were introduced in \eqref{eq:psi} and Definition~\ref{def:delta_minus}.
\item Write $\ws = \s|_{X_{\rho\s^e}}$ and $\wrho = \rho|_{X_{\rho\s^e}}$. Observe that $X_{\rho\s^e}\ws = \PDO^\e_{2m}(q)\th$, noting that when $\e=-$ we are making the usual identifications justified by the isomorphism $\Psi\:X_{r\p^f} \to \PDO^-_{2m}(q)$ given in Lemma~\ref{lem:algebraic_finite_minus} (see Remark~\ref{rem:not_o_Ia}(iii)).
\end{enumerate}
\end{remarkx}

\clearpage

We now choose the elements for Case~I(b) in the following proposition (see Remark~\ref{rem:o_cases} for an explanation of the statement and Table~\ref{tab:o_Ib_elt}).

\begin{table}
\centering
\caption{Case~I(b): The element $y$ for the automorphism $\th$} \label{tab:o_Ib_elt}
{\renewcommand{\arraystretch}{1.2}
\begin{tabular}{ccc}
\hline    
\multicolumn{3}{c}{Generic case} \\
\hline  
\multicolumn{2}{c}{$m \mod{4}$} & $y$ \\
\hline
\multicolumn{2}{c}{$0$ or $2$} & ${\,}^a(2m-2)^+ \perp {\,}^ar^\e$                     \\
\multicolumn{2}{c}{$1$}        & ${\,}^a(m-3)^+ \perp {\,}^a(m+1)^+ \perp {\,}^ar^\e$  \\
\multicolumn{2}{c}{$3$}        & ${\,}^a(m-5)^+ \perp {\,}^a(m+3)^+ \perp {\,}^ar^\e$  \\
\hline \\[-2pt]
\hline
\multicolumn{3}{c}{Specific cases} \\
\hline 
$m$        & $\th$                 & $y$                                 \\
\hline
$5$ or $7$ & $r\p^i, \psi^i$       & $(4)^- \perp (2m-6)^- \perp r^\e$   \\
           & $\d r\p^i, \d\psi^i$  & $D_{2m-2}^+ \perp {\,}^\Delta r^\e$ \\
\hline
\end{tabular}}
\\[5pt]
{\small Note: we describe $y$ by its type over $\F_{q_0}$ and $D_{2m-2}^+$ is defined in Remark~\ref{rem:o_Ib_elt}(ii)}
\end{table}

\begin{propositionx}\label{prop:o_Ib_elt}
Let $T=\POm^\e_{2m}(q)$ and let $\th$ be an automorphism from Table~\ref{tab:o_cases} (in Case~I(iii) or~(v)). If $y$ is the element in Table~\ref{tab:o_Ib_elt}, then there exists $t \in T$ that centralises the decomposition $\<e_1,\dots,f_{m-1}\> \perp \<e_m,f_m\>$ such that $(t\th)^e$ is $X$-conjugate to $y$. Moreover, if $H \leq G$, then the number of $G$-conjugates of $H$ that contain $t\th$ is at most $|C_{\PDO^{-\e}_{2m}(q_0)}(y^d)|$.
\end{propositionx}

\begin{proof}
In each case, $(\r\s^e)^d = \s^{ed}$. For instance, if $\e=+$ and $\th = \d r\p^i$, then 
\[
(\r\s^e)^d = (r\Delta^{-1}\Delta(r\p^i)^e)^d = (\p^f)^{2(q_0-1)_2} = (\p^{2f})^{(q_0-1)_2}
\] 
and
\[
\s^{ed} = (\d r\p^i)^{ed} = (\Delta (r\p^i)^e)^d = (\Delta r\p^f)^d = (\Delta\Delta^r\p^{2f})^{(q_0-1)_2} = (\p^{2f})^{(q_0-1)_2}.
\]
It is also easy to verify that $y\wrho \in Z_\s$. Therefore, Lemma~\ref{lem:shintani_substitute} implies that there exists $g \in Z_{\s^e} \leq \PSO^\e_{2m}(q) \leq X_{\rho\s^e}$ such that $(g\ws)^e$ is $X$-conjugate (indeed $Z$-conjugate) to $y$ and if $H \leq G$, then the number of conjugates of $H$ that contain $g\ws$ is at most $|C_{\PDO^{-\e}_{2m}(q_0)}(y^d)|$.

If $q^m \not\equiv \e \mod{4}$, then $\PSO^\e_{2m}(q)=T$ and $\ws=\th$, so $g\ws \in T\th$, as required (see \eqref{eq:discriminant_condition}). Otherwise, $g \in \PSO^\e_{2m}(q) = T \cup T\rsq\rns$, so we may choose $\th \in \{ \ws, \rsq\rns\ws \}$ such that $g\ws \in T\th$, which proves the claim, by Lemmas~\ref{lem:o_out_plus_facts} and~\ref{lem:o_out_minus_facts}.
\end{proof}

\begin{remarkx}\label{rem:o_Ib_elt}
We comment on the definition of $t\th$ when $m \in \{5,7\}$.
\begin{enumerate}
\item Let $m=5$ and let $\th \in \{ r\p^i, \psi^i\}$. By Table~\ref{tab:o_Ib_elt}, $y= y_1 \perp y_2 \perp r^-$, centralising a decomposition $\F_{q_0}^{10} = U_1 \perp U_2 \perp U_3$, where $y_1$ and $y_2$ both have type ${\,}^\Delta(4)^-$. By \cite[Lemma~6.1]{ref:BambergPenttila08}, we can fix a primitive prime divisor $\ell$ of $q_0^4-1$ that is strictly greater than $5$. Let $\Lambda$ be the set of elements of order $\ell$ in $\F_{q_0}^\times$. Then $|\Lambda| \geq 8$, so we can, and will, assume that $y_1$ and $y_2$ have distinct sets of eigenvalues. This implies that $U_1$ and $U_2$ are nonisomorphic $\F_{q_0}\<y\>$-modules.
\item Let $q$ be odd and let $\th \in \{ \d r\p^i, \d\psi^i\}$. We need to define $D_{2m-2}^+$. We define $D_{2m-2}^+$ to be an element $\b A \perp A^{-\tr}$ where $A$ is an irreducible element, whose order is a primitive prime divisor of $q_0^{m-1}-1$. This is like, but not exactly the same as, an element of type ${\,}^\Delta (2m-2)^+$ (which does not exist when $m$ is odd).
\end{enumerate}
\end{remarkx}

Continue to let $T$ be the simple group $\POm^\e_{2m}(q)$ and let $\th$ be an automorphism from Table~\ref{tab:o_cases}. Fix $y$ from Table~\ref{tab:o_Ib_elt} and $t\th \in G = \<T,\th\>$ from Proposition~\ref{prop:o_Ib_elt}. The following result describes $\M(G,t\th)$.

\begin{theoremx}\label{thm:o_Ib_max}
The maximal subgroups of $G$ which contain $t\th$ are listed in Tables~\ref{tab:o_Ib_max} and~\ref{tab:o_Ib_max_57}, where $m(H)$ is an upper bound on the multiplicity of the subgroups of type $H$ in $\M(G,t\th)$.
\end{theoremx}

\begin{table}
\centering
\caption{Case~I(b): Description of $\M(G,t\th)$ for $m \not\in \{5,7\}$}\label{tab:o_Ib_max}
{\renewcommand{\arraystretch}{1.2}
\begin{tabular}{clccc}
\hline   
       & \multicolumn{1}{c}{type of $H$}              & $m(H)$  & \multicolumn{2}{c}{conditions}  \\
\hline
$\C_1$ &                                              &         & $m \mod{4}$ & $q$               \\
\hline
       & $\O^\up_2(q) \times \O^{\e\up}_{2m-2}(q)$    & $1$     &             &                   \\
       & $\Sp_{2m-2}(q)$                              & $1$     &             & even              \\
       & $\O_{2m-1}(q)$                               & $2$     &             & odd               \\[5.5pt]
       & $P_{m-1}$                                    & $2$     & even        &                   \\
       &                                              & $4$     & odd         &                   \\[5.5pt]
       & $\O^\up_{m-3}(q) \times \O^{\e\up}_{m+3}(q)$ & $1$     & $1$         &                   \\
       & $\O_{m-2} \times \O_{m+2}$                   & $2$     & $1$         & odd               \\
       & $\O^\up_{m-1}(q) \times \O^{\e\up}_{m+1}(q)$ & $1$     & $1$         &                   \\[5.5pt]
       & $P_{(m-3)/2}$                                & $2$     & $1$         &                   \\
       & $P_{(m+1)/2}$                                & $2$     & $1$         &                   \\[5.5pt]
       & $\O^\up_{m-5}(q) \times \O^{\e\up}_{m+5}(q)$ & $1$     & $3$         &                   \\
       & $\O_{m-4} \times \O_{m+4}$                   & $2$     & $3$         & odd               \\
       & $\O^\up_{m-3}(q) \times \O^{\e\up}_{m+3}(q)$ & $1$     & $3$         &                   \\[5.5pt]
       & $P_{(m-5)/2}$                                & $2$     & $3$         &                   \\
       & $P_{(m+3)/2}$                                & $2$     & $3$         &                   \\
\hline       
$\C_2$ & $\O^{\up}_{2m/k}(q) \wr S_k$                 & $N$     & \multicolumn{2}{c}{$k \div m$, \, $k > 1$, \, ${\up}^k \in \e$} \\
       & $\O_{2m/k}(q) \wr S_k$                       & $N$     & \multicolumn{2}{c}{$k \div 2m$, \, $2m/k > 1$ odd} \\
       & $\GL_m(q)$                                   & $N$     & \multicolumn{2}{c}{$m$ odd, \, $\e = +$} \\[5.5pt]
$\C_5$ & $\O^{\up}_{2m}(q^{1/k})$                     & $N$     & \multicolumn{2}{c}{$k \div f$, \, $k$ is prime, \, ${\up}^k=\e$}                        \\ 
\hline
\end{tabular}}
\\[5pt]
{\small Note: $N=|C_{\PDO^{-\e}_{2m}(q_0)}(y^2)|$ and in $\C_1$ there is a unique choice of $\up$}
\end{table}

\begin{table}
\centering
\caption{Case~I(b): Description of $\M(G,t\th)$ for $m \in \{5,7\}$}\label{tab:o_Ib_max_57}
{\renewcommand{\arraystretch}{1.2}
\begin{tabular}{clcccc}
\hline
       & \multicolumn{1}{c}{type of $H$}            & $m(H)$  & \multicolumn{3}{c}{conditions}        \\
\hline 
$\C_1$ &                                            &         & $\th$                    & $m$ & $q$  \\
\hline
       & $\O^\up_2(q) \times \O^{\e\up}_{2m-2}(q)$  & $1$     &                          &     &      \\
       & $\O_{2m-1}(q)$                             & $2$     &                          &     & odd  \\
       & $\Sp_{2m-1}(q)$                            & $1$     &                          &     & even \\[5.5pt]
       & $P_{m-1}$                                  & $2$     & $\d r\p^i$ or $\d\psi^i$ &     &      \\[5.5pt]
       & $\O^\up_4(q) \times \O^{\e\up}_{2m-4}(q)$  & $1$     & $r\p^i$ or $\psi^i$      &     &      \\
       & $\O^\up_6(q) \times \O^{\e\up}_{2m-6}(q)$  & $1$     & $r\p^i$ or $\psi^i$      &     &      \\
       & $\O_5(q) \times \O_7(q)$                   & $2$     & $r\p^i$ or $\psi^i$      & $7$ & odd  \\ 
\hline 
$\C_2$ & $\O_2^-(q) \wr S_m$                        & $N$     & \multicolumn{3}{c}{$e$ is even, \, $\e=-$}                                     \\
       &                                            &         & \multicolumn{3}{c}{($e=2$ only if $m=5$ and $\th \in \{r\p^i, \psi^i\}$)}      \\      
       & $\O_m(q) \wr S_2$                          & $N$     & \multicolumn{3}{c}{$q$ is odd}                                                 \\[5.5pt]
$\C_3$ & $\O_m(q^2)$                                & $N$     & \multicolumn{3}{c}{$\th \in \{\d r\p^i, \d\psi^i\}$, \, $e$ is odd}            \\
       & $\GU_m(q)$                                 & $N$     & \multicolumn{3}{c}{$\th \in \{\d r\p^i, \d\psi^i\}$, \, $e$ is odd, \, $\e=-$} \\[5.5pt]
$\C_5$ & $\O^{\up}_{2m}(q^{1/k})$                   & $N$     & \multicolumn{3}{c}{$k \div f$, \, $k$ is prime, \, ${\up}^k=\e$}               \\ 
\hline      
\end{tabular}}
\\[5pt]
{\small Note: $N=|C_{\PDO^{-\e}_{2m}(q_0)}(y^2)|$}
\end{table}

Theorem~\ref{thm:o_Ib_max} will be proved in parts. As before, write $\widetilde{G} = \< X_{\s^e}, \ws \>$. We will make use of Proposition~\ref{prop:o_I_conjugacy} in this section. We begin with reducible subgroups.
 
\begin{propositionx}\label{prop:o_Ib_max_reducible}
Theorem~\ref{thm:o_Ib_max} is true for reducible subgroups.
\end{propositionx}

\begin{proof}
Let us divide this proof into four parts.

\emph{Part 1: Setup.} Let $\mathcal{D}$ be the decomposition 
\[
V = V_1 \perp V_2 \quad\text{where}\quad V_1 = \<e_1,\dots,f_{m-1}\> \quad\text{and}\quad V_2 = \<e_m,f_m\>.
\] 
Observe that $\th$ centralises $\mathcal{D}$, and write $\th_i = \th|_{V_i}$. By Proposition~\ref{prop:o_Ib_elt}, $t$ also centralises $\mathcal{D}$, so we may write $t\th = t_1\th_1 \perp t_2\th_2$ with respect to $\mathcal{D}$. Let us also write $y = y_1 \perp {\,}^ar^\e$. We begin by studying the $\<t_i\th_i\>$-invariant subspaces of $V_i$.

\emph{Part 2: Subspaces of $V_1$.} Let $U_1$ be a $\<t_1\th_1\>$-invariant subspace of $V_1$. We will apply Lemma~\ref{lem:shintani_substitute}(ii)(b). 

For the sake of exposition, let us assume that $m \geq 9$ and $m \equiv 1 \mod{4}$; the other cases are very similar and we comment on them below. In this case, the element $y_1$ has type ${\,}^a(m-3)^+_{q_0} \perp {\,}^a(m+1)^+_{q_0}$, where $a$ is empty or $\Delta$. Write $S = \<e_1,\dots,f_{m-1}\>_{\F_{q_0}}$. Then $y_1$ centralises a decomposition $S = (S_1 \oplus S_2) \perp (S_3 \oplus S_4)$, where the $S_i$ are pairwise nonisomorphic irreducible $\F_{q_0}\<y_1\>$\=/modules (here $\dim{S_1} = \dim{S_2} = \frac{m-3}{2}$ and $\dim{S_3} = \dim{S_4} = \frac{m+1}{2}$). Therefore, by Lemma~\ref{lem:c1}, the only $\<y_1\>$-invariant subspaces of $W$ are direct sums of $S_1$, $S_2$, $S_3$ and $S_4$. 

We now proceed as in the proof of Proposition~\ref{prop:o_Ia_max_reducible} (see that proof for more details), but we use Lemma~\ref{lem:shintani_substitute}(ii)(b) in place of Lemma~\ref{lem:shintani_descent_fix}. In particular, Lemma~\ref{lem:shintani_substitute}(ii)(b) establishes that the only possibilities for $U_1$ are direct sums of four pairwise nonisomorphic irreducible $\<t\th_1\>$\=/invariant subspaces $U_{1,1}$, $U_{1,2}$, $U_{1,3}$ and $U_{1,4}$ (where $\dim{U_{1,1}} = \dim{U_{1,2}} = \frac{m-3}{2}$ and $\dim{U_{1,3}} = \dim{U_{1,4}} = \frac{m+1}{2}$). Moreover, we can deduce that these subspaces are totally singular but $U_{1,1} \oplus U_{1,2}$ and $U_{1,3} \oplus U_{1,4}$ are nondegenerate. 

The other cases are very similar. In all cases $U_1$ is a direct sum of pairwise nonisomorphic irreducible $\F_{q}\<y_1\>$-submodules of dimension at least three. In particular, this implies that
\begin{equation} \label{eq:codimension}
\dim{V_1} - \dim{U_1} \not\in \{ 1, 2\}.
\end{equation}
  
\emph{Part 3: Subspaces of $V_2$.} Next let $U_2$ be a $\<t_2\th_2\>$-invariant subspace of $V_2$. Note that a power of $t_2\th_2$ is ${\,}^ar^\e$. Therefore, if $q$ is even, then Lemma~\ref{lem:elt_r_even} implies that there is at most one proper nonzero $\F_q\<t_2\th_2\>$-invariant subspace of $V_2$. Similarly, if $q$ is odd, then Lemma~\ref{lem:elt_r_odd} implies that there are at most two $\F_q\<t_2\th_2\>$-invariant proper nonzero subspaces of $V_2$.

\emph{Part 4: Subspaces of $V$.} Now let $U$ be a $\<t\th\>$-invariant subspace of $V$. Let $\pi_i\:U \to V_i$ be the projection map of $U$ onto $V_i$. Then $U_i=\pi_i(U)$ is a $\<t_i\th_i\>$-invariant subspace of $V_i$.

Suppose that $U_2 \neq 0$ and $U_2 \not\leq U$. We mimic the proof of Lemma~\ref{lem:goursat}. Let $W_i = U \cap U_i$. Let $u_1 \in U_1$ and let $u_2, v_2 \in U_2$ satisfy $u_1 + u_2 \in U$ and $u_1 + v_2 \in U$. Then $u_2-v_2 \in U$, so $u_2 - v_2 \in W_2$. Therefore, there is a well-defined function $L\:U_1 \to U_2/W_2$ where $L(u_1) = \{ u_2 \in U_2 \mid u_1+u_2 \in U\}$.

If $u_1, v_1 \in U_1$ and $u_2, v_2 \in U_2$ satisfy $u_1+u_2 \in U$ and $v_1 + v_2 \in U$, then for all $\l \in \FF_q$ we have $(u_1+u_2)+\l(v_1+v_2) = (u_1+\l v_1)+(u_2+\l v_2)$, so 
\[
L(u_1+\l v_1) = W + (u_2 + \l v_2) = L(u_1) + \l L(v_1).
\]
Therefore, $L$ is linear.

For $u_1 \in U_1$, $L(u_1) = W_2$ if and only if $u_1 \in U$, so $\ker{L} = W_1$. Since $U_2 \not\leq U$ we know that $U_2/W_2 \neq 0$. This implies that $\dim{W_1} = \dim{U_1} - \dim{U_2/W_2} \in \{ 2m-3, 2m-4 \}$. However, $W_1$ is a $\<t_1\th_1\>$-invariant subspace of $V_1$ and \eqref{eq:codimension} implies that $V_1$ does not have a $\<t_1\th_1\>$-invariant subspace of dimension $2m-3$ or $2m-4$, so we have obtained a contradiction. 

Therefore, either $U_2 = 0$ or $U_2 \leq U$. This implies that $U = U_1 \oplus U_2$, the possibilities for which follow from Parts~2 and~3. These exactly correspond to the subgroups given in Tables~\ref{tab:o_Ib_max} and~\ref{tab:o_Ib_max_57}.
\end{proof}

We now turn to irreducible subgroups.

\begin{propositionx}\label{prop:o_Ib_max_irreducible}
Theorem~\ref{thm:o_Ib_max} is true for irreducible subgroups.
\end{propositionx}

\begin{proof}
\emph{Case 1: $m \not\in \{5,7\}$.} By construction, a suitable power of $t\th$ is $X$-conjugate to $y$. We begin by demonstrating that we can fix a power $z$ of $y$ satisfying $|z|=2$ and $1 \leq \nu(z) \leq 2$.  If $(\e,\th) \in \{ (+, r\p^i), (-, \psi^i) \}$, then a power $z$ of $y$ has type $I_{2m-2} \perp r^\e$ and evidently $\nu(z)=1$. Otherwise $(\e,\th) \in \{ (+, \d r\p^i), (-, \d\psi^i) \}$ and raising $y^{(q-1)_2}$ to a suitable power gives an element of type $I_{2m-2} \perp -I_2$ and $\nu(z)=2$. 

Let $H \in \M(G,t\th)$ be irreducible. We proceed as in the proof of Proposition~\ref{prop:o_Ia_max_primitive}, using Theorem~\ref{thm:aschbacher}. In particular, let us quickly handle the cases that are essentially identical to those in that previous proof. Observe that $\C_6$ and $\C_8$ are empty, $z$ is not contained in an $\S$ family subgroup by \cite[Theorem~7.1]{ref:GuralnickSaxl03} and $\C_5$ subgroups have type $\O^{{\up}}_{2m}(q_1)$ where $q=q_1^k$ for a prime $k$ and a sign ${\up} \in \{+,-\}$ such that ${\up}^k = \e$.

The possible types of $\C_2$ subgroups are those given in Table~\ref{tab:o_Ib_max} (see \cite[Tables~3.5.E and~3.5.F]{ref:KleidmanLiebeck}). The restriction on $\GL_m(q)$ subgroups arises for the reason given in the proof of Proposition~\ref{prop:o_Ia_max_imprimitive} for $(\e,\eta)=(+,-)$.  

Consider $\C_3$ subgroups. In this case, $H$ is a field extension subgroup of type $\O^{\up}_{2m/k}(q^k)$ or $\GU_m(q)$. Write $H \cap T = B.k$. Lemma~\ref{lem:c3}(ii) implies that $z \in B$, and Lemma~\ref{lem:c3}(i) implies that $k = 2$ since $\nu(z) \leq 2$. Now let $w$ be a power of $y$ of type $(2d)^+ \perp I_{2m-2d}$ where $d \in \{ m-1, \frac{m+1}{2}, \frac{m+3}{2} \}$ is odd. Lemma~\ref{lem:c3}(ii) implies that $w \in B$ and Corollary~\ref{cor:c3_subfield} implies that $z \not\in B$ since $d$ is odd, which is a contradiction. Therefore, $H \not\in \C_3$.

For $\C_4$ subgroups, suppose that $H$ is the centraliser of a decomposition $V_1 \otimes V_2$ where $\dim{V_1} \geq \dim{V_2} > 1$. Since $z \in H$, we may write $z = z_1 \otimes z_2$. If $\nu(z) = 1$, then we have a contradiction to \cite[Lemma~3.7]{ref:LiebeckShalev99}. Otherwise $z = -I_2 \perp I_{2m-2}$ and we quickly deduce that $\e=+$, $H$ has type $\Sp_2(q) \otimes \Sp_m(q)$ and $\nu(z_1) = 1$, which is not possible. Therefore, $H \not\in \C_4$.

For $\C_7$ subgroups we may assume that $\e=+$. Suppose that $H = B.S_k$ is the stabiliser of a decomposition $U_1 \otimes U_2 \otimes \cdots \otimes U_k$. From the previous paragraph, $z \not \in B$. However, \cite[Lemma~5.7.2]{ref:BurnessGiudici16} implies that $z$ does not cyclically permute the $k$ factors, which is a contradiction. Therefore, $H \not\in \C_7$.

To complete the proof when $m \not\in \{5,7\}$, we note that the stated upper bounds on the multiplicities of nonsubspace subgroups follow from Lemma~\ref{lem:shintani_substitute}(ii)(a) and Propositions~\ref{prop:o_I_conjugacy} and~\ref{prop:o_Ib_elt}.

\emph{Case 2: $m \in \{5,7\}$.} Let $H \in \M(G,t\th)$ be irreducible. We proceed as in the previous case. In particular, note that a power $z$ of $y$ satisfies $\nu(z) \leq 2$, so by \cite[Theorem~7.1]{ref:GuralnickSaxl03} $H \not\in \S$. Therefore, $H$ is a geometric subgroup and by considering the possible types we see that it suffices to consider subgroups in $\C_2$, $\C_3$ and $\C_5$. The result is clear for $\C_5$ subgroups. Note also that the multiplicities, as usual, follow from Lemma~\ref{lem:shintani_substitute}(ii)(a) and Propositions~\ref{prop:o_I_conjugacy} and~\ref{prop:o_Ib_elt}.

First assume that $H$ has type $\O^\e_2(q) \wr S_m$ stabilising a decomposition $\mathcal{D}$ of $V$ into $m$ nondegenerate $2$-spaces. If $e$ is odd, then a power of $y$ has one of the following types: 
\[ 
I_2 \perp (4)^-_q \perp (2m-6)^-_q, \quad I_2 \perp (8)^+_q, \quad I_2 \perp (12)^+_q, \quad I_2 \perp (6)^+_q \perp (6)^+_q.
\] 
By \cite[Lemma~5.2.6]{ref:BurnessGiudici16}, $y$ must centralise $\mathcal{D}$, which is a contradiction, since elements of these types act irreducibly on a space of dimension strictly greater than 2. Therefore, $e$ is even. Now assume that $m = 7$ or $\th \in \{\d\rsq\p^i,\d\psi^i \}$. If $e=2$, then a power of $y$ has one of the following types
\[
I_6 \perp (4)^-_q \perp (4)^-_q, \quad I_2 \perp (4)^+_q \perp (4)^+_q, \quad I_2 \perp (6)^+_q \perp (6)^+_q,
\]
and again we obtain a contradiction.

Next assume that $\e=+$ and $H$ has type $\GL_m(q)$. Let $H$ be the stabiliser of the decomposition $V = V_1 \oplus V_2$, where $V_1$ and $V_2$ are maximal totally singular subspaces of $V$. Record that $e$ is odd since $\e=+$. If $\th \in \{ r\p^i,\psi^i \}$, then a power of $y$ has type $I_2 \perp (4)^-_q \perp (2m-6)^-_q$, noting that $2m-6 \in \{4,8\}$, so $y$ has odd order and does not stabilise a maximal totally singular subspace, which is a contradiction. Now assume that $\th \in \{ \d\rsq\p^i, \d\psi^i \}$. In this case, $y$ has type ${\,}^\Delta r \perp {\,}^\Delta(2m-2)^+_{q_0}$. Therefore, $y$ has type $M \perp (8)^+_q$ or $M \perp (6)^+_q \perp (6)^+_q$, depending on whether $m$ is $5$ or $7$, where $M$ acts irreducibly on a $2$-space (see Lemma~\ref{lem:elt_r_odd}). Now $y^2$ centralises the decomposition and we may assume that $U \subseteq V_1$, where $U$ is a totally singular subspace of dimension $4$ or $3$ that is stabilised by $y^2$ and on which $y^2$ acts irreducibly. However, $U$ is stabilised by $y$, so $y$ stabilises $V_1$ and hence centralises the decomposition. However, since $M$ is irreducible, $y$ does not stabilise a maximal totally singular subspace, which is a contradiction. Therefore, $t\th$ is not contained in a subgroup of type $\GL_m(q)$.

Now we may assume that $H$ is a $\C_3$ subgroup. If $\th \in \{ r\p^i,\psi^i \}$, then a power $z$ of $y$ satisfies $\nu(z) = 1$, so $y$ is not contained in $H$ (see Lemma~\ref{lem:c3}). Now assume $\th \in \{ \d r\p^i, \d\psi^i \}$ and $H$ has type $\O_m(q)$ or $\GU_m(q)$. Note that $\e=-$ in the latter case (see \cite[Table~3.5.E]{ref:KleidmanLiebeck}). Since $y$ has type ${\,}^\Delta(2m-2)^+_{q_0} \perp {\,}^\Delta r_{q_0}$, $y$ has exactly two eigenvalues, $\l$ and $-\l$, of order $2(q_0+1)_2$. Lemma~\ref{lem:c3} implies that $y$ arises from an element $g \in \CU_m(q^2)$ or $\GO_m(q^2)$ with exactly one eigenvalue of order $2(q_0+1)_2$. Therefore, $\l^q = -\l$, so $e$ is odd. This completes the proof.
\end{proof}

We have now proved Theorem~\ref{thm:o_Ib_max} and are, consequently, in the position to prove Theorems~\ref{thm:o_main} and~\ref{thm:o_asymptotic} in Case~I(b). We consider two cases depending on whether $m \in \{5,7\}$.

\begin{propositionx}\label{prop:o_Ib}
Let $G = \<T, \th\> \in \A$ where $T=\POm^\e_{2m}(q)$ with $m \not\in \{5,7\}$. In Case~I(b), $u(G) \geq 2$ and as $q \to \infty$ we have $u(G) \to \infty$. 
\end{propositionx}

\begin{proof}
Let $x \in G$ have prime order. Theorem~\ref{thm:o_Ib_max} gives a superset of $\M(G,t\th)$. Using the fixed point ratios from Theorem~\ref{thm:fpr_s} and Proposition~\ref{prop:fpr_ns_o}(i), we will prove that $P(x,t\th) < \frac{1}{2}$ and $P(x,t\th) \to 0$ as $q \to \infty$. For brevity, we will not explicitly note that $P(x,t\th) \to 0$ as $q \to \infty$ separately in each case. Write $d(n)$ for the number of proper divisors of $n$.

\emph{Case~1: $m$ is even.} In this case,
\[
P(x,t\th) \leq \frac{(2,q-1)}{q} + \frac{1}{q^2} + \frac{20}{q^{m-2}} + (1+\log\log{q} + 2d(2m)) \cdot (q_0+1)(q_0^{m-1}-1) \cdot \frac{2}{q^{m-2}},
\]
so $P(x,t\th) < \frac{1}{2}$ unless either $(m,q) \in \{(4,8), (4,27), (6,8)\}$, or $e=f=2$ and $m \leq 10$. 

Consider the former case. The unique type of $\C_5$ subgroup is $\O^\e_{2m}(p)$. First assume $m=6$ and $q=8$, then a suitable power $z$ of $y$ has type $10^+_2 \perp I_2 = 10^+_8 \perp I_2$, which has odd prime order and acts irreducibly on a totally singular $5$-space. This implies that $z$, and hence $t\th$, is not contained in a $\C_2$ subgroup. Therefore, in this case,
\[
P(x,t\th) \leq \frac{1}{8} + \frac{1}{8^2} + \frac{20}{8^4} + (2+1)(2^5-1) \cdot \frac{2}{8^4} < \frac{1}{2}.
\]
Next assume that $m=4$ and $q \in \{8,27\}$. For now assume that $\e=-$. The subgroups of type $\O^-_8(p)$ are the only nonsubspace subgroups containing $t\th$. By Proposition~\ref{prop:fpr_ns_o}, for subgroups $H$ of this type we have $\fpr(x,G/H) < 3/q^3$ provided that $\nu(x) \neq 1$ and a direct calculation demonstrates that this bound also holds when $\nu(x)=1$ in this case. With this, together with better bounds extracted from Theorem~\ref{thm:fpr_s}, we obtain
\[
P(x,t\th) \leq \frac{(2,q-1)}{q} + \frac{9}{q^2} + \frac{14}{q^3} + (q_0+1)(q_0^3-1) \cdot \frac{3}{q^3} < \frac{1}{2}.
\] 
If $\e=+$, then we must also take into account the $\C_2$ subgroups. Here $y = y_1 \perp y_2 \perp y_3 \perp y_4$, centralising $\F_q^8 = U_1 \perp U_2 \perp U_3 \perp U_4$, where $y$ acts on pairwise nonisomorphically on the $U_i$ and each $y_i$ acts on $U_i$ by centralising the decomposition into two totally singular $1$-spaces, acting nontrivially on both. In particular, $y$ is not contained in any imprimitive subgroups of type $\O^-_2(q) \wr S_4$ or $\O^-_4(q) \wr S_2$ and is contained in at most one subgroup of type $\O^+_2(q) \wr S_4$ and at most three of type $\O^+_4(q) \wr S_2$. Therefore,
\[
P(x,t\th) \leq \frac{(2,q-1)}{q} + \frac{9}{q^2} + \frac{14}{q^3} + (q_0+1)(q_0^3-1) \cdot \frac{3}{q^3} + 4 \cdot \frac{2}{q^{12/5}} < \frac{1}{2}.
\]

Now assume that $e=f=2$ and $m \in \{4,6,8,10\}$. Here $\e=-$ since $e$ is even. Therefore, since $f=2$, $G$ has no $\C_5$ subgroups. We will now show that $t\th$ is not contained in any $\C_2$ subgroups. Note that $D(Q) = \nonsquare$ since $q^m \equiv 1 \mod{4}$, so any $\C_2$ subgroup has type $\O^-_{2m/k}(q) \wr S_k$ where $k$ is odd and $2m/k$ is even (see \cite[Table~3.5.F]{ref:KleidmanLiebeck}). If $m \in \{4,8\}$, then no such subgroups arise. Now assume that $m \in \{6,10\}$. The unique possible type of $\C_2$ subgroup is $\O^-_4(q) \wr S_{m/2}$. A power $z$ of $y$ has type $(2m-2)^+_{q_0} \perp I_2 = (2m-2)^+_{q} \perp I_2$ since $e=2$ and $m-1$ is odd (see Lemma~\ref{lem:elt_splitting}). By \cite[Lemma~5.2.6]{ref:BurnessGiudici16}, $z$ must centralise a decomposition $U_1 \perp \cdots \perp U_{m/2}$ where $\dim{U_i} = 4$, which is impossible since $y$ acts irreducibly on a totally singular subspace of dimension $m-1 \geq 5$. Therefore, $t\th$ is contained in no nonsubspace subgroups. Accordingly,
\[
P(x,t\th) \leq \frac{(2,q-1)}{q} + \frac{1}{q^2} + \frac{20}{q^{m-2}},
\]
so $P(x,t\th) < \frac{1}{2}$ unless $(m,q) = (4,4)$. If $T = \Omega^-_8(4)$, Proposition~\ref{prop:o_computation} implies that $u(G) \geq 2$.

\emph{Case~2: $m$ is odd.} If $m \equiv 1 \mod{4}$ and $m \geq 9$, then
\begin{align*}
P(x,t&\th) \leq \frac{(2,q-1)}{q} + \frac{1}{q^2} + \frac{2}{q^{(m-3)/2}} + \frac{6}{q^{(m-1)/2}} + \frac{56}{q^{m-3}} \\ &+ (1+\log\log{q} + 2d(2m) + q) \cdot (q_0+1)(q_0^{(m-3)/2}-1)(q_0^{(m+1)/2}-1) \cdot \frac{2}{q^{m-2}},
\end{align*}
which proves that $P(x,t\th) < \frac{1}{2}$ unless $(m,q) = (9,4)$. In this exceptional case, $\e = -$ since $e$ is even, so the only nonsubspace subgroup to arise has type $\O^-_2(q) \wr S_9$, so
\[
P(x,t\th) \leq \frac{1}{4} + \frac{1}{4^2} + \frac{2}{q^3} + \frac{6}{4^4} + \frac{56}{4^6} + (2+1)(2^3-1)(2^5-1) \cdot \frac{2}{4^7} < \frac{1}{2}.
\] 
If $m \equiv 3 \mod{4}$ and $m \geq 11$, then
\begin{align*}
P(x,t&\th) \leq \frac{(2,q-1)}{q} + \frac{1}{q^2} + \frac{2}{q^{(m-5)/2}} + \frac{6}{q^{(m+1)/2}} + \frac{56}{q^{m-5}} \\ &+ (1+\log\log{q} + 2d(2m) + q) \cdot ( q_0+1)(q_0^{(m-5)/2}-1)(q_0^{(m+3)/2}-1) \cdot \frac{2}{q^{m-2}},
\end{align*}
which proves that $P(x,t\th) < \frac{1}{2}$ unless $(m,q) = (11,4)$. In this case, as above, $\e = -$, the only type of nonsubspace subgroup to occur is $\O^-_2(q) \wr S_{11}$ and adjusting the bound accordingly demonstrates that $P(x,t\th) < \frac{1}{2}$. This completes the proof.
\end{proof}

\begin{propositionx}\label{prop:o_Ib_57}
Let $G = \<T, \th\> \in \A$ where $T=\POm^\e_{2m}(q)$ with $m \in \{5,7\}$. In Case~I(b), $u(G) \geq 2$ and as $q \to \infty$ we have $u(G) \to \infty$. 
\end{propositionx}

\begin{proof}
Let $x \in G$ have prime order. We proceed as in the previous proof. Theorem~\ref{thm:o_Ib_max} gives a superset of $\M(G,t\th)$,  Theorem~\ref{thm:fpr_s} and Proposition~\ref{prop:fpr_ns_o} give bounds on the associated fixed point ratios, and we will use this information to prove that $P(x,t\th) < \frac{1}{2}$ and $P(x,\th) \to 0$ as $q \to \infty$.

\emph{Case~1: $\th \in \{ \d r\p^i, \d\psi^i \}$.} In this case $q$ is odd and 
\[
P(x,t\th) \leq \frac{2}{q} + \frac{1}{q^2} + \frac{10}{q^{m-2}} + \frac{10}{q^{m-1}} + (3+q+M)\cdot(q_0+1)(q_0^{m-1}-1)\cdot\frac{2}{q^{m-2}},
\]
where $M$ is the number of types of subfield subgroups. Notice that
\[
M \leq \left\{
\begin{array}{ll}
0               & \text{if $f$ is a power of $2$}     \\
1               & \text{if $f$ is an odd prime power} \\
1+\log{\log{q}} & \text{otherwise}                    \\
\end{array}
\right.
\]
where in the first case $\e=-$ since $e$ is even. With this bound on $M$ we see that $P(x,t\th) < \frac{1}{2}$ unless $(m,q) \in \{ (7,3^2), (7,5^2)\}$, or $m=5$ and either $f=e=3$ or $e=2$. If $(m,q) \in \{ (7,3^2), (7,5^2)\}$, then $t\th$ is contained in no $\C_3$ or $\C_5$ subgroups; adjusting the bound on $P(x,t\th)$ accordingly proves that $P(x,t\th) < \frac{1}{2}$. 

Next assume that $m=5$ and $f=e=3$. If $\e=+$, then there are no subgroups of type $\GU_m(q)$, so
\[
P(x,t\th) \leq \frac{2}{q} + \frac{1}{q^2} + \frac{10}{q^3} + \frac{10}{q^4} + 4\cdot(q_0+1)(q_0^4-1)\cdot\frac{2}{q^3} < \frac{1}{2}.
\]
Therefore, assume that $\e=-$. If $x \not\in \PGL(V)$ or $\nu(x) \geq 2$, then by Proposition~\ref{prop:fpr_ns_o}(ii)
\[
P(x,t\th) \leq \frac{2}{q} + \frac{1}{q^2} + \frac{10}{q^3} + \frac{10}{q^4} + (4+q^2)\cdot(q_0+1)(q_0^4-1)\cdot\frac{3}{q^5} < \frac{1}{2},
\]
while if $x \in \PGL(V)$ and $\nu(x)=1$, then $\fpr(x,G/H)=0$ for $\C_3$ subgroups $H$ (see Lemma~\ref{lem:c3}) and
\[
P(x,t\th) \leq \frac{2}{q} + \frac{1}{q^2} + \frac{10}{q^3} + \frac{10}{q^4} + 3\cdot(q_0+1)(q_0^4-1)\cdot\frac{2}{q^3} < \frac{1}{2}.
\]

Now assume that $m=5$ and $e=2$. In this case, the only type of nonsubspace subgroup to arise is $\O_5(q) \wr S_2$. We will now bound the number of subgroups of this type that contain $t\th$. Note that a suitable power $z$ of $y$ has type
\[
I_2 \perp (8)^+_{q_0} = I_2 \perp (4)^+_q \perp (4)^+_q.
\]
Let $E$ be the $1$-eigenspace of $z$. Then $z$ stabilises $q-1$ nondegenerate subspaces of $E$ and consequently stablises exactly $2(q-1)$ nondegenerate $5$-spaces of $V$ (see Lemma~\ref{lem:goursat}). Therefore, $z$ is contained in at most $q-1$ subgroups of type $\O_5(q) \wr S_2$, and thus
\[
P(x,t\th) < \frac{2}{q} + \frac{1}{q^2} + \frac{10}{q^3} + \frac{10}{q^4} + (1+3(q-1))\cdot\frac{2}{q^3} < \frac{1}{2}.
\]

\emph{Case~2: $\th \in \{ r\p^i, \psi^i \}$.} If $q$ is even, then
\begin{align*}
P(x,t\th) \leq \frac{1}{q} &+ \frac{1}{q^2} + \frac{1}{q^4} + \frac{1}{q^{(m-1)/2}} + \frac{1}{q^{m-3}} + \frac{9}{q^{m-2}} + \frac{6}{q^{m-1}} \\
                           &+ (2 + \log{\log{q}})\cdot(q_0+1)(q_0^2+1)(q_0^{m-3}+1)\cdot\frac{2}{q^{m-2}},
\end{align*} 
and if $q$ is odd, then
\begin{align*}
P(x,t\th) \leq \frac{2}{q} &+ \frac{1}{q^2} + \frac{1}{q^4} + \frac{2}{q^5} +  \frac{1}{q^{(m-1)/2}} + \frac{3}{q^{m-3}} + \frac{15}{q^{m-2}} + \frac{10}{q^{m-1}} \\
                           &+ (3 + \log{\log{q}})\cdot(q_0+1)(q_0^2+1)(q_0^{m-3}+1)\cdot\frac{2}{q^{m-2}}.
\end{align*}
This proves that $P(x,t\th) < \frac{1}{2}$ unless $(m,q) = (5,8)$ or $e=2$. If $(m,q) = (5,8)$, then there is a unique type of subfield subgroups and $t\th$ is not contained in a subgroup of type $\O^\e_2(q) \wr S_5$; adjusting the bound accordingly gives $P(x,t\th) < \frac{1}{2}$. 

Finally assume that $e=2$. In this case $\e=-$ and no subfield subgroups arise. If $m=7$, then $t\th$ is not contained in a subgroup of type $\O^-_2(q) \wr S_7$, and adjusting the bound above accordingly, proves that $P(x,t\th) < \frac{1}{2}$. If $m=5$, then $y$ has type 
\[
(4)^-_{q_0} \perp (4)^-_{q_0} \perp r^- = (2)^-_q \perp (2)^-_q \perp (2)^-_q \perp (2)^-_q \perp r^-,
\] 
so $y$ is contained in a unique $\C_2$ subgroup of type $\O^\e_2(q) \wr S_5$. Therefore, if $q$ is even, then
\[
P(x,t\th) \leq \frac{1}{q} + \frac{3}{q^2} + \frac{9}{q^3} + \frac{7}{q^4} + \frac{2}{q^3} < \frac{1}{2}.
\]
Now assume that $q$ is odd. Let $H$ be a subgroup of type $\O_5(q) \wr S_2$ stabilising a decomposition $V_1 \perp V_2$. Now $y^2$ centralises the decomposition and we may assume that $U \subseteq V_1$, where $U$ is one of the $2$-spaces $y^2$ stabilises and on which $y$ acts irreducibly. However, $U$ is stabilised by $y$, so $y$ stabilises $V_1$ and hence centralises the decomposition. However, by considering the number of choices for the stabilised $5$-space containing the $1$-eigenspace of $y$, we see that $y$ is contained in at most $\binom{4}{2}=6$ subgroups of type $\O_5(q) \wr S_2$. Therefore,
\[
P(x,t\th) \leq \frac{2}{q} + \frac{3}{q^2} + \frac{1}{q^3} + \frac{4}{q^4} + \frac{15}{q^8} + 7 \cdot \frac{2}{q^{m-2}} < \frac{1}{2}.
\]
This completes the proof.
\end{proof}

\clearpage
\section{Case II: linear automorphisms} \label{s:o_II}

In this section, we begin with Case~II. Accordingly, write $G=\<T,\th\>$ where $T = \POm^\e_{2m}(q)$ for $m \geq 4$ and $\e \in \{+,-\}$ and where $\th \in \PGO^\e_{2m}(q)$. Recall the cases
\begin{enumerate}[(a)]
\item $G \leq \PDO^\e_{2m}(q)$
\item $G \not\leq \PDO^\e_{2m}(q)$.
\end{enumerate}
We will consider Cases~II(a) and~II(b) in Sections~\ref{ss:o_IIa} and~\ref{ss:o_IIb}, respectively.

\subsection{Case II(a)} \label{ss:o_IIa}

Let $m \geq 4$ and $\e \in \{+,-\}$. In this section, we focus on the groups $\POm^\e_{2m}(q) \leq G \leq \PDO^\e_{2m}(q)$ and prove Theorems~\ref{thm:o_main} and \ref{thm:o_asymptotic} in Case~II(a). In \cite{ref:BreuerGuralnickKantor08}, Breuer, Guralnick and Kantor proved that $s(T) \geq 2$. As they point out \cite[p.447]{ref:BreuerGuralnickKantor08}, their proofs, in fact, prove that $s(G) \geq 2$. The following result is motivated by this comment (see \cite[Theorem~3.1]{ref:BurnessGuest13} for a similar argument).

\begin{propositionx}\label{prop:o_IIa}
Let $G \in \A$. In Case~II(a), $u(G) \geq 2$ and as $q \to \infty$ we have $u(G) \to \infty$.
\end{propositionx}

\begin{proof}
If $G = T$, then the result follows from \cite[Propositions~5.13--5.18]{ref:BreuerGuralnickKantor08} (the fact that $u(G) \to \infty$ as $q \to \infty$ is evident from the proofs). Now assume that $q$ is odd and $\th \in \{ \rsq\rns, \d \}$. In the proofs of \cite[Propositions~5.13--5.18]{ref:BreuerGuralnickKantor08}, it is shown that for all prime order elements $x \in T$, we have that $P(x,s) < \frac{1}{2}$ and $P(x,s) \to 0$ as $q \to \infty$, for a suitable semisimple element $s \in T$. In each case, by Lemmas~\ref{lem:irreducible_minus} and \ref{lem:irreducible_minus_similarity}, there exists $g \in T\th$ such that a suitable power of $g$ is $s$. It is straightforward to verify that for all $x \in G$, we also have $P(x,g) < \frac{1}{2}$ and $P(x,g) \to 0$ and $q \to \infty$ and consequently $u(G) \geq 2$ and $u(G) \to \infty$ as $q \to \infty$. We give the details when $\e=+$ and $m \geq 7$ is odd with $\th \in \{ \d, \rsq\rns\d \}$. The other cases are similar.

Assume that $\e = +$, $m \geq 7$ is odd and $\th \in \{ \d, \rsq\rns\d \}$. Let $V = \F_q^{2m}$ be the natural module for $T$. By Lemma~\ref{lem:irreducible_minus_similarity}, there exists $x = x_1 \perp x_2 \in \DO^+_{2m}(q)$ centralising $V_1 \perp V_2$, where $V_1$ and $V_2$ are nondegenerate subspaces of dimensions $m-1$ and $m+1$, $x_1$ has order $(q-1)(q^{(m-1)/2}+1)$ acting irreducibly on $V_1$, $x_2$ has order $(q-1)(q^{(m+1)/2}+1)$ acting irreducibly on $V_2$ and $\t(x_1) = \t(x_2) = \a$ (where $\F_q^\times = \< \a \>$). Since $\tau(x) = \a \not\in (\F_q^\times)^2$, $g = xZ(\DO^+_{2m}(q)) \in \PDO^+_{2m}(q) \setminus \PSO^+_{2m}(q)$. Consequently, $g \in T\d \cup T\rsq\rns\d$, but $\ddot{r}_\square\ddot{r}_\nonsquare\ddot{\d}$ are $\Out(T)$-conjugate.

The order of $g$ is divisible a primitive prime divisor $\ell$ of $q^{m+1}-1$, which by \cite[Lemma~6.1]{ref:BambergPenttila08} we may assume satisfies $\ell > 2m+3$. Therefore, by \cite[Theorem~2.2]{ref:GuralnickMalle12JAMS}, all of the subgroups in $\M(G,g)$ are reducible, subfield or field extension subgroups. Since $m+1 > m$ and $(m+1,m) = 1$, the prime $\ell$ does not divide the order of any subfield or field extension subgroup of $G$. Therefore, we conclude that $\M(G,g)$ contains only reducible subgroups. Moreover, Lemma~\ref{lem:c1} implies that the only proper nonzero subspaces of $V$ that are stabilised by $g$ are $V_1$ and $V_2$. Consequently, $\M(G,g) = \{ H \}$, where $H$ has type $\O^-_{m-1}(q) \times \O^-_{m+1}(q)$.

Now Theorem~\ref{thm:fpr_s} implies that for each prime order element $x \in G$ we have
\[
P(x,g) \leq \fpr(x,G//H) < \frac{1}{q^{(m+1)/2}} + \frac{2}{q^{m-2}} + \frac{2}{q^{m-1}} < \frac{1}{2}.
\]
By Lemma~\ref{lem:prob_method} we conclude that $u(G) \geq 2$. Moreover, as $q \to \infty$ we have $P(x,g) \to 0$ and consequently $u(G) \to \infty$.
\end{proof}

\subsection{Case II(b)} \label{ss:o_IIb}

We now turn to Case~II(b). By Proposition~\ref{prop:o_cases}, we may assume that $G$ is $\<T, \th\>$ where $T \in \T$ and $\th \in \{ r, \d r \}$. 

Recall the reflection $r^\e$ defined in Definition~\ref{def:elt_r}, and if $q$ is odd, the diagonal element $\d^\e$ defined in Definitions~\ref{def:delta} and~\ref{def:delta_minus}. Unless there is ambiguity, we write $r = r^\e$ and $\d = \d^{\e}$. If $q$ is odd, fix the the element $\b \in \F_q^\times$ of order $(q-1)_2$.

\begin{remarkx}\label{rem:o_IIb}
By Proposition~\ref{prop:o_computation} implies that $u(G) \geq 2$ when $G$ is one of 
\begin{equation}\label{eq:o_IIb}
\O^\pm_8(2), \ \ \<\POm^\pm_8(3),r\> \ \ \O^\pm_{10}(2), \ \ \O^\pm_{12}(2).
\end{equation}
Therefore, for the remainder of this section, we may assume that $G$ does not appear in \eqref{eq:o_IIb}.
\end{remarkx}

We apply the probabilistic method, so we begin by selecting an element. For now assume that $T \neq \POm^\pm_8(5)$. Let
\[
y = \left\{ \begin{array}{ll}
A \perp r                                & \text{if $\th = r$} \\
{\,}^\Delta (2m-2)^- \perp {\,}^\Delta r & \text{if $\th = \d r$,} 
\end{array}
\right.
\]
where $A$ has type $(2m-2)^-$, unless $q=2$, in which case $A$ has order $2^{m-1}+1$. If $T = \POm^\pm_8(5)$, then let $y = A \perp r$ where $A$ has order $(5^3+1)/2=63$ if $\th=r$ and $(5^3+1)4=504$ if $\th = \d r$.

\begin{propositionx}\label{prop:o_IIb_elt}
Let $G = \<T,\th\>$ for $T \in \T$ and $\th \in \{r,\d r\}$. Assume that $G$ is not one of the groups in \eqref{eq:o_IIb}. 
\begin{enumerate}
\item If $\th=r$, then $y \in Tr$.
\item If $q$ is odd and $y$ has type ${\,}^\Delta (2m-2)^- \perp {\,}^\Delta r$, then $y \in T\d r$.
\end{enumerate}
\end{propositionx}

\begin{proof}
Part~(i) is immediate since $I_2 \perp (2m-2)^- \in T$, by Lemma~\ref{lem:omega_a}, and $I_2 \perp A$ is clearly in $T$ when $q=2$ and $(m,q)=(4,5)$. Now consider part~(ii), so $q$ is odd. Let $x_1 \in \DO^-_{2m-2}(q)$ have type ${\,}^\Delta(2m-2)^-$, so $\tau(x_1) = \b$ and $\det(x_1) = \b^{m-1}$. Additionally, by Lemma~\ref{lem:elt_r_odd}(vi), $\tau({\,}^\Delta r^{-\e}) = \b$ and $\det({\,}^\Delta r^{-\e}) = -\b$. Therefore, the element $x = x_1 \perp {\,}^\Delta r^{-\e}$ has type ${\,}^\Delta (2m-2)^- \perp {\,}^\Delta r^{-\e}$ and satisfies $\tau(x) = \b$ and $\det(x) = -\b^m$. Let $y = xZ(\DO^\e_{2m}(q))$. Now $\tau(r) = 1$ and $\det(r) = -1$. Moreover, we saw in Remarks~\ref{rem:delta} and~\ref{rem:delta_minus} that $\tau(\d)=\b$ and $\det(\d) = \b^m$. Therefore, $\tau(\d r) = \b$ and $\det(\d r) = -\b^m$. Consequently, $y \in \PSO^\e_{2m}(q)\d r$, or in other words $y \in T\d r \cup T\rsq\rns\d r$, but $\ddot{\d}\ddot{r}$ and $\ddot{r}_\square\ddot{r}_\nonsquare\ddot{\d}\ddot{r}$ are $\Out(T)$-conjugate.
\end{proof}

\begin{theoremx}\label{thm:o_IIb_max}
The maximal subgroups of $G$ that contain $y$ are listed in Table~\ref{tab:o_IIb_max}, where $m(H)$ is an upper bound on the multiplicity of the subgroups of type $H$ in $\M(G,y)$.
\end{theoremx}

\begin{table}
\centering
\caption{Case~II(b): Description of $\M(G,y)$} \label{tab:o_IIb_max}
{\renewcommand{\arraystretch}{1.2}
\begin{tabular}{ccc}
\hline
type of $H$                           & $m(H)$ & conditions                   \\
\hline
$\O^{-\e}_2(q) \times \O^-_{2m-2}(q)$ & $1$    &                              \\
$\O_{2m-1}(q)$                        & $2$    & $q$ odd, $\th = r$           \\
$\Sp_{2m-2}(q)$                       & $1$    & $q$ even                     \\
$\O_m(q^2)$                           & $4$    & $q$ odd, $m$ odd, $\th=\d r$ \\
\hline
\end{tabular}}
\end{table}

We will prove Theorem~\ref{thm:o_IIb_max} in two parts, considering the reducible and irreducible maximal overgroups of $y$ separately.

\begin{propositionx} \label{prop:o_IIb_max_reducible}
Theorem~\ref{thm:o_IIb_max} is true for reducible subgroups.
\end{propositionx}

\begin{proof}
First assume that $q$ is odd and $\th = \d r$. Then $y$ centralises an orthogonal decomposition $V=U \perp U^\perp$, where $U$ is a nondegenerate $2$-space. Moreover, $y$ acts irreducibly on $U$ and $U^\perp$ (see Lemma~\ref{lem:elt_r_odd}(iv)). Therefore, by Lemma~\ref{lem:c1}, the only proper nonzero subspaces of $V$ stabilised by $y$ are $U$ and $U^\perp$, so the only reducible maximal overgroup of $y$ is one of type $\O^{-\e}_2(q) \times \O^-_{2m-2}(q)$. 

Next assume that $q$ is odd and $\th = r$. In this case, the element $y$ centralises a decomposition $V = U_1 \perp U_2 \perp (U_1+U_2)^\perp$, where $U_1$ and $U_2$ are nondegenerate $1$-spaces. Moreover, $y$ acts irreducibly on $(U_1+U_2)^\perp$, and acts as $1$ and $-1$ on $U_1$ and $U_2$, respectively. Therefore, by Lemma~\ref{lem:c1}, the only subspaces stabilised by $y$ are direct sums of $U_1$, $U_2$ and $(U_1+U_2)^\perp$. Consequently, the reducible maximal overgroups of $y$ are two of type $\O_{2m-1}(q)$ (the stabilisers of $U_1$ and $U_2$) and one of type $\O^{-\e}_2(q) \times \O^-_{2m-2}(q)$ (the stabiliser $U_1+U_2$).

Finally assume that $q$ is even and $\th=r$. In this case, $y$ centralises the decomposition $V=U \perp U^\perp$, where $U$ is a nondegenerate $2$-space. In this case, $y$ acts irreducibly on $U^\perp$. However, $y$ acts indecomposably on $U$ and stabilises a unique $1$-dimensional (nonsingular) subspace $W$ of $U$ (see Lemma~\ref{lem:elt_r_even}). Since there are no $\F_q\<y\>$-homomorphisms between $U^\perp$ and any $\F_q\<y\>$-subquotient of $U$, Corollary~\ref{cor:goursat} implies that the only proper nonzero subspaces of $V$ stabilised by $y$ are $W$, $U$, $U^\perp$ and $U^\perp+W$. From this we deduce that the reducible maximal overgroups of $y$ are one of type $\Sp_{2m-2}(q)$ (the stabiliser of $W$) and one of type $\O^{-\e}_2(q) \times \O^-_{2m-2}(q)$ (the stabiliser of $U$).
\end{proof}

\begin{propositionx} \label{prop:o_IIb_max_irreducible}
Theorem~\ref{thm:o_IIb_max} is true for irreducible subgroups.
\end{propositionx}

\begin{proof}
Let $H \in \M(G,y)$ be an irreducible subgroup. If $\th = r$, then $y = y_1 \perp r$ where $|y_1|$ is divisible by a primitive prime divisor of $q^{2m-2}-1$ (in fact, $|y_1| \in \ppd(q,2m-2)$ unless $q=2$ or $(m,q)=(4,5)$). Now assume that $q$ is odd and $\th=\d r$. Recall that ${\,}^\Delta r$ has order $2(q-1)_2$ and $|y_1| = (q^{m-1}+1)_2(q-1)_2\ell$ for $\ell \in \ppd(q,2m-2)$. Therefore, $y^{(q^m+1)_2(q-1)_2}$ has order $\ell$. Consequently, in both cases, we can fix a power $z$ of $y$ of order $\ell \in \ppd(q,2m-2)$.

Let us also note that if $\th = r$, then a power of $y^\ell$ is $r$ and $\nu(r)=1$.

By Theorem~\ref{thm:aschbacher} either $H$ is a geometric subgroup contained in $\C_2 \cup \cdots \cup \C_8$ or $H$ is an almost simple subgroup in $\S$. We begin by considering the geometric maximal overgroups $H$ of $y$ in $G$. Since $y$ has order divisible by $\ell \in \ppd(q,2m-2)$, the main theorem of \cite{ref:GuralnickPenttilaPraegerSaxl97} implies that  the possibilities for $H$ feature in \cite[Examples~2.1--2.5]{ref:GuralnickPenttilaPraegerSaxl97}. Let us consider these possibilities in turn.

For orthogonal groups, Example~2.1 consists of subfield subgroups, none of which arise since for all proper divisors $k$ of $f$, if $q_0=p^k$, then $\ell$ does not divide 
\[
|\O^{\e}_{2m}(q_0)| = 2 q_0^{m^2-m} (q_0^m-\e) \prod_{i=1}^{m-1} (q_0^{2i}-1).
\]

All subgroups in Example~2.2 are reducible.

Example~2.3 features the imprimitive subgroups of type $\O_1(q) \wr S_n$. For these we insist that $\e=+$, $q=p \geq 3$ and $\ell=2m-1$; however, by \cite[Lemma~6.1]{ref:BambergPenttila08}, this implies that $T=\POm^+_8(5)$ (noting that, by Remark~\ref{rem:o_IIb}, we are not considering $T = \POm^+_8(3)$). Now suppose that $T = \POm^+_8(5)$ and $H$ has type $\O_1(5) \wr S_8$. Then $s^4$ has order at least $63$, but there are no elements of this order in $H \cap T = 2^7.A_8$. Therefore, no subgroups arise from Example~2.3.

The only possible field extension subgroup $H$ in Example~2.4 is $\O^\eta_{m}(q^2)$ where $\eta = \e$ if $m$ is even and $\eta=\circ$ if $m$ is odd. If $\th = r$, then $\nu(z^\ell)=1$, so $y$ is not contained in such a subgroup, by \cite[Lemma~5.3.2]{ref:BurnessGiudici16}. Now assume that $\th = \d r$. If $m$ is even, then $\ell$ does not divide the order of $H$. 

Therefore, if $H$ is a field extension subgroup containing $y$, then $q$ is odd, $\th = \d r$, $m$ is odd and $H$ has type $\O_m(q^2)$. We will now prove that, in this case, $y$ is contained in four $G$-conjugates of $H$. Note that $y$ is a semisimple element with eigenvalue multiset $\Lambda \cup \Lambda^q \cup \{\mu,\mu^q\}$, where $\Lambda = \{ \l^{q^{2i}} \mid 0 \leq i \leq m-1 \}$ for a scalar $\l \in \FF_p^\times$ of order $(q^m+1)_2(q-1)_2\ell$ (where $\ell \in \ppd(q,2m)$) and $\mu \in \FF_p^\times$ has order $2(q-1)_2$. Let $\pi\:H \to G$ be the field extension embedding and write $H=B.\phi$, where $\phi$ is the field automorphism $\xi \mapsto \xi^q$. By \cite[Lemma~5.3.2]{ref:BurnessGiudici16}, if $\pi(\tilde{y}) = y$, then $\tilde{y}$ has one of the following eigenvalue sets 
\[
S_1=\Lambda \cup \{\mu\}, \quad S_2=\Lambda \cup \{\mu^q\}, \quad S_3=\Lambda^q \cup \{\mu^q\}, \quad S_4=\Lambda^q \cup \{\mu\}.
\] 
Let $\tilde{y}_i$ have eigenvalue set $S_i$. By \cite[Propositions~3.4.3 and~3.5.4]{ref:BurnessGiudici16}, $y^G \cap H = \bigcup_{i=1}^4 \tilde{y}_i^B$. Note that $\phi$ fuses $\tilde{y}_1^B$ with $\tilde{y}_3^B$ and fuses $\tilde{y}_2^B$ with $\tilde{y}_4^B$. Therefore, $y^G \cap H = \tilde{y}_1^H \cup \tilde{y}_2^H$. Since an element of type ${\,}^\Delta r^\e$ is self-centralising in $\GO^\e_2(q)$, Lemma~\ref{lem:centraliser} and \cite[Appendix~B]{ref:BurnessGiudici16} yield $|C_G(y)| = (q^{m-1}+1)(q-1)2 = 2|C_H(y)|$. Now Lemma~\ref{lem:prob_method} implies that the number of $G$-conjugates of $H$ that contain $y$ is
\[
\frac{|y^G \cap H|}{|y^G|} \frac{|G|}{|H|} = \frac{2|C_G(y)|}{|C_H(y)|} = 4.
\]

We now consider subgroups $H$ contained in the $\S$ family. First assume that $\th = r$. Suppose that $q$ is prime and $H$ arises from the fully deleted permutation module. For now assume that $T \neq \POm_8^\pm(5)$ and recall that $T \not\in \{\POm_8^\pm(2),\POm_8^\pm(3)\}$ (see Remark~\ref{rem:o_IIb}). If $q > 2$, then, by \cite[Lemma~6.1]{ref:BambergPenttila08}, $y$ has order $2\ell$ where $\ell \geq 4m-3$ is prime. If $q = 2$, then $y$ has order $2(2^{m-1}+1)$, which is divisible by a prime at least $2m-1$. In both cases, $S_{2m+2}$ does not contain an element of order $|y|$, so we conclude that $H \not\in \S$. If $T \neq \POm_8^\pm(5)$, then $s^4$ has order at least $63$, but $H \cap T \cong A_{10}$ has no elements of order $63$. Therefore $H$ does not arise from the fully deleted permutation module. Therefore, since $\nu(y^\ell)=1$, \cite[Theorem~7.1]{ref:GuralnickSaxl03} implies that $T=\POm^+_8(q)$ with $q=p \geq 5$ and $\soc(H) = \POm^+_8(2)$ (noting $G=\<T,\th\>$ does not have absolutely irreducible maximal subgroups of type $\Omega_7(q)$ or ${\,}^3D_4(q^{1/3})$, see \cite[Table~8.50]{ref:BrayHoltRoneyDougal}) but again $\soc(H)$ contains no elements of order $|s^4|$.

Now assume that $\th=\d r$. If, $T \neq \POm^\pm_8(5)$, then Theorem~\ref{thm:zsigmondy} implies that $\ell > 4m-3$, so \cite[Theorem~2.2]{ref:GuralnickMalle12JAMS} eliminates all possibilities for $H$ (see \cite[Table~1]{ref:GuralnickMalle12JAMS}, noting that $\<\POm_8^\pm(q),\d r\>$ does not have any maximal absolutely irreducible subgroups of type $\Omega_7(q)$ or $\PSU_3(q)$, see \cite[Tables~8.50 and~8.53]{ref:BrayHoltRoneyDougal}). If $T = \POm^\pm_8(5)$, then $s^4 \in T$ has order at least $63$ and no maximal subgroup in $\S$ contains an element of such an order. Therefore, no $\S$ family subgroups occur in this case either.
\end{proof}

Next we handle a special case in a more concrete fashion.

\begin{propositionx} \label{prop:o_IIb_reflection}
Let $G = \< T, r\>$ with $m \geq 5$. Let $x_1,x_2 \in G$ have prime order and satisfy $\nu(x_1)=1$ and $\nu(x_2) \leq 2$. Then there exists $g \in G$ such that $\<x_1,y^g\> = \<x_2,y^g\> = G$.
\end{propositionx}

\begin{proof}
We prove the claim when $q$ is odd; the case where $q$ is even is similar. We work in terms of the bases $\B^\e$ in \eqref{eq:B_plus} and \eqref{eq:B_minus}.

Let us fix three particular vectors. First let $t_1,t_{m-1} \in \<e_1,f_1,e_{m-1},f_{m-1}\>$ be nonsingular vectors such that $(e_i-f_i,t_i)=0$ and $\<e_i-f_i,t_i\>$ is a nondegenerate minus-type $2$-space. Next let $t_2 \in \<e_1,f_1,e_2-f_2,e_{m-1},f_{m-1}\>^\perp$ with the property that $\<e_2-f_2,t_2\>$ is a nondegenerate minus-type $2$-space. 

Recall that the element $y$ has type $r^\e \perp (2m-2)^-$, centralising a decomposition $U_1 \perp U_2$. If $\e=-$, then we may assume that $r^+=r_{e_1-f_1}$ and 
\[
U_1 = \<e_1,f_1\> \quad\text{and}\quad U_2 = \<e_2,\dots,f_{m-1},u_m,v_m\>.
\] 
If $\e=+$, then we may assume that $r^- = r_{e_1-f_1}$ and
\[
U_1 = \<e_1-f_1, t_1\> \quad\text{and}\quad U_2 = \<e_2,\dots,f_{m-2},e_{m-1}-f_{m-1},t_{m-1},e_m,f_m\>.
\]

\emph{Case~1: $\nu(x_2) = 1$.} In this case, $x_1$ and $x_2$ are reflections in nonsingular vectors. If $u_1$ and $u_2$ are nonsingular vectors, then $r_{u_1} = r_{u_2}$ if and only if $\<u_1\> = \<u_2\>$. Therefore, it suffices to prove the claim for $x_1 = r_{u_1}$ and $x_2 = r_{u_2}$ for orbit representatives $(\<u_1\>,\<u_2\>)$ for the action of $G$ on pairs of distinct nonsingular 1-spaces of $V$. We may assume that $u_1 = e_1-f_1$. Now $V = \<u_1\> \perp \<u_1\>^\perp$ and $G_{\<u_1\>}$ acts transitively on the sets of nonzero vectors of a given norm in $\<u_1\>^\perp$. Therefore, we may assume that $u_2 = \xi u_1 + \eta (e_1+f_1)$ or $u_2 = \xi u_1 + \eta e_3$ for scalars $\xi, \eta \in \F_q$. This amounts to the following two cases
\begin{enumerate}
\item $u_2 = e_1 - \l f_1$ for $\l \in \F_q \setminus \{0,1\}$
\item $u_2 = e_1 + f_1 + \l e_3$ for $\l \in \F_q^\times$
\end{enumerate}

First assume that $\e=-$. Let $z$ have type $r_v \perp (2m-2)^-$ centralising the decomposition $\<v,w\> \perp \<v,w\>^\perp$ where $v = e_1+e_2-f_2$ and $w = e_1+e_2+f_2$. Note that $v$ is nonsingular and $\<v,w\>$ is a nondegenerate plus-type $2$-space. By Theorem~\ref{thm:o_IIb_max}, $\M(G,z) \subseteq \{ G_{\<v\>}, G_{\<w\>}, G_{\<v,w\>} \}$. Observe that $vx_1 = f_1+e_2+f_2$ and $wx_1 = f_1+e_2-f_2$, neither of which is contained in $\<v,w\>$. Therefore, $x_1$ does not stabilise $\<v\>$, $\<w\>$ nor $\<v,w\>$. Consequently, $\<x_1,z\> = G$. Moreover, in the two possible cases above
\begin{enumerate}
\item $vx_2 = \l f_1 + e_2 + f_2$ and $wx_2 = \l f_1 + e_2 - f_2$
\item $vx_2 = - f_1 + e_2 + f_2 - \l e_3$ and $wx_2 = - f_1 + e_2 + f_2 - \l e_3$ 
\end{enumerate}
In both cases, $vx_2$ and $wx_2$ are not contained in $\<v,w\>$, so, as above, $\<x_2, z\> = G$. It remains to observe that since $Q(e_1-f_1) = -2 =  Q(e_1+e_2-f_2)$, there exists $g \in G$ such that $\<e_1,f_1\>g = \<v,w\>$ and $(e_1-f_1)g = e_1+e_2-f_2$. This implies that $r_{(e_1-f_1)g} = r_v$ and $y^g = z$.

Now assume that $\e=+$. In this case, let $z$ have type $r_v \perp (2m-2)^-$ centralising $\<v,w\> \perp \<v,w\>^\perp$ where $v = e_1+e_2-f_2$ and $w = e_1+t_2$, noting that $\<v,w\>$ is a nondegenerate minus-type $2$-space. Arguing as in the previous case we see that $\<x_1,z\> = \<x_2,z\> = G$. Moreover, there exists $g \in G$ such that $\<e_1+f_1,t_1\>g = \<v,w\>$ and $(e_1-f_1)g = e_1+e_2-f_2$, so $y^g = z$. This completes the proof in Case~1.

\emph{Case~2: $\nu(x_2) =2$ and $x_2$ is semisimple.} In this case, $x_1$ is a reflection and $x_2$ centralises a decomposition $W \perp W^\perp$ where $W$ is a nondegenerate $2$-space. Moreover, if $|x_2|=2$, then we may assume that $x_2 = -I_2 \perp I_{2m-2}$ and if $|x_2|$ is odd, then $x_2 = A \perp I_{2m-2}$ where $A$ is irreducible. As in Case~1, it suffices to assume that $x_1 = r_u$ where $u=e_1-f_1$ and consider orbit representatives $W$ of the action of $G_{\<u\>}$ on nondegenerate $2$-subspaces of $V$. Considering that $W$ is either plus- or minus-type, and by separating into the cases where
\[
\text{(i)} \ \ \<u\> \leq W \qquad \text{(ii)} \ \ W \leq \<u\>^\perp \qquad \text{(iii)} \ \ \<u\> \not\leq W \not\leq \<u\>^\perp
\]
we may assume that $W$ is one of the following
\begin{enumerate}
\item $W = \< e_1,f_1 \>$ or $W = \< e_1-f_1, t_1 \>$
\item $W = \< e_2, f_2 \>$ or $W = \< e_2-f_2, t_2\>$
\item $W = \< e_2-f_2 + \l u, e_2+f_2\>$ or $W = \< e_2-f_2 + \l u, t_2 \>$ where $\l \in \F_q^\times$.
\end{enumerate}

As in Case~1, let $z$ be an element of type $r_v \perp (2m-2)^-$, centralising a decomposition $\<v,w\> \perp \<v,w\>^\perp$ where $v = e_1+e_2-f_2$. Moreover, let $w = e_1+e_2+f_2$ if $\e=-$ and $w=e_1+t_2$ if $\e=+$. Note that $\<v,w\>$ is a nondegenerate $(-\e)$-type $2$-space. Consequently, we have $\<x_1,z\> = G$. Since $x_2$ fixes $W^\perp$ pointwise and either negates or acts irreducibly on $W$, we see that $\<x_2,z\> = G$ also. 

\emph{Case~3: $\nu(x_2) =2$ and $x_2$ is unipotent.} Here we need to consider the cases where $x_2$ has Jordan form $[J_2^2,J_1^{2m-4}]$ and $[J_3,J_1^{2m-3}]$. The latter case is very similar to Case~2, so we provide the details in the case where $x_2$ has Jordan form $[J_2^2,J_1^{2m-4}]$. 

As before, $x_1$ is a reflection. In this case, $x_2$ centralises a decomposition $W \perp W^\perp$ where $W = W_1 \oplus W_2$ for totally singular $2$-spaces $W_1$ and $W_2$. Moreover, $x_2$ acts trivially on $W^\perp$ and acts indecomposably on $W_i$ stabilising a unique $1$-space $\<w_i\> \leq W_i$. As in the previous cases, it suffices to assume that $x_1 = r_u$ where $u=e_1-f_1$ and consider orbits of the action of $G_{\<u\>}$. In this way, we may assume that one of the following holds
\begin{enumerate}
\item $W_1 = \<e_1,e_2\>$ with $w_1=e_1$ and $W_2 = \<f_1,f_2\>$ with $w_2=f_2$
\item $W_1 = \<e_1,e_2\>$ with $w_1=e_1+e_2$ and $W_2 = \<f_1,f_2\>$ with $w_2=f_2$
\item $W_1 = \<e_2,e_3\>$ with $w_1=e_2$, and $W_2 = \<f_2,f_3\>$ with $w_2 = f_3$
\item $W_1 = \<e_2,e_1+e_3\>$ with $w_1=e_2$ and  $W_2 = \<f_2,f_3\>$ with $w_2 = f_3$
\item $W_1 = \<e_1+e_2,e_3\>$ with $w_1=e_1+e_2$ and  $W_2 = \<f_2,f_3\>$ with $w_2 = f_3$
\end{enumerate}

As in the previous cases, let $z$ have type $r_v \perp (2m-2)^-$ centralising a decomposition $\<v,w\> \perp \<v,w\>^\perp$ where $v = e_1+e_2-f_2$, and let $w = e_1+e_2+f_2$ if $\e=-$ and $w=e_1+t_2$ if $\e=+$. Consequently, we have $\<x_1,z\> = G$. It is also easy to see that the action of $x_2$ on the decomposition $(W_1 \oplus W_2) \perp W^\perp$ ensures that $x_2$ stabilises none of $\<v\>$, $\<w\>$ and $\<v,w\>$. 

For example, consider case~(i). Here  
\[
x_2 = \left( \begin{array}{cc} 1 & 0 \\ 1 & 1 \end{array} \right) \oplus \left( \begin{array}{cc} 1 & -1 \\ 0 & 1 \end{array} \right)  \perp I_{2m-4}.
\] 
with respect to $(\<e_1,e_2\> \oplus \<f_1,f_2\>) \perp \<e_1,f_1,e_2,f_2\>^\perp$. Therefore, $x_2$ fixes $e_1$ and $f_2$ and maps $e_2 \mapsto e_1+e_2$ and $f_1 \mapsto f_1-f_2$. Therefore, $vx_2, wx_2 \not\in \<v,w\>$. Therefore, we conclude that  $\<x_2,z\> = G$. 
\end{proof}

\begin{propositionx} \label{prop:o_IIb}
Let $G \in \A$. In Case~II(b), $u(G) \geq 2$ and as $q \to \infty$ we have $u(G) \to \infty$.
\end{propositionx}

\begin{proof}
We will apply the probabilistic method encapsulated by Lemma~\ref{lem:prob_method}. Theorem~\ref{thm:o_IIb_max} gives the members of $\M(G,y)$. Let $x \in G$ have prime order. We now use fixed point ratio bounds from Section~\ref{s:fpr_s} to obtain an upper bound on $P(x,y)$.

If $\th = \d r$, then $q$ is odd and
\[
P(x,y) \leq \frac{1}{q^2} + \frac{1}{q^{m-1}-1} + \frac{4}{q^{2m-3}} + \frac{1}{q^{2m-2}} + N_m \, \frac{2}{q^{m-2}} < \frac{1}{2}
\] 
where $N_m$ is $4$ if $m$ is odd and $0$ if $m$ is even. In addition, $P(x,y) \to 0$ as $q \to \infty$.

From now on we may assume that $\th = r$. By Remark~\ref{rem:o_IIb}, we may assume that $G$ does not appear in \eqref{eq:o_IIb}. First assume that $q$ is odd. For brevity, write
\[
P_1(m,q) = \frac{1}{q^{m-1}-1} + \frac{4}{q^m-1} + \frac{4}{q^{2m-3}}.
\]
In this case,
\[
P(x,y) \leq 2q^{-1} + q^{-2} + q^{-(2m-2)} + 2q^{-(2m-1)} + P_1(m,q).
\]
Now $P(x,y) \to 0$ as $q \to \infty$, and if $q > 3$, then $P(x,y) < \frac{1}{2}$. Now assume that $q=3$ and therefore $m \geq 5$. Making use of the dependence on $\nu(x)$ in the fixed point ratio bounds in Proposition~\ref{prop:fpr_s_o12}, we obtain
\[
P(x,y) \leq
\left\{\!
\begin{array}{ll}
2q^{-3} + q^{-6} + q^{-(2m-6)} + 2q^{-(2m-3)} + P_1(m,q) < 0.120 & \text{if $\nu(x) \geq 3$} \\       
2q^{-2} + q^{-4} + q^{-(2m-4)} + 2q^{-(2m-2)} + P_1(m,q) < 0.268 & \text{if $\nu(x) = 2$}    \\
2q^{-1} + q^{-2} + q^{-(2m-2)} + 2q^{-(2m-1)} + P_1(m,q) < 0.809 & \text{if $\nu(x) = 1$}    \\
\end{array}
\right.
\]
Now let $x_1,x_2 \in G$ have prime order. If
\[
P(x_1,y)+P(x_2,y) > 1
\] 
then we can assume that $\nu(x_1) = 1$ and $\nu(x_2) \leq 2$. In the latter case, Proposition~\ref{prop:o_IIb_reflection} implies that there exists $y \in G$ such that $\< x_1, y \> = \< x_2, y \> = G$. Therefore, $u(G) \geq 2$.

Now assume that $q$ is even. We proceed as when $q$ is odd. In this case, write
\[
P_2(m,q) = \frac{1}{q^{m-1}-1} + \frac{2}{q^m-1} + \frac{4}{q^{2m-3}}.
\]
Here
\[
P(x,y) \leq q^{-1} + q^{-2} + P_2(m,q).
\]
Now $P(x,y) \to 0$ as $q \to \infty$, and if $q > 2$, then $P(x,y) < \frac{1}{2}$. Now assume that $q=2$ and therefore $m \geq 7$. Now
\[
P(x,y) \leq
\left\{
\begin{array}{ll}
q^{-3} + q^{-6} + P_2(m,q) < 0.175 & \text{if $\nu(x) \geq 3$} \\       
q^{-2} + q^{-4} + P_2(m,q) < 0.347 & \text{if $\nu(x) = 2$}    \\
q^{-1} + q^{-2} + P_2(m,q) < 0.784 & \text{if $\nu(x) = 1$}    \\
\end{array}
\right.
\]
As above, for $x_1,x_2 \in G$ of prime order, if
\[
P(x_1,y)+P(x_2,y) > 1
\] 
then we can assume that $\nu(x_1) = 1$ and $\nu(x_2) \leq 2$, in which case, Proposition~\ref{prop:o_IIb_reflection} implies that there exists $y \in G$ such that $\< x_1, y \> = \< x_2, y \> = G$. Therefore, we conclude that $u(G) \geq 2$.
\end{proof}

\clearpage 
\section{Case III: triality automorphisms} \label{s:o_III}

This section sees the completion of the proofs of Theorems~\ref{thm:o_main} and~\ref{thm:o_asymptotic}. Write $G=\<T,\th\>$ where $T = \POm^+_8(q)$ and $\th \in \Aut(T) \setminus \PGaO^+_8(q)$. By Proposition~\ref{prop:o_cases}, in Case~III, it suffices to consider the following three cases
\begin{enumerate}[(a)]
\item $\th = \t\p^i$ where $i$ is a proper divisor of $f$ and $3$ divides $f/i$
\item $\th = \t\p^i$ where $i$ is a proper divisor of $f$ and $3$ does not divide $f/i$
\item $\th = \t$.
\end{enumerate}

For Cases~III(a) and~III(b), we will apply Shintani descent and the application of Shintani descent will be very similar to that in Cases~II(a) and~II(b) respectively. In Case~III(c), $\th$ is a graph automorphism and the argument will be more reminiscent of Case~I(b). It is worth noting that in all three cases $\nu(x) > 1$ for all $x \in G \cap \PGO^+_8(q)$. Cases~III(a)--(c) will be considered in turn in Sections~\ref{ss:o_IIIa}--\ref{ss:o_IIIc}, respectively.

\subsection{Case III(a)}\label{ss:o_IIIa}

Write $q=p^f$ where $f \geq 2$. Let $V = \F_q^8$. Fix the simple algebraic group $X = \Spin_8(\FF_p)$, the standard Frobenius endomorphism $\p = \p_{\B^+}$ of $X$ and the standard triality graph automorphism $\t$ of $X$ such that $C_X(\t) = G_2(\FF_p)$. 

Write $\sigma = \t\p^i$ and $e=f/i$ and $q=q_0^e$. In Case~III(a), we assume that $3$ divides $e$. Let $F$ be the Shintani map of $(X,\s,e)$, so
\[
F\: \{ (g\ws)^{X_{\s^e}} \mid g \in X_{\s^e} \} \to \{ x^{X_\s} \mid x \in X_\s \}.
\]
Observe that $X_{\s^e} = T$, since $3$ divides $e$, and $X_{\s}$ is $T_0 = C_{T}(\p^i\t) = {}^3D_4(q_0)$, the \emph{Steinberg triality group}. Let $y \in T_0$ have order $q_0^4-q_0^2+1$ and let $t \in T$ satisfy $F(t\th) = y$. 

\begin{propositionx}\label{prop:o_IIIa}
Let $G = \<T, \th\> \in \A$. In Case~III(a), $u(G) \geq 2$ and as $q \to \infty$ we have $u(G) \to \infty$.
\end{propositionx}

\begin{proof}
First note that the order of $y$ does not divide the order of any parabolic subgroup of $T_0$. Therefore, by Lemma~\ref{lem:shintani_descent_fix}, we deduce that $t\th$ is not contained in any parabolic subgroups of $G$ (see Example~\ref{ex:shintani_descent_fix}). From \cite[Table~II]{ref:Kleidman883D4}, we see that $|C_{T_0}(y)| = q_0^4-q_0^2+1$. By \cite[Table~8.50]{ref:BrayHoltRoneyDougal}, there are at most $10+\log\log{q}$ classes of maximal nonparabolic subgroups of $G$. Note that all nonparabolic maximal subgroups of $G$ are nonsubspace, see for example \cite[Table~3.1]{ref:Burness074}. Therefore, noting that $e \geq 3$, for all prime order $x \in G$ we have
\[
P(x,t\th) < (10+\log\log{q})(q_0^4-q_0^2+1) \cdot \frac{3}{q^{15/4}} < \frac{1}{2}
\]
and $P(x,\th) \to 0$ as $q \to \infty$. Therefore, $u(G) \geq 2$ and $u(G) \to \infty$ as $q \to \infty$, as claimed.
\end{proof}

\subsection{Case III(b)}\label{ss:o_IIIb}

Write $q=p^f$ where $f \geq 2$. Fix the simple algebraic group $X = \Spin_8(\FF_p)$, the standard Frobenius endomorphism $\p = \p_{\B^+}$ and the triality automorphism $\t$. Let $Z$ be the centraliser $C_X(\t) = G_2(\FF_p)$. Write $\sigma = \t\p^i$ and $e=f/i$ and $q=q_0^e$. In Case~III(b), we assume that $3$ does not divide $e$. 

\begin{propositionx}\label{prop:o_IIIb_elt}
Let $T=\POm^+_8(q)$ and let $\th = \t\p^i$ where $f/i$ is not divisible by $3$. Let $y$ have order $7$ if $q_0=2$ and $q_0^2-q_0+1$ if $q_0 > 2$. Then there exists $t \in T$ that commutes with $\t$ such that $(t\th)^e$ is $X$-conjugate to $y\t^{-1}$. Moreover, if $H \leq G$, then the number of $G$-conjugates of $H$ that contain $t\th$ is at most $|C_{{}^3D_4(q_0)}(y^3)|$.
\end{propositionx}

\begin{proof}
Since $(\t\s^e)^3 = \p^{3f} = \s^{3e}$ and $y \in G_2(q_0) = Z_\s$, by Lemma~\ref{lem:shintani_substitute}, there exists $t \in Z_{\s^e} \leq \POm^+_8(q) \leq X_{\t\s^e}$ such that $(t\ws)^e$ is $X$-conjugate (indeed $Z$-conjugate) to $y\t^{-1}$ and if $H \leq G$, then the number of conjugates of $H$ that contain $t\ws$ is at most $|C_{{}^3D_4(q_0)}(y^3)|$.
\end{proof}

\begin{lemmax}\label{lem:o_IIIb_centraliser}
Assume that $q_0 > 2$. Let $z = y^3$, where $y \in G_2(q_0) \leq {}^3D_4(q_0)$ has order $q_0^2-q_0+1$. Then $C_{{\,}^3D_4(q_0)}(z) = C_{q_0^2-q_0+1} \times C_{q_0^2-q_0+1}$.
\end{lemmax}

\begin{proof}
We may assume that $z \in \SU_3(q_0) < G_2(q_0) < {}^3D_4(q_0)$, and consequently $z \in W < Z < X$, where $W = \SL_3(\FF_p)$, $Z = G_2(\FF_p)$ and $X = \PSO_8(\FF_p)$ are the corresponding algebraic groups. Let $V$ and $U$ be the natural modules for $X$ and $Y$, respectively, and observe that $V|_{W} = U \oplus U^* \oplus 0^2$, where $0$ is the trivial module. By first considering the eigenvalues of $z$ on $U$, and then on $V$ via the given decomposition, we deduce that $C_{X}(z)^\circ$ is a maximal torus. In particular, this implies that $z$ is a regular semisimple element of ${\,}^3D_4(q_0)$ and by inspecting \cite[Table~II]{ref:Kleidman883D4}, we deduce that $C_{{\,}^3D_4(q_0)}(z)$ is either $C_{q_0^2-q_0+1} \times C_{q_0^2-q_0+1}$ or $C_{q_0^3+1} \times C_{q_0+1}$. Finally, we observe that the $\SU_{3}(q_0)$ subgroup of $G_2(q_0)$ containing $z$ is centralised in ${}^3D_4(q_0)$ by a torus of order $q_0^2-q_0+1$ and this rules out the latter possibility.
\end{proof}

\begin{propositionx}\label{prop:o_IIIb}
Let $G = \<T, \th\> \in \A$. In Case~III(b), $u(G) \geq 2$ and as $q \to \infty$ we have $u(G) \to \infty$.
\end{propositionx}

\begin{proof}
Write $z = y^3$. First assume that $q_0 > 2$. By Lemma~\ref{lem:o_IIIb_centraliser}, we have $|C_{{\,}^3D_4(q_0)}(z)| = (q_0^2-q_0+1)^2$, and note that $|z|$ is divisible by a primitive prime divisor $r$ of $q_0^6-1$. The maximal subgroups of ${\,}^3D_4(q_0)$ are given by the main theorem of \cite{ref:Kleidman883D4} (see also \cite[Table]{ref:BrayHoltRoneyDougal}). The only maximal parabolic subgroup of ${\,}^3D_4(q_0)$ with order divisible by $r$ has type $H_0=q^{1+8}{:}\SL_{2}(q^3).(q-1)$, but the maximal tori of $\SL_{2}(q^3)$ have order $q^3 \pm 1$, so there are no elements in $H_0$ with the appropriate centraliser in ${\,}^3D_4(q_0)$. Therefore, $z$ is not contained in a maximal parabolic subgroup of ${\,}^3D_4(q_0)$. Now assume that $q_0=2$. In this case $y$ and $z$ have order $7$ and it is straightforward to check that $|C_{{\,}^3D_4(2)}(z)| = 7^2$ and again that $z$ is not contained in any parabolic subgroup of ${\,}^3D_4(2)$. 

Suppose that $t\th$ is contained in a parabolic subgroup of $G$. Then $t\th$ is contained in a parabolic subgroup of $\PDO^+_8(q^3){:} \< \th \>$. Let $F\: \PDO^+_8(q^3)\th \to {\,}D_4(q_0)$ be the Shintani map of $(X,\s,3e)$. Then Lemma~\ref{lem:shintani_descent_fix} implies that $F(t\th) = y^3 = z$ (see Lemma~\ref{lem:shintani_powers}(ii)) is contained in a parabolic subgroup of ${\,}^3D_4(q_0)$, which we know is false. Thus we conclude that $t\th$ is not contained in a parabolic subgroup of $G$.

Let $M$ be $7^2$ if $q_0=2$ and $(q_0^2-q_0+1)^2$ if $q_0 > 2$. There are at most $10+\log\log{q}$ classes of maximal nonparabolic subgroups of $G$, so for all prime order $x \in G$,
\[
P(x,t\th) < (10+\log\log{q})\cdot M \cdot \frac{3}{q^{15/4}} \to 0
\]
as $q \to \infty$ and $P(x,\th) < \frac{1}{2}$, unless $q=4$. When $q=4$, Proposition~\ref{prop:o_computation} implies that $u(G) \geq 2$.
\end{proof}

\subsection{Case III(c)}\label{ss:o_IIIc}

Write $q=p^f$ where $f \geq 1$. Let $\t$ be the triality graph automorphism and recall that $C_T(\t) = G_2(q)$. If $q=2$, then for $t \in G_2(2)$ of order $7$, Proposition~\ref{prop:o_computation} gives $u(G) \geq 2$. From now on, assume that $q > 2$. In this case, let $t \in G_2(q)$ have order $q^2-q+1$. 

\begin{propositionx}\label{prop:o_IIIc}
Let $G = \<T, \t\> \in \A$. Then $u(G) \geq 2$ and as $q \to \infty$ we have $u(G) \to \infty$.
\end{propositionx}

\begin{proof}
Let $z = (t\t)^3 = t^3$. Since $q > 2$, the order $|z| = (q^2-q+1)/(q^2-q+1,3)$ is divisible by some $r \in \ppd(q,6)$. Let $H \in \M(G,t\t)$. The possibilities for $H$ are given in \cite[Table~8.50]{ref:BrayHoltRoneyDougal}. The only $G$-classes of subgroup that have order divisible by $r \in \ppd(q,6)$ are those of type $\O_2^-(q) \times \GU_3(q)$ and $G_2(q)$, and if $q \equiv 2 \mod{3}$, then also an absolutely irreducible almost simple group with socle $\PSU_3(q)$. 

First assume that $H$ has type $\O_2^-(q) \times \GU_3(q)$. Let $S$ be a maximal torus of $\PSO^+_8(q)$ that contains $z$. Since $r$ divides $|S|$, we have $|S| = (q^3+1)(q+1)$. Therefore, $z = A \perp B$ with respect to $\F_q^8 = U \perp W$, where $U$ and $W$ are nondegenerate minus-type subspaces of dimensions $6$ and $2$. Moreover, $A$ has order dividing $q^3+1$ and $B$ has order dividing $q+1$. Since $(q+1,q^2-q+1)=1$, we deduce that $|A|=|z|$ and $|B|=1$. Therefore, $z = A \perp I_2$, and since $r$ divides $|z|$, Lemma~\ref{lem:technical} implies that $A$ acts irreducibly on $U$. Write $H_0 = H \cap T$. Then $H_0 = K \cap K^\t \cap K^{\t^2}$, where $K$ is the stabiliser in $T$ of a nondegenerate minus-type $2$-space of $\F_q^8$. Since $z$ stabilises a unique such subspace, $K$ is the unique $T$-conjugate of $K$ containing $z$. Therefore, $H$ is the unique $G$-conjugate of $H$ containing $z$.

Next let $H = G_2(q) \times C_3$. By \cite[Table~II]{ref:Kleidman88G2}, any element of $G_2(q)$ with order $(q^2-q+1)/(q^2-q+1,3)$ and centraliser in $G_2(q)$ of order $(q^2-q+1)/d$ for some $d \in \{1,(q^2-q+1,3)\}$, in fact has a centraliser of order $q^2-q+1$. Let $M$ be the number of $H$-classes that $z^G \cap H$ splits into. By consulting \cite{ref:Chang68,ref:Enomoto69}, we see that there are at most $(q^2-q)/6$ classes in $G_2(q)$ of elements whose centraliser has order $q^2-q+1$. In addition, by arguing as in \cite[Lemma~4.5]{ref:LawtherLiebeckSeitz02}, $z^G \cap H$ splits into at most $|W(D_4)/W(G_2)|=16$, classes. Therefore, $M \leq \min\left\{ (q^2-q)/6, 16 \right\}$ and the number of $G$-conjugates of $H$ that contain $z$ is 
\[
\frac{|G|}{|H|}\frac{|z^G \cap H|}{|z^G|} \leq M \frac{|C_G(x)|}{|C_H(x)|} = M \frac{3(q^3+1)(q+1)}{3(q^2-q+1)} = M(q+1)^2.
\]

Now assume that $q \equiv 2 \mod{3}$ and $H = \PGU_3(q) \times C_3$. The elements in $\PGU_3(q)$ of order $q^2-q+1$ act irreducibly on the natural module $\F_q^3$ and have centraliser in $\PGU_3(q)$ of order $q^3+1$. Each $\PGU_3(q)$-class of such elements corresponds to an orbit under $\l \mapsto \l^q$ on the set $\Lambda$ of elements of $\F_{q^6}^\times$ of order $q^2-q+1$. Since each of these orbits has size three, there are at most $(q^2-q)/3$ such classes. Therefore, we the number of $G$-conjugates of $H$ that contain $z$ is
\[
\frac{|G|}{|H|}\frac{|z^G \cap H|}{|z^G|} \leq \frac{q^2-q}{3} \frac{|C_G(x)|}{|C_H(x)|} = \frac{q^2-q}{3} \frac{3(q^3+1)(q+1)}{3(q^3+1)} = (q^3-q)/3.
\]

Let $x \in G$ have prime order. By \cite[Theorem~7.1]{ref:GuralnickSaxl03}, if $H \leq G$ has type $G_2(q)$ or $\PGU_3(q)$, then $x \in H \cap T$ only if $\nu(x) \geq 3$. Therefore, by Proposition~\ref{prop:fpr_ns_o},
\[
P(x,t\t) < \frac{2}{q^{12/5}} + M(q+1)^2 \frac{2}{q^{9/2}} + \d_{2,(q \mod{3})}\frac{q^3-q}{3} \frac{2}{q^{9/2}} < \frac{1}{2}
\]
and $P(x,t\t) \to 0$ and $q \to \infty$.
\end{proof}

Combining Propositions~\ref{prop:o_Ia}, \ref{prop:o_Ia_4_plus}, \ref{prop:o_Ib} and \ref{prop:o_Ib_57} in Case~I, Propositions~\ref{prop:o_IIa} and~\ref{prop:o_IIb} in Case~II, and Propositions~\ref{prop:o_IIIa}, \ref{prop:o_IIIb} and~\ref{prop:o_IIIc} in Case~III, establishes Theorems~\ref{thm:o_main} and~\ref{thm:o_asymptotic}.

\chapter{Linear and Unitary Groups} \label{c:u}

\section{Introduction} \label{s:u_intro}

In this final chapter we complete the proof of Theorems~\ref{thm:main} and~\ref{thm:asymptotic} by considering the unitary groups. Write $q=p^f$ and
\begin{gather}
\T_- = \{ \PSU_n(q) \mid \text{$n \geq 3$ and $(n,q) \not\in (3,2)$} \} \label{eq:u_t_minus} \\
\A_- = \{ \< T, \th \> \mid \text{$T \in \T$ and $\th \in \Aut(T)$} \}. \label{eq:u_a_minus}
\end{gather}
The subscript $-$ in this notation will be explained in \eqref{eq:u_t_ep} and \eqref{eq:u_a_ep}. Note that we exclude the group $\PSU_3(2)$ from $\T$ since it is isomorphic to $3^2.Q_8$.

We now present the main theorems of this chapter.

\begin{thmchap}\label{thm:u_main}
If $G \in \A_-$, then $u(G) \geq 2$.
\end{thmchap}

\begin{thmchap}\label{thm:u_asymptotic}
Let $(G_i)$ be a sequence of groups in $\A_-$ with $\soc(G_i) = \PSU_{n_i}(q_i)$. Then $u(G_i) \to \infty$ if $q_i \to \infty$.
\end{thmchap}

This chapter is organised similarly to Chapter~\ref{c:o}. We partition our proof of Theorems~\ref{thm:u_main} and~\ref{thm:u_asymptotic} into two cases
\begin{enumerate}[I]
\item $\th \in \PGaU_n(q) \setminus \<\PGU_n(q),\gamma\>$
\item $\th \in \<\PGU_n(q),\gamma\>$ 
\end{enumerate}
where $\gamma$ is the standard involutory graph automorphism of $\PGU_n(q)$.

As in Chapter~\ref{c:u}, we define two subcases of both Cases~I and~II
\begin{enumerate}[(a)]
\item $G \cap \<\PGU_n(q),\gamma\> \leq \PGU_n(q)$
\item $G \cap \<\PGU_n(q),\gamma\> \not\leq \PGU_n(q)$.
\end{enumerate}

As we explain in Remark~\ref{rem:linear}, one case in the proof of \cite[Theorem~2]{ref:BurnessGuest13} was omitted, and we take the opportunity to provide the proof of this case. That is we prove the following.

\begin{thmchap}\label{thm:linear}
Let $T=\PSL_n(q)$, where $n$ is even and $q=p^f$ is odd. Let $\th \in \PGL_n(q)\gamma\p^i$, where $\p$ is the standard field automorphism of $T$, $\g$ is the standard graph automorphism of $T$ and $i$ is a proper divisor of $f$ such that $f/i$ is odd. Then $u(\<T,\th\>) \geq 2$ and $u(\<T,\th\>) \to \infty$ if $q \to \infty$.
\end{thmchap}

We proceed as in Chapter~\ref{c:o}. Sections~\ref{s:u_cases} and~\ref{s:u_elements} establish general information about automorphisms and elements of linear and unitary groups. We then prove Theorems~\ref{thm:u_main} and~\ref{thm:u_asymptotic} in Cases~I and~II in Sections~\ref{s:u_I} and~\ref{s:u_II}, respectively, and Theorem~\ref{thm:linear} in Section~\ref{s:linear}.

\section{Automorphisms} \label{s:u_cases}

The aim of this section is to determine the automorphisms $\th \in \Aut(\PSU_n(q))$ it is sufficient to consider to prove Theorems~\ref{thm:u_main} and~\ref{thm:u_asymptotic}. It will be convenient to simultaneously consider $\PSL^+_n(q) = \PSL_n(q)$ and $\PSL^-_n(q) = \PSU_n(q)$, where $q=p^f$ and $n \geq 3$ in both cases. 

Write $V = \F_{q^d}^n$ where
\begin{equation}
d = \left\{ 
\begin{array}{ll}
1 & \text{if $\e=+$} \\
2 & \text{if $\e=-$.}
\end{array}
\right.
\end{equation}

Let $\B = (v_1,\dots,v_n)$ be a basis for $V$, and if $\e=-$, then assume that $\B$ is the basis from \eqref{eq:B_u_ef}. 

Recall from Definition~\ref{def:phi_gamma_r}, the standard Frobenius endomorphism defined as $\p = \p_\B\: (x_{ij}) \mapsto (x_{ij}^p)$ and the standard graph automorphism $\g = \g_\B\: x \to (x^{-\tr})^J$, where $J$ is the antidiagonal matrix with entries $1,-1,1,-1, \dots, (-1)^{n+1}$.

By \cite[Tables~4.3.1 and~4.5.1]{ref:GorensteinLyonsSolomon98}, 
\begin{align}
C_{\PGL_n(\FF_p)}(\gamma) &= \left\{
\begin{array}{ll}
\PGSp_n(\FF_p) & \text{if $n$ is even}                \\
\PSO_n(\FF_p)  & \text{if $n$ is odd}                 \\
\end{array}
\right. \label{eq:u_graph_centraliser_algebraic} \\
C_{\PGL^\e_n(q)}(\gamma)  &= \left\{
\begin{array}{ll}
\PGSp_n(q)   & \text{if $n$ is even}                 \\
\PSO_n(q)    & \text{if $n$ is odd and $q$ is odd}   \\
\Sp_{n-1}(q) & \text{if $n$ is odd and $q$ is even.} \\
\end{array}
\right. \label{eq:u_graph_centraliser_finite}
\end{align}

Let $\alpha_\e \in \F_{q^2}^\times$ satisfy $|\alpha_\e| = q-\e$. We define one further element 

\begin{definitionx}\label{def:u_delta}
Let $\hat{\d}_\e \in \GL^\e_n(q)$ be the antidiagonal matrix with entries $(-1)^{\left\lfloor\frac{n}{2}\right\rfloor}\a_\e,1,1,\dots,1$ (from top-right to bottom-left), written with respect to $\B$. Let $\d_\e \in \PGL^\e_n(q)$ be the image of $\hat{\d}_\e$. If $\e$ is understood, then we write $\d = \d_\e$.
\end{definitionx}

\begin{remarkx}\label{rem:u_delta}
It is easy to check that $\det(\hat{\d}_\e) = \a_\e$ and $\d_\e^\g = \d_\e^{-1}$.
\end{remarkx}

As in Chapter~\ref{c:o}, for $g \in \Aut(T)$, we write $\ddot{g} = Tg \in \Out(T)$. From \cite[Proposition~2.3.5]{ref:KleidmanLiebeck}, we have the outer automorphism groups
\begin{equation}\label{eq:l_out}
\Out(\PSL_n(q)) = \< \ddot{\d}, \ddot{\g}, \ddot{\p} \>  = C_{(n,q-1)}{:}(C_2 \times C_f)
\end{equation}
where $|\ddot{\d}|=(n,q-1)$, $|\ddot{\g}|=2$, $|\ddot{\p}|=f$, $\ddot{\d}^{\ddot{\g}}=\ddot{\d}^{-1}$ and $\ddot{\d}^{\ddot{\p}} = \ddot{\d}^p$, and
\begin{equation}\label{eq:u_out}
\Out(\PSU_n(q)) = \<\ddot{\d}, \ddot{\p}\> = C_{(n,q+1)}{:}C_{2f} 
\end{equation}
where $|\ddot{\d}|=(n,q+1)$, $|\ddot{\p}|=2f$ and $\ddot{\d}^{\ddot{\p}} = \ddot{\d}^p$. 

We now present two similar lemmas that will be crucial to our case analysis.

\begin{lemmax} \label{lem:u_out_facts}
Let $T = \PSU_n(q)$ and let $i$ divide $f$. Then the following hold
\begin{enumerate}
\item if $(n,q+1)$ is odd, then $(\ddot{\p}^i)^{\Out(T)} = \<\ddot{\d}\>\ddot{\p}^i$
\item if $(n,q+1)$ is even, then $(\ddot{\p}^i)^{\Out(T)} = \<\ddot{\d}^2\>\ddot{\p}^i$ and $(\ddot{\d}\ddot{\p}^i)^{\Out(T)} = \<\ddot{\d}^2\>\ddot{\d}\ddot{\p}^i$, so in particular $\<\ddot{\d}\>\ddot{\p}^i$ is the disjoint union $(\ddot{\p}^i)^{\Out(T)} \cup (\ddot{\d}\ddot{\p}^i)^{\Out(T)}$.
\end{enumerate}
\end{lemmax}

\begin{proof}
Begin by observing that $(\ddot{\p}^i)^{\Out(T)} \subseteq \<\ddot{\d}\>\ddot{\p}^i$. Now let $j$ be a divisor of $(n,q+1)$ and note that $(\ddot{\d}^j)^{\ddot{\p}^i} = \ddot{\d}^{jp^i}$ if and only if $(n,q+1)$ divides $(p^i-1)j$. It is easy to see that $(q+1,p^i-1)=(p-1,2)$.  

For (i), let us assume that $(n,q+1)$ is odd. Then $(n,q+1)$ and $p^i-1$ are coprime, so $(n,q+1)$ divides $(p^i-1)j$ if and only if $(n,q+1)$ divides $j$, that is, when $\ddot{\d}^j = 1$. Consequently, $C_{\Out(T)}(\ddot{\p}^i) = \<\ddot{\p} \>$, so $(\ddot{\p}^i)^{\Out(T)} = \<\ddot{\d}\>\ddot{\p}^i$.

For (ii), we now assume that $(n,q+1)$ is even. In this case, $((n,q+1),p^i-1)=2$, so $(n,q+1)$ divides $(p^i-1)j$ if and only if $(n,q+1)/2$ divides $j$. Therefore, $C_{\Out(T)}(\ddot{\p}^i) = \< \ddot{\p}, \ddot{\d}^{(n,q+1)/2}\>$, so $\<\ddot{\d}\>\ddot{\p}^i$ must consist of exactly two $\Out(T)$-classes of equal size.

Let $h \in \<\ddot{\d}^2\>\ddot{\p}^i$ and write $h = \ddot{\d}^{2j}\ddot{\p}^i$. Note that $h^{\ddot{\p}^k} = \ddot{\d}^{2jp^k}\ddot{\p}^i \in \<\ddot{\d}^2\>\ddot{\p}^i$ and $h^{\ddot{\d}^k} = \ddot{\d}^{2j+k(p^i-1)}\ddot{\p}^i \in \<\ddot{\d}^2\>\ddot{\p}^i$. Therefore, $h^{\Out(T)} \in \<\ddot{\d}^2\>\ddot{\p}^i$. This implies that $\<\ddot{\d}^2\>\ddot{\p}^i$ is a union of conjugacy classes. However, since $|\ddot{\d}|=(n,q+1)$ is even, $\<\ddot{\d}\>\ddot{\p}^i$ is the disjoint union of $\<\ddot{\d}^2\>\ddot{\p}^i$ and $\<\ddot{\d}^2\>\ddot{\d}\ddot{\p}^i$, so these must be the two $\Out(T)$-classes in $\<\ddot{\d}\>\ddot{\p}^i$. Therefore,  $(\ddot{\p}^i)^{\Out(T)} = \<\ddot{\d}^2\>\ddot{\p}^i$ and $(\ddot{\d}\ddot{\p}^i)^{\Out(T)} = \<\ddot{\d}^2\>\ddot{\d}\ddot{\p}^i$.
\end{proof}

\begin{lemmax} \label{lem:l_out_facts}
Let $T = \PSL_n(q)$ and let $i$ divide $f$ and assume that $f/i$ is odd. Then the following hold
\begin{enumerate}
\item if $(n,q-1)$ is odd, then $(\ddot{\g}\ddot{\p}^i)^{\Out(T)} = \<\ddot{\d}\>\ddot{\g}\ddot{\p}^i$
\item if $(n,q-1)$ is even, then we have the classes $(\ddot{\g}\ddot{\p}^i)^{\Out(T)} = \<\ddot{\d}^2\>\ddot{\g}\ddot{\p}^i$ and $(\ddot{\d}\ddot{\g}\ddot{\p}^i)^{\Out(T)} = \<\ddot{\d}^2\>\ddot{\d}\ddot{\g}\ddot{\p}^i$, so in particular $\<\ddot{\d}\>\ddot{\g}\ddot{\p}^i$ is the disjoint union $(\ddot{\g}\ddot{\p}^i)^{\Out(T)} \cup (\ddot{\d}\ddot{\g}\ddot{\p}^i)^{\Out(T)}$.
\end{enumerate}
\end{lemmax}

\begin{proof}
We argue just as in the proof of Lemma~\ref{lem:u_out_facts}. First note $(\ddot{\g}\ddot{\p}^i)^{\Out(T)} \subseteq \<\ddot{\d}\>\ddot{\g}\ddot{\p}^i$. Now let $j$ divide $(n,q-1)$ and note that $(\ddot{\d}^j)^{\ddot{\g}\ddot{\p}^i} = \ddot{\d}^{-jp^i}$ if and only if $(n,q-1)$ divides $(p^i+1)j$. Note that $(q-1,p^i+1)=(p-1,2)$.  

First assume that $(n,q-1)$ is odd. Then $(n,q-1)$ and $p^i+1$ are coprime, so $(n,q-1)$ divides $(p^i+1)j$ if and only if $(n,q-1)$ divides $j$, so $C_{\Out(T)}(\ddot{\g}\ddot{\p}^i) = \< \ddot{\g}, \ddot{\p} \>$ and we conclude that $(\ddot{\g}\ddot{\p}^i)^{\Out(T)} = \<\ddot{\d}\>\ddot{\g}\ddot{\p}^i$.

Next assume that $(n,q-1)$ is even. Here $((n,q-1),p^i+1)=2$, so $(n,q-1)$ divides $(p^i+1)j$ if and only if $(n,q-1)/2$ divides $j$, so $C_{\Out(T)}(\ddot{\p}^i) = \< \ddot{\g}, \ddot{\p}, \ddot{\d}^{(n,q-1)/2}\>$, and $\<\ddot{\d}\>\ddot{\g}\ddot{\p}^i$ consists of exactly two $\Out(T)$-classes.

By arguing as we did in the proof of Lemma~\ref{lem:u_out_facts}, it is straightforward to show that $\<\ddot{\d}^2\>\ddot{\g}\ddot{\p}^i$ is a union of conjugacy classes. Therefore, $(\ddot{\g}\ddot{\p}^i)^{\Out(T)} = \<\ddot{\d}^2\>\ddot{\g}\ddot{\p}^i$ and $(\ddot{\d}\ddot{\g}\ddot{\p}^i)^{\Out(T)} = \<\ddot{\d}^2\>\ddot{\d}\ddot{\g}\ddot{\p}^i$.
\end{proof}

\begin{remarkx} \label{rem:u_delta_2}
Assume that $(n,q-\e)$ is even. Recall that $\d \in \PGL^\e_n(q)$ lifts to an element $\hat{\d} \in \GL^\e_n(q)$ of order $q-\e$. Now $|\d| = (n,q-\e)$ and we define 
\[
\d_2 = \d^{\frac{q-\e}{(q-\e)_2}}.
\] 
\begin{enumerate}
\item Note that $|\d_2| = |\d|_2 = (n,q-\e)_2$.
\item Since $\d^\g = \d^{-1}$ (see Remark~\ref{rem:u_delta}), we have $\d_2^\g = \d_2^{-1}$, so $|\d_2\g|=2$.
\item As $\ddot{\d}_2$ is not a square in $\<\ddot{\d}\>$, in view of Lemmas~\ref{lem:u_out_facts} and~\ref{lem:l_out_facts}, the following hold
\begin{enumerate}[(a)]
\item $(\ddot{\d}_2\ddot{\p}^i)^{\Out(T)}          = (\ddot{\d}\ddot{\p}^i)^{\Out(T)}$          if $\e=-$ and $i$ divides $f$
\item $(\ddot{\d}_2\ddot{\g}\ddot{\p}^i)^{\Out(T)} = (\ddot{\d}\ddot{\g}\ddot{\p}^i)^{\Out(T)}$ if $\e=+$ and $i$ divides $f$ with $f/i$ odd.
\end{enumerate}
\item As $\ddot{\g}$ and $\ddot{\d}\ddot{\g}$ are not $\Out(T)$-conjugate $C_{\PGL_n(\FF_p)}(\d_2\g) = \PGO_n(\FF_p)$ and $C_{\PGL_n(q)}(\d_2\g) = \PGO^\eta_n(q)$ with $\eta = (-)^{\frac{n(q-\e)}{4}+1}$ (see \cite[Table~4.5.1]{ref:GorensteinLyonsSolomon98}).
\end{enumerate}
\end{remarkx}

We will now determine the almost simple linear and unitary groups it is sufficient to consider to prove our main theorems.

For a sign $\e \in \{+,-\}$, write
\begin{gather}
\T_\e = \{ \PSL^\e_n(q) \mid \text{$n \geq 3$ and $T \neq \PSU_3(2)$} \} \label{eq:u_t_ep} \\
\A_\e = \{ \< T, \th \> \mid \text{$T \in \T$ and $\th \in \Aut(T)$} \}, \label{eq:u_a_ep}
\end{gather}
noting that this agrees with the definition of $\T_-$ and $\A_-$ in \eqref{eq:u_t_minus} and \eqref{eq:u_a_minus}. 

\begin{propositionx}\label{prop:u_cases}
Let $G \in \A_\e$ with $\soc(G) = T = \PSL^\e_n(q)$. Then $G$ is $\Aut(T)$-conjugate to $\<T,\th\>$ for one of the following
\begin{enumerate}
\item $\th$ in Row~(1) of Table~\ref{tab:u_cases}
\item $\th$ in Row~(2) of Table~\ref{tab:u_cases}, if $q$ is odd and $n$ is even.
\end{enumerate}
\end{propositionx}

\begin{table}
\centering
\caption{The relevant automorphisms $\th$ in when $T=\PSL^\e_n(q)$} \label{tab:u_cases}
{\renewcommand{\arraystretch}{1.2}
\begin{tabular}{ccccccccc}
\cline{1-8}    
           &                  &                 &                 & I(a)            & I(b)           & II(a)          & II(b)    &     \\
\cline{1-8} 
$\e$       & $+$              & $+$             & $+$             & $-$             & $-$            &                &          &     \\
\cline{1-8}  
$\th$      & $\d^\ell\p^i$    & $\d^\ell\g\p^i$ & $\g\p^i$        & $\d^\ell\p^i$   & $\p^i$         & $\d^\ell$      & $\g$     & (1) \\
           &                  &                 & $\d_2\g\p^i$    &                 & $\d_2\p^i$     &                & $\d_2\g$ & (2) \\
\cline{1-8} 
$df/i$     & any              & even            & odd             & odd             & even           &                &          &     \\
\cline{1-8} 
\end{tabular}} 
\\[5pt]
{ \small Note: $i$ is a proper divisor of $df$ and $0 \leq \ell \leq (n,q-\e)$}
\end{table}

\begin{proof}
Write $G = \<T,g\>$ where $g \in \Aut(T)$. We first consider $T=\PSU_n(q)$. From the description of $\Out(T)$, we see that we may write $\ddot{g} = \ddot{\d}^\ell\ddot{\p}^i$ where $0 \leq \ell < (n,q+1)$ and $0 \leq i < 2f$. By Lemma~\ref{lem:division} we may assume that $i=0$ or $i$ divides $2f$. If $i=0$, then $\ddot{g} = \ddot{\d}^\ell$ and we are in Case~II(a), and if $i > 0$ and $2f/i$ is odd, then we are in Case~I(a). Therefore, from now on, we will assume that $i$ divides $f$. 

First assume that $n$ is odd or $q$ is even, so $(n,q+1)$ is odd. By Lemma~\ref{lem:u_out_facts}(i), $\ddot{g} = \ddot{\d}^\ell\ddot{\p}^i$ is $\Out(T)$-conjugate to $\ddot{\p}^i$. If $i < f$, then we are in Case~I(b). If $i=f$, then noting that $\ddot{g} = \ddot{\p}^f = \ddot{\gamma}$, we see that we are in Case~II(b). 

Now assume that $n$ is even and $q$ is odd, so $(n,q+1)$ is even. In this case, Lemma~\ref{lem:u_out_facts}(ii) implies that $\ddot{g}$ is $\Out(T)$-conjugate to either $\ddot{\p}^i$ or $\ddot{\d}_2\ddot{\p}^i$, where $\d_2 = \d^{\frac{q+1}{(q+1)_2}}$. If $i < f$, then we are in Case~I(b). Since $\ddot{\p}^f = \ddot{\gamma}$ and $\ddot{\d}_2\ddot{\p}^f = \ddot{\d}_2\ddot{\gamma}$, if $i=f$, then we deduce that we are in Case~II(b). This completes the proof for $T = \PSU_n(q)$.

It remains to consider $T = \PSL_n(q)$. As usual, we may assume that $\ddot{g} = \ddot{h}\ddot{\p}^i$ where $h$ is a product of diagonal and graph automorphisms and where either $i=0$ or $i$ divides $f$. We claim that there is an automorphism $\th$ in the statement such that $\ddot{g} = \ddot{h}\ddot{\p}^i$ is $\Out(T)$-conjugate to $\ddot{\th}$. This is clear if $h$ is diagonal or $f/i$ is even. Therefore, assume that $\ddot{h}=\ddot{\d}^\ell\ddot{\g}$ and $f/i$ is odd. Then Lemma~\ref{lem:l_out_facts} implies that $\ddot{g}$ is $\Out(T)$-conjugate to $\ddot{\g}\ddot{\p}^i$ or, if $(n,q-1)$ is even, $\ddot{\d}_2\ddot{\g}\ddot{\p}^i$ where $\d_2 = \d^{\frac{q-1}{(q-1)_2}}$, as required.
\end{proof}

\clearpage
\begin{remarkx}\label{rem:u_cases}
While we do not require this information, it is easy to check that the automorphisms $\th$ given in Proposition~\ref{prop:u_cases} when $\e=+$ and $f/i \geq 1$ is odd or $\e=-$ and $2f/i \geq 2$ is even are pairwise not $\Out(T)$-conjugate, and the $\Out(T)$-classes in the remaining cases are as follows
\[
{\renewcommand{\arraystretch}{1.2}
\begin{array}{cccc}
\hline
\e & g                    & \text{conditions} & \text{classes in $\Outdiag(T)g$} \\
\hline
+  & \ddot{\p}^i          & \text{none}       & \text{$\{\ddot{\d}^j,\ddot{\d}^{-j}, \dots,\ddot{\d}^{jp^{f-1}}, \ddot{\d}^{-jp^{f-1}}\}\<\ddot{\d}^{p^i-1}\>$ for $0 \leq j \leq \frac{(n,k)}{n}\frac{p^i-1}{2}$} \\
   & \ddot{\g}\ddot{\p}^i & \text{$f/i$ even} & \text{$\{\ddot{\d}^j,\ddot{\d}^{-j}, \dots,\ddot{\d}^{jp^{f-1}}, \ddot{\d}^{-jp^{f-1}}\}\<\ddot{\d}^{p^i+1}\>$ for $0 \leq j \leq \frac{(n,k)}{n}\frac{p^i+1}{2}$} \\
-  & \ddot{\p}^i          & \text{$2f/i$ odd} & \text{$\{\ddot{\d}^j, \dots, \ddot{\d}^{jp^{f-1}}\}\<\ddot{\d}^{p^i+1}\>$                                      for $0 \leq j \leq \frac{(n,k)}{n}\frac{p^i+1}{2}$} \\
\hline
\end{array}
}
\]
\vspace{5.5pt}
\end{remarkx}

\begin{remarkx}\label{rem:linear}
The main result of \cite{ref:BurnessGuest13} is that $u(G) \geq 2$ for all almost simple linear groups $G \in \A_+$. Referring to Table~\ref{tab:u_cases}, the automorphisms in columns 1, 2, 3, 6, 7 are considered in Sections~4, 5.1, 5.2, 3, 6 of \cite{ref:BurnessGuest13}, respectively. When $q$ is odd and $n$ is even, in Section~6, the authors consider both $\gamma$ and $\d_2\gamma$, but, in Section~5.2 where $f/i$ is odd, only $\p^i$ is considered, since it was claimed that $\ddot{\g}\ddot{\p}^i$ and $\ddot{\d}_2\ddot{\g}\ddot{\p}^i$ were $\Out(T)$-conjugate, but we know that this does not hold by Lemma~\ref{lem:l_out_facts}. The basis of this claim was \cite[Theorem~7.2]{ref:GorensteinLyons83}, which states that for a finite simple group of Lie type $K$ and a field or graph-field automorphism $\phi \in \Aut(K)$, if $\phi' \in \phi\Inndiag(K)$ has the same order as $\phi$, then $\phi$ and $\phi'$ are conjugate under $\Inndiag(K)$. However, this statement is false for general elements of composite order, as the example of $\g\p^i$ and $\d_2\g\p^i$ in $\Aut(\PSL_n(q))$ when $f/i$ is odd highlights. Since $\ddot{\d}_2\ddot{\g}\ddot{\p}^i$ is not $\Out(T)$-conjugate to $\ddot{\g}\ddot{\p}^i$, the group $\<T, \d_2\g\p^i\>$ is not $\Aut(T)$-conjugate to, and hence not isomorphic to \cite[Lemma~3]{ref:BrayHoltRoneyDougal09}, $\<T,\g\p^i\>$, so we must consider this case. In proving Theorem~\ref{thm:linear} in Section~\ref{s:linear}, we do exactly this.
\end{remarkx}

As in Chapter~\ref{c:o}, we can deal with some small cases computation. More precisely, via computation in \textsc{Magma} (see Section~\ref{s:p_computation}) we prove the following.

\begin{propositionx}\label{prop:u_computation}
Let $G \in \A_-$ with socle $\PSU_n(q)$. Then $u(G) \geq 2$ if $n \in \{3,4\}$ and $q \leq 9$, $n \in \{5,6\}$ and $q \leq 4$ or $n \in \{7,8\}$ and $q \leq 3$.
\end{propositionx}

\clearpage
\section{Elements} \label{s:u_elements}

In Section~\ref{s:o_elements} we introduced types of elements of symplectic and orthogonal groups. We now define types of semisimple elements in unitary groups. 

For this section, write $V = \F_{q^2}^n$ where $n \geq 3$ and $q=p^f$, and let $\a = \a_{-} \in \F_{q^2}^\times$ have order $q+1$.

\begin{definitionx} \label{def:elt_u}
Let $n \geq 3$ be odd. An element $g \in \GU_n(q)$ has \emph{type $[n]^-_q$} if $V$ is an irreducible $\F_{q^2}\<g\>$-module and $g$ has order $q^n+1$ and determinant $\alpha$.
\end{definitionx}

\begin{lemmax}
Let $n \geq 3$ be odd. Then $\GU_n(q)$ has an element of type $[n]^-_q$. 
\end{lemmax}

\begin{proof}
Fix a field extension embedding $\pi\:\GU_1(q^n) \to \GU_n(q)$. Note that
\[
\GU_1(q^n) = \{ (\mu) \in \GL_1(q^{2n}) \mid \mu^{q^n+1} = 1 \} \cong C_{q^n+1}.
\]
Let $N\:\F_{q^{2n}}^\times \to \F_{q}^\times$ be the norm map. Let $H \leq \F_{q^{2n}}^\times$ have order $q^n+1$. Since $n$ is odd, $N(H)$ has order 
\[
(q^{n}+1)/\left(\frac{q^{2n}-1}{q^2-1},q^n+1\right) = q+1,
\]
so there exists a generator $\l$ of $H$ such that $N(\l) = \alpha$, and the element $g=\pi((\l))$ has order $q^n+1$. Now the determinant of $g$ is $N(\l) = \alpha$ and $\l$ is an eigenvalue $g$. Therefore, Lemma~\ref{lem:technical} implies that $g$ is irreducible on $V$, so $g$ has type $[n]^-_q$. 
\end{proof}

Let $n=2m$ be even. Then $V$ admits a decomposition $\mathcal{D}(V)$ 
\[
V = V_1 \oplus V_2 \quad \text{where} \quad V_1 = \<e_1,\dots,e_m\> \quad \text{and} \quad V_2 = \<f_1,\dots,f_m\>,
\]
noting that $V_1$ and $V_2$ are totally singular $m$-spaces. The centraliser of $\mathcal{D}(V)$ is
\begin{equation}
(\GU_{2m}(q))_{(\mathcal{D}(V))} = \{ (g_{ij}) \oplus (g_{ij}^q)^{-\tr} \mid g = (g_{ij}) \in \GL_m(q^2) \}.
\end{equation}

\begin{definitionx}
Let $n=2m \geq 4$. An element $g \in \GU_n(q)$ has \emph{type $[n]^+_q$} if $g$ has order $q^n-1$, determinant $\alpha$ and centralises a decomposition $V = V_1 \oplus V_2$ where $V_1$ and $V_2$ are totally singular subspaces on which $g^i$ acts irreducibly for all divisors $i$ of $q+1$. 
\end{definitionx}

\begin{lemmax}
Let $n \geq 4$ be even. Then $\GU_n(q)$ has an element of type $[n]^+_q$.
\end{lemmax}

\begin{proof}
Fix a field extension embedding $\pi\:\GL_1(q^{2m}) \to \GL_m(q^2)$ where we write $n=2m$. Since $|\alpha|=q+1$, we may write $\alpha = \mu^{1-q}$ for a generator $\mu$ of $\F_{q^2}^\times$. Let $N\:\F_{q^{2m}}^\times \to \F_{q^2}^\times$ be the norm map and let $\l \in \F_{q^{2m}}^\times$ satisfy $N(\l) = \mu$.

Let $g=\pi((\l)) \oplus \pi((\l))^{-(q)\tr}$, and note that $g$ has order $q^n-1$. The determinant of $\pi((\l))$ is $N(\l) = \mu$, so the determinant of $g$ is $\mu\mu^{-q} = \alpha$. 

Let $i$ divide $q+1$. Now $\l^i$ is an eigenvalue of $\pi(\l)^i$, and since $|\l^i| = (q^{2m}-1)/i$ is a primitive divisor of $q^{2m}-1$, Lemma~\ref{lem:technical} implies that $\pi(\l)^i$ acts irreducibly on $V_1$, and hence $g^i$ acts irreducibly on both $V_1$ and $V_2$. Therefore, $g$ has type $[n]^+_q$.
\end{proof}

The following proof is based on the arguments in \cite[Chapter~3]{ref:Brookfield14}.

\begin{lemmax} \label{lem:brookfield}
Let $n=2m \geq 4$. Let $g \in \GU_{2m}(q)$ have type $[n]_q^+$, centralising the decomposition $V = V_1 \oplus V_2$. For a divisor $i$ of $q+1$, the only $\F_{q^2}\<g^i\>$-submodules of $V$ are $0$, $V_1$, $V_2$ and $V$.
\end{lemmax}

\begin{proof}
Evidently, it suffices to prove the lemma when $i=q+1$, so write $h = g^{q+1}$. If $m$ is even, then $V_1$ and $V_2$ are nonisomorphic, since $h$ has different eigenvalues on these two submodules. Since $V_1$ and $V_2$ are irreducible $\F_q\<h\>$-modules, the result follows from Lemma~\ref{lem:c1}. Therefore, for the remainder of the proof we will assume that $m$ is odd.

Consider $V_1$ and $V_2$ as copies of $\F_{q^m}$ where the action of $h$ on $V$ is given as $(v_1,v_2)h = (\mu v_1, \mu^{-q} v_2)$ for some $\mu \in \F_{q^{2m}}$ of order $(q^{2m}-1)/(q+1)$. 

For a contradiction, let $0 < U < V$ be an $\F_{q^2}\<h\>$-submodule different from $V_1$ and $V_2$. In particular, $U \cap V_1 = U \cap V_2 = 0$. Therefore, by Lemma~\ref{lem:goursat}, $U$ is isomorphic to $V_1$ and $V_2$, so, in particular, $U$ is $m$-dimensional. This means that for all $v_1 \in V_1$, there exists a unique $v_2 \in V_2$ such that $(v_1,v_2) \in U$. In this way, we can define a map $L\: V_1 \to V_2$ as $L(v_1) = v_2$ where $(v_1,v_2) \in U$.

Fix $(1,u) \in U$. Since $m$ is odd, $q+1$ divides $q^m+1$, so $\F_{q^m}^\times \leq \<\mu\>$. Let $\<\zeta\> = \F_{q^m}^\times$, noting that $\zeta \neq -1$ since $m \geq 2$. Now $(\zeta,\zeta^{-q}u) \in U$, so $(1+\zeta,(1+\zeta^{-q})u) \in U$. Now $1+\zeta \in \F_{q^m} \leq \<\mu\>$, so $(1+\zeta,(1+\zeta)^{-q}u) \in U$. Therefore, $1+\zeta^{-q} = (1+\zeta)^{-q} = (1+\zeta^{q})^{-1}$. This implies that $1+\zeta+\zeta^2=1$. Therefore $\zeta^3=1$. However, $|\zeta|=q^m-1$, so $q^m-1 \leq 3$, which implies that $q=m=2$. We can check that $g \in \SU_4(2)$ of order $(4^2-1)/(2+1) = 5$ does not stabilise such a subspace $U$, which gives a contradiction.
\end{proof}

\clearpage
\section{Case I: semilinear automorphisms} \label{s:u_I}

In this section, we prove Theorems~\ref{thm:u_main} and~\ref{thm:u_asymptotic} in Case~I. To this end, write $G=\<T,\th\> \in \A_-$ where $T=\PSU_n(q)$ and $\th \in \Aut(T) \setminus \Inndiag(T)$.

We separate into two cases, which will be considered in Sections~\ref{ss:u_IIa} and~\ref{ss:u_IIb}:
\begin{enumerate}[(a)]
\item $G \cap \<\PGU_n(q),\gamma\> \leq \PGU_n(q)$
\item $G \cap \<\PGU_n(q),\gamma\> \not\leq \PGU_n(q)$.
\end{enumerate}

\subsection{Case I(a)} \label{ss:u_Ia}

As in Case~I of Chapter~\ref{c:o}, Shintani descent (see Chapter~\ref{c:shintani}) is the central tool in the identification of the element $t\th$. Consequently, we need to fix our notation relating to Shintani descent for Case~I(a). \vspace{5pt}

\begin{shbox}
\begin{notationx} \label{not:u_Ia}
\begin{enumerate}[label={}, leftmargin=0cm, itemsep=3pt]
\item Write $q=p^f$ where $f \geq 2$. Let $V = \F_{q^2}^n$. 
\item Fix the basis $\B$ in \eqref{eq:B_u_ef}.
\item Fix the simple algebraic group $X = \PSL_n(\FF_p)$. 
\item Fix the Frobenius endomorphism $\p = \p_\B$, the standard graph automorphism $\g=\g_\B$ and the antidiagonal element $\d = \d_-$ (see Definitions~\ref{def:phi_gamma_r} and~\ref{def:u_delta}).
\item Fix $\a = \a_- \in \F_{q^2}^\times$ satisfying $|\a| = q+1$.
\end{enumerate}
\end{notationx}
\end{shbox} \vspace{5pt}

Our approach is like that for minus-type orthogonal groups in Section~\ref{ss:o_Ia}. By Proposition~\ref{prop:u_cases}, we can assume that $\th \in \PGU_n(q)\p^i$ where $2f/i$ is odd. Therefore, $i$ is even and for $j=i/2$ we have $2f/(2f,f+j) = 2f/(2f,i)$, so we will work with $\th = \th_0\gamma\p^j$ for some $\th_0 \in \PGU_n(q)$, noting that $j$ divides $f$ and $f/j = 2f/i$ is odd. \vspace{5pt}

\begin{shbox}
\notacont{\ref{not:u_Ia}} Write $q=q_0^e$ where $e=f/j=2f/i$.
\begin{enumerate}[label={}, leftmargin=0cm, itemsep=3pt]
\item Let $\s = \gamma\p^j$. 
\item Let $F$ be the Shintani map of $(X,\s,e)$, so
\[
F\: \{ (g\s)^{X_{\s^e}} \mid g \in X_{\s^e} \} \to \{ x^{X_{\s}} \mid x \in X_{\s} \}.
\]
Note $X_{\s^e} = \PGU_n(q) = \Inndiag(T)$ and $X_\s = \PGU_n(q_0) = \Inndiag(T_0)$, where $T_0 = \PSU_n(q_0)$.
\item Fix the antidiagonal element $\d_0 \in \PGU_n(q_0)$ corresponding to $\d \in \PGU_n(q)$.
\item Let $\alpha_0 \in \F_{q_0^2}^\times$ satisfy $|\alpha_0|=q_0+1$ 
\end{enumerate}
\end{shbox} \vspace{5pt}

Let $N\:\F_{q^2} \to \F_{q_0^2}$ be the norm map. Then 
\[
|N(\a)| = (q+1)/\left(\frac{q^2-1}{q_0^2-1},q+1\right) = q_0+1
\]
noting that $e$ is odd. Since the normal closure of $\< \PSU_n(q), \p^i \>$ in $\<\PGU_n(q), \p^i\>$ is $\< \PSU_n(q), \d^{q_0+1}, \p^i\>$, there is a well-defined bijection between the normal unions of cosets of $\PSU_n(q)$ in $\PGU_n(q)\s$ and the (necessarily normal) cosets of $\PSU_n(q_0)$ in $\PGU_n(q_0)$. The following demonstrates that the Shintani map preserves this bijection (compare with \cite[Lemmas~4.2 and~5.3]{ref:BurnessGuest13}).

\begin{lemmax} \label{lem:det_u}
Let $x \in X_{\s}$. Write $\det(x) = \l^{1+q_0^2+\cdots+q_0^{2e-2}}$, where $\l \in \F_{q^2}^\times$ with $\l^{q+1}=1$. There exists $g \in X_{\s^e}$ such that $F((g\s)^{X_{\s^e}}) = x^{X_{\s}}$ and $\det(g) = \l$.
\end{lemmax}

\begin{proof}
There exists $h \in X_{\s^e}$ such that $F((h\s)^{X_{\s^e}}) = x^{X_{\s}}$. Now 
\[
\l^{1+q_0^2+\cdots+q_0^{2(e-1)}} = \det(x) = \det(a^{-1}(h\s)^ea) = \det(h)^{1+q_0^2+\cdots+q_0^{2(e-1)}}.
\]
Therefore, $\det(h) = \l \cdot \mu^{q_0^2-1}$ for some $\mu \in \F_{q^2}^\times$ such that $\mu^{q+1}=1$. Now let $z \in \PGU_n(q)=X_{\s^e}$ satisfy $\det(z)=\mu^{-q_0^2}$ and write $g = zhz^{-\s^{-1}}$. Then
\[
g\s = (zhz^{-\s^{-1}})\s = (h\s)^{z^{-1}} \in (h\s)^{X_{\s^e}},
\]
so $F((g\s)^{X_{\s^e}}) = F((h\s)^{X_{\s^e}}) = x^{X_{\s}}$, and $g \in X_{\s^e}$ satisfies the statement since
\[
\det(g) = \det(z)\det(h)\det(z)^{-q_0^{2e-2}} = \mu^{-q_0^2}\l\mu^{q_0^2-1}\mu = \l. \qedhere 
\]
\end{proof}

\begin{propositionx}\label{prop:u_Ia_elt}
Let $T \in \T_-$ and let $\th = \d^\ell\gamma\p^j$ where $1 \leq \ell \leq (q+1,n)$ and $j$ is a proper divisor of $f$. Write $y = y_0^\ell$, for the element $y_0 \in \PGU_n(q_0)$ in Table~\ref{tab:u_Ia_elt}. Then there exists $t \in T$ such that $(t\th)^e$ is $X$-conjugate to $y$.
\end{propositionx} 

\begin{table}
\centering
\caption{Case~I(a): The element $y_0$ for the automorphism $\th$} \label{tab:u_Ia_elt}
{\renewcommand{\arraystretch}{1.2}
\begin{tabular}{cc}
\hline   
$n$             & $y_0$ \\
\hline
$n \geq 7$ odd  & $J_2 \perp [n-2]^-_{q_0}$ \\ 
$n \geq 6$ even & $J_2 \perp [n-2]^+_{q_0}$ \\
$5$             & $[5]^-_{q_0}$             \\
$4$             & $I_1 \perp [4]^-_{q_0}$   \\
$3$             & $[3]^-_{q_0}$             \\
\hline
\end{tabular}}
\end{table}

\begin{proof}
Note that $y \in \PGU_n(q_0)$ and $\det(y) = \alpha_0^{\ell}$. Without loss of generality, assume that $\alpha_0 = \alpha^{1+q_0^2+\cdots+q_0^{2e-2}}$, so by Lemma~\ref{lem:det_u}, there exists $g \in \PGU_n(q)$ such that $\det(g) = \alpha^\ell$ and $F(g\s) = y$. Therefore, we may write $g = t\d^\ell$ where $t \in T$, so $g\s = t\d^\ell\p^i = t\th$. Now $y = a^{-1}(t\th)^ea = F(t\th)$ for some $a \in X$, as claimed.
\end{proof}

\begin{propositionx} \label{prop:u_Ia_max}
Assume that $n \geq 6$. The maximal subgroups of $G$ which contain $t\th$ are listed in Table~\ref{tab:u_Ia_max}, where $m(H)$ is an upper bound on the multiplicity of the subgroups of type $H$ in $\M(G,t\th)$.
\end{propositionx}

\begin{table}
\centering
\caption{Case~I(a): Description of $\M(G,t\th)$}\label{tab:u_Ia_max}
{\renewcommand{\arraystretch}{1.2}
\begin{tabular}{cccc}
\hline
       & type of $H$                       & $m(H)$  & conditions                   \\
\hline
$\C_1$ & $\GU_2(q) \times \GU_{n-2}(q)$    & $1$     &                              \\
       & $P_1$                             & $1$     &                              \\
       & $P_{m-1}$                         & $2$     & $n=2m$                       \\
       & $P_m$                             & $2$     & $n=2m$                       \\[5.5pt]
$\C_2$ & $\GU_{n/k}(q) \wr S_k$            & $N$     & $k > 1$, \, $k \div n$       \\
       & $\GL_m(q^2)$                      & $N$     & $n=2m$                       \\[5.5pt]
$\C_5$ & $\GU_n(q^{1/k})$                  & $N$     & $k$ odd prime, \, $k \div f$ \\
       & $\Sp_n(q)$                        & $N$     & $n$ even                     \\
\hline
\end{tabular}}
\\[5pt]
{ \small Note: $N = |C_{\PGU_n(q_0)}(y)|$}
\end{table}

\begin{proof}
Let $H \in \M(G,t\th)$ and note that $T \not\leq H$. First assume that $H \not\in \C_1$. A power of $t\th$ is $X$-conjugate to $y$, a power of which is $J_2 \perp I_{n-2}$. Therefore, Proposition~\ref{prop:u_max} implies that $H \in \C_2 \cup \C_5$, noting that $q > p$. Moreover, $H$ does not have type $\O^\e_n(q)$ since orthogonal groups do not contain elements of Jordan form $[J_2,J_1^{n-2}]$. All other possible types of such subgroups are given in Table~\ref{tab:u_Ia_max}. All geometric subgroups of $G$ of a given type are $\<X_{\s^e},\s\>$-conjugate by \cite[Theorem~4.0.2]{ref:KleidmanLiebeck} and the upper bound on the multiplicity $m(H)$ is provided by Proposition~\ref{lem:centraliser_bound}.

Now assume that $H \in \C_1$. By Lemma~\ref{lem:c1} (in conjunction Lemma~\ref{lem:brookfield}), the reducible subgroups of $X_\s$ that contain $y$ are one of type $\GU_2(q_0) \perp \GU_{n-2}(q)$, one of type $P_1$ and if $n$ is even also two of type $P_{n/2-1}$ and two of type $P_{n/2}$. By applying Lemma~\ref{lem:shintani_descent_fix} with $Y$ as the (connected) subgroup of $X$ of type $\GL_k(\FF_p) \times \GL_{n-k}(\FF_p)$ or $P_{k,n-k}$ for each $1 \leq k < n/2$, we conclude that the reducible subgroups of $G$ that contain $t\th$ are those in the statement.
\end{proof}

\begin{propositionx}\label{prop:u_Ia}
Let $G = \<T,\th\> \in \A$ where $T = \PSU_n(q)$. In Case~I(a), $u(G) \geq 2$ and as $q \to \infty$ we have $u(G) \to \infty$. 
\end{propositionx}

\begin{proof}
Let $x \in G$ have prime order. As usual, write $\M(G,t\th)$ for the set of maximal subgroups of $G$ that contain $t\th$.

First assume that $n \geq 6$. Then $\M(G,t\th)$ is described by Proposition~\ref{prop:u_Ia_max}. By Lemma~\ref{lem:centraliser}, $|C_{X_{\s}}(y)| \leq q_0^{n-1} + q_0$. Write $d(k)$ for the number of divisors of $k$. 

Upper bounds on the fixed point ratios for subspace actions are given by Theorem~\ref{thm:fpr_s}. Now assume that $H \leq G$ is a maximal irreducible subgroup. If $n \geq 7$, then Proposition~\ref{prop:fpr_ns_u} implies that $\fpr(x,G/H) < 2q^{-(n-3+2/n)}$ and Theorem~\ref{thm:fpr_ns_u_low} implies that the same conclusion holds for $n=6$ too.

Applying Lemma~\ref{lem:centraliser_bound}, if $n \geq 7$ is odd, then
\begin{align*}
P(x,t\th) < \frac{1}{q^2} &+ \frac{1}{q^4} + \frac{1}{q^{n/2-3/2}} + \frac{1}{q^{n-3}} + \frac{4}{q^{n-2}} + \frac{1}{q^n} \\ 
&+ (d(n) + \log\log{q} + 1) \cdot (q_0^{n-1}+q_0) \cdot \frac{2}{q^{n-3+2/n}} < \frac{1}{2}
\end{align*}
and $P(x,t\th) \to 0$ and $q \to \infty$  (recall that $q=q_0^e$ where $e = f/j \geq 3$ is odd). Similarly, if $n \geq 6$ is even, then
\begin{align*}
P(x,t\th) < \frac{1}{q^2} &+ \frac{1}{q^4} + \frac{2}{q^{n-4}} + \frac{5}{q^{n/2-1}} + \frac{11}{q^{n-2}} + \frac{5}{q^{n-1}} \\ 
&+ (d(n) + \log\log{q} + 1) \cdot (q_0^{n-1}+q_0) \cdot \frac{2}{q^{n-3+2/n}} < \frac{1}{2}
\end{align*}
and $P(x,t\th) \to 0$ as $q \to \infty$. 

Next assume that $n \in \{3,5\}$. We begin by determining the possible types of subgroups in $\M(G,t\th)$. First consider reducible subgroups. Since $y$ is not contained in any reducible subgroups of $\PGU_n(q_0)$, by Lemma~\ref{lem:shintani_descent_fix}, $t\th$ is not contained in any reducible subgroups of $G$. Therefore, if $H \in \M(G,t\th)$, then consulting \cite[Tables~8.5, 8.6, 8.20 and~8.21]{ref:BrayHoltRoneyDougal}, we see that $H$ has one of the following types: $\GU_1(q) \wr S_n$, $\GU_1(q^n)$ and $\GU_n(q^{1/k})$ for $k$ dividing $f$. 

With the bound on fixed point ratios from Theorem~\ref{thm:fpr_ns_u_low}, if $n=5$, then
\[
P(x,t\th) \leq (2+\log{\log{q}}) \cdot \frac{q_0^5+1}{q_0+1} \cdot \frac{4}{3q^4} < \frac{1}{2},
\]
and if $n=3$, then by Proposition~\ref{prop:u_computation} we assume that $q \neq 8$ and we obtain
\[
P(x,t\th) \leq (2+\log{\log{q}}) \cdot \frac{q_0^3+1}{q_0+1} \cdot \frac{1}{q^2-q+1} < \frac{1}{2}.
\]
In both cases, $P(x,t\th) \to 0$ as $q \to \infty$.

Finally assume that $n=4$. By Proposition~\ref{prop:u_computation}, we may assume that $q \neq 8$. Since $y$ is contained in a unique reducible subgroup of $\PGU_4(q_0)$ (of type $\GU_3(q_0)$), by Lemma~\ref{lem:shintani_descent_fix}, we know that $t\th$ is contained in a unique reducible subgroup of $G$ (of type $\GU_3(q)$). From \cite[Table~8.10 and~8.11]{ref:BrayHoltRoneyDougal}, the types of irreducible maximal subgroups of $G$ are $\GU_1(q) \wr S_4$, $\GU_2(q) \wr S_2$, $\GL_2(q^2)$, $\Sp_4(q)$, $\O^\pm_4(q)$ and $\GU_4(q^{1/k})$ for $k$ dividing $f$. Notice that $y$ is not contained in a subgroup of type $\Sp_4(q)$ or $\O^\pm_4(q)$ since such groups do not contain elements with a $1$-dimensional $1$-eigenspace. Therefore, using the fixed point ratio bounds in \eqref{eq:fpr} (for the subspace subgroup) and Theorem~\ref{thm:fpr_ns_u_low} (for the nonsubspace subgroups), we obtain
\[
P(x,t\th) < \frac{4}{3q} + (3+\log{\log{q}}) \cdot (q_0^3+1) \cdot \frac{1}{q^2-q+1} < \frac{1}{2}
\]
and $P(x,t\th) \to 0$ as $q \to \infty$. The desired result now follows by Lemma~\ref{lem:prob_method}.
\end{proof}

\clearpage
\subsection{Case I(b)} \label{ss:u_Ib}

As in Section~\ref{ss:o_Ib}, we need a variant on Shintani descent to identify an element of $T\th$. Let us fix our notation for this section. \vspace{5pt}

\begin{shbox}
\begin{notationx} \label{not:u_Ib}
\begin{enumerate}[label={}, leftmargin=0cm, itemsep=3pt]
\item Write $q=p^f$ where $f \geq 1$. Let $V = \F_{q^2}^n$.
\item Fix the basis $\B$ from \eqref{eq:B_u_ef}.
\item Fix the simple algebraic group $X = \PSL_n(\FF_p)$.
\item Fix the Frobenius endomorphism $\p = \p_\B$ and the standard graph automorphism $\g=\g_\B$ (see Definition~\ref{def:phi_gamma_r}).
\item If $(n,q+1)$ is even, then fix the antidiagonal element $\d_2 = \d^{\frac{q+1}{(q+1)_2}}$, where $\d$ is given in Definition~\ref{def:u_delta}, so $|\d_2| = (n,q+1)_2$ (see Remark~\ref{rem:u_delta_2}).
\end{enumerate}
\end{notationx}
\end{shbox} \vspace{5pt}

By Proposition~\ref{prop:u_cases}, we can assume that $\th = \p^i$ or, if $(n,q+1)$ is even, $\th = \d_2\p^i$ where $i$ is a proper divisor of $f$. \vspace{5pt}

\begin{shbox}
\notacont{\ref{not:u_Ib}}  Write $q = q_0^e$ where $e=f/i$.
\begin{enumerate}[label={}, leftmargin=0cm, itemsep=3pt]
\item Fix the Steinberg endomorphism $\s$ and the graph automorphism $\r$ according to the following two cases. In both cases, write $Z = C_X(\r)^\circ$.
\[
{\renewcommand{\arraystretch}{1.2}
\begin{array}{ccccc}
\hline
\text{case} & \th      & e           & \s       & \r     \\   
\hline
\text{(i)}  & \begin{array}{c} \p^i \\ \d_2\p^i \end{array} & \begin{array}{c} \text{all} \\ \text{even} \end{array} & \begin{array}{c} \p^i \\ \d_2\p^i \end{array} & \begin{array}{c} \g \\ \g \end{array}  \\[11pt]  
\text{(ii)} & \d_2\p^i & \text{odd}  & \d_2\p^i & \d_2\g \\ 
\hline
\end{array}}
\]
\end{enumerate}
\end{shbox} \vspace{5pt}

\begin{propositionx}\label{prop:u_Ib_z}
The automorphism $\rho$ is an involution that commutes with $\s$ and the isomorphism type of $Z_\s$ is given in Table~\ref{tab:u_Ib_z}.
\end{propositionx}

\begin{table}[b]
\centering
\caption{Case~I(b): The group $Z_\s$} \label{tab:u_Ib_z}
{\renewcommand{\arraystretch}{1.2}
\begin{tabular}{cccc}
\hline   
$n$  & $q$  & $Z_\s$             & condition \\
\hline
odd  & even & $\PSp_{n-1}(q_0)$  &           \\
     & odd  & $\PSO_n(q_0)$      &           \\
even & even & $\PSp_n(q_0)$      &           \\
     & odd  & $\PGSp_n(q_0)$     & case~(i)  \\
     &      & $\PDO^\eta_n(q_0)$ & case~(ii) \\
\hline
\end{tabular}}
\\[5pt]
{ \small Note: in the final row, $\eta = (-)^{\frac{n(q+1)}{4}+1}$}
\end{table}

\begin{proof}
If $\th=\p^i$, then $\s=\p^i$ and $\r = \g$, so clearly $|\r|=2$ and $\r^\s=\r$, and $Z_\s$ is given by \eqref{eq:u_graph_centraliser_finite}. For the remainder of the proof, we will assume that $\th = \g\p^i$, so, in particular, $n=2m$ is even and $q$ is odd. 

First assume that $e$ is even, so $\s=\d_2\p^i$ and $\r=\g$. Clearly $\r$ is an involution. Since $e$ is even, $q \equiv 1 \mod{4}$, so $|\d_2| = 2$. Therefore, $\d_2^\p = \d_2^\g = \d_2$, which implies that $\s$ and $\r$ commute. By \eqref{eq:u_graph_centraliser_algebraic}, $Z = \PGSp_n(\FF_p)$. Since $Z$ is connected and $\d_2 \in Z$, we know that $Z_{\d_2\p^i} \cong Z_{\p^i} = \PGSp_n(q_0)$.

Now assume that $e$ is odd, so $\s=\d_2\p^i$ and $\r=\d_2\g$. Since $e$ is odd, the quotient $(q+1)/(q_0+1) = q_0^{e-1}-q_0^{e-2}+\dots-q_0+1$ is odd, so $|\d_2| = (n,q+1)_2 = (n,q_0+1)_2$. In particular, $\d_2^{\p^i} = \d_2^\g = \d_2^{-1}$, so $(\d_2\g)^2 = 1$ and $(\d_2\p^i)(\d_2\g) = \p^i\g = (\d_2\g)(\d_2\p^i)$. Finally, by Remark~\ref{rem:u_delta_2}(iv)
\begin{align*}
C_X(\r)_\s &= \{ x \in X \mid \text{$x^{\d_2\p^i} = x$ and $x^{\d_2\g} = x$} \} \\
           &= \{ x \in X \mid \text{$x^{\g\p^i} = x$   and $x^{\d_2\g} = x$} \} \\
           &= C_{\PGU_n(q_0)}(\d_2\g) \\
           &= \PGO^\eta_n(q),
\end{align*}
where $\eta = (-)^{\frac{n(q_0+1)}{2}+1} = (-)^{\frac{n(q+1)}{2}+1}$ as $e$ is odd. Therefore, $Z_\s = \PDO^\eta_n(q_0)$.
\end{proof}

\begin{propositionx} \label{prop:u_Ib_elt}
Let $y$ be the element from Table~\ref{tab:u_Ib_elt}. Then there exists $t \in T$ that commutes with $\rho$ such that $(t\th)^e$ is $X$-conjugate to $y\rho$.
\end{propositionx}

\begin{table}
\centering
\caption{Cases~I(b) and~II(b): The element $y$} \label{tab:u_Ib_elt}
{\renewcommand{\arraystretch}{1.2}
\begin{tabular}{cccc}
\hline    
\multicolumn{4}{c}{Generic case: $n \geq 7$ odd or $n \geq 12$ even}           \\ 
\hline   
case & $n$         & $q$         & $y$                                         \\
\hline
(i)  & odd         & even        & $(n-3)^-_{q_0} \perp (2)^-_{q_0} \perp I_1$ \\
     &             & odd         & $(n-3)^-_{q_0} \perp J_3$                   \\
     & even        &             & $(n-2)^-_{q_0} \perp (2)^-_{q_0}$           \\
(ii) & $0 \mod{4}$ &             & $(n-2)^+_{q_0} \perp (2)^-_{q_0}$           \\
     & $2 \mod{4}$ & $1 \mod{4}$ & $(n-2)^-_{q_0} \perp (2)^-_{q_0}$           \\
     &             & $3 \mod{4}$ & $\left(\frac{n}{2}+1\right)^-_{q_0} \perp \left(\frac{n}{2}-3\right)^-_{q_0} \perp (2)^-_{q_0}$ \\
\hline
\\[-2pt]
\hline
\multicolumn{4}{c}{Specific cases: $n \leq 5$ odd or $n \leq 10$ even}   \\
\hline 
case & \multicolumn{2}{c}{$n$}  & $y$                         \\
\hline
(i)  & \multicolumn{2}{c}{$5$}  & $(4)^-_{q_0} \perp I_1$     \\
     & \multicolumn{2}{c}{$3$}  & $g \perp I_1$               \\  
     & \multicolumn{2}{c}{even} & $(n)^-_{q_0}$               \\
(ii) &           &              & $(n)^\eta_{q_0}$            \\
\hline
\end{tabular}}
\\[5pt]
{\small Note: when $n=3$, let $g \in \SO_2^-(q_0)$ have order $q_0+1$}
\end{table}

\begin{proof}
By Proposition~\ref{prop:u_Ib_z}, $\r$ is an involution that commutes with $\s$, so $(\r\s)^{2e} = \s^{2e}$ and $Z = C_X(\r)^\circ$ is $\s$-stable. Moreover, from the structure of $Z_\s$ given in Table~\ref{tab:u_Ib_z}, it is easy to that $Z_\s$ contains elements of the type given in Table~\ref{tab:u_Ib_elt}. Therefore, Lemma~\ref{lem:shintani_substitute}(i) implies that there exists $t \in Z_{\s^e} \leq X_{\g\s^e}$ such that $(t\s)^e$ is $X$-conjugate (indeed $Z$-conjugate) to $y\r$. If $n$ is odd or $q$ is even, then $Z_{\s^e} \leq T$, so we, in fact, have $t \in T$. Now assume that $n$ is even and $q$ is odd. In case~(ii), by applying Lemma~\ref{lem:o_Ia_tau} to the Shintani map of $(Z,\s,e)$, we see that $t \in \PSO_n^\pm(q) \leq \PSU_n(q)$ since $y \in \PSO^\eta_n(q_0)$. Similarly, in case~(i), we apply \cite[Proposition~2.4]{ref:Harper17} (see Example~\ref{ex:shintani_quotients}) to conclude that $t \in \PSp_n(q) \leq \PSU_n(q)$ since $y \in \PSp_n(q_0)$.
\end{proof}

\begin{remarkx} \label{rem:u_Ib_elt_3}
Let $g \in \SO^-_2(q_0)$ have type $(2)^-_{q_0}$. If $q_0$ is not Mersenne, then $g$ and $g^2$ have odd prime order and two distinct eigenvalues. Now assume that $q_0$ is Mersenne. In this case, $|g| = q_0+1$, which is a power of two. Therefore, $g$ has two distinct eigenvalues, and $g^2$ also has two distinct eigenvalues, unless $q_0=3$, in which case $g^2 = -I_2$. For this reason, in several arguments in this section (and those that follow), the case where $q_0=3$ will require particular attention.
\end{remarkx}

\begin{propositionx} \label{prop:u_Ib_max}
Assume that $n \geq 7$ is odd or $n \geq 12$ is even. Then the maximal subgroups of $G$ that contain $t\th$ are listed in Table~\ref{tab:u_Ib_max}, where $m(H)$ is an upper bound on the multiplicity of the subgroups of type $H$ in $\M(G,t\th)$.
\end{propositionx}

\begin{table}
\centering
\caption{Case~I(b): Description of $\M(G,t\th)$}\label{tab:u_Ib_max}
{\renewcommand{\arraystretch}{1.2}
\begin{tabular}{cccc}
\hline
       & type of $H$                                            & $m(H)$ & conditions                                          \\
\hline   
$\C_1$ & $P_1$                                                  & $1$    & $n$ odd, \, $q$ odd                                 \\
       & $P_2$                                                  & $1$    & $n$ odd, \, $q$ odd                                 \\
       & $P_{\frac{n}{2}-1}$                                    & $2$    & case~(ii), \, $n \equiv 0 \mod{4}$                  \\
       & $\GU_1(q) \times \GU_{n-1}(q)$                         & $1$    & $n$ odd, \, $q$ even                                \\
       &                                                        & $4$    & $n$ even, \, $q_0=3$                                \\
       & $\GU_2(q) \times \GU_{n-2}(q)$                         & $1$    & $n$ even or $q$ even                                \\
       & $\GU_3(q) \times \GU_{n-3}(q)$                         & $1$    & $n$ odd                                             \\
       & $\GU_{\frac{n}{2}-3}(q) \times \GU_{\frac{n}{2}+3}(q)$ & $1$    & case~(ii), \, $n \equiv 2 \mod{4}$, \, $q \equiv 3 \mod{4}$ \\
       & $\GU_{\frac{n}{2}-2}(q) \times \GU_{\frac{n}{2}+2}(q)$ & $4$    & case~(ii), \, $n \equiv 2 \mod{4}$, \, $q_0=3$              \\         
       & $\GU_{\frac{n}{2}-1}(q) \times \GU_{\frac{n}{2}+1}(q)$ & $1$    & case~(ii), \, $n \equiv 2 \mod{4}$, \, $q \equiv 3 \mod{4}$ \\[5.5pt]
$\C_2$ & $\GU_{n/k}(q) \wr S_k$                                 & $N$    & $k > 1$, \, $k \div n$                              \\
       & $\GL_{n/2}(q^2)$                                       & $N$    & $n$ even                                            \\[5.5pt]
$\C_4$ & $\GU_2(q) \otimes \GU_{n/2}(q)$                        & $N$    & $n$ even, \, $q_0=3$                                \\[5.5pt]
$\C_5$ & $\GU_n(q^{1/k})$                                       & $N$    & $k$ odd prime, \, $k \div f$                        \\
       & $\Sp_n(q)$                                             & $N$    & $n$ even                                            \\
       & $\O^\upsilon_n(q)$                                     & $N$    & $q$ odd, \, $\upsilon \in \{+,\circ,-\}$            \\
\hline
\end{tabular}}
\\[5pt]
{\small Note: $N = |C_{\PGL_n(q_0)}(y^2)|$}
\end{table}

\begin{proof}
First assume that $H \not\in \C_1$. An $X$-conjugate of $(t\th)^{2e}$ is $(y\gamma)^2 = y^2$. If $n$ is odd or $q_0$ is not Mersenne, then a suitable power of $y^2$ has type $J_3 \perp I_{n-3}$ or $(2)^-_{q_0} \perp I_{n-2}$, so Proposition~\ref{prop:u_max}, implies that $H$ appears in Table~\ref{tab:u_Ib_max}. Now assume that $n$ is even and $q_0$ is Mersenne. Here the order of $(2)^-_{q_0}$ is $q_0+1$, which is a power of two, so a suitable power of $y^2$ is $[\l, \l^{q_0}] \perp I_{n-2}$ where $|\l| = (q_0+1)/2$. Again Proposition~\ref{prop:u_max} gives the possible types for $H$ but note that when $q_0=3$ we have $\l=\l^{q_0} = -1$, so the type $\GU_2(q) \otimes \GU_{n/2}(q)$ also arises. In both cases, all maximal subgroups of a given type are $\<X_{\g\s^e},\s\>$-conjugate, so Lemma~\ref{lem:shintani_substitute}(ii)(a) gives $m(H)$, noting that Proposition~\ref{prop:u_Ib_z} implies $(\r\s^e)^2 = \s^{2e}$.

We now assume that $H \in \C_1$. We will provide the details when $n$ and $q$ are even; the other cases are similar. By Lemma~\ref{lem:c1}, the proper nonzero subspaces of $\F_{q_0}^n$ stabilised by $y$ are $U_0$ and $U_0^\perp$, orthogonal nondegenerate subspaces of dimensions $2$ and $n-2$. Moreover, these are also the subspaces stabilised by $y^2$. We now apply Lemma~\ref{lem:shintani_substitute}(ii)(b). Since the unique $\gamma$-stable $\C_1$ subgroup of $\GL_n(q_0)$ containing $y^2$ has type $\GL_2(q_0) \times \GL_{n-2}(q_0)$, there is a unique $\gamma$-stable reducible subgroup of $\<\PGL_n(q^2),\s\>$ containing $t\th$ and moreover it has type $\GL_2(q^2) \times \GL_{n-2}(q^2)$. Since any reducible subgroup of $\widetilde{G} = \<\PGU_n(q),\s\>$ containing $t\th$ gives rise to an $\gamma$-stable subgroup of $\< \PGL_n(q^2), \s \>$ containing $t\th$, we deduce that the unique reducible subgroup of $\widetilde{G}$ containing $t\th$ has type $\GU_2(q) \times \GU_{n-2}(q)$, so the same conclusion holds for $G$, as we claim in Table~\ref{tab:u_Ib_max}.
\end{proof}

We obtain a more precise bound on in one case.

\begin{propositionx} \label{prop:u_Ib_max_sp4}
Assume that $T = \PSU_4(q)$ and $e$ is even. Then $t\th$ is contained in at most $k(q_0+1)$ maximal subgroups of $G$ of type $\Sp_4(q)$ where $k=2$ if $q_0 \equiv 3 \mod{4}$ and $k=1$ otherwise.
\end{propositionx}

\begin{proof}
Let $H$ be a maximal subgroup of $G$ of type $\Sp_4(q)$ and let $m$ be the number of maximal subgroups of $G$ of type $\Sp_4(q)$ that contain $t\th$. Recall the algebraic groups $X = \PGL_4(\FF_p)$ and $Z = \PGSp_4(\FF_p)$, the Frobenius endomorphisms $\s \in \{ \p^i, \d_2\p^i\}$ and the standard graph automorphism $\r = \g$. We have 
\[
G = \< \PSU_4(q), \th \> \leq \< X_{\g\s^e}, \ws\> = \widetilde{G}
\]
and we may assume that 
\[
H \leq \< \PGSp_4(q), \th \> = \< Z_{\s^e}, \ws \> = \widetilde{H}.
\] 
All subgroups of $G$ of type $\Sp_4(q)$ are $\widetilde{G}$-conjugate, so, by Lemma~\ref{lem:fpr_subgroups}, 
\[
m \leq \frac{|\widetilde{G}|}{|\widetilde{H}|}\frac{|(t\th)^{\widetilde{G}} \cap H|}{|(t\th)^{\widetilde{G}}|} = \sum_{i=1}^{k}\frac{|C_{\widetilde{G}}(t\th)|}{|C_{\widetilde{H}}(t_i\th)|}
\]
where $(t\th)^{\widetilde{G}} \cap H$ is the disjoint union $\cup_{i=1}^{k} (t_i\th)^{\widetilde{H}}$.

Let $s\th \in (t\th)^{\widetilde{G}} \cap H$.  Let $E$ and $F$ be the Shintani maps of $(X,\s,2e)$ and $(Z,\s,e)$ respectively. Since $s\th$ is $\<\PGL_4(q^2),\th\>$-conjugate to $t\th$, by Shintani descent, $E(s\th) = F(s\th)^2$ is $\PGL_4(q_0)$-conjugate to $E(t\th) = F(t\th)^2 = y^2$ (see Lemma~\ref{lem:shintani_powers}). Now $y^2$ is semisimple, so by \cite[Lemma~3.4.2]{ref:BurnessGiudici16}, $F(s\th)^2$ is $\PGSp_4(q_0)$-conjugate to $y^2$. The eigenvalues of $y$ are 
\[
\l_1 = \l, \ \l_2 = \l^{q_0}, \ \l_3 = \l^{q_0^2}, \  \l_4 = \l^{q_0^3}
\] 
where $\l \in \F_{q_0^4}^\times$ satisfies $|\l| \in \ppd(q_0,4)$. Therefore, the eigenvalues of $y^2$, and hence $F(s\th)^2$ are $\l_1^2,\l_2^2,\l_3^2,\l_4^2$. 

For now assume that $q_0 \not\equiv 3 \mod{4}$. In this case, modulo scalars, the eigenvalues of $F(s\th)$ are $\l_1,\l_2,\l_3,\l_4$, so $F(s\th)$ is $\PGSp_4(q_0)$-conjugate to $y$. Therefore, by Shintani descent, $s\th$ is $\widetilde{H}$-conjugate to $t\th$. In this case, write $y_1 = y$ and $t_1 = t$.

Now assume that $q_0 \equiv 3 \mod{4}$. Here $F(s\th) \in y_1^{\PGSp_4(q_0)} \cup y_2^{\PGSp_4(q_0)}$ where $y_1 = y = [\l_1,\l_2,\l_3,\l_4]$ and $y_2 = [\xi\l_1,-\xi\l_2,\xi\l_3,-\xi\l_4]$ where $\xi \in \F_{q_0^2}^\times$ satisfies $|\xi|=4$, and we conclude that $s\th$ is $\widetilde{H}$-conjugate to $t_1\th$ or $t_2\th$, where $F(t_1\th)=y_1$ and $F(t_2\th)=y_2$. 

By Theorem~\ref{thm:shintani_descent} (in conjunction Remark~\ref{rem:shintani_descent})
\[
|C_{\widetilde{G}}(t\th)| \leq |C_{\<\PGL_4(q^2),\th\>}(t\th)| = 2e|C_{\PGL_4(q_0)}(y)| = 2e(q_0+1)(q_0^2+1)
\]
and for $i \in \{1,2\}$
\[
|C_{\widetilde{H}}(t_i\th)| = 2e|C_{\PGSp_4(q_0)}(y_i)| = e(q_0^2+1).
\]

Therefore, letting $k=2$ if $q_0 \equiv 3 \mod{4}$ and $k=1$ otherwise, we obtain
\[
m \leq \sum_{i=1}^{k}\frac{|C_{\widetilde{G}}(t\th)|}{|C_{\widetilde{H}}(t_i\th)|} = k(q_0+1). \qedhere
\]
\end{proof}

We now establish the main result of this section.

\begin{propositionx}\label{prop:u_Ib}
Let $G = \<T,\th\> \in \A$ where $T = \PSU_n(q)$. In Case~I(b), $u(G) \geq 2$ and as $q \to \infty$ we have $u(G) \to \infty$.
\end{propositionx}

\begin{proof}
Let $x \in G$ have prime order. As usual, we apply Lemma~\ref{lem:prob_method}. Bounds on the relevant fixed point ratios are given in Theorem~\ref{thm:fpr_s} (for subspace actions) and Proposition~\ref{prop:fpr_ns_u} and Theorem~\ref{thm:fpr_ns_u_low} (for nonsubspace actions). As usual, we write $d(k)$ for the number of divisors of $k$.

First assume that $n \geq 7$ is odd and $q$ is even. In this case, 
\[
|C_{\PGL_n(q_0)}(y^2)| = (q_0^2-1)(q_0^{n-3}-1),
\] 
so by Proposition~\ref{prop:u_Ib_max},
\begin{align*}
P(x,t\th) < \frac{1}{q^2}&+\frac{1}{q^4}+\frac{1}{q^6} +\frac{8}{q^{n-3}}+\frac{1}{q^{n-1}}+\frac{3}{q^n} \\ &+ (d(n)-1+\log\log{q}) \cdot (q_0^{n-3}-1)(q_0^2-1) \cdot \frac{2}{q^{n-3+2/n}} < \frac{1}{2}
\end{align*}
unless $n=7$ and $q=4$. In this remaining case, then there is one class of nonsubspace subgroups (of type $\GU_1(q) \wr S_7$), and by replacing $(d(n)-1 + \log\log{q})$ with $1$ we obtain the desired result.

Next assume $n \geq 7$ and $q$ are odd. Here 
\[
|C_{\PGL_n(q_0)}(y^2)| = q_0^2(q_0^{n-3}-1)\
\]
and, by Proposition~\ref{prop:u_Ib_max},
\begin{align*}
P(x,t\th) < \frac{1}{q^2}&+\frac{1}{q^4}+\frac{1}{q^6} +\frac{2}{q^{n/2}}+\frac{3}{q^{n-3}}+\frac{4}{q^{n-1}}+\frac{1}{q^n} \\ &+ (d(n)-1+\log\log{q}) \cdot (q_0^{n-3}-1)q_0^2 \cdot \frac{2}{q^{n-3+2/n}} < \frac{1}{2}
\end{align*}
unless $e=2$ and $n=7$, when there are two classes of nonsubspace subgroups in $\M(G,t\th)$ (of types $\GU_1(q) \wr S_7$ and $\O_n(q)$) and replacing $(d(n) + \log\log{q})$ with $2$ gives the result.

Now assume that $n \in \{3,5\}$. By Proposition~\ref{prop:u_computation}, we will assume that $(n,q) \not\in \{(3,2^2),(3,2^3),(3,3^2),(5,2^2)\}$. Here a power of $y^2$ has type $(n-1)^-_{q_0} \perp I_1$, so
\[
|C_{\PGL_n(q_0)}(y^2)| = q_0^{n-1}-1.
\]
By applying Lemma~\ref{lem:shintani_substitute}(ii)(b) in the usual way, we see that $t\th$ is contained in a unique reducible maximal subgroup, and by \cite{ref:BrayHoltRoneyDougal}, there are at most $2+\d_{n,5}+\log\log{q}$ classes of irreducible maximal subgroups, so by \eqref{eq:fpr} and Theorem~\ref{thm:fpr_ns_u_low},
\[
P(x,t\th) < \frac{4}{3q} + (2 + \d_{n,5} + \log\log{q})\cdot(q_0^{n-1}-1)\cdot\frac{4}{3q^{n-1}} < \frac{1}{2}.
\]

For the remainder of the proof we may assume that $n$ is even. For now assume that $n \geq 12$. If we are in case~(i), then
\[
|C_{\PGL_n(q_0)}(y^2)| = q_0^{\d_{q_0,3}}(q_0+1)(q_0^{n-2}-1)
\]
and, by Proposition~\ref{prop:u_Ib_max},
\begin{align*}
P(x,t\th) &< \d_{q_0,3}\left( \frac{3}{q^2} + \frac{6}{q^{n-1}} + \frac{6}{q^{n-4}} \right) +  \frac{1}{q^4}+\frac{2}{q^{n-4}}+\frac{1}{q^{n-2}}+\frac{1}{q^{n-1}} \\ 
&+ (d(n)+\log\log{q}+4-3\d_{2,p}) \cdot q_0^{\d_{q_0,3}}(q_0^{n-2}-1)(q_0+1) \cdot \frac{2}{q^{n-3+2/n}} < \frac{1}{2}.
\end{align*}

Now consider case~(ii). Write $N = C_{\PGL_n(q_0)}(y^2)$. Then
\[
N \leq \left\{ 
\begin{array}{ll}
q_0(q_0+1)(q_0^{n-2}-1)                  & \text{if $\eta=+$}                           \\
q_0(q_0+1)(q_0^{n/2-1}-1)^2              & \text{if $\eta=-$ and $n \equiv 0 \mod{4}$}  \\
q_0(q_0+1)(q_0^{n/2+1}-1)(q_0^{n/2-3}-1) & \text{if $\eta=-$ and $n \equiv 2 \mod{4}$.} \\
\end{array}
\right.
\]
Therefore, 
\begin{align*}
P(x,t\th) < \frac{4}{q^2}&+\frac{1}{q^4}+\frac{4}{q^{n/2-1/2}}+\frac{1}{q^{n/2+1/2}}+\frac{1}{q^{n/2+3}}+\frac{1}{q^{n-6}}+\frac{6}{q^{n-4}}+\frac{13}{q^{n-2}} \\
&+ (d(n)+\log\log{q}+4) \cdot N \cdot \frac{2}{q^{n-3+2/n}} < \frac{1}{2}.
\end{align*}

It remains to assume that $n \in \{4,6,8,10\}$. For now assume that $n \geq 4$. Now $y^2$ has type $(n)^\pm_{q_0}$ and 
\[
|C_{\PGL_n(q_0)}(y^2)| \leq \frac{q_0^n-1}{q_0-1}.
\]
Since $y^2$ is not contained in any $\C_1$ subgroups of $\PGL_n(q_0)$, by Lemma~\ref{lem:shintani_substitute}(ii)(b) implies that $t\th$ is not contained in any $\C_1$ subgroups of $G$.  By consulting \cite{ref:BrayHoltRoneyDougal}, we see that $G$ has at most $4$ classes of $\C_2$ subgroups, at most $k-2\d_{2,p}+\log\log{q}$ further classes of irreducible maximal subgroups, and together Proposition~\ref{prop:fpr_ns_u} and Theorem~\ref{thm:fpr_ns_u_low}, establish $\fpr(x,G/H) \leq f(q)$ for all $H \in \M(G,t\th)$, where
\[
k = \left\{
\begin{array}{ll}
7  & \text{if $n=10$} \\ 
3  & \text{if $n=8$}  \\ 
5  & \text{if $n=6$}  \\
\end{array}
\right.
\text{ and } \
f(q) = \left\{
\begin{array}{ll}
2q^{-7.2}                 & \text{if $n=10$} \\
2q^{-5.25}                & \text{if $n=8$}  \\
(q^4-q^3+q^2-q+1)^{-1}    & \text{if $n=6$.} \\
\end{array}
\right.
\]
Therefore, if $e \geq 3$, then
\[
P(x,t\th) \leq (4+k-2\d_{2,p}+\log\log{q}) \cdot \frac{q_0^n-1}{q_0-1} \cdot f(q) < \frac{1}{2}.
\]
Now assume that $e=2$. Since $|y| \in \ppd(q,\frac{n}{2})$ and $\frac{n}{2} \not\equiv 2 \mod{4}$, by \cite[Proposition~3.3.2]{ref:BurnessGiudici16}, $y$ centralises the decomposition $\F_{q^2}^n = U \oplus U^*$ where $U$ is a totally singular $\frac{n}{2}$-space on which $y$ acts irreducibly. Therefore, $y$ is contained in a unique subgroup of type $\GL_{\frac{n}{2}}(q^2)$ and no further $\C_2$ subgroups. Therefore, we obtain
\[
P(x,t\th) \leq \left(1+(k-2\d_{2,p}+\log\log{q}) \cdot \frac{q_0^n-1}{q_0-1}\right) \cdot f(q) < \frac{1}{2}.
\]

Finally assume that $n=4$. By Proposition~\ref{prop:u_computation}, we will assume that $q \not\in \{2^2,2^3,3^2\}$. Now $y$ has type $(4)^-_{q_0}$, so 
\[
|C_{\PGL_4(q_0)}(y^2)| = (q_0+1)(q_0^2+1).
\] 
By Lemma~\ref{lem:shintani_substitute}(ii)(b), $t\th$ is not contained in any reducible maximal subgroup of $G$ since $y^2$ is not contained in any reducible maximal subgroups of $\PGL_4(q_0)$. From \cite{ref:BrayHoltRoneyDougal}, there are at most $5+\log\log{q}$ classes of irreducible maximal subgroups. If $e \geq 3$, then, by Theorem~\ref{thm:fpr_ns_u_low},
\[
P(x,t\th) \leq \frac{(4 + \log\log{q}) \cdot (q_0^2-q_0+1)}{q^2-q+1} + \frac{(2,q+1)(q_0^2-q_0+1)}{q} \ < \frac{1}{2}.
\]
If $e=2$, then let $k$ be $2$ if $q_0 \equiv 3 \mod{4}$ and $1$ otherwise, so by Proposition~\ref{prop:u_Ib_max_sp4},
\[
P(x,t\th) \leq \frac{(4 + \log\log{q}) \cdot (q_0^2-q_0+1)}{q^2-q+1} + \frac{k(q_0+1)\cdot(2,q+1)(q^4+1)}{q^5+q^2} \ < \frac{1}{2}.
\]
In all cases $P(x,\th) \to 0$ as $q \to \infty$.
\end{proof}

\clearpage
\section{Case II: linear automorphisms} \label{s:u_II}

We now turn to Case~II. In this section, we write $G=\<T,\th\>$ where $T = \PSU_n(q)$ for $n \geq 3$ and where $\th \in \<\PGU_n(q), \gamma\>$. Recall the case distinction
\begin{enumerate}[(a)]
\item $G \leq \PGU_n(q)$
\item $G \not\leq \PGU_n(q)$.
\end{enumerate}
Cases~II(a) and~II(b) will be considered in Sections~\ref{ss:u_IIa} and~\ref{ss:u_IIb}, respectively.

\subsection{Case II(a)} \label{ss:u_IIa}

Let $T = \PSU_n(q)$ and let $G = \<T,\th\> \in \A_-$ in Case~II(a). Therefore, $G=T$ or $(n,q+1) > 1$ and we may write $G = \<T,\d^\ell\>$ for some $0 < \ell < (n,q+1)$. As in Section~\ref{ss:o_IIa}, in the following proof, we are closely following \cite[Sections~5.10 and~5.11]{ref:BreuerGuralnickKantor08}.

\begin{propositionx} \label{prop:u_IIa}
Let $G = \<T, \th\> \in \A_-$. In Case~II(a), $u(G) \geq 2$ and as $q \to \infty$ we have $u(G) \to \infty$.
\end{propositionx}

\begin{proof}
If $G = T$, then the result follows from \cite[Propositions~5.20 and~5.21]{ref:BreuerGuralnickKantor08}. Therefore, for the remainder of the proof, we may assume that $(n,q+1) > 1$ and we will write $G = \<T,\d^\ell\>$ for some fixed $0 < \ell < (n,q+1)$. Let $s = y^\ell$ where $y$ has type $[n]^-$ if $n$ is odd and $[n-1]^- \perp I_1$ if $n$ is even. Note that $s \in T\d^\ell$, since $\det(s) = \alpha^\ell$. By Proposition~\ref{prop:u_computation}, we may assume that $q \geq 11$ if $n \in \{3,4\}$, $q \geq 4$ if $n \in \{5,6,7,8\}$.

If $n$ is odd, then $s$ acts irreducibly on $V = \F_{q^2}^n$, and if $n$ is even, then, by Lemma~\ref{lem:c1}, $s$ is contained in a unique reducible maximal subgroup of $G$, of type $\GU_1(q) \perp \GU_{n-1}(q)$. Now let $H \in \M(G,s)$ be irreducible. Then the order of $s$ is divisible by some $r \in \ppd(q^2,k)$ where $k \in \{ n-1, n\}$ is odd. Moreover, by \cite[Lemma~6.1]{ref:BambergPenttila08}, we may assume that $r > 2k+1$. Now applying \cite[Theorem~2.2]{ref:GuralnickMalle12JAMS}, we see that $H$ is a subfield or field extension subgroup. In the former case, it is straightforward to see that $r$ does not divide the order of $H$. Now consider the latter case. Here the degree of the field extension divides $(n,k)$, so $n$ must be odd and $H$ has type $\GU_{n/k}(q^k)$ for some prime $k$ dividing $n$. Then \cite[Lemma~2.12]{ref:BreuerGuralnickKantor08} implies that $s$ is contained in a unique subgroup of type $\GU_{n/k}(q^k)$ for each possible $k$. 

Let $x \in G$ have prime order. For now assume that $n \geq 4$ is even. Then $\M(G,s) = \{ H \}$ for $H$ of type $\GU_1(q) \times \GU_{n-1}(q)$. If $n \geq 6$,  by Theorem~\ref{thm:fpr_s},
\[
P(x,s) < \frac{1}{q^2} + \frac{2}{q^{n-4}} + \frac{1}{q^{n-2}} + \frac{1}{q^{n-1}} < \frac{1}{2}
\]
since $(n,q) \not\in \{ (6,2), (6,3) \}$ and if $n \geq 4$, then, by \eqref{eq:fpr}, $P(x,s) < \frac{4}{3q} < \frac{1}{2}$ since $q \geq 4$. Moreover, $P(x,s) \to 0$ as $q \to \infty$.

Now assume $n \geq 3$ is odd. Then $\M(G,s) = \{ H_k \mid \text{$k$ is a prime divisor of $n$} \}$, where $H_k$ has type $\GU_{n/k}(q^k)$. By Proposition~\ref{prop:fpr_ns_u}, if $n \geq 7$, then
\[
P(x,s) < (n-2) \cdot \frac{2}{q^{n-3}} < \frac{1}{2},
\]
and if $n \in \{5,7\}$, then $P(x,s) < \frac{2}{q^{n-3}} \leq \frac{1}{2}$. If $n=3$, then $q \geq 11$ and Theorem~\ref{thm:fpr_ns_u_low} implies that 
\[
P(x,s) \leq \frac{1}{q^2-q+1} < \frac{1}{2}
\]
and $P(x,s) \to 0$ as $q \to \infty$.
\end{proof}

\clearpage
\subsection{Case II(b)} \label{ss:u_IIb}

This section completes the proof of Theorems~\ref{thm:u_main} and~\ref{thm:u_asymptotic}, by considering Case~II(b). In this case, $G = \<T, \th\>$ where $\th$ is either $\g$ or $\d_2\g$ (recall that $(n,q+1)$ is even in the latter case). 

To avoid repetition, we refer to some tables in Section~\ref{ss:u_Ib}, with the convention that $i=f$, so $e=1$, $\p^i=\g$ and $q_0=q$. By \eqref{eq:u_graph_centraliser_finite} and Remark~\ref{rem:u_delta_2}(iv), the centraliser $C_{\PGU_n(q)}(\th)$ is given in the $Z_\s$ column of Table~\ref{tab:u_Ib_z}. Let $t \in C_{\PSU_n(q)}(\th)$ be the element $y$ in Table~\ref{tab:u_Ib_elt}.

\begin{propositionx} \label{prop:u_IIb_max}
Assume that $n \geq 7$ is odd or $n \geq 12$ is even. Then the maximal subgroups of $G$ that contain $t\th$ are listed in Table~\ref{tab:u_IIb_max}, where $m(H)$ is an upper bound on the multiplicity of the subgroups of type $H$ in $\M(G,t\th)$.
\end{propositionx}

\begin{table}[b]
\centering
\caption{Case~II(b): Description of $\M(G,t\th)$}\label{tab:u_IIb_max}
{\renewcommand{\arraystretch}{1.2}
\begin{tabular}{cccc}
\hline
       & type of $H$                                            & $m(H)$ & conditions                                    \\
\hline
$\C_1$ & $P_1$                                                  & $1$    & $n$ odd,  \, $q$ odd                          \\
       & $P_2$                                                  & $1$    & $n$ odd,  \, $q$ odd                          \\
       & $P_{\frac{n}{2}-1}$                                    & $2$    & $\th=\d_2\g$, \, $n \equiv 0 \mod{4}$         \\
       & $\GU_1(q) \times \GU_{n-1}(q)$                         & $1$    & $n$ odd,  \, $q$ even                         \\
       &                                                        & $4$    & $n$ even, \, $q=3$                            \\
       & $\GU_2(q) \times \GU_{n-2}(q)$                         & $1$    & $n$ even or $q$ even                          \\
       & $\GU_3(q) \times \GU_{n-3}(q)$                         & $1$    & $n$ odd                                       \\
       & $\GU_{\frac{n}{2}-3}(q) \times \GU_{\frac{n}{2}+3}(q)$ & $1$    & $\th=\d_2\g$, \, $n \equiv 2 \mod{4}$, \, $q \equiv 3 \mod{4}$ \\
       & $\GU_{\frac{n}{2}-2}(q) \times \GU_{\frac{n}{2}+2}(q)$ & $4$    & $\th=\d_2\g$, \, $n \equiv 2 \mod{4}$, \, $q = 3$              \\  
       & $\GU_{\frac{n}{2}-1}(q) \times \GU_{\frac{n}{2}+1}(q)$ & $1$    & $\th=\d_2\g$, \, $n \equiv 2 \mod{4}$, \, $q \equiv 3 \mod{4}$ \\[5.5pt] 
$\C_2$ & $\GL_{\frac{n}{2}}(q^2)$                               & $q+1$  & $n$ even                                      \\
       & $\GU_{\frac{n}{2}}(q) \wr S_2$                         & $q+1$  & $n$ even                                      \\
       & $\GU_3(q) \wr S_3$                                     & $1$    & $n=9$                                         \\[5.5pt]
$\C_5$ & $\O_n(q)$                                              & $M$    & $n$ odd,  \, $q$ odd                          \\
       & $\Sp_n(q)$                                             & $M$    & $n$ even                                      \\
       & $\O^+_n(q)$                                            & $M$    & $\th=\g$, \, $n$ even, \, $q$ odd             \\
       & $\O^\eta_n(q)$                                         & $M$    & $\th=\d_2\g$, \, $n$ even, \, $q$ odd         \\        
       & $\GU_9(q^{1/3})$                                       & $M$    & $n=9$ and $3 \div f$                          \\
\hline
\end{tabular}}
\\[5pt]
{ \small Note: $M = |C_{\Inndiag(\soc(G))}(t^2):C_{\Inndiag(\soc(H))}(t^2)|$}
\end{table}

\begin{proof}
Let $H \in \M(G,t\th)$. A suitable power of $t^2$ has type $J_3 \perp I_{n-3}$ or $(2)^-_q \perp I_{n-2}$ unless $n$ is even and $q_0=3$, in which case a suitable power of $t^2$ has type $-I_2 \perp I_{n-2}$. In particular, a power of $t^2$ is an element $z$ that satisfies $\nu(z) = 2$, so Proposition~\ref{prop:u_max} implies that one of the following holds
\begin{enumerate}
\item $H \in \C_1 \cup \C_2 \cup \C_5$
\item $H \in \C_4$ has type $\GU_2(q) \otimes \GU_{n/2}(q)$ and $n$ is even
\item $H \in \S$ has socle $\PSU_3(3)$ with $n=7$ and $q=p \equiv 2 \mod{3}$ odd.
\end{enumerate}

We begin by eliminating the possibilities in (ii) and (iii). For (ii), let $n=2m$ be even, $q=3$ and $H$ have type $\GU_2(q) \otimes \GU_m(q)$. Write $t^2 = -I_2 \perp x$ and suppose $g \otimes h  = t^2$. Then $-1 = \l\mu$ for some eigenvalues $\l$ and $\mu$ of $g$ and $h$, respectively. Therefore, $\mu = -\l \in \F_{3^2}^\times$, which is a contradiction since no eigenvalue of $x$ is contained in $\F_{3^2}^\times$.  

For (iii), let $\PSU_3(3) \leq H \leq \Aut(\PSU_3(3))$. The only prime divisors of $|H|$ are $2$, $3$ and $7$, but $|t^2|$ is divisible by $r \in \ppd(q,4)$, which satisfies $r \equiv 1 \mod{4}$, so $t^2 \not\in H$. Therefore, $H \in \C_1 \cup \C_2 \cup \C_5$.

First assume that $H \in \C_1$. We will apply Shintani descent. Let $X$ be the simple algebraic group $\PGL_n(\FF_p)$ and let $\s$ be the Frobenius endomorphism $\p^f$. Notice that $t\th \in \PGU_n(q)\p^f \subseteq \PGL_n(q^2)\p^f = X_{\s^2}\s$. Moreover, $H \cap \PGU_n(q) \leq Y_{\s^2}$ for a closed connected $\s$-stable subgroup $Y$ of $X$. By Lemma~\ref{lem:shintani_descent_fix}, the $\<X_{\s^2},\ws\>$-conjugates of $Y_{\s^2}$ that are normalised by $t\th$ correspond to the $X_\s$-conjugates of $Y_\s$ that contain $t^2$. It is easy to determine the maximal reducible overgroups of $t^2$ in $\PGL_n(q)$ and these give the maximal reducible overgroups of $t\th$ in $G$ that feature in Table~\ref{tab:u_IIb_max} (see the proof of Proposition~\ref{prop:u_Ib_max} for further details).

Next assume that $H \in \C_2$. Write $H = N_G(H_0)$ where $H_0 = H \cap T$ is the stabiliser in $T$ of a direct sum decomposition $\F_{q^2}^n = U_1 \oplus \cdots \oplus U_k$ where $\dim{U_i} = n/k$ and $k >1$. Let $B$ be the index $k!$ subgroup of $H_0$ that centralises this decomposition. For now assume that $\th=\g$ or $\eta=+$. Let $m$ be $2$ if $n$ is even and $3$ if $n$ is odd. Then we may fix a suitable power $z$ of $t^2$ of type $(n-m)^-_q \perp I_m$. The order of $z$ is a primitive prime divisor $r$ of $q^{n-m}-1$. Since $r \geq n-m+1 > k$, we see that $z \in B$. However, $z = z_1 \oplus z_2 \oplus I_m$ with respect to a decomposition $\F_{q^2}^n = Z_1 \oplus Z_2 \oplus Z$, where $\dim{Z_i} = (n-m)/2$ and $\dim{Z}=m$, and $z_i$ acts irreducibly on $Z_i$. This implies that $n$ is even and $k=2$ (and $t^2$ is contained in at most $q+1$ subgroups of a given type) or $n=9$ and $k=3$ (and $t^2$ is contained in a unique such subgroup). 

We may now assume that $\th=\d_2\g$ and $\eta=-$. We proceed as in the previous case and the argument is similar. First assume that $n \equiv 0 \mod{4}$. Let $z$ be a power of $t^2$ of type $(n-2)^+_q \perp I_2$ of order $r \in \ppd(q,n/2-1)$, which is at least $2n-3 > k$ (see \cite[Lemma~6.1]{ref:BambergPenttila08}). Therefore, $z \in B$. Now $z = z_1 \oplus z_2 \oplus I_2$ with respect to a decomposition $\F_{q^2}^n = Z_1 \oplus Z_2 \oplus Z$, where $\dim{Z_i} = (n-2)/2$ and $\dim{Z}=2$, and $z_i$ acts irreducibly on $Z_i$. As above, this implies that $k=2$ and $t^2$ is contained in at most $q+1$ subgroups of a given type. Now assume that $n \equiv 2 \mod{4}$. Let $z$ be a power of $t^2$ of type $\left(\frac{n}{2}+1\right)^-_q \perp \left(\frac{n}{2}-3\right)^- \perp I_2$ of order $rs$ where $r \in \ppd(q,n/2+1)$ and $s \in \ppd(q,n/2-3)$. By \cite[Lemma~6.1]{ref:BambergPenttila08}), $r \geq n+3 > k$ and $s \geq n-5 > k$, so $z \in B$. Now $z = z_1^1 \oplus z_1^2 \oplus z_2 \oplus z_2^2 \oplus I_2$ with respect to a decomposition $\F_{q^2}^n = Z_1^1 \oplus Z_1^2 \oplus Z_2^1 \oplus Z_2^2 \oplus Z$, where $\dim{Z_1^j} = (n+2)/4$, $\dim{Z_2^j} = (n-6)/4$ and $z_i^j$ acts irreducibly on $Z_i^j$. This implies that $k=2$ and $t^2$ is contained in at most $2(q+1)$ subgroups of a given type.

Finally assume that $H \in \C_5$. We postpone the analysis of the subgroups of type $\GU_n(q^{1/k})$ for now, so we may assume that $n$ is even or $q$ is odd. If $n$ and $q$ are even, then $H$ has type $\Sp_n(q)$; if $n$ and $q$ are odd, then $H$ has type $\O_n(q)$; and if $n$ is even and $q$ is odd, then $H$ has type $\Sp_n(q)$ or $\O^\upsilon_n(q)$, where $\upsilon=+$ if $\th=\g$ or $e$ is even and $\upsilon=\eta$ if $\th=\d_2\g$. 

Write $H_0 = H \cap T$, so $|G:T|=|H:H_0|=2$. Let $c$ be the number of $G$-classes of subgroups of $G$ of type $H$. From \cite[Propositions~4.5.5 and~4.5.6]{ref:KleidmanLiebeck}, we see that 
\[
c = \frac{|\Inndiag(T):T|}{|\Inndiag(H_0):H_0|}.
\]
From the description of the conjugacy classes of elements of prime order in \cite[Chapter~3]{ref:BurnessGiudici16}, we see that $(t^2)^{T} = (t^2)^{\Inndiag(T)}$ and $(t^2)^{H_0} = (t^2)^{\Inndiag(H_0)}$, so
\[
\frac{|C_{\Inndiag(T)}(t^2):C_{T}(t^2)|}{|C_{\Inndiag(H_0)}(t^2):C_{H_0}(t^2)|} = \frac{|\Inndiag(T):T|}{|\Inndiag(H_0):H_0|}.
\]
Moreover, $(t^2)^{T} \cap H_0 = (t^2)^{H_0}$. Therefore, the number of subgroups of $G$ of type $H$ that contain $t^2$ is
\[
c \cdot \frac{|G|}{|H|}\frac{|(t^2)^G \cap H|}{|(t^2)^G|} = c \cdot \frac{|T|}{|H_0|}\frac{|(t^2)^{H_0}|}{|(t^2)^{T}|} = c \cdot \frac{|C_{T}(t^2)|}{|C_{H_0}(t^2)|} = \frac{|C_{\Inndiag(T)}(t^2)|}{|C_{\Inndiag(H_0)}(t^2)|}.
\]

It remains to assume that $H$ has type $\GU_n(q^{1/k})$ for an odd prime divisor $k$ of $f$ (recall that $q=p^f$). In order for $|t^2|$ to divide the order of $|\GU_n(q^{1/k})|$ we must have $(n,k)=(9,3)$, and arguing as in the previous case we see that $t^2$ is contained in $|C_{\Inndiag(T)}(t^2):C_{\Inndiag(H_0)}(t^2)|$ subgroups of this type.
\end{proof}

\begin{propositionx} \label{prop:u_IIb}
Let $G = \<T, \th\> \in \A$ where $T=\PSU_n(q)$. In Case~II(b), $u(G) \geq 2$ and as $q \to \infty$ we have $u(G) \to \infty$.
\end{propositionx}

\begin{proof}
Let $x \in G$ have prime order. We begin by computing the parameter $M$ that features in Table~\ref{tab:u_IIb_max}. If $n \geq 7$ is odd and $q$ is odd, then $t^2$ has type $(n-3)^- \perp J_3$ and
\[
\frac{|C_{\PGU_n(q)}(t^2)|}{|C_{\PSO_n(q)}(t^2)|} = \frac{q^2(q^{(n-3)/2}+1)(q^{(n-3)/2}-(-)^{(n-3)/2})}{q(q^{(n-3)/2}+1)} \leq q^{(n-1)/2}+q.
\]
Similarly, if $n \geq 12$ is even and either $\th=\g$ or $\th=\d_2\g$ and $\eta=+$, then $t^2$ has type $(n-2)^- \perp (2)^-$ and
\[
\frac{|C_{\PGU_n(q)}(t^2)|}{|C_{\PGSp_n(q)}(t^2)|} = \frac{|C_{\PGU_n(q)}(t^2)|}{|C_{\PDO^+_n(q)}(t^2)|} \leq q^{(n-2)/2}+1,
\]
if $\th=\d_2\g$, $\eta=-$ and $n \equiv 0 \mod{4}$, then
\[
\frac{|C_{\PGU_n(q)}(t^2)|}{|C_{\PGSp_n(q)}(t^2)|} = \frac{|C_{\PGU_n(q)}(t^2)|}{|C_{\PDO^-_n(q)}(t^2)|} \leq q^{(n-2)/2}+1,
\]
and if $\th=\d_2\g$, $\eta=-$ and $n \equiv 2 \mod{4}$, then
\[
\frac{|C_{\PGU_n(q)}(t^2)|}{|C_{\PGSp_n(q)}(t^2)|} = \frac{|C_{\PGU_n(q)}(t^2)|}{|C_{\PDO^-_n(q)}(t^2)|} \leq (q^{(n+2)/4}+1)(q^{(n-6)/4}+1).
\]
Finally, if $n=9$ and $3$ divides $f$, then $t^2$ either has type $(6)^- \perp (2)^- \perp I_1$ or $6^- \perp J_3$, but in either case
\[
\frac{|C_{\PGU_9(q)}(t^2)|}{|C_{\PGU_9(q^{1/3})}(t^2)|} \leq \frac{(q^6-1)(q+1)^2}{(q^2-1)(q^{1/3}+1)^2} \leq q^{4/3}(q^4+q^2+1).
\]

First assume that $n \geq 7$ is odd. By Proposition~\ref{prop:u_computation}, we will assume that $(n,q) \neq (7,2)$. Let $\b = 1$ if $n=9$ and $3$ divides $f$. If $q$ is even, then
\begin{align*}
P(x,t\th) < \frac{1}{q^2}&+\frac{1}{q^4}+\frac{1}{q^6}+\frac{1}{q^{n-3}}+\frac{7}{q^{n-2}}+\frac{1}{q^{n-1}}+\frac{3}{q^{n+1}} \\ &+ (\d_{n,9} + \b q^{4/3}(q^4+q^2+1))\cdot\frac{2}{q^6} < \frac{1}{2},
\end{align*}
and if $q$ is odd, then
\begin{align*}
P(x,t\th) < \frac{1}{q^2}&+\frac{1}{q^4}+\frac{1}{q^6}+\frac{1}{q^{n-3}}+\frac{2}{q^{n-2}}+\frac{4}{q^{n-4}}+\frac{2}{q^{n/2}}+\frac{1}{q^{n+1}} \\ &+ (\d_{n,9} + \b q^{4/3}(q^4+q^2+1) + q^{(n-1)/2}+q)\cdot\frac{2}{q^{n-3+2/n}} < \frac{1}{2}.
\end{align*}

Next consider $n=5$, where, by Proposition~\ref{prop:u_computation}, we will assume that $q \geq 4$. Here $t^2$ has type $4^-_q \perp I_1$, which has order $r \in \ppd(q,4)$ that satisfies $r \geq 13$. Arguing as in the proof of Proposition~\ref{prop:u_IIb_max}, via Lemma~\ref{lem:shintani_descent_fix}, $t\th$ is contained in a $\GU_1(q) \times \GU_4(q)$ subgroup and no further reducible subgroups. Inspecting \cite[Tables~8.20 and~8.21]{ref:BrayHoltRoneyDougal}, all irreducible maximal subgroups of $G$ that contain $t^2$ have type. Arguing as in the proof of Proposition~\ref{prop:u_IIb_max}, we see the the number of subgroups of type $\SO_5(q)$ that contain $t^2$ is $|C_{\PGU_5(q)}(t^2):C_{\PSO_5(q)}(t^2)| = q^2-1$. Therefore,
\[
P(x,t\th) < \frac{4}{3q} + (q^2-1)\cdot\frac{4}{3q^4} < \frac{1}{2}.
\]

Now consider $n=3$, where, by Proposition~\ref{prop:u_computation}, we will assume that $q \geq 11$. Here $t^2 = g \perp I_1$ where $|g|=(q+1)/(q+1,2) > 2$. Therefore, $t\th$ is contained in a $\GU_1(q) \times \GU_3(q)$ subgroup and no further reducible subgroups. Arguing as in the proof of Proposition~\ref{prop:u_IIb_max}, we see the the number of subgroups of type $\SO_3(q)$ that contain $t^2$ is $|C_{\PGU_3(q)}(t^2):C_{\PSO_3(q)}(t^2)| = q+1$. Since $q \geq 13$, for $k \in \{2,3\}$ we have $|g|/(|g|,k) > 2$, so $t^{2k} = [\l,\l^q,1]$ where $\l \in \F_{q^2}^\times \setminus \F_{q}^\times$. Therefore, $g$ is contained in at most one subgroup of type $\GU_1(q) \wr S_3$ and no subgroups of type $\GU_1(q^3)$. Consulting, \cite[Tables~8.20 and~8.21]{ref:BrayHoltRoneyDougal}, all remaining maximal subgroups of $G$ do not contain elements of order $|t^2|$. Therefore,
\[
P(x,t\th) < \frac{4}{3q} + \frac{q+1}{q^2-q+1} < \frac{1}{2}.
\]

For the remainder of the proof, we may assume that $n$ is even. For now assume that $n \geq 12$. If $\th=\g$, then
\begin{align*}
P(x,t\th) < \frac{4\d_{q,3}}{q^2} &+\frac{1}{q^4}+\frac{2+8\d_{3,q}}{q^{n-4}}+\frac{1+4\d_{3,q}}{q^{n-2}}+\frac{1+4\d_{3,q}}{q^{n-1}} \\ &+ (2q+2 + (2-\d_{2,p})(q^{(n-2)/2}+1))\cdot\frac{2}{q^{n-3}} < \frac{1}{2},
\end{align*}
and if $\th=\d\p^i$, then
\begin{align*}
P(x,t\th) < \frac{4}{q^2} &+ \frac{1}{q^4} + \frac{1}{q^{n-6}} + \frac{26}{q^{n-4}} + \frac{2}{q^{n-2}} + \frac{15}{q^{n-1}} + \frac{6}{q^{(n+2)/2}} \\ &+ (2q+2 + 2(q^{(n+2)/4}+1)(q^{(n-6)/4}+1))\cdot\frac{2}{q^{n-3}} < \frac{1}{2}.
\end{align*}

We now handle the remaining cases where $n$ is even. First assume that $n=10$, so $t^2$ has type $(10)^\eta$ (with the convention that $\eta=-$ if $\th=\g$). Arguing as in the proof of Proposition~\ref{prop:u_IIb_max}, $t\th$ is not contained in any reducible maximal subgroups. The order of $t^2$ is a primitive prime divisor $r$ of either $q^{10}-1$ or $q^5-1$, but in either case, by \cite[Lemma~6.1]{ref:BambergPenttila08}, $r \geq 31$. Therefore, inspecting \cite[Tables~8.62 and~8.63]{ref:BrayHoltRoneyDougal}, the only possible types of irreducible maximal subgroup of $G$ that could contain $t^2$ are those of type $\GU_2(q^5)$, $\GL_5(q^2)$, $\GU_5(q) \wr S_2$, $\Sp_{10}(q)$ and $\SO^\eta_{10}(q)$. 

The number of subgroups of types $\Sp_{10}(q)$ and $\SO^\eta_{10}(q)$ that contain $t^2$ is
\[
M = \frac{|C_{\PGU_{10}(q)}(t^2)|}{|C_{\PGSp_{10}(q)}(t^2)|} = \frac{|C_{\PGU_n(q)}(t^2)|}{|C_{\PDO^\eta_n(q)}(t^2)|} = \frac{q^5+1}{q+1}.
\]

Now let us determine the multiplicities of $\C_2$ and $\C_3$ subgroups. First assume that $H \in \C_2$. Since $t^2$ has odd order, if $t^2 \in H$ then $t^2$ centralises the decomposition $\F_{q^2}^{10} = U_1 \oplus U_2$ where $\dim{U_1} = \dim{U_2} = 5$. If $\eta=+$, then $\{ U_1,U_2 \}$ must be a dual pair of totally singular subspaces and $t^2$ centralises a unique such decomposition, so $t^2$ is contained in a unique subgroup of type $\GL_5(q^2)$ and no subgroups of type $\GU_5(q) \wr S_2$. If $\eta=-$, then $U_1$ and $U_2$ must be orthogonal nondegenerate subspaces and $t^2$ and again $t^2$ centralises a unique such decomposition, so $t^2$ is contained in a unique subgroup of type $\GU_5(q) \wr S_2$ and no subgroups of type $\GL_5(q^2)$. 

Now assume that $H$ has type $\GU_2(q^5)$. Write $H \cap T = H_0 = B.5$ and let $\pi$ be the field extension embedding. Since $r = |t^2| > 5$, we know that $t^2 \in B$. Let $b \in B$ satisfy $\pi(b) = t^2$. Write $\Lambda = \{ \l, \l^{q^2}, \l^{q^4}, \l^{q^6}, \l^{q^8} \}$ where $|\l|=r$. For now assume that $\eta=+$, so $t^2 = [\Lambda,\Lambda^{-1}]$. Then $b=[\l^{q^i},\l^{-q^i}]$ where $0 \leq i \leq 4$, so there are $5$ possibilities for $b$ up to $B$-conjugacy and consequently $1$ possibility up to $H_0$-conjugacy. Therefore, $|(t^2)^T \cap H_0| = |b^{H_0}|$. In addition, $|C_{\GU_{10}(q)}(t^2)| = (q^{10}-1) = |C_{\GU_2(q^5)}(b)|$, so $t^2$ is contained in a unique subgroup of type $\GU_2(q^5)$. If $\eta=-$, then $t^2 = [\Lambda,\Lambda^q]$, so $b=[\l^{q^i},(\l^q)^{q^j}]$ where $0 \leq i,j \leq 4$. In this case, there are $25$ possibilities for $b$ up to $B$-conjugacy and $5$ up to $H_0$-conjugacy, so arguing as before we deduce that $t^2$ is contained in $5$ subgroups of $G$ of this type.

Therefore,
\[
P(x,t\th) < \left(1+5+2 \cdot \frac{q^5+1}{q+1}\right) \cdot \frac{2}{q^7} < \frac{1}{2}.
\]

The cases $n \in \{6,8\}$ are very similar. In both cases, by Proposition~\ref{prop:u_computation} we can assume that $q \geq 4$. If $n=8$, then $t^2$ has type $(8)^-$ of order $r \in \ppd(q,8)$ satisfying $r \geq 41$ (see \cite[Lemma~6.1]{ref:BambergPenttila08}), and our usual arguments allow us to conclude that the maximal subgroups of $G$ containing $t^2$ are one of type $\GL_4(q^2)$ and $(q+1)(q^2+1)$ of types $\Sp_8(q)$ and $\SO^-_8(q)$ (where $q$ is odd in the latter case), so we obtain
\[
P(x,t\th) < (1+2(q+1)(q^2+1))\cdot\frac{2}{q^5} < \frac{1}{2}.
\]
If $n=6$, then $t^2$ has type $(6)^\eta$ (again, with the convention that $\eta=-$ if $\th=\g$) and we deduce that the maximal subgroups of $G$ that contain $t^2$ are a unique subgroup of type $\GL_3(q^2)$ if $\eta=+$ and of type $\GU_3(q) \wr S_2$ if $\eta=-$ and $q^2-q+1$ subgroups of types $\Sp_6(q)$ and $\SO^\eta_6(q)$ ($q$ odd), so
\[
P(x,t\th) \leq \frac{q^2-q+3}{q^4-q^3+q^2-q+1} < \frac{1}{2}.
\]

Finally assume that $n=4$. By Proposition~\ref{prop:u_computation}, we can assume that $q \geq 11$. Since $t^2$ has type $(4)^-$, our usual application of Lemma~\ref{lem:shintani_substitute}(ii)(b) implies that $t\th$ is not contained in any reducible maximal overgroups. The order $r$ of $t^2$ satisfies $r \in \ppd(q,4)$ and $r \geq 13$. Therefore, consulting the list of maximal subgroups of $G$ in \cite[Tables~8.10 and~8.11]{ref:BrayHoltRoneyDougal}, we see that the only types of maximal subgroup that could contain $t^2$ are $\GL_2(q^2)$, $\Sp_4(q)$ and, if $q$ is odd, $\SO^-_4(q)$. Arguing as in the previous cases, $t^2$ stabilises a unique decomposition $\F_{q^2}^4 = U \oplus U^*$ where $U$ is a maximal totally singular subspace, so $t\th$ is contained in at most one subgroup of type $\GL_2(q^2)$. Moreover, $t\th$ is contained in at most $|C_{\PGU_4(q)}(t^2):C_{\PSO^-_4(q)}(t^2)| = q-1$ subgroups of type $\SO^-_4(q)$. It remains to estimate the number $m$ of subgroups of type $\Sp_4(q)$ that contain $t\th$. There is a unique $\widetilde{G}$-class of such subgroups, so $m = \sum_{i=1}^k |C_{\widetilde{G}}(t_i\th):C_{\widetilde{H}}(t_i\th)|$ where $\widetilde{G} = \<\PGU_4(q), \g\>$ and $\widetilde{H} = N_{\widetilde{G}}(\widetilde{H}) = C_{\widetilde{G}}(\g)$, and where $(t\th)^{\widetilde{G}} \cap H = \cup_{i=1}^{k} (t_i\th)^{\widetilde{H}}$. If $g$ centralises $t\th$, then $g$ centralises the power $\g$, so $C_{\widetilde{G}}(t_i\th)=C_{\widetilde{H}}(t_i\th)$. If $t\th$ is $\widetilde{G}$-conjugate to $s\th$, then $t^2$ and $s^2$ have the same eigenvalues, so as we argued in the proof of Proposition~\ref{prop:u_Ib_max_sp4}, $t$ is $\widetilde{H}$-conjugate to $s$ if $q \not\equiv 3 \mod{4}$ and there are at most two choices for $t$ up to $\widetilde{H}$-conjugacy if $q \equiv 3 \mod{4}$. Therefore, $m = k \leq (2,q+1)$. Now using the fixed point ratio bounds in Theorem~\ref{thm:fpr_ns_u_low} we conclude that
\[
P(x,t\th) \leq \frac{q+1}{q^2-q+1} + \frac{(2,q+1)^2(q^4+1)}{q^5+q^2} < \frac{1}{2}
\]

In every case, $P(x,t\th) \to 0$ as $q \to \infty$. This completes the proof.
\end{proof}

Combining Propositions~\ref{prop:u_Ia}, \ref{prop:u_Ib}, \ref{prop:u_IIa}, \ref{prop:u_IIb} yields Theorems~\ref{thm:u_main} and~\ref{thm:u_asymptotic}.

\clearpage
\section{Linear groups} \label{s:linear}

In this final section we prove Theorem~\ref{thm:linear}, which concerns a particular family of almost simple linear groups. Let $T = \PSL_n(q)$ where $n \geq 4$ is even and $q$ is odd. We follow Section~\ref{ss:u_Ib} very closely. Let us fix some notation. \vspace{5pt}

\begin{shbox}
\begin{notationx}\label{not:linear}
\begin{enumerate}[label={}, leftmargin=0cm, itemsep=3pt]
\item Write $q=p^f$ where $f \geq 2$. Let $V = \F_q^n$.
\item Fix a basis $\B = (v_1,\dots,v_n)$ for $V$.
\item Fix the simple algebraic group $X = \PSL_n(\FF_p)$.
\item Fix the Frobenius endomorphism $\p = \p_\B$ and the standard graph automorphism $\g=\g_\B$ (see Definition~\ref{def:phi_gamma_r}).
\item Fix the antidiagonal element $\d_2 = \d^{\frac{q-1}{(q-1)_2}}$, where $\d$ is given in Definition~\ref{def:u_delta}, so $|\d_2| = (n,q-1)_2$ (see Remark~\ref{rem:u_delta_2}).
\end{enumerate}
\end{notationx}
\end{shbox} \vspace{5pt}

In light of Remark~\ref{rem:linear}, to prove Theorem~\ref{thm:linear}, we can assume that $\th = \d_2\g\p^i$ where $i$ divides $f$ and $f/i \geq 3$ is odd. \vspace{5pt}

\begin{shbox}
\notacont{\ref{not:linear}}
\begin{enumerate}[label={}, leftmargin=0cm, itemsep=3pt]
\item Write $q = q_0^e$ where $e=f/i$
\item Fix the Steinberg endomorphism $\s = \d_2\g\p^i$ and the automorphism $\r=\d_2\g$.
\item Let $Z = C_X(\r)^\circ$.
\end{enumerate}
\end{shbox} \vspace{5pt}

\begin{propositionx}\label{prop:l_z}
The automorphism $\rho$ is an involution that commutes with $\s$ and $Z_\s \cong \PDO^\eta_n(q_0)$ where $\eta = (-)^{\frac{n(q-1)}{4}+1}$.
\end{propositionx}

\begin{proof}
Since $e$ is odd, $q_0-1$ divides $q-1$ and $(q-1)/(q_0-1) = q_0^{e-1}+\dots+q_0+1$ is odd, so $|\d_2| = (n,q-1)_2 = (n,q_0-1)_2$. In particular, this implies that $\d_2^{\p^i} = \d_2$ and $\d_2^\g = \d_2^{-1}$. Therefore, $(\d_2\g)^2 = 1$ and $(\d_2\g\p^i)(\d_2\g) = \g\p^i\g = (\d_2\g)(\d_2\g\p^i)$. Finally, by Remark~\ref{rem:u_delta_2}(iv)
\[
C_X(\r)_\s = \{ x \in X \mid \text{$x^{\d_2\g\p^i} = x$ and $x^{\d_2\g} = x$} \} = C_{\PGL_n(q_0)}(\d_2\g) = \PGO^\eta_n(q_0),
\]
and $Z_\s = \PDO^\eta_n(q_0)$.
\end{proof}

\begin{propositionx}\label{prop:l_elt}
Let $T=\PSL_n(q)$ and $\th = \d_2\g\p^i$, where $n \geq 4$ is even and $f/i$ is odd. Let $y \in \PSO^\eta_n(q_0) \leq T$ be the element in Table~\ref{tab:l_elt}. Then there exists $t \in T$ that commutes with $\d_2\g$ such that $(t\th)^e$ is $X$-conjugate to $y\d_2\g$. 
\end{propositionx}

\begin{table}[b]
\centering
\caption{Linear groups: The element $y$} \label{tab:l_elt}
{\renewcommand{\arraystretch}{1.2}
\begin{tabular}{cccc} 
\hline  
\multicolumn{2}{c}{$n$}   & $q$         & $y$                                                                                             \\
\hline 
$n \leq 10$ &             &             & $(n)^\eta_{q_0}$                                                                                \\
$n \geq 12$ & $0 \mod{4}$ &             & $(n-2)^+_{q_0} \perp (2)^-_{q_0}$                                                               \\
            & $2 \mod{4}$ & $1 \mod{4}$ & $\left(\frac{n}{2}+1\right)^-_{q_0} \perp \left(\frac{n}{2}-3\right)^-_{q_0} \perp (2)^-_{q_0}$ \\
            &             & $3 \mod{4}$ & $(n-2)^-_{q_0} \perp (2)^-_{q_0}$                                                               \\
\hline
\end{tabular}}
\end{table}

\begin{proof}
From Proposition~\ref{prop:l_z}, we see that $Z = C_X(\r)^\circ$ is $\s$-stable. By Lemma~\ref{lem:shintani_substitute}(i), there exists $t \in \PDO^\eta_n(q_0) = Z_{\s^e} \leq X_{\g\s^e} = \PGL_n(q)$ such that $(t\s)^e$ is $X$-conjugate to $y\r$. Moreover, since $y \in \PSO^\eta_n(q_0)$, by Lemma~\ref{lem:o_Ia_tau} we deduce that $t \in \PSO^\eta_n(q) \leq \PSU_n(q)$.
\end{proof}

\begin{remarkx}\label{rem:l_elt_split}
This remark will help us understand how $y \in \PDO^\eta_n(q_0)$ from Table~\ref{tab:l_elt_split} acts on $V_0 = \F_{q_0^2}^n$ as an element of $\PGU_n(q_0)$. 

We begin with some preliminaries, where we use \cite[Proposition~3.3.2]{ref:BurnessGiudici16}. Let $g \in \SO^\e_{2d}(q_0)$ have type $(2d)^\e_{q_0}$ where $\e \in \{+,-\}$ and $d \geq 1$ (with $d$ odd if $\e=-$). If $\e=-$ and $d$ is even, then $|g| \in \ppd(q_0,2d)$, and if $\e=+$ and $d$ is odd, then $|g| \in \ppd(q_0,d)$; in both cases, $g$ centralises a decomposition $V_0 = U \oplus U^*$, where $\{U, U^*\}$ is a dual pair of totally singular $d$-spaces that are nonisomorphic irreducible $\F_{q_0^2}\<g\>$-modules. Now assume that $\e=-$ and $d$ is odd. Here $|g|$ is a primitive divisor of $q^{2d}-1$ and $2d \equiv 2 \mod{4}$, so $g$ centralises a decomposition $V_0 = U_1 \oplus U_2$, where $U_1$ and $U_2$ are nondegenerate $d$-spaces that are nonisomorphic irreducible $\F_{q_0^2}\<g\>$-modules. 

This allows us to obtain a decomposition of $V_0$ centralised by $y^2$, which we present in Table~\ref{tab:l_elt_split}. Let us explain our notation. For any symbol $X$, the subspaces $X^1$ and $X^2$ are equidimensional. The subspaces $W^j$ and $W_i^j$ are nondegenerate and the subspaces $U$ and $U_i$ are totally singular. In every decomposition, the summands are pairwise nonisomorphic irreducible $\F_{q_0^2}\<g\>$-modules, except when $n \geq 12$ and $q_0=3$, where $y^2$ acts as $-I_2$ on $W_0^1 \perp W_0^2$.
\end{remarkx}

\begin{table}
\centering
\caption{Linear groups: Decomposition centralised by $y^2$} \label{tab:l_elt_split}
{\renewcommand{\arraystretch}{1.2}
\begin{tabular}{cccc} 
\hline  
$n$         & $\eta$                & $y$                                                                   & conditions           \\
\hline 
$n \leq 10$ & $(-)^{\frac{n}{2}+1}$ & $U \oplus U^*$                                                        &                      \\ 
            & $(-)^{\frac{n}{2}}$   & $W^1 \perp W^2$                                                       &                      \\   
$n \geq 12$ & $(-)^{\frac{n}{2}+1}$ & $U \oplus U^* \perp W_0^1 \perp W_0^2$                                &                      \\
            & $-$                   & $W_1^1 \perp W_1^2 \perp W_2^1 \perp W_2^2 \perp W_0^1 \perp W_0^2$   & $n \equiv 2 \mod{8}$ \\ 
            & $-$                   & $(U_1 \oplus U_1^*) \perp (U_2 \oplus U_2^*) \perp W_0^1 \perp W_0^2$ & $n \equiv 6 \mod{8}$ \\
\hline
\end{tabular}}
\\[5pt]
{\small Note: $\dim{U_1} = \dim{W_1^j} = (n+2)/4$ and $\dim{U_2} = \dim{W_2^j} = (n-6)/4$, see Remark~\ref{rem:l_elt_split} }
\end{table}

\begin{propositionx} \label{prop:l_max}
Assume that $n \geq 12$. Then the maximal subgroups of $G$ that contain $t\th$ are listed in Table~\ref{tab:l_max}, where $m(H)$ is an upper bound on the multiplicity of the subgroups of type $H$ in $\M(G,t\th)$.
\end{propositionx}

\begin{table}
\centering
\caption{Linear groups: Description of $\M(G,t\th)$}\label{tab:l_max}
{\renewcommand{\arraystretch}{1.2}
\begin{tabular}{cccc}
\hline
       & type of $H$                     & $m(H)$ & conditions                                          \\
\hline
$\C_1$ & $\GL_1(q) \times \GL_{n-1}(q)$  & $2$    &                                                     \\
       & $P_{1,n-1}$                     & $2$    & $q_0=3$                                             \\
       & $\GL_2(q) \times \GL_{n-2}(q)$  & $1$    &                                                     \\
       & $P_{(n-2)/2,(n+2)/2}$           & $2$    & $\eta=(-)^{\frac{n}{2}+1}$                          \\
       & $P_{k,n-k}$                     & $6$    & $n \equiv 2 \mod{4}$, \, $\eta=-$, \, $1 < k < n/2$ \\
       & $\GL_k(q) \times \GL_{n-k}(q)$  & $6$    & $n \equiv 2 \mod{4}$, \, $\eta=-$, \, $1 < k < n/2$ \\[5.5pt]
$\C_2$ & $\GL_{n/k}(q) \wr S_k$          & $N$    & $k > 1$, \, $k \div n$                              \\[5.5pt]
$\C_3$ & $\GL_{n/2}(q^2)$                & $N$    &                                                     \\[5.5pt]
$\C_4$ & $\GL_2(q) \otimes \GL_{n/2}(q)$ & $N$    & $q_0=3$                                             \\[5.5pt]
$\C_5$ & $\GL_n(q^{1/k})$                & $N$    & $k$ prime, \, $k \div f$                            \\[5.5pt]
$\C_8$ & $\Sp_n(q)$                      & $N$    &                                                     \\
       & $\O^\upsilon_n(q)$              & $N$    & $\upsilon \in \{+,-\}$                              \\
       & $\GU_n(q^{1/2})$                & $N$    & $f$ even                                            \\
\hline
\end{tabular}}
\\[5pt]
{ \small Note: $N = |C_{\PGU_n(q_0)}(y^2)|$}
\end{table}

\begin{proof}
First assume that $H \in \C_1$. It is straightforward to determine the maximal reducible subgroups of $\PGU_n(q_0)$ that contain $y^2$ by using Remark~\ref{rem:l_elt_split} (if $n \equiv 2 \mod{4}$ and $\eta=-$, then there are several but we simply note that $y^2$ is contained in at most $6$ of any given type). Lemma~\ref{lem:shintani_substitute}(ii)(b) now implies that the $\C_1$ subgroups of $G$ that contain $t\th$ are the corresponding subgroups that that appear in Table~\ref{tab:l_max} (see the proof of Proposition~\ref{prop:u_Ib_max} for further details). 

Now assume that $H \not\in \C_1$. An $X$-conjugate of $(t\th)^{2e}$ is $(y\gamma)^2 = y^2$. If $q_0 > 3$, then a suitable power $z$ of $y^2$ has type $[\l,\l^{q_0}] \perp I_{n-2}$ where $\l \in \F_{q_0^2}^\times$ satisfies $\l \neq \l^{q_0}$, and if $q_0 = 3$, then a power $z$ of $y^2$ is $-I_2 \perp I_{n-2}$. In both cases $\nu(z)=2$ and Proposition~\ref{prop:u_max} implies that $H$ appears in Table~\ref{tab:u_Ib_max}. Since geometric maximal subgroups of $G$ of a given type are $\<\PGL_n(q),\th\>$-conjugate, Lemma~\ref{lem:shintani_substitute}(ii)(a) gives $m(H)$, noting that Proposition~\ref{prop:u_Ib_z} implies $(\r\s^e)^2 = \s^{2e}$. 
\end{proof}

\begin{proof}[Proof of Theorem~\ref{thm:linear}]
We proceed as normal, applying Lemma~\ref{lem:prob_method}. Let $x \in G$ have prime order. By \cite[Corollary~1]{ref:Burness071}, if $n \geq 8$ and $H \leq G$ is a maximal nonsubspace subgroup, then 
\[
\fpr(x,G/H) < \frac{2}{q^{n-3}},
\]
and by \cite[Theorem~2.7]{ref:BurnessGuest13}, if $H$ has type $\GL_k(q) \times \GL_{n-k}(q)$ or $P_{k,n-k}$, with $k < n/2$, then
\[
\fpr(x,G/H) \leq \left\{ 
\begin{array}{ll}
q^{-1} + q^{-(n-1)} & \text{if $k=1$}    \\
2q^{-k}             & \text{if $k > 1$.} \\
\end{array}
\right.
\]

First assume that $n \geq 12$ and $\eta = (-)^{\frac{n}{2}+1}$. From Remark~\ref{rem:l_elt_split}, we see
\[
|C_{\PGU_n(q_0)}(y^2)| = \left\{ 
\begin{array}{ll}
(3^2-1)(3^{n-1}-3)  & \text{if $q_0=3$} \\
(q_0+1)(q_0^{n-2}-1) & \text{otherwise.}  \\
\end{array}
\right.
\]
Therefore, writing $d(n)$ for the number of divisors of $n$, we have
\begin{align*}
P(x,t\th) <  (d(n)&+\log\log{q}+4) \cdot (q_0^2-1)(q_0^{n-1}-q_0) \cdot \frac{2}{q^{n-3}} \\ &+ \frac{1}{q} + \frac{1}{q^{n-1}} + \frac{2}{q^2} + \frac{4}{q^{(n-2)/2}} < \frac{1}{2}.
\end{align*}

Next assume that $\eta=-$ and $n \geq 14$ satisfies $n \equiv 2 \mod{4}$. Then
\[
|C_{\PGU_n(q_0)}(y^2)| = \left\{ 
\begin{array}{ll}
(q_0+1)(q_0^{(n+2)/4}+1)^2(q_0^{(n-6)/4}+1)^2 & \text{if $n \equiv 2 \mod{8}$} \\
(q_0+1)(q_0^{(n+2)/2}-1)(q^{(n-6)/2}-1)       & \text{if $n \equiv 6 \mod{8}$.} \\
\end{array}
\right.
\]
Therefore,
\begin{align*}
P(x,t\th) <  (d(n)+&\log\log{q}+4) \cdot (q_0+1)(q_0^{(n+2)/4}+1)^2(q_0^{(n-6)/2}+1)^2 \cdot \frac{2}{q^{n-3}} \\ &+ \frac{1}{q} + \frac{1}{q^{n-1}} + \frac{2}{q^2} + 6\sum_{k\geq 2} \frac{2}{q^k} < \frac{1}{2}.
\end{align*}

Now assume that $n \in \{6,8,10\}$. Arguing as in the proof of Proposition~\ref{prop:l_max}, via Lemma~\ref{lem:shintani_substitute}(ii)(b), $t\th$ is not contained in any reducible maximal subgroups of $G$. From \cite{ref:BrayHoltRoneyDougal} we see that $G$ has at most $13+\log\log{q}$ classes of irreducible maximal subgroups. Note that 
\[
|C_{\PGU_n(q_0)}(y^2)| \leq \frac{(q_0^{n/2}+1)^2}{q_0+1}.
\] 
Using the fixed point ratio bound from \cite[Corollary~2.9]{ref:BurnessGuest13}, we obtain
\[
P(x,t\th) < (13+\log\log{q})\cdot(q_0^{n/2}+1)\cdot\left( \frac{q-1}{(q^{n-1}-1)(q^n-1)} \right)^{1/2-1/n} < \frac{1}{2}.
\]

Finally assume that $n=4$. As in the previous cases, $t\th$ is not contained in any reducible maximal subgroups of $G$, there are at most $6+\log\log{q}$ classes of irreducible maximal subgroups and $|C_{\PGU_4(q_0)}(y^2)| = (q_0-1)(q_0^2+1)$. As with the $4$-dimensional unitary groups, the subgroups of type $\Sp_4(q)$ present a special challenge. If $H \in \M(G,t\th)$ does not have type $\Sp_4(q)$, then \cite[Corollary~2.9]{ref:BurnessGuest13} implies that
\[
\fpr(x,G/H) < \left( (q+1)(q^2+1)(q^3-1)^2 \right)^{-1/4},
\]
and of $H$ has type $\Sp_4(q)$, then \cite[Lemma~2.11]{ref:BurnessGuest13}
\[
\fpr(x,G/H) < \frac{q^2}{(2,q-1)(q^3-1)}.
\]
Therefore,
\[
P(x,t\th) < \frac{(6+\log\log{q}) \cdot (q_0-1)(q_0^2+1)}{\left((q+1)(q^2+1)(q^3-1)^2 \right)^{-1/4}} + \frac{(q_0-1)(q_0^2+1) \cdot q^2}{(2,q-1)(q^3-1)} < \frac{1}{2}.
\]

As usual, in all cases $P(x,t\th) \to 0$ as $q \to \infty$.
\end{proof}

It remains to note that Theorem~\ref{thm:main} is a combination of Theorems~\ref{thm:o_main} and~\ref{thm:u_main}, and similarly Theorem~\ref{thm:asymptotic} is a combination of Theorems~\ref{thm:o_asymptotic} and Theorem~\ref{thm:u_asymptotic}. Moreover, Theorems~\ref{thm:us_main} and~\ref{thm:us_asymptotic} follow from Theorems~\ref{thm:main} and ~\ref{thm:asymptotic}, together with the relevant results on linear groups in \cite{ref:BurnessGuest13} and Theorem~\ref{thm:linear} and the relevant results on symplectic and odd-dimensional orthogonal groups in \cite{ref:Harper17}. Theorem~\ref{thm:3/2-generation} is a corollary of Theorem~\ref{thm:us_main} (noting that $s(S_6)=2$).

\appendix

\chapter{Magma Code} \label{c:code}

In this appendix, we give the \textsc{Magma} \cite{ref:Magma} code for our computational methods. See Section~\ref{s:p_computation} for further information.

The function \texttt{FixedPointRatio}  calculates the fixed point ratio $\fpr(g,G/H)$ of an element $g \in G$ in the action of $G$ on $G/H$. It takes as input a group $G$, a subgroup $H \leq G$ and an element $g \in G$. The function returns the fixed point ratio $\fpr(g,G/H)$.

\begin{verbatim}
function FixedPointRatio( G, H, g )
  count:=0;
  classreps:=Classes(H);
  for rep in classreps do
    if (rep[1] eq Order(g)) then
      if IsConjugate(G,g,rep[3]) then
        count:=count+rep[2];
      end if;
    end if;
  end for;
  return count*Order(Centraliser(G,g))/Order(G);
end function;
\end{verbatim}

The function \texttt{MaximalOvergroups} provides information about the maximal overgroups of an element. The input is a group $G$ and an element $s \in G$. The function returns a pair of lists $[H_1,\dots,H_m]$ and $[k_1,\dots,k_m]$ where $H_i$ are pairwise non-conjugate maximal subgroups of $G$ and $k_i$ is the number of conjugates of $H_i$ which contain $s$.

\begin{verbatim}
function MaximalOvergroups( G, s )
  groups:=[];
  mults:=[];
  maxes:=MaximalSubgroups(G : OrderMultipleOf:=Order(s));
  for M in maxes do
    H:=M`subgroup;
    count:=FixedPointRatio(G,H,s)*Order(G)/Order(H);
    if (count ne 0) then
      groups:=Append(groups,H);
      mults:=Append(mults,count);   
    end if;
  end for;
  return <groups, mults>;
end function;
\end{verbatim}

The function \texttt{ClassRepTuples} is based heavily on an algorithm of Breuer \cite[Section~3.3]{ref:Breuer07}. The input is a group $G$ and a list $[x_1,\dots,x_k]$ of elements of $G$. The function returns a list of orbit representatives for the diagonal conjugation action of $G$ on $x_1^G \times \cdots \times x_k^G$.

\begin{verbatim}
function ClassRepTuples( G, list )
  cents:=[];
  for x in list do
    cents:=Append(cents,Centraliser(G,x));
  end for;
  function OrbReps(G, reps, intersect, i, cents, list )
    if (i gt #list) then
      L:=[reps];
    else
      L:=[];
      for r in DoubleCosetRepresentatives(G, cents[i], intersect) do
        L:=L cat OrbReps(G, Append(reps,list[i]^r), 
          (intersect meet cents[i]^r), i+1, cents, list );  
      end for;
    end if;
    return L;
  end function;
  return OrbReps(G,[list[1]],cents[1],2,cents,list);
end function;
\end{verbatim}

The function \texttt{RandomCheck} is a randomised algorithm that plays a role in determining the uniform spread of a group. The input is a group $G$, an element $s \in G$, a list $[x_1,\dots,x_k]$ of elements in $G$ and a nonnegative integer $N$. The claim to be tested is: for every list $[y_1,\dots,y_k]$ with $y_i \in x_i^G$, there exists $z \in s^G$ such that $\<y_1,z\> = \cdots = \<y_k,z\> = G$. If the function returns \texttt{true}, then this claim is true, and if the function returns \texttt{false}, then the result is inconclusive. The claim is tested by random selections of elements in $G$, the number of which depends on the parameter $N$.

\begin{verbatim}
function RandomCheck( G, s, list, N )
  classtuples:=ClassRepTuples(G,list);
  for X in classtuples do
    found:=false;
    for i in [1..N] do
      h:=Random(G);
      found:=true;
      for x in X do
        H:=sub<G|[x,s^h]>;
        if not (Order(H) eq Order(G)) then
          found:=false;
          break;
        end if;
      end for;
      if (found) then   
        break;
      end if;
    end for; 
    if (not found) then
      return false;
    end if;
  end for;
  return true;
end function;
\end{verbatim}

The function \texttt{ProbabilisticMethod} is our main computational tool for studying the uniform spread of a group. The input is a group $G$, an element $s \in G$ and nonnegative integers $k$ and $N$. First, the function implements the probabilistic method described in Section~\ref{s:p_prob} to determine whether $u(G) \geq k$ with respect to the class $s^G$. If successful, the function returns \texttt{true}; otherwise the second phase commences. Here \texttt{RandomCheck} is employed to verify that for all $(y_1,\dots,y_k)$ with $y_i \in x_i^G$ there exists $z \in s^G$ such that $\<y_1,z\> = \cdots = \<y_k,z\>$, for all $k$-tuples $(x_1^G,\dots,x_k^G)$ of conjugacy classes for which this was not proved in the first phase. If successful, the function returns \texttt{true}. If \texttt{false} is returned, then the result is inconclusive. A variety of helpful data from the computation is printed to the standard output.

\begin{verbatim}
function ProbabilisticMethod( G, s, k, N )
  maxandmult:=MaximalOvergroups(G,s);
  max:=maxandmult[1];
  mult:=maxandmult[2];
   
  print "-------------- \nMAXIMAL SUBGROUPS \n-------------- \n ";   
  for i in [1..#max] do
    print [Order(max[i]), mult[i]];
  end for;
  print " ";

  classes:=Classes(G);
  primeclasses:=[];
  sums:=[];
   
  print "-------------- \nCONJUGACY CLASSES \n-------------- \n ";
  for class in classes do
    if (IsPrime(class[1])) then
      primeclasses:=Append(primeclasses,class[3]);
      ratios:=[];
      for H in max do
        ratios:=Append(ratios,FixedPointRatio(G,H,class[3]));
      end for;
      sum:=0;
      for i in [1..#max] do
        sum:=sum+ratios[i]*mult[i];
      end for;
      sums:=Append(sums,sum);
      print "Order:", class[1];
      print "Size:", class[2];
      print "Fixed Point Ratios:", ratios;
      print "Sum of FPRs:", sum;
      print " \n--------------\n ";
    end if;
  end for;
   
  print "-------------- \nBAD TUPLES \n-------------- \n ";

  tuples:=[];  
  if exists{sum: sum in sums | sum ge 1/k} then
    markers:=[1 .. #sums];
    ind:=[[]];
    for i in [1 .. k] do
      newind:=[];
      for y in ind do
        for x in markers do
          if (i eq 1) or (x ge y[i-1]) then
            z:=Append(y,x);
            newind:=Append(newind,z);
          end if;
        end for;
      end for;
      ind:=newind;
    end for;
    seq:=[];
    for I in ind do
      elt:=[];
      for i in I do
        elt:=Append(elt,sums[i]);
      end for;
      seq:=Append(seq,elt);
    end for;
    for i in [1 .. #seq] do
      tot:=0; 
      for x in seq[i] do
        tot:=tot+x;
      end for;
      if tot ge 1 then
        tuples:=Append(tuples,ind[i]);
      end if;
    end for;
  end if;
  
  print "Bad Tuples:", tuples;
  print " ";
  if N gt 0 then 
    badtuples:=[];
    for tuple in tuples do
      list:=[];
      for t in tuple do
        list:=Append(list, primeclasses[t]);
      end for;
      if not RandomCheck(G,s,list,N) then
        badtuples:=Append(badtuples,tuple);
      end if;
    end for;
    print "Bad tuples remaining after", N, 
      "random checks:", badtuples;
    print " ";
  else
    badtuples:=tuples;
  end if;
  
  return (badtuples eq []);
end function;
\end{verbatim}

We sometimes want to work with groups that cannot be handled with the command \texttt{MaximalSubgroups}. In this case, we use the function \texttt{ClassicalMaximals}. For example, to obtain the maximal subgroups of $\O^+_{12}(2)$ we use 

\begin{verbatim}
ClassicalMaximals("O+", 12, 2 : general:=true);
\end{verbatim}

\backmatter
\bibliographystyle{../../documents/bib}

\begin{thebibliography}{10}

\bibitem{ref:Aschbacher84}
M.~Aschbacher, \emph{On the maximal subgroups of the finite classical groups},
  Invent. Math. \textbf{76} (1984), 469--514.

\bibitem{ref:Aschbacher00}
M.~Aschbacher, \emph{Finite Group Theory}, Cambridge Studies in Advanced
  Mathematics, vol.~10, 2nd ed., Cambridge University Press, 2000.

\bibitem{ref:AschbacherGuralnick84}
M.~Aschbacher and R.~Guralnick, \emph{Some applications of the first cohomology
  group}, J. Algebra \textbf{90} (1984), 446--460.

\bibitem{ref:AschbacherSeitz76}
M.~Aschbacher and G.~M. Seitz, \emph{Involutions in {C}hevalley groups over
  fields of even order}, Nagoya Math. J. \textbf{63} (1976), 1--91.

\bibitem{ref:BambergPenttila08}
J.~Bamberg and T.~Penttila, \emph{Overgroups of cyclic {S}ylow subgroups of
  linear groups}, Comm. Algebra \textbf{36} (2008), 2503--2543.

\bibitem{ref:Magma}
W.~Bosma, J.~Cannon and C.~Playoust, \emph{The {\textsc{magma}} algebra system
  {I}: {T}he user language}, J. Symbolic Comput. \textbf{24} (1997), 235--265.

\bibitem{ref:BrayHoltRoneyDougal09}
J.~N. Bray, D.~F. Holt and C.~M. Roney-Dougal, \emph{Certain classical groups
  are not well-defined}, J. Group Theory \textbf{12} (2009), 171--180.

\bibitem{ref:BrayHoltRoneyDougal}
J.~N. Bray, D.~F. Holt and C.~M. Roney-Dougal, \emph{The Maximal Subgroups of
  the Low-Dimensional Finite Classical Groups}, London Math. Soc. Lecture Notes
  Series, vol. 407, Cambridge University Press, 2013.

\bibitem{ref:BrennerWiegold75}
J.~L. Brenner and J.~Wiegold, \emph{Two generator groups, {I}}, Michigan Math.
  J. \textbf{22} (1975), 53--64.

\bibitem{ref:Breuer07}
T.~Breuer, \emph{{GAP} computations concerning probabilistic generation of
  finite simple groups}, {arXiv:\url{0710.3267}}, 2007.

\bibitem{ref:BreuerGuralnickKantor08}
T.~Breuer, R.~M. Guralnick and W.~M. Kantor, \emph{Probabilistic generation of
  finite simple groups, {II}}, J. Algebra \textbf{320} (2008), 443--494.

\bibitem{ref:BreuerGuralnickLucchiniMarotiNagy10}
T.~Breuer, R.~M. Guralnick, A.~Lucchini, A.~Mar{\'o}ti and G.~P. Nagy,
  \emph{Hamiltonian cycles in the generating graphs of finite groups}, Bull.
  Lond. Math. Soc. \textbf{42} (2010), 621--633.

\bibitem{ref:Brookfield14}
T.~Brookfield, \emph{Overgroups of a linear {S}inger cycle in classical
  groups}, {PhD thesis, University of Birmingham, 2014.}

\bibitem{ref:Burness071}
T.~C. Burness, \emph{Fixed point ratios in actions of finite classical groups,
  {I}}, J. Algebra \textbf{309} (2007), 69--79.

\bibitem{ref:Burness072}
T.~C. Burness, \emph{Fixed point ratios in actions of finite classical groups,
  {II}}, J. Algebra \textbf{309} (2007), 80--138.

\bibitem{ref:Burness073}
T.~C. Burness, \emph{Fixed point ratios in actions of finite classical groups,
  {III}}, J. Algebra \textbf{314} (2007), 693--748.

\bibitem{ref:Burness074}
T.~C. Burness, \emph{Fixed point ratios in actions of finite classical groups,
  {IV}}, J. Algebra \textbf{314} (2007), 749--788.

\bibitem{ref:Burness16}
T.~C. Burness, \emph{Simple groups, fixed point ratios and applications}, in
  \emph{Local Representation Theory and Simple Groups}, {EMS} Series of
  Lectures in Mathematics, European Mathematical Society, 2018, 267--322.

\bibitem{ref:Burness19}
T.~C. Burness, \emph{Simple groups, generation and probabilistic methods}, in
  \emph{Proceedings of Groups St Andrews 2017}, London Math. Soc. Lecture Note
  Series, vol. 455, Cambridge University Press, 2019, 200--229.

\bibitem{ref:BurnessGiudici16}
T.~C. Burness and M.~Giudici, \emph{Classical Groups, Derangements and Primes},
  Aust. Math. Soc. Lecture Note Series, vol.~25, Cambridge University Press,
  2016.

\bibitem{ref:BurnessGuest13}
T.~C. Burness and S.~Guest, \emph{On the uniform spread of almost simple linear
  groups}, Nagoya Math. J. \textbf{209} (2013), 35--109.

\bibitem{ref:BurnessGuralnickHarper}
T.~C. Burness, R.~M. Guralnick and S.~Harper, \emph{The
  spread of a finite group}, {s}ubmitted ({arXiv:\url{2005.01421}}).

\bibitem{ref:BurnessHarper19}
T.~C. Burness and S.~Harper, \emph{On the uniform domination number of a finite
  simple group}, Trans. Amer. Math. Soc., \textbf{372} (2019), 545--583.

\bibitem{ref:BurnessHarper}
T.~C. Burness and S.~Harper, \emph{Finite groups, $2$-generation and the
  uniform domination number}, Israel J. Math. {to appear}.

\bibitem{ref:ButurlakinGrechkoseeva07}
A.~A. Buturlakin and M.~A. Grechkoseeva, \emph{The cyclic structure of maximal
  tori of the finite classical groups}, Algebra Logic \textbf{46} (2007),
  73--89.

\bibitem{ref:CabanesSpath19}
M.~Cabanes and B.~Sp{\"{a}}th, \emph{Descent equalities and the inductive
  {M}c{K}ay condition for types {B} and {E}}, Adv. Math. \textbf{356} (2019),
  106820.

\bibitem{ref:Carter72}
R.~W. Carter, \emph{Simple Groups of Lie Type}, John Wiley and Sons, 1972.

\bibitem{ref:Chang68}
B.~Chang, \emph{The conjugate classes of {C}hevalley groups of type {$(G_2)$}},
  J. Algebra \textbf{9} (1968), 190--211.

\bibitem{ref:DigneMichel94}
F.~Digne and J.~Michel, \emph{Groupes r\'eductifs non connexes}, Ann. Sci.
  {\'E}cole Norm. Sup. (4) \textbf{27} (1994), 345--406.

\bibitem{ref:DonovenHarper}
C.~Donoven and S.~Harper, \emph{Infinite $\frac{3}{2}$-generated groups},
  Bull. Lond. Math. Soc. {to appear}.

\bibitem{ref:Enomoto69}
H.~Enomoto, \emph{The conjugacy classes of {C}hevalley groups of type {$(G_2)$}
  over finite fields of characteristic 2 or 3}, J. Fac. Sci. Univ. Tokyo Sect.
  I \textbf{16} (1969), 497--512.

\bibitem{ref:Erdem18}
F.~Erdem, \emph{On the generating graphs of symmetric groups}, J. Group Theory
  \textbf{21} (2018), 629--649.

\bibitem{ref:Fairbairn12}
B.~Fairbairn, \emph{The exact spread of ${\rm {M}}_{23}$ is $8064$}, Int. J.
  Group Theory \textbf{1} (2012), 1--2.

\bibitem{ref:FrohardtMagaard00}
D.~Frohardt and K.~Magaard, \emph{Grassmannian fixed point ratios}, Geom.
  Dedicata \textbf{82} (2000), 21--104.

\bibitem{ref:FrohardtMagaard01}
D.~Frohardt and K.~Magaard, \emph{Composition factors of monodromy groups},
  Ann. of Math. \textbf{154} (2001), 327--345.

\bibitem{ref:GorensteinLyons83}
D.~Gorenstein and R.~Lyons, \emph{The local structure of finite groups of
  characteristic 2 type} Mem. Amer. Math. Soc. \textbf{276} (1983).

\bibitem{ref:GorensteinLyonsSolomon98}
D.~Gorenstein, R.~Lyons and R.~Solomon, \emph{The Classification of the Finite
  Simple Groups, Number 3}, Mathematical Surveys and Monographs, vol.~40, Amer.
  Math. Soc., 1998.
 
\bibitem{ref:GuralnickKantor00}
R.~M. Guralnick and W.~M. Kantor, \emph{Probabilistic generation of finite
  simple groups}, J. Algebra \textbf{234} (2000), 743--792.

\bibitem{ref:GuralnickMalle12JAMS}
R.~M. Guralnick and G.~Malle, \emph{Products of conjugacy classes and fixed
  point spaces}, J. Amer. Math. Soc. \textbf{25} (2012), 77--121.

\bibitem{ref:GuralnickPenttilaPraegerSaxl97}
R.~M. Guralnick, T.~Penttila, C.~E. Praeger and J.~Saxl, \emph{Linear groups
  with orders having certain large prime divisors}, Proc. Lond. Math. Soc.
  \textbf{78} (1997), 167--214.

\bibitem{ref:GuralnickSaxl03}
R.~M. Guralnick and J.~Saxl, \emph{Generation of finite almost simple groups by
  conjugates}, J. Algebra \textbf{268} (2003), 519--571.

\bibitem{ref:GuralnickShalev03}
R.~M. Guralnick and A.~Shalev, \emph{On the spread of finite simple groups},
  Combinatorica \textbf{23} (2003), 73--87.

\bibitem{ref:Harper17}
S.~Harper, \emph{On the uniform spread of almost simple symplectic and
  orthogonal groups}, J. Algebra \textbf{490} (2017), 330--371.

\bibitem{ref:Kawanaka77}
N.~Kawanaka, \emph{On the irreducible characters of the finite unitary groups},
  J. Math. Soc. Japan \textbf{29} (1977), 425--450.

\bibitem{ref:Kessar04}
R.~Kessar, \emph{Shintani descent and perfect isometries for blocks of finite
  general linear groups}, J. Algebra \textbf{276} (2004), 493--501.

\bibitem{ref:Kleidman87}
P.~B. Kleidman, \emph{The maximal subgroups of the finite 8-dimensional
  orthogonal groups $\mathrm{P}{\Omega}^+_8(q)$ and of their automorphism
  groups}, J. Algebra \textbf{110} (1987), 173--242.

\bibitem{ref:Kleidman88G2}
P.~B. Kleidman, \emph{The maximal subgroups of the {C}hevalley groups
  {$G_2(q)$} with $q$ odd, the {R}ee groups {${}^2G_2(q)$} and their
  automorphism groups}, J. Algebra \textbf{117} (1988), 30--71.

\bibitem{ref:Kleidman883D4}
P.~B. Kleidman, \emph{The maximal subgroups of the steinberg triality groups
  ${}^3{D}_4(q)$ and of their automorphism groups}, J. Algebra \textbf{115}
  (1988), 182--199.

\bibitem{ref:KleidmanLiebeck}
P.~B. Kleidman and M.~W. Liebeck, \emph{The Subgroup Structure of the Finite
  Classical Groups}, London Math. Soc. Lecture Note Series, vol. 129, Cambridge
  University Press, 1990.

\bibitem{ref:Lang02}
S.~Lang, \emph{Algebra}, Graduate Texts in Mathematics, vol. 211, 3rd ed.,
  Springer-Verlag, 2002.

\bibitem{ref:LawtherLiebeckSeitz02}
R.~Lawther, M.~W. Liebeck and G.~M. Seitz, \emph{Fixed point ratios in actions
  of finite exceptional groups of {L}ie type}, Pacific J. Math. \textbf{205}
  (2002), 393--463.

\bibitem{ref:Liebeck13}
M.~W. Liebeck, \emph{Probabilistic and asymptotic aspects of finite simple
  groups}, in \emph{Probabilistic Group Theory, Combinatorics, and Computing},
  Lecture Notes in Math., vol. 2070, Springer, 2013, 1--34.

\bibitem{ref:LiebeckSaxl91}
M.~W. Liebeck and J.~Saxl, \emph{Minimal degrees of primitive permutation
  groups, with an application to monodromy groups of covers of {R}iemann
  surfaces}, Proc. Lond. Math. Soc. \textbf{63} (1991), 266--314.

\bibitem{ref:LiebeckShalev99}
M.~W. Liebeck and A.~Shalev, \emph{Simple groups, permutation groups, and
  probability}, J. Amer. Math. Soc. \textbf{12} (1999), 497--520.

\bibitem{ref:Piccard39}
S.~Piccard, \emph{Sur les bases du groupe sym\'etrique et du groupe alternant},
  Math. Ann. \textbf{116} (1939), 752--767.

\bibitem{ref:Shalev05}
A.~Shalev, \emph{Probabilistic group theory and {F}uchsian groups}, in
  \emph{Infinite Groups: Geometric, Combinatorial and Dynamical Aspects},
  Progr. Math., vol. 248, Birkh{\"a}user, 2005, 363--388.

\bibitem{ref:Shintani76}
T.~Shintani, \emph{Two remarks on irreducible characters of finite general
  linear groups}, J. Math. Soc. Japan \textbf{28} (1976), 396--414.

\bibitem{ref:Shoji95}
T.~Shoji, \emph{Character sheaves and almost characters of reductive groups},
  Adv. Math. \textbf{111} (1995), 244--313.

\bibitem{ref:Steinberg68}
R.~Steinberg, \emph{Endomorphisms of linear algebraic groups} Mem. Amer. Math.
  Soc. \textbf{80} (1968).

\bibitem{ref:Steinberg62}
R.~Steinberg, \emph{Generators for simple groups}, Canadian J. Math.
  \textbf{14} (1962), 277--283.

\bibitem{ref:Woldar94}
A.~J. Woldar, \emph{{$\frac{3}{2}$}-generation of the sporadic simple groups},
  Comm. Algebra \textbf{22} (1994), 675--685.

\bibitem{ref:Zsigmondy82}
K.~Zsigmondy, \emph{Zur {T}heorie der {P}otenzreste}, Monat. Math. Physik
  \textbf{3} (1892), 265--284.
  
\end{thebibliography}

\end{document}